\documentclass[11pt]{aims}

\usepackage{amsmath, amssymb, mathrsfs}
  \usepackage{paralist}
  \usepackage{graphics} %% add this and next lines if pictures should be in esp format
  \usepackage{epsfig} %For pictures: screened artwork should be set up with an 85 or 100 line screen
\usepackage{graphicx}
  \usepackage{epstopdf}%This is to transfer .eps figure to .pdf figure; please compile your paper using PDFLeTex or PDFTeXify.
 \usepackage[colorlinks=true]{hyperref}
 \hypersetup{urlcolor=blue, citecolor=red}
 \usepackage[toc,page]{appendix}
\usepackage{chngcntr}
\usepackage{tikz-cd}
\usepackage{dsfont}

\usepackage{titlesec}
\titleformat{\section}{\Large\bfseries}{\thesection.}{4pt}{}
\titleformat{\subsection}{\large\bfseries}{\thesection.\arabic{subsection}.}{4pt}{}
\titleformat{\subsubsection}{\bfseries}{\thesection.\arabic{subsection}.\arabic{subsubsection}.}{4pt}{}
\titleformat*{\paragraph}{\bfseries}
\titleformat*{\subparagraph}{\bfseries}
 \usepackage[margin=1.1in]{geometry}

  \def\currentyear{}
   \def\currentmonth{}
    \def\ppages{}
     \def\DOI{}
\def\RR{\mathbb{R}}

\newcommand{\be}{\begin{equation}}
\newcommand{\ee}{\end{equation}}

\newcommand{\Ls}{\mathscr{L}}

\newcommand{\e}{\varepsilon} 
\newcommand{\la}{\langle} 
\newcommand{\ra}{\rangle} 
\newcommand{\pa}{\partial} 
 
\newcommand{\lab}{\label} 
\newcommand{\ba}{\begin{array}}
\newcommand{\ea}{\end{array}}
\newcommand{\bee}{\begin{eqnarray*}}
\newcommand{\eee}{\end{eqnarray*}}
\newcommand{\bea}{\begin{eqnarray}}
\newcommand{\eea}{\end{eqnarray}}
\newcommand{\non}{\nonumber}
\newcommand{\lb}{\lambda}

\newcommand{\NL}{\text{NL}}

\newcommand{\mH}{\mathcal{H}}
\newcommand{\tH}{\tilde{\mathcal{H}}}
\def\fref#1{{\rm (\ref{#1})}}

\newcommand{\sfs}{\mathsf{s}}

\newcommand{\sfF}{\mathsf{F}}
\newcommand{\sfu}{\mathsf{u}}

\newcommand{\sfY}{\mathsf{Y}}

\newcommand{\sfw}{\mathsf{w}}
\newcommand{\sfR}{\mathsf{R}}
\newcommand{\sfS}{\mathsf{S}}
\newcommand{\sfM}{\mathsf{M}}
\newcommand{\sfE}{\mathsf{E}}
\newcommand{\sfK}{\mathsf{K}}
\newcommand{\sfL}{\mathsf{L}}

\newtheorem{theoremmain}{Theorem}
\newtheorem{theorem}{Theorem}[section]

\newtheorem{corollary}[theorem]{Corollary}

\newtheorem{definition}[theorem]{Definition}

\newtheorem{lemma}[theorem]{Lemma}

\newtheorem{proposition}[theorem]{Proposition}
\newtheorem{remark}[theorem]{Remark}

\numberwithin{equation}{section}

\begin{document}
\title{On singularity formation for the two dimensional unsteady Prandtl system around the axis}

\author[C. Collot]{Charles Collot}
\address{Department of Mathematics, New York University in Abu Dhabi, Saadiyat Island, P.O. Box 129188, Abu Dhabi, United Arab Emirates.}
\email{cc5786@nyu.edu}
\author[T.-E. Ghoul]{Tej-Eddine Ghoul}
\address{NYUAD Research Institute, New York University Abu Dhabi, PO Box 129188, Abu Dhabi,   United Arab Emirates.}
\email{teg6@nyu.edu}
\author[S. Ibrahim]{Slim Ibrahim}
\address{Department of Mathematics and Statistics, University of Victoria, 3800 Finnerty Road, Victoria, B.C., Canada V8P 5C2.}
\email{ibrahims@uvic.ca}
\author[N. Masmoudi]{Nader Masmoudi}
\address{NYUAD Research Institute, New York University Abu Dhabi, PO Box 129188, Abu Dhabi,   United Arab Emirates.
 Courant Institute of Mathematical Sciences, New York University, 251 Mercer Street, New York, NY 10012, USA}
\email{masmoudi@cims.nyu.edu}

\keywords{Prandtl's equations, blow-up, singularity, self-similarity, stability, analyticity, blowup rate}
\subjclass{35B44, 35Q35, 35K58, 35B40, 35B35, 35A20, 35B36}
\maketitle

\begin{abstract} 
We consider the two dimensional unsteady Prandtl system. For  a special  class  of outer Euler flows
 and solutions of the Prandtl system, the trace of the tangential derivative of the tangential velocity along the transversal axis solves a closed one dimensional equation. First, we give a precise description of singular solutions for this reduced problem. A stable blow-up pattern is found, in which the blow-up point is ejected to infinity in finite time, and the solutions form a plateau with growing length. Second, in the case where, for a general analytic solution, this trace of the derivative on the axis  follows the stable blow-up pattern, we show persistence of analyticity around the axis up to the blow-up time, and establish a universal lower bound of $(T-t)^{7/4}$ for its radius of analyticity. 
\end{abstract}

\section{Introduction}

We consider the two dimensional unsteady Prandtl boundary layer equations:
\begin{align}\label{2DPrandtl}
\left\{\begin{array}{ll}
&u_t-u_{yy}+uu_x+vu_y=-p^E_x\quad\quad (t,x,y)\in[0,T)\times\mathbb R \times\RR_+,\\
&u_x+v_y=0,\\
&u\arrowvert_{y=0}=v\arrowvert_{y=0}=0,\quad\quad u\arrowvert_{y\rightarrow\infty}=u^E,
\end{array}
 \right.
\end{align}
where $\overrightarrow{u}=(u,v)$ is the velocity field,
$u^E$ and $p^E$ are the traces at the boundary of the tangential component of the underlying inviscid velocity field and the pressure. Prandtl in \cite{Prandtl} introduced this model to describe the behaviour of a fluid close to a physical boundary for high Reynolds numbers. He obtained this model as a formal limit of the Navier-Stokes equation when the viscosity goes to zero. He proposed the appearance of a boundary layer where the viscosity is still effective, describing the solution between the boundary and the interior part where  the dynamics is inviscid. The leading order term in the expansion in the boundary layer solves \eqref{2DPrandtl}, see for example \cite{SC,SC2,M14} for more on the derivation of the system.\\

\subsection{On singularity formation for the $2$-dimensional Prandtl equations}

In this paper we are interested in the formation of a singularity in  the Prandtl system. The fact that a singularity can appear in this system is a physical phenomenon that is called the unsteady separation. Van Dommelen and Shen \cite{VanShen80} obtained the first reliable numerical result in this direction, and explained how the separation is linked to the formation of singularity. They described the singularity as being a consequence of particles squashed in the streamwise direction, with a compensating expansion in the normal direction of the boundary. We refer to \cite{Cow83,RLS,GSS,HH} and references therein for additional numerical results on the singularity formation.\\

\noindent Singularity formation is one problem out of many others regarding the  Prandtl boundary layer system. 
 The system  is locally well-posed in the analytical setting  \cite{SC,LCS,KV}, or Gevrey setting \cite{DG}. Under monotonicity assumptions, the well-posedness holds in Sobolev regularity \cite{Oleinki68,MW,AWXY} and global weak solutions  also exist  globally \cite{XZ}. Note that the solutions we consider here do not satisfy the monotonicity assumption. In this case, the equation can be ill-posed in Sobolev regularity \cite{GVD}. Similar instabilities prevent Prandtl's system from being a good approximation of the Navier-Stokes equations at high Reynolds number in certain cases \cite{GGN}. Indeed, monotonicity and/or Gevrey regularity in the tangential x-variable are necessary to insure that this approximation holds. We refer to \cite{SC,GMM} and the references therein. Finally, let us mention that the Goldstein singularity in the steady case has been recently constructed in \cite{DM}.\\

\noindent The precise description of the formation of singularity is still an open problem. However, E and Engquist \cite{EE} proved that blow-up can happen. They make some symmetry assumptions and consider a trivial inviscid flow in the outer region ($u^E=p^E=0$). In this case, the trace of the tangential derivative of the horizontal component of the velocity along the transversal axis solves a closed one dimensional equation \eqref{1DPrandtl}. They proved existence of blow-up for this reduced problem. Their approach is by contradiction and does not provide any information about the mechanism that leads to the singularity. For a more general class of non-trivial inviscid outer flows  $(u^E,p^E)$ but still with a suitable assumption of symmetry, such a  reduction remains still possible, and the corresponding one dimensional problem still admits blow-up solutions as shown in \cite{KVW}. The authors of \cite{KVW} also use a convexity argument that does not give details about the singularity.\\

\noindent In this paper, our first results are a complete description of the mechanism that leads to the singularity for the reduced one dimensional problem, including the case of nontrivial inviscid flows in the outer region. In particular, we prove the existence of a stable blow-up pattern, and other unstable ones.\\

\noindent Our approach is inspired by the description of the so-called ODE blow-up for the semi-linear heat equation, see \cite{GK,BK,HV,MZ} in particular. Note that the incompressibility condition generates difficulties through the appearance of a nonlocal nonlinear transport term. Actually, this nonlocal term will induce two new effects, the singular point is ejected to infinity in finite time, and the solution forms a plateau with a growing length. Another difficulty comes from the boundary. Indeed, the blow-up is not localized near a single point but happens on a large zone. We perform a careful treatment near the boundary to show that the solution stays bounded in its vicinity. \\

\noindent The reduced one dimensional problem \eqref{1DPrandtl} with a different domain and boundary conditions also appears in a special class of infinite energy solutions to the Navier-Stokes equations \cite{GV}. The authors proved the existence of a similar stable blow-up pattern as the one we describe here, for a particular class of solutions. Their approach is based on parabolic methods and maximum principles, allowing for a non-perturbative argument, but requires many special assumptions. In particular, their argument does not seem to apply to the problem that we consider in the present paper. In addition, our approach based on energy methods is more robust, since it allows us to prove the stability of the fundamental profile, to construct  unstable blow-ups and to derive weighted estimates.\\

\noindent One may wonder how the one dimensional reduction is related to the full two dimensional problem. From the numerics in \cite{GSS} it seems that for certain solutions with symmetries the blow-up indeed happens on the vertical axis. However, for other solutions, such as the singularity considered by Van Dommelen and Shen, the numerics show that another singularity appears before the one on the vertical axis. Our second result shows that for analytic solutions, if the solution of the reduced one dimensional problem blows up with the aforementioned stable blow-up pattern, then the solution exists up to this blow-up time in a suitable neighborhood of the vertical axis with a universal lower bound on its local analyticity radius. This justifies that the one-dimensional profile constructed in Theorem \ref{th:main} describes blowing up solutions for the two-dimensional Prandtl system \fref{2DPrandtl}.\\

\noindent In \cite{CGM} we treated a two dimensional Burgers model with transverse viscosity. This corresponds to a simplified version of the Prandtl system with a trivial flow at infinity $u^E=p^E=0$ and no vertical velocity $v=0$. A similar one dimensional reduction can be made. More interestingly we were able to prove that the one dimensional problem captures the main features of the two dimensional singularity. As a result we obtained a complete description of the mechanism that leads to singularity for the two dimensional problem.\\

\noindent In the present work, we show that the viscosity is asymptotically negligible during the singularity formation. This indicates that the full $2$-d blow-up could correspond to leading order to that of the inviscid Prandtl equations. This has been proposed for the Van Dommelen and Shen singularity in \cite{VC,ESC,CSW}. In the recent paper \cite{CGM2}, Collot, Ghoul and Masmoudi studied the self-similar blow-up profiles of the inviscid $2$-d Prandtl equations. In particular, they show that there exists  one  of the form
$$
u(t,x,y)=(T-t)^{\frac 12} \Theta\left(\frac{x}{(T-t)^{\frac 32}},\frac{y}{(T-t)^{-\frac 12}} \right)
$$
where $T$ is the blow-up time, and the profile $\Theta (X,Y)$ satisfies $\pa_X \Theta(0,Y)=-\sin^2(Y/2)\mathds 1_{0\leq Y\leq 2\pi}$. Our main result in Theorem \ref{th:main} shows that this is precisely the profile of the reduced one dimensional equation. Therefore our result can be understood as a partial stability result for the profile $\Theta$. In a forthcoming paper, we will pursue its stability analysis for the full two dimensional viscous Prandlt's system.

\subsection{A first result on the blow-up of the derivative along the vertical axis} \label{sec:blowupresult}

Without loss of generality, we consider a trivial vanishing outer flow $u^E=p^E=0$. Our result adapts straightforwardly to more general outer flows, as they just generate additional lower order terms, see comments below. Consider an initial datum $u_0(x,y)$ of the horizontal component of the velocity field for the Prandtl equation that is odd in $x$. Consequently, the corresponding solution $u(t,x,y)$ is also odd in $x$ and
$$
u(t,0,y)=u_{xx}(t,0,y)=0.
$$
This allows one to consider only the dynamic of the tangential derivative of $u$ along the $y$-axis.
To do so, we set
\begin{align}\label{tangderdef}
\xi(t,y)=-u_x(t,0,y),
\end{align}
which obeys the following equation for $y\in [0,+\infty)$:
\begin{align}\label{1DPrandtl}
\left\{\begin{array}{ll}
\xi_t-\xi_{yy}-\xi^2+\left(\int_0^y \xi \right)\xi_y=0,\\
\xi(t,0)=0, \\
 \xi(0,y)=\xi_0(y).
\end{array}
 \right.
\end{align}
The local well-posedness for the above equation is standard, see for example Proposition \ref{pr:cauchy} which adapts the result of \cite{W}. In particular, solutions for initial data in $L^1([0,+\infty))$ exist, are instantaneously regularised and there holds the following blow-up criterion. If the maximal time $T$ of existence of the solution is finite, then
\be \lab{id:blowupcriterion}
\limsup_{t\uparrow T} \| \xi (t,\cdot)\|_{L^{\infty}([0,+\infty))}=+\infty.
\ee
Our first main result is the precise description of the singularity formation for the reduced one-dimensional problem \fref{1DPrandtl}.

\begin{theoremmain}[Stable blow-up for Equation \fref{1DPrandtl}] \lab{th:main}

There exists $\lb_0^*\gg 1$ such that for all $\lb_0\geq \lb_0^*$, an $\epsilon (\lb_0)>0$ exists with the following property. For an initial datum of the form:
\be \lab{id:condtion initiale}
\xi_0(y)=\lb_0^2 \cos^2 \left(\frac{y-\lb_0 \pi}{2\lb_0} \right)\mathds 1_{0\leq y\leq 2\lb_0 \pi}+\tilde \xi_0(y), \ \ \text{with} \ \ \| \tilde \xi_0 \|_{L^1([0,+\infty))} \leq \epsilon (\lb_0),
\ee
the unique solution to \fref{1DPrandtl} blows up at some time $T>0$ \footnote{note that our proof will show $T\rightarrow 0$ as $\lb_0\rightarrow +\infty$} with:
$$
\xi(t,y)=\lb^2(t) \cos^2 \left(\frac{y-y^*(t)}{2\lb(t)\mu(t)} \right)\mathds 1_{-\pi \leq \frac{y-y^*}{\lb \mu}\leq \pi }+\tilde \xi,
$$
where, for some $\mu_{\infty}>0$:
\be \lab{th:bd para}
\lb (t)=\frac{1}{\sqrt{T-t}}+O((T-t)^{3/2}), \ \ \mu(t)=\mu_{\infty}+O((T-t)), \ \ y^*(t)=\frac{\mu_{\infty}\pi}{\sqrt{T-t}}+O((T-t)^{-1/4}),
\ee
and
\be \lab{th:bd xi}
\| \tilde \xi \|_{L^{\infty}}\leq (T-t)^{-1+\frac 18}.
\ee
Moreover, on any compact set, the solution remains uniformly regular up to time $T$, so that for any $y\in [0,+\infty)$, the limit $\lim_{t\uparrow T} \xi (t,y)=\xi^*(y)$ exists and satisfies:
\be \lab{eq:blowupprofile}
\xi^*(y)\sim \frac{y^2}{4\mu_{\infty}^2} \ \ \text{as} \ \ y\rightarrow +\infty.
\ee

\end{theoremmain}

\begin{remark}

Our analysis could be extended to show the existence of other unstable blow-up dynamics for \eqref{1DPrandtl}. We show in Proposition \ref{pr:F} that there exists a countable family of blow up profiles $(G_k)_{k\geq 1}$, with $G_1(Z)=\cos ^2 (Z/2)\mathds 1_{-\pi\leq Z\leq \pi}$. We thus mention here as an open problem to show the existence of solutions to \eqref{1DPrandtl} blowing up with a $G_k$ profile for $k\geq 2$ according to:
$$
\xi(t,y)=(T-t)^{-1} G_k \left( \frac{y-y^*(t)}{\mu_\infty (T-t)^{\frac{1}{2k}-1}} \right)\mathds 1_{-a_k \leq \frac{y-y^*(t)}{\lb \mu}\leq a_k}+l.o.t.,
$$
where $a_k>0$ is defined in Proposition \ref{pr:F}, $y^*(t)= \mu_{\infty}a_k(T-t)^{\frac{1}{2k}-1}$, and $\mu_\infty>0$. A sketch of proof is given in arXiv:1808.05967v1 version 1.

\end{remark}

Let us make the following comments on the results of Theorem \ref{th:main}.

\vspace{0.2cm}

\noindent{\it 1. On the implication for Prandtl's boundary layer}. Our result shows that the blow-up does not happen at the boundary, nor at a finite distance from it, but the singularity is ejected to infinity. This fact is rarely emphasized, but can be seen on numerical results, see \cite{GSS} for example. This suggests that the boundary layer should interact with the outer Euler flow in connection  with other high order boundary layer models like the Triple Deck model \cite{IV}, which has been proposed to
describe flow regimes where the Prandtl theory is expected to fail. 

Moreover, Prandtl's equations are derived neglecting the viscosity effects in the horizontal direction $x$. Since the $x$-derivative becomes unbounded in our result, the approximation of the Navier-Stokes equations by the Prandtl system  is not valid just before the singularity formation. \\

\noindent{\it 2. On the symmetry assumptions and the stable singularity formation}. The reduction to the one-dimensional problem \fref{1DPrandtl} breaks down in the general case without symmetry assumptions. Hence our stability result in Theorem \ref{th:main} should be understood within the symmetry class of odd solutions. Actually, the stable $2$-d singularity is expected to be a non-symmetrical one from \cite{VanShen80,VC,ESC,CSW}. In particular, the blow-up scales in the transversal $y$ direction are different from the one of Theorem \ref{th:main}, see \cite{GSS}. \\

\noindent{\it 3. On more general outer flows}. Our results could be extended to other non-trivial outer flows satisfying suitable symmetry assumptions (e.g. $u^E$ odd and $p^E$ even in $x$). Indeed, this will just induce the presence of new terms that are of lower order asymptotically during singularity formation, and will not perturb the blow-up mechanism. Hence the statement of Theorem \ref{th:main} would remain true. This is the case, for example, of the impulsively started cylinder \cite{VanShen80} $u^E=\kappa \sin x$ and $p^E=(\kappa^2/4) \cos (2x)$, for which the reduced equation \fref{1DPrandtl} becomes:
\be \label{eq:1D more gen}
\left\{\begin{array}{ll}
\xi_t-\xi_{yy}-\xi^2+\left(\int_0^y \xi \right)\xi_y=-\kappa^2,\\
\xi(t,0)=0, \ \ \xi(t,y)\underset{y\rightarrow +\infty}{\longrightarrow}-\kappa.
\end{array}
 \right.
\ee

\noindent{\it 4. Displacement thickness}. The displacement thickness $\delta^*$ is a quantity that measures the effect of the Prandtl layer on the outer Eulerian flow. It is defined as:
$$
\delta^*(t,x)=\int_0^{\infty}\left(1-\frac{u(t,x,y)}{u^E(t,x)} \right)dy,
$$
see for example \cite{SG,VanShen80}. For the aforementioned flow $u^E(t,x)=\kappa \sin x$, we have $\delta^*(t,0,\kappa)=\int_0^{\infty}(1+\frac{\xi(t,y)}{\kappa} )dy$ (using L'Hopital's rule). Kukavica, Vicol and Wang in \cite{KVW} proved the existence of blow-up solutions to \fref{eq:1D more gen}, by establishing that a quantity similar to $\delta^*(t,0)$ could blow-up in finite time. For $u^E=0$, the analogous quantity is $\int_0^\infty \xi (t,y)dy$ (which is $\lim_{\kappa \to 0} \kappa \delta^*(t,0,\kappa) $). For initial data more localised than $L^1$, we obtain that this quantity blows-up as $t\uparrow T$ and give an equivalent, see Proposition \ref{pr:stability weighted}.

\subsection{A second result on a general quantitative persistence of analyticity around the vertical axis up to the blow-up time}

In the sequel, as in Subsection \ref{sec:blowupresult}, we restrict ourselves to solutions of \fref{2DPrandtl} that are odd in $x$, with vanishing outer flow $u^E=p^E=0$ (again, this second assumption is for simplicity only). We consider higher order derivatives restricted to the vertical axis and introduce for $i\geq 0$:
\be \label{def:xii}
\xi_{i}(t,y):=\pa_x^{2i+1} u(t,0,y)
\ee
(hence $\xi=-\xi_{0}$ with this notation). They solve the following system for $i\geq 0$ and $y\in[0,+\infty)$:
\be \label{eq:evolutionxii}
\left\{
\begin{array}{l l}
\pa_t \xi_{i} =\pa_{yy}\xi_{i}-\sum_{j=0}^{i}  \binom{2i+1}{2j+1} \xi_{j} \xi_{i-j}+\sum_{j=0}^{i} \binom{2i+1}{2j} (\pa_y^{-1}\xi_{j})\pa_y \xi_{i-j}, \\
\xi_{i} (t,0)=0,\\
 \xi_{i} (0,y)=\pa_x^{2i+1}u(0,0,y).
\end{array}
\right.
\ee
Our second result describes solutions $u$ to \fref{2DPrandtl} around the axis $\{x=0\}$, combining the study of \fref{eq:evolutionxii} and an analytic extension. It shows that if $u_0$ is any initial datum to \fref{2DPrandtl} that is analytic in $x$ around the axis $\{x=0\}$ at time $t=0$, and such that $\pa_x u_{|x=0}$, defined as the solution to \eqref{1DPrandtl}, blows up at any time $T$ satisfying the properties in the conclusion of Theorem \ref{th:main}, then there is a local analytic solution up to time $T$ on a $2$ dimensional set around the vertical axis, with a radius of analyticity greater than $ (T-t)^{7/4}$. This justifies the blow-up profile of Theorem \ref{th:main} on a two-dimensional set with universal size (i.e. regardless\footnote{More precisely this set is $\{|x| \leq \tau (T-t)^{7/4}\}$, and higher order derivatives than $\pa_x u_{0|_{x=0}}$ only influence $\tau$.} of other information on $u_0$ other than $\pa_x u_{0|_{x=0}}$), establishing a bound for the blow-up rate for the analyticity radius. Moreover, this set is causal regarding the finite speed of propagation of the Prandtl equations. Other singularities of $u$ might form before time $T$, but this shows that they cannot happen too close to the vertical axis.

Given a function $\tau \in \mathcal C^{0}([0,T],(0,\infty))$, we introduce the set $E_{T,\tau}$:
\be \label{statement:id:defE}
 E_{T,\tau}:=\{(t,x,y)\in [0,T)\times \mathbb R \times [0,\infty), \ |x|\leq \tau (t) (T-t)^{7/4}\}.
\ee
Note that $\tau\geq \tau^*>0$ for some $\tau^*>0$. Writing $\langle a \rangle=\sqrt{1+a^2}$, we have

\begin{theoremmain} \label{th:analytic}

Assume $p^E=u^E=0$. Assume that $u_0:\mathbb R\times \mathbb R_+\rightarrow \mathbb R$ is odd in $x$, and analytic in $x$ on the set $\{|x|<\delta\}$ for some $\delta>0$, and satisfies the following hypotheses:
\begin{itemize}
\item[(i)] \emph{Analytic bound on the axis at initial time:} There exist $C_0,\tau_0>0$ such that for all $i\geq 0$, $\pa_x^{2i+1}u_0\in C([0,\infty))$ with for all $y\geq 0$:
\be \label{bd:analytichypothesis}
|\pa_x^{2i+1}u_0(0,y)|\leq C_0 \tau^{-2i-1}_0 (2i+1)! \langle y \rangle^{-2}.
\ee
\item[(ii)] \emph{Stable blow-up behaviour on the axis:} There exist $T,\mu,\iota,C_0'>0$ such that the solution $\xi$ to \fref{1DPrandtl} with initial datum $\xi_0(y)=-\pa_x u_{0}(0,y)$ blows up at time $T$ with:
$$
\xi (t,y)= \frac{1}{T-t} \cos^2 \left(\frac{y-\mu\pi (T-t)^{-\frac 12}}{2\mu (T-t)^{-\frac 12}} \right)\mathds 1_{-\pi \leq \frac{y-\mu\pi (T-t)^{-1/2}}{\mu (T-t)^{-1/2}}\leq \pi }+\tilde \xi (t,y)
$$
where for all $t\in [0,T)$:
\begin{align}
&\label{bd:lowerglobal} |\tilde \xi (t,y)| + (T-t)^{-\frac 12 }|\pa_y \tilde \xi (t,y)|\leq C_0' (T-t)^{-1+\iota}\langle y\sqrt{T-t} \rangle^{-2} \qquad  \quad \mbox{for } y\in [0,+\infty), \\
&\label{id:lowernearboundary} |\xi (t,y)|+|\pa_y \xi (t,y)|\leq C_0' \qquad \qquad \qquad   \quad \qquad  \qquad \qquad  \qquad \qquad  \ \mbox{for } y\in [0,1/2].
\end{align}
\end{itemize}

Then, there exists $\tau \in C^0([0,T], (0,+\infty))$ and a function $u\in C(E_{T,\tau})$ (as defined by \eqref{statement:id:defE}) with $u\in C^\infty(E_{T,\tau}\cap \{t>0\})$ such that:
\begin{itemize}
\item[(i)] $u$ is a classical solution to \fref{2DPrandtl} on $E_{T,\tau}\cap \{t>0\}$ and $u=u_0$ on $E_{t,\tau}\cap \{t=0\}$.
\item[(ii)] There exist $C_1,\tau_1>0$ such that for all $t\in [0,T)$ and $y\geq 0$:
\begin{align} \label{bd:analytiresult}
& |\pa_x^{3} u(t,0,y)|\leq C_1 (T-t)^{-4}  \\
&  |\pa_x^{2i+1} u(t,0,y)| \leq C_1 (T-t)^{-\frac 72i-\frac 18}\tau^{-2i-1}_1 (2i+1)! \qquad \qquad \mbox{for }i\geq 2.
\label{bd:analytiresult2}
\end{align}
\item[(iii)] the set $E_{T,\tau}$ is causal in the sense that at its boundary:
\be \label{id:causality}
|u_{|_{\left\{x=\pm \tau(t)(T-t)^{7/4} \right\}}}|<|\frac{d}{dt}( \tau(t) (T-t)^{7/4})|.
\ee
\end{itemize}

\end{theoremmain}

\noindent{\it 1. Uniqueness}. Assume that $u_0$ is everywhere $x$-analytic, with $|\pa_x^{i} u_0(x,y)|\leq \bar Ci!\bar \tau^{-i} \langle y\rangle^{-2}$ for all $(x,y)\in \mathbb R\times \mathbb R_+$, for some $\bar C,\bar \tau>0$. In this case, there exists $T_0>0$ and an everywhere $x$-analytic solution $\bar u$ to \fref{2DPrandtl} on $[0,T_0]\times \mathbb R\times \mathbb R_+$, as proved in \cite{KV}. Then the solution $u$ of Theorem \ref{th:analytic} coincides with $\bar u$ as long as it is defined, i.e. $u=\bar u$ on $E_{T,\tau}\cap \{ t\leq T_0\}$. This is because both solutions can be obtained by the same Picard iteration scheme.\\

\noindent{\it 2. On the assumptions}. Note that there are no conditions imposed on the parameters $T_0$, $T$, $\mu$, $\iota$, $C_0$, $C'_0$ and $\tau_0$. Thus $\xi(t=0)$ can, at the initial time, be away from the blow-up regime, in the sense that both $T$ and $\tilde \xi(t=0)$ can be arbitrarily large. The existence of solutions satisfying (i) is obtained as an easy extension of the proof of Theorem \ref{th:main}. We shall prove that for initial data in the space $\mathcal B$ with norm $\| f\|_{\mathcal B}=\sup_{y\geq 0} (|f(y)|+|\pa_y f(y)|)\langle y \rangle^{2} $:

\begin{proposition} \lab{pr:stability weighted}
There exists an open set in $\mathcal B$ of initial data $\xi_0$ such that the solution $\xi$ to \fref{1DPrandtl} satisfies the assumption (i) of Theorem \ref{th:analytic}. Moreover, there holds: $\int_0^\infty \xi(t,y)dy\sim (T-t)^{-3/2}\mu \pi$ as $t\uparrow T$.
\end{proposition}

\noindent{\it 3. Optimality of the lower bound}. We believe that the exponent $7/4$ is optimal. This value comes from optimal bounds for the linearised dynamics induced by assumption (i), and from certain nonlinear bounds for what we identify as the worst terms, which we believe are optimal, see the formal computation in Subsubsection \ref{subsubsec:analytic}. This value was critical for the analysis and reaching it required a delicate treatment.\\

\noindent{\it 4. Causality}. Prandtl's equations have finite speed of propagation along the tangential direction, see for example \cite{KMVW}. The inequality \fref{id:causality} states that at the boundary of $E_{T,\tau}$, the vector field $\pa_t+u\pa_x$ points outward.

\subsection{Strategy of the proof and organisation of the paper}

The proof of Theorems \ref{th:main} relies on a perturbative bootstrap argument around the blow-up profile. The maximum of the solution is the most sensitive location, where the viscosity effects are non negligible at the parabolic scale. There, the dynamic is given by an elliptic operator with compact resolvent \fref{eq:def L} in a suitable weighted space, as in \cite{GK,BK,HV,MZ}. A decomposition of the solution onto the eigenmodes allows to derive modulation equations for the parameters and decay for the remainder due to a spectral gap. In the midrange zone, away from the maximum but still on the support of the blow-up profile, the viscosity is negligible and we face a singularly perturbed problem \fref{eq:u}. We use a new Lyapunov functional with an adapted weight and take derivatives with a suitable vector field, which are the main technical novelties of the present paper. Finally, the solution is studied near the boundary via a no blow-up argument inspired from \cite{GK,HV2,MZ2}.

The proof of Theorem \ref{th:analytic} relies on the study of all $x$ and $y$ derivatives $\xi_{i,k}=\pa_y^k\xi_i$. Analyticity in $y$ is first obtained by a parabolic regularisation argument. Then, linear bounds for the dynamics of $\xi_{i,k}=\pa_y^k\xi_i$ are showed, using maximum principle and an explicit treatment of a non-local term. Then, a suitable analytic norm based on a weighted $L^{\infty}$ space is defined. It is controlled using a bootstrap type argument. The analytic norm  controls the nonlinear effects, including what we think are the worst ones, for which the $y$-derivatives act as forcing terms for the $x$-derivatives. To control boundary terms at $y=0$, we rely, classically, on the fact that controlling $t$-derivatives allows to control the $y$-ones for parabolic equations. Implementing this argument is delicate around the blow-up time $T$, and we use the fact that we are away from the blow-up zone $y\sim \mu \pi (T-t)^{-1/2}$ to obtain smallness in certain terms. \\

\noindent The paper is organized as follows. In Section \ref{sec:formal}, we give a heuristic argument for the derivation of the blow-up profiles and some of their properties in Proposition \ref{pr:F}. Section \ref{sec:main} is devoted to the proof Theorem \ref{th:main}. A bootstrap argument is described in Subsection \ref{subsec:bootstrap} and Proposition \ref{pr:bootstrap} states the stability result in renormalised variables. The analysis near the maximum is in Subsection \ref{subsec:max}, the modulation equations and the interior Lyapunov functional are established in Lemmas \ref{lem:modulation} and \ref{lem:dseL2rho}. The midrange zone $y\sim (T-t)^{-1/2}$ is analyzed in Subsection \ref{subsec:Z}, the exterior Lyapunov functionals are established in Lemmas \ref{lem:exteleft1} and \ref{lem:exteleft2}. The solution is studied on compact sets in the original variable in Lemma \ref{lem:noblowup}. The main Proposition \ref{pr:bootstrap} is proved in Subsection \ref{subsec:prbootstrap}, allowing to prove Theorem \ref{th:main} in the same subsection, and Proposition \ref{pr:stability weighted} in Subsection \ref{subsec:propadd}.

Theorem \ref{th:analytic} is proved in Section \ref{sec:analytic}. Linear bounds are first established in Proposition \ref{pr:linearanalytic}. Then the third order derivative and higher order derivatives for the full problem are bounded in Propositions \ref{pr:paxxx} and \ref{pr:analyticbootstrap} respectively, yielding the proof of Theorem \ref{th:analytic} in Subsection \ref{subsubsec:analytic}. The proof of Theorem \ref{th:analytic} uses that solutions to \fref{eq:evolutionxii} become instantaneously analytic in $y$, what is proved in Section \ref{sec:regularisation}.

\subsection*{Acknowledgements} The authors would like to thank the anonymous referees for their comments that helped improve the presentation and results of this paper. C. Collot is supported by the ERC-2014-CoG 646650 SingWave, S. Ibrahim is partially supported by NSERC Discovery grant \# 371637-2019 and N. Masmoudi is supported by NSF grant DMS-1716466, and by Tamkeen under the NYU Abu Dhabi Research Institute grant of the center SITE. Part of this work was done while C. Collot, T.-E. Ghoul and N. Masmoudi were visiting IH\'ES and they thank the institution. S. Ibrahim is grateful to New York University in Abu Dhabi for hosting him.  The authors thank V. T. Nguyen for helpful comments.

%%%%%%%%%%%%%%%%%%%%%%%%%%%%%%%%%%%%%%%%%%%%%%%%%%%%%%%%%%%%%%%%%%%%%
%%%%%%%%%%%%%%%%%%%%%%%%%%%%%%%%%%%%%%%%%%%%%%%%%%%%%%%%%%%%%%%%%%%%%
%%%%%%%%%%%%%%%%%%%%%%%%%%%%%%%%%%%%%%%%%%%%%%%%%%%%%%%%%%%%%%%%%%%%%
%%%%%%%%%%%%%%%%%%%%%%%%%%%%%%%%%%%%%%%%%%%%%%%%%%%%%%%%%%%%%%%%%%%%%

\section{Notation}

Let the measure
\be
\label{measure}
\rho(Y)=\frac 12 \sqrt{\frac 3 \pi}e^{-\frac{3Y^2}{4}}.
\ee
For a function $h$ defined on some half line $[Y_0,+\infty)$, we will write with an abuse of notation:
\be \lab{def:L2rho}
\| h \|_{L^2_\rho}^2=\int_{Y_0}^{+\infty} h^2(Y)\rho (Y)dY, \ \ \| h \|_{H^1_\rho}^2=\int_{Y_0}^{+\infty} (h^2(Y)+|\pa_Y h(Y)|^2)\rho(Y)dY,
\ee
and the value of $Y_0$ (being the image of the boundary $y=0$ in \fref{1DPrandtl} in the original variables $y$ by a change of variable), will always be clear from the context. We denote the primitive of a function integrated from the origin by
$$
\pa_y^{-1}h (y) =\int_0^y h(\tilde y)d\tilde y, \ \ \pa_Y^{-1}h(Y)=\int_0^Yh (\tilde Y)d\tilde Y, \ \ \pa_Z^{-1}h(Z)=\int_0^Zh(\tilde Z)d\tilde Z,
$$
the integration being with respect to the variables $y$, $Y$ or $Z$ to be defined later on. Note that the origin will not be preserved by the change of variables: $y=0$ does not correspond to $Y=0$ and the integrals do not start from the same point. Recall the Hermite polynomials:
\be \lab{eq:defhi}
h_0=1, \ \ h_1=\sqrt 3 Y, \ \ h_2=3Y^2-2.
\ee
The heat kernel will be denoted by:
\be \lab{eq:defheatkernel}
K_t (x)=\frac{1}{(4\pi t)^{\frac 12}}e^{-\frac{x^2}{4t}}.
\ee
We write $A\leq CB$ if $A,B\geq 0$ with a positive constant $C$ that is independent of all other parameters at stake in the analysis, we call such a constant "universal". The value of such a constant $C$ then may vary from one line to another.
%    and if the constant $C$ is independent of the other parameters, or which is independent of the initial renormalized time $s_0$, and its value will change from one line to another. 
We also write $A\lesssim B$ if $A\leq C B$, and $O(B)$ means a quantity that is $\lesssim B$. We write $C(K)$ for example to precise that the constant depends only on some parameter $K$. Finally,  $A\approx B$ if $A\lesssim B$ and $B\lesssim A$.

\section{Formal analysis and blow-up profiles} \lab{sec:formal}

In this section we formally derive the blow-up profile for \fref{1DPrandtl}. This approach relying on matched asymptotics is inspired by \cite{VGH,FK,BK,MZ,HV,GV}. Let us first perform a formal computation for the effect of the viscosity near the maximum of the solution, and for the obtention of the suitable self-similar variables. Assume that the solution to \fref{1DPrandtl} blows up at time $T$, with its maximum at a point $y^*(t)$, and that the speed of this point is given by the transport part of the equation: $y^*_t=\pa_y^{-1} \xi (y^*)$. We then use parabolic self-similar variables:
$$
Y=\frac{y-y^*}{\sqrt{T-t}}, \ \ s=-\log(T-t), \ \ f(s,Y)=(T-t)\xi(t,y)
$$
and find that $f$ solves, assuming that one can neglect the boundary condition,
$$
f_s+f+\frac Y2 \pa_Y f -f^2+\pa_Y^{-1}f \pa_Y f-\pa_{YY}f=0.
$$
An obvious solution of the above equation is the constant in space-time solution $f=1$, which corresponds to $\xi=1/(T-t)$ in the original variables (which solves \fref{1DPrandtl} but does not satisfy the boundary condition). Assuming that $1$ is a good approximation of the solution for some large zone in the variable $Y$, we compute the evolution of the correction $\varepsilon=f-1$:
\be \lab{eq:def L}
\varepsilon_s+\Ls \varepsilon=NL, \ \ \Ls\varepsilon:=-\varepsilon+\frac 32 Y\pa_Y\varepsilon-\varepsilon_{YY}, \ \ NL=\varepsilon^2-\pa_Y^{-1}\e \pa_Y \e.
\ee
The linearised operator $\Ls$ is well known.

\begin{proposition} \lab{pr:Ls}

The operator $\Ls:H^2_\rho \rightarrow L^2_\rho$ is essentially self-adjoint with compact resolvent. Its spectrum is $\{-1+3i/2, \ i=0,1,2,... \}$, with associated eigenfunctions
$$
h_{i}(Y)=H_{i} \left(\sqrt{3}Y \right)=\sum_{j=0}^{\left[\frac i2 \right]} \frac{i!}{j!(i-2j)!}3^{\frac{i-2j}{2}}(-1)^{j} Y^{i-2j}
$$
where $H_{j}$ is a Hermite polynomial.

\end{proposition}

\begin{proof}

Changing variables and setting $u(Y)=w(z)$, $z=\sqrt 3 Y$ gives $ \Ls u=-3(\hat {\Ls } w) (z)$ where $\hat{\Ls}:=\pa_{zz}-z \pa_z+1/3$ and the result follows from the corresponding result on $\tilde{\Ls}$ whose eigenbasis consists of Hermite polynomials, see \cite{MZ}.

\end{proof}

From Proposition \ref{pr:Ls} one sees that the linearised dynamics possesses one unstable direction, and an infinite number of stable modes. The unstable direction corresponds to the constant in space mode $1$, and is related to a symmetry of the equation: the invariance by time translation. One can assume that the blow-up time has been chosen well, so that this mode is not excited. Neglecting the nonlinear effects, one can assume from Proposition \ref{pr:Ls} that one mode dominates:
$$
\varepsilon (s, Y)\approx C e^{(1-\frac 32 i)s}h_i(Y), \ \ i\geq 1.
$$
From the behaviour at infinity of the polynomials $h_i$, the fact that $1+\varepsilon$ is maximal near the origin implies that $C=-c<0$ and that $i=2k$ is an even positive integer (the modes associated to odd integers are related to another symmetry of the equation: the invariance by space translation). Therefore, $\varepsilon (s,Y)\approx -ce^{(1-3 k)s}h_{2k}(Y)\approx -ce^{(1-3 k)s}Y^{2k}$ for $Y$ large. The correction $\varepsilon$ then starts to be of the same size as the leading order term $1$ in the zone
$$
|Y|\sim e^{(\frac 32 -\frac{1}{2k})s}, \ \ \text{i.e.} \ \ y-y^*\sim (T-t)^{-1+\frac{1}{2k}}.
$$
This suggests to introduce the new variables:
$$
Z:=\frac{Y}{e^{\left(\frac 32 -\frac{1}{2k} \right)s}}=(T-t)^{1-\frac{1}{2k}}(y-y^*), \ \ F(s,Z):=f(s,Y)
$$
and $F$ solves
$$
F_s+F-F^2+\left(-\left(1-\frac{1}{2k}\right)Z+\int_0^Z F(s,\tilde Z)d\tilde Z\right)\pa_Z F-e^{-\left(3-\frac 1k \right)s}\pa_{ZZ}F=0.
$$
Assuming that $F$ is the correct rescaled unknown, the viscosity is asymptotically negligible and $F$ should converge to a stationary solution of the self-similar inviscid equation, we obtain
\be \lab{eq:F1}
F-F^2+\left(-\left(1-\frac{1}{2k}\right)Z+\int_0^Z F(\tilde Z)d\tilde Z\right)\frac{d}{dZ} F=0.
\ee
In other words, in the renormalised variables, $F$ should tend to a self-similar solution of \fref{1DPrandtl} without viscosity and boundary which is:
\be \lab{eq:inviscidprandtl1d}
\psi_t-\psi^2+\left(\int_{-\infty}^y\psi \right) \pa_y \psi=0.
\ee
This equation admits a four-parameter group of symmetries: invariance by space and time translation and a two-parameter scaling group. Namely, if $\psi(t,x)$ is a solution then so is
$$
\frac{1}{\lambda} \psi \left(\frac{t-t_0}{\lambda},\frac{y-y_0}{\mu} \right), \ \ (t_0,y_0,\mu,\lambda)\in \mathbb R^2\times (0,+\infty)^2.
$$
Note that this contains the action of scaling subgroups of the form $\lambda^{2k/(2k-1)}\psi (\lambda^{2k/(2k-1)}t,y/\lambda)$ for $k\geq 0$. The following proposition describes the solutions to Equation \fref{eq:F1}, and is essentially taken from \cite{GV}.

\begin{proposition} \lab{pr:F}

Let $k\in \mathbb N$. Equation \fref{eq:F1} admits a one-parameter family of solutions
\be \lab{id:Fk rescaled}
G_k\left( \frac Z \mu \right), \ \ \mu>0.
\ee
For $k\geq 2$, $G_k$ is even, compactly supported on $[-a_k,a_k]$ with $a_k=\pi/(2k\sin(\pi/2k))$, positive and increasing on $(-a_k,0)$, of class $C^{1+1/(2k-1)-\epsilon}$ on $\mathbb R$, and satisfies the asymptotic expansions
$$
G_k(Z)\sim (2k-1)^{1+\frac{1}{2k-1}} (Z+a_k)^{1+\frac{1}{2k-1}} \ \ \text{as} \ Z\rightarrow -a_k, \ \ G_k(Z)=1-Z^{2k}+O(Z^{4k}) \ \ \text{as} \ \ Z\rightarrow 0.
$$
For $k=1$ one has the explicit formula, with a different scaling than $k\geq 2$ to ease notation:
\be \lab{id:F1}
G_1(Z)=\cos ^2 \left(\frac{Z}{2} \right)\mathds 1_{-\pi\leq Z\leq \pi}.
\ee

\end{proposition}

\begin{remark} \lab{re:F}

As this will be clear from the proof of Proposition \ref{pr:F} provided in below, we have $\int_0^{a_k}F(Z)dZ=(1-1/(2k))a_k$. Using this fact, one sees that equation \fref{eq:F1} admits other solutions of the form $G_k((Z-\mu a_k)/\mu)$. It also admits the trivial solutions $0$ and $1$. We claim that all other bounded solutions of \fref{eq:F1} can be obtained by gluing a finite or an infinite number of these solutions, when they attain $1$ or $0$. For example, the function:
$$
F(Z)=\left\{\begin{array}{l l} 1 \ \ \text{for} \ Z\leq 0, \\
G_k(Z) \ \ \text{for} \ 0\leq Z \leq a_k,\\
G_k\left( \frac{Z-\mu a_k-a_k}{\mu}\right), \ \ \text{for} \ a_k\leq Z  \end{array} \right.
$$
is also a solution with the same regularity.

The solutions $G_k$ of \fref{eq:F1} are also well defined for $k>0$ and $k\notin \mathbb N$. There is then a continuum of blow-up profiles for equation \fref{eq:inviscidprandtl1d}, but we expect that adding viscosity would prevent the appearance of non-smooth blow-up profiles.

\end{remark}

\begin{proof}
%The formula for $k=1$ is a direct computation. 
We perform the change of variables on $[0,+\infty)$:
$$
\frac{d\xi}{dZ}=\frac{\xi}{-\left(1-\frac{1}{2k}\right)Z+\int_0^Z G(\tilde Z)d\tilde Z}, \ \ H(\xi):=G(Z),
$$
so that equation \fref{eq:F1} becomes
$$
H-H^2+\xi\pa_\xi H=0
$$
whose solution is $H=(1+\xi)^{-1}$ (renormalizing the constant of integration). 
Notice that the function 1 is the only constant solution to \eqref{eq:F1}, and that for non-constant solutions, there should be a non-empty neighborhood such that  $-\left(1-\frac{1}{2k}\right)Z+\int_0^Z G(\tilde Z)d\tilde Z\neq0$. This justifies the above change of variables.
Unwinding the transformation one finds
$$
\frac{dZ}{d\xi} =\frac{1}{\xi} \left[-\left(1-\frac{1}{2k}\right)Z+\int_0^Z G(\tilde Z)d\tilde Z \right]
$$
which gives
$$
\frac{d^2 Z}{d\xi^2} =-\frac 1 \xi \frac{dZ}{d\xi}-\left(1-\frac{1}{2k}\right)\frac 1 \xi \frac{dZ}{d\xi}+\frac 1 \xi \frac{dZ}{d\xi} F(Z)=\frac{dZ}{d\xi}\left[-\left( 2-\frac{1}{2k} \right)\frac{1}{\xi}+\frac{1}{\xi+\xi^2} \right]
$$
and hence
$$
\frac{d}{dZ} \left(\log \frac{dZ}{d\xi}\right)=-\left( 2-\frac{1}{2k} \right)\frac{1}{\xi}+\frac{1}{\xi+\xi^2}.
$$
An integration yields
$$
\log \frac{dZ}{d\xi}=C+\log \left(\xi^{-\left(2-\frac{1}{2k}\right)}\right)+\log \xi -\log (\xi+1)
$$
with an integration constant $C$. Because of the invariance of the equations by scaling, we can without loss of generality, consider
%From this one deduces that, taking $C$ to get a constant equal %to $1$ in the expression below
$$
\frac{dZ}{d\xi} =\frac{\xi^{-\left(1-\frac{1}{2k}\right)}}{1+\xi}, \ \ Z(0)=0.
$$
Since $Z(0)=0$, one deduces that
$$
\lim_{\xi \rightarrow +\infty} Z(\xi)=\int_0^{\xi} \frac{\tau^{-\left(1-\frac{1}{2k}\right)}}{1+\tau}d\tau =\frac{\pi}{\sin \left(\frac{\pi}{2k}\right)}.
$$
and that as $\xi \rightarrow 0$,
$$
Z=2k\xi^{\frac{1}{2k}}(1+O(\xi))
$$
and that as $\xi \rightarrow +\infty$:
$$
Z=\frac{\pi}{\sin \left(\frac{\pi}{2k}\right)}-\frac{\xi^{-1+\frac{1}{2k}}}{1-\frac{1}{2k}}(1+O(\xi^{-1})).
$$
Near the origin $Z\sim 0$, this yields
$$
\xi =\left(\frac{Z}{2k}\right)^{2k}(1+O(Z^{2k}))
$$
and as $Z\to \frac\pi{\sin(\frac\pi{2k})}$:
$$
\xi=\left(1-\frac{1}{2k}\right)^{-\frac{2k}{2k-1}}\left(\frac{\pi}{\sin \left(\frac{\pi}{2k}\right)}-Z \right)^{-\frac{2k}{2k-1}}\left(1+O\left(\frac{\pi}{\sin \left(\frac{\pi}{2k}\right)}-Z \right)^{\frac{2k}{2k-1}}\right).
$$
Therefore near the origin $Z\sim 0$, $G(Z)=1-(2k)^{-2k}Z^{2k}+O(Z^{4k})$ and near $Z\sim \pi/\sin(\pi/(2k))$, we have
$$
G(Z)=\left(1-\frac{1}{2k}\right)^{\frac{2k}{2k-1}}\left(\frac{\pi}{\sin \left(\frac{\pi}{2k}\right)}-Z\right)^{\frac{2k}{2k-1}}(1+O((a_k-Z)^{\frac{2k}{2k-1}})). 
$$
For $k\geq 2$, we finally define $G_k(Z)=G(2kZ)$ where $G$ is defined above. $G_k$ also solves \fref{eq:F1} by scaling invariance, its support is $[-a_k,a_k]$ for $a_k=\pi/(2k\sin(\pi/(2k)))$ and it satisfies the desired asymptotic behaviour near $-a_k$ and $0$.
The computation in the case $k=1$ is more explicit, and gives $Z=2\tan^{-1}\sqrt\xi$, that is $Z=\frac1{\tan^2(Z/2)+1}=\cos^2(Z/2)$. Therefore the result follows.

\end{proof}

From Proposition \ref{pr:F} and Remark \ref{re:F}, equation \fref{eq:inviscidprandtl1d} then admits a family of backward self-similar profiles for $k\in \mathbb N$ which are smooth on their support:
$$
\psi(t,y)=\frac{1}{T-t} G_k\left((y-y^*(t))\frac{(T-t)^{1-\frac{1}{2k}}}{\mu} \right), \ \ y^*(t)=\frac{\mu a_k}{(T-t)^{1-\frac{1}{2k}}}+y_0^*, \ \ \mu>0.
$$
They blow up in finite time and their support, which is $y\in [y_0^*,y_0^*+2a_k/(\mu(T-t)^{1-1/2k})]$, is growing to infinity. The formal analysis we just performed indicates that they could be at the heart of the blow-up phenomenon.

%%%%%%%%%%%%%%%%%%%%%%%%%%%%%%%%%%%%%%%%%%%%%%%%%%%%%%%%%%%%%%%%%
%%%%%%%%%%%%%%%%%%%%%%%%%%%%%%%%%%%%%%%%%%%%%%%%%%%%%%%%%%%%%%%%%
%%%%%%%%%%%%%%%%%%%%%%%%%%%%%%%%%%%%%%%%%%%%%%%%%%%%%%%%%%%%%%%%%
%%%%%%%%%%%%%%%%%%%%%%%%%%%%%%%%%%%%%%%%%%%%%%%%%%%%%%%%%%%%%%%%%

\section{Equation on the axis} \lab{sec:main}

In this section we aim at proving Theorem \ref{th:main}. First, let us give the following local well-posedness result which is an adaptation of \cite{W}. Note that if $\xi$ solves \fref{1DPrandtl}, then $\lb^2\xi(\lb^2t,\lb y)$ is also a solution. The scaling transformation $h\mapsto \lb^2h(\lb y)$ is an isometry on $L^{1/2}([0,+\infty))$ and \fref{1DPrandtl} is then said to be $L^{1/2}$-critical.

\begin{proposition}[Local well-posedness] \lab{pr:cauchy}

Let $\xi_0\in L^1([0,+\infty))$. There exists $T(\| \xi_0\|_{L^1})>0$ and a unique solution of the Duhamel formulation of \fref{1DPrandtl} such that $\xi \in C([0,T],L^1([0,+\infty)))$, $\xi(0,\cdot)=\xi_0(\cdot)$ and\footnote{with a multiplicative constant that depends on $\| \xi_0\|_{L^1([0,+\infty)}$} $\| \pa_y \xi(t) \|_{L^1}\lesssim t^{-1/2}$. Moreover, $\xi\in C^{\infty}((0,T]\times [0,+\infty))$ and for each $k\in \mathbb N$, $\pa_y^k\xi \in C((0,T],L^1([0,+\infty)))$. For any $k\in \mathbb N$ and $0<T_1\leq T$, the solution map is locally uniformly continuous from $L^1$ into $C([T_1,T],W^{k,1}[0,+\infty))$.

\end{proposition}

Solutions associated to initial data of the form \fref{id:condtion initiale} are thus well-defined and we now turn to the proof of Theorem \ref{th:main}. We will use sometimes alternative formula for the profile:
\bea
\lab{id:G1} G_1(Z)&=&\cos^2 \left(\frac Z2\right)\mathds 1_{-\pi\leq Z \leq \pi}=\left(\frac 12 +\frac 12 \cos (Z)\right)\mathds 1_{-\pi\leq Z \leq \pi}\\
\non &=& 1-\frac{Z^2}{4}+\frac{Z^4}{48}+O(|Z|^6) \ \ \text{as} \ Z \rightarrow 0.
\eea
The proof of Theorem \ref{th:main} relies on a bootstrap argument performed near the blow-up profile. First we explain how to suitably decompose a solution near the blow-up profile and then set up the bootstrap procedure. The fact that such solutions satisfy the properties of Theorem \ref{th:main} is then showed at the end of this section.

%%%%%%%%%%%%%%%%%%%%%%%%%%%%%%%%%%%%%%%%%%%%%%%%%%%%%%%%%%%%%%%%%
%%%%%%%%%%%%%%%%%%%%%%%%%%%%%%%%%%%%%%%%%%%%%%%%%%%%%%%%%%%%%%%%%
%%%%%%%%%%%%%%%%%%%%%%%%%%%%%%%%%%%%%%%%%%%%%%%%%%%%%%%%%%%%%%%%%

\subsection{Adapted geometrical decomposition and renormalised flow}

The following lemma states that in a suitable neighborhood of the set of self-similar profiles, there exists a unique way to project the solution onto this set using adapted orthogonality conditions.

\begin{lemma}[Geometrical decomposition] \lab{lem:decomposition}

There exist $\lambda^*,\delta,K>0$ such that for all $\lb_0\geq \lambda^*$ and $Y_0\leq -\lambda_0^2$, for any regular $\varepsilon \in B_{L^2_\rho}(\delta \lambda_0^{-4})$ with $\varepsilon (Y_0)=-G_1(Y_0/\lb_0^2)$, there exist $(\lambda,\mu,\tilde Y_0)\in (0,+\infty)^2\times \mathbb R$, such that the following decomposition holds
$$
G_1\left(\frac{Y}{\lambda_0^2}\right)+\varepsilon (Y)=\lambda^2G_1\left(\frac{Y-\tilde Y_0}{\lambda^2\mu}\right) +\tilde \varepsilon (Y-\tilde Y_0)\ \ \text{with} \ \ \tilde \varepsilon \perp h_0,h_1,h_2 \ \text{in} \ L^2_\rho .
$$
Moreover, these are the only such parameters satisfying $|\lambda-1|\lambda_0^4+|\mu|+|\tilde Y_0|\leq K$. This defines a mapping $\varepsilon \mapsto (\lambda,\mu,\tilde Y_0)$, which is of class $C^1$ in $L^2_\rho$.

\end{lemma}

\begin{remark}

One has to keep track of the free boundary in the $Y$ variable, and we made a slight abuse of notation in Lemma \ref{lem:decomposition}. Indeed, note that the space $L^2_\rho$ in which $\varepsilon$ belongs is given by \fref{def:L2rho} with boundary at $Y_0$, whereas the space $L^2_\rho$ in which $\tilde \varepsilon$ belongs, and in which it enjoys orthogonality condition is defined by \fref{def:L2rho} with boundary at $Y_0-\tilde Y_0$.

\end{remark}

The proof of the above lemma is a standard combination of the implicit function theorem and a Taylor expansion of $G_1$ near the origin. It is relegated to Appendix \ref{ap:decomposition}. \\

\noindent For a function $\xi:[0,T)\times [0,+\infty)\rightarrow \mathbb R$, given parameters $(\lambda,\mu,y^*)\in C^{1}([0,T),(0,+\infty)^2\times \mathbb R)$, we define two renormalizations. The first one is the parabolic self-similar renormalization close to the blow-up point:
\be \lab{def:renormalisationpara}
s=s_0+\int_{t_0}^t \lambda^2(\tilde t)d\tilde t, \ \ Y=\lambda (y-y^*), \ \ f(s,Y)=\frac{1}{\lambda^2}\xi (t,y).
\ee
The second one is the renormalization associated to the leading order of the profile:
\be \lab{def:renormalisationpara2}
Z=\frac{y-y^*}{\lambda \mu}=\frac{Y}{\lambda^2\mu}, \ \ F(s,Z)=\frac{1}{\lambda^2}\xi(t,y)=f(s,Y).
\ee
The function $\xi$ solves \fref{1DPrandtl} if and only if the functions $f$ and $F$ solve the equations
\be \lab{eq:f}
\left\{ 
\begin{array}{l l }
f_s+\frac{\lambda_s}{\lambda}(2+Y\pa_Y)f-f^2 +\pa_Y^{-1}f\pa_Yf+\left(\int_{-\lb y^*}^0 f -\lb y^*_s \right)\pa_Yf-\pa_{YY}f =0,\\
f(s,-(\pi+a)\lb^2\mu)=0, \end{array}\right.
\ee
and
\be \lab{eq:F}
\left\{ 
\begin{array}{l l }
 F_s+\frac{\lb_s}{\lb}(2-Z\pa_Z)F-\frac{\mu_s}{\mu}Z\pa_Z F-F^2+\pa_Z^{-1}F\pa_Z F+\left(\int_{-\frac{y^*}{\lb \mu}}^0 F - \frac{y^*_s}{\lb \mu} \right)\pa_Zf -\frac{1}{\lambda^4\mu^2}\pa_{ZZ}F=0,\\
F(s,-(\pi+a))=0,
\end{array}\right.
\ee
respectivelly. Since $\lambda$ will behave like $(T-t)^{-1/2}$, and the blow-up point will behave like $\pi \mu (T-t)^{-1/2}$, we introduce the correction $a$:
\be \lab{def:a}
y^*=\lambda \mu (\pi+a).
\ee
We adopt the following different notation for the remainder:
\be \lab{id:decomposition vp}
f(s,Y)=G_1(Z)+\e(s,Y), \ \ F(s,Z)=G_1(Z)+u(s,Z), \ \ \text{so} \ \text{that} \ \e(s,Y)=u(s,Z).
\ee

%%%%%%%%%%%%%%%%%%%%%%%%%%%%%%%%%%%%%%%%%%%%%%%%%%%%%%%%%%%%%%%%%
%%%%%%%%%%%%%%%%%%%%%%%%%%%%%%%%%%%%%%%%%%%%%%%%%%%%%%%%%%%%%%%%%

\subsection{The weighted norm and derivative outside the blow-up point} \lab{subsec:wq}

To control the solution, we need a special weight and a special vector field to take derivatives, both adapted to the linearised operator in the $Z$ variable. We refer to subsection \ref{subsec:Z} and Lemma \ref{lem:exteriorlinear} for the motivation regarding these choices. Let $q:\mathbb R \rightarrow [0,+\infty)$ be an even function satisfying the following properties. $q\in C^2((0,+\infty))$, $q(0)=0$, $q'>0$ on $(0,\pi)$ with a limit on the right of the origin that exists and satisfies $\lim_{Z\downarrow 0} q'(Z)>0$, $q'(\pi)=0$, $q''(\pi)<0$, and $q(Z)=q(\pi)=1$ for $Z\geq \pi$. Define the weight $w$ on $(0,+\infty)\times \mathbb R^*$ by:
\be \lab{eq:def w}
w(s,Z):= \left\{ \begin{array}{l l}
\frac{1+\cos Z}{(1-\cos Z)\sin^4 Z} \frac{1}{\sin (-Z)} 4(\pi+Z)^3 \frac{1}{s^{q(Z)}} \ \  \ \text{if} \ Z\in (-\pi,0),\\
\frac{1+\cos Z}{(1-\cos Z)\sin^4 Z} \frac{1}{\sin Z} 4(\pi-Z)^3 \frac{1}{s^{q(Z)}} \ \ \ \text{if} \ Z\in (0,\pi),\\
\frac{1}{s}, \ \ \ \text{if} \ |Z|\geq \pi .
\end{array} \right.
\ee
Note that the weight $w(s,\cdot)$ is even, of class $C^1$ on $(0,+\infty)$, and $C^2$ on $(0,\pi)$ and $(\pi,+\infty)$. To take derivatives in a suitable way, we will use the vector field $A\pa_Z$, where:
\be \lab{eq:def A}
A(Z) := \left\{ \begin{array}{l l} 
-1  \ \ \text{for} \  Z\leq -\frac \pi 2,\\
\sin Z  \ \ \text{for} \ -\frac \pi 2 \leq Z \leq \frac \pi 2, \\
1 \ \ \text{for} \  \frac \pi 2 \leq Z.
\end{array} \right.
\ee
Note that one has the following sizes for $s>0$ and $Z\in [-\pi,\pi]$:
\be \lab{bd:w}
w\approx \frac{1}{|Z|^7s^{q(Z)}}, \ \ |A|\approx |Z|.
\ee

%%%%%%%%%%%%%%%%%%%%%%%%%%%%%%%%%%%%%%%%%%%%%%%%%%%%%%%%%%%%%%%%%
%%%%%%%%%%%%%%%%%%%%%%%%%%%%%%%%%%%%%%%%%%%%%%%%%%%%%%%%%%%%%%%%%

\subsection{The bootstrap regime} \lab{subsec:bootstrap}

The solution we will construct will be close to the blow-up profile in the following sense. At initial time we require the following bounds, involving parameters which will be fixed later on. Note that Lemma \ref{lem:decomposition} will imply the uniqueness of the decomposition used below:
\be \lab{eq:orthogonalite}
f (s,Y)=G_1 \left(\frac{Y}{\lb^2\mu}\right)+\e(s,Y), \ \ \e \perp_\rho (h_0,h_1,h_2),
\ee

\begin{definition}[Initial closeness] \lab{def:ini}

Let $M\gg 1$, $s_0\gg 1$ such that $M^3e^{-s_0}\ll 1$,  $0<\nu \ll 1$ and $\xi_0\in C^{\infty}([0,+\infty),\mathbb R)$ with $\xi_0(0)=\pa_{yy}\xi_0 (0)=0$. We say that $\xi_0$ is initially ($t=t_0$ i.e. $s=s_0$) close to the blow-up profile if there exists $\lambda_0>0$, $a_0\in \mathbb R$ and  $\mu_0>0$ such that the following properties are verified. In the variables \fref{def:renormalisationpara} one has the decomposition \fref{eq:orthogonalite} with $(s,\e(s),\lambda,\mu)=(s_0,\e_0,\lambda_0,\mu_0)$, where the remainder and the parameters satisfy:
\begin{itemize}
\item[(i)] \emph{Initial values of the modulation parameters:}
\be \lab{bd:parametersini}
\frac 12 e^{\frac{s_0}{2}}< \lb_0 < 2 e^{\frac{s_0}{2}}, \ \ \frac 12 <\mu_0 < 2, \ \ |a_0|<e^{-\frac 12 s_0}.
\ee
\item[(ii)] \emph{Initial smallness of the remainder in parabolic variables:}
\be \lab{bd:eini}
\| \e_0 \|_{L^2_{\rho}} < e^{-\frac 72 s_0}, \ \ \| \e_0 \|_{H^3(|Y|\leq M^3)} < e^{-\frac 72 s_0}.
\ee
\item[(iii)] \emph{Initial smallness of the remainder in the inviscid self-similar variables:}
\be \lab{bd:eini2}
\int_{-\pi-a_0}^{-Me^{-s_0}} u^2wdZ+\int_{Me^{-s_0}}^{+\infty} u^2wdZ < e^{-2\left(\frac 12 -\nu \right)s_0}, \ \ \int_{-\pi-a_0}^{-Me^{-s_0}} |A\pa_Zu|^2wdZ+\int_{Me^{-s_0}}^{+\infty} |A\pa_Zu|^2wdZ < e^{2\nu s_0}.
\ee
\item[(iv)] \emph{Initial regularity close to the origin $y=0$ in original variables:}
\be \lab{bd:eini3}
\| \xi_0\|_{W^{1,\infty}([0,2])}< 1.
\ee

\end{itemize}

\end{definition}

We aim at proving that solutions which are initially close to the blow-up profile in the sense of Definition \ref{def:ini} will stay close to this blow-up profile up to modulation. The proximity at later times is defined as follows.

\begin{definition}[Trapped solutions] \lab{def:trap}

We say that a solution $f(s,Y)=F(s,Z)$ is trapped on $[s_0,s_1]$ with $s_0<s_1<\infty$, if it satisfies the properties of Definition \ref{def:ini} at time $s_0$ with the parameters $\nu$ and $M$ and if, for $K\gg 1$ and $0<\nu'\ll \nu $, and for all $s\in [s_0,s_1]$, $f(s,\cdot)$ can be decomposed as in \fref{id:decomposition vp} and \fref{eq:orthogonalite} with:
\begin{itemize}
\item[(i)] \emph{Values of the modulation parameters:}
\be \lab{bd:parameterstrap}
\frac 1K e^{\frac{s}{2}}< \lb < K e^{\frac{s}{2}}, \ \ \frac 1K <\mu < K, \ \ |a|<Ke^{-\left(\frac 12 -2\nu \right)s}.
\ee
\item[(ii)] \emph{Smallness of the remainder in parabolic variables:}
\be \lab{bd:etrap}
\| \e \|_{L^2_{\rho}} < Ke^{-\frac 72 s}, \ \ \| \e \|_{H^3(|Y|\leq M^2)} < Ke^{-\left(\frac 72 -\nu' \right) s}.
\ee
\item[(iii)] \emph{Smallness of the remainder in the inviscid self-similar variables:}
\be \lab{bd:etrap2}
\int_{-\pi-a}^{-Me^{-s}} u^2wdZ+\int_{Me^{-s}}^{+\infty} u^2wdZ < K^2e^{-2\left(\frac 12 -\nu \right)s}, \ \ \int_{-\pi-a}^{-Me^{-s}} |A\pa_Zu|^2wdZ+\int_{Me^{-s}}^{+\infty} |A\pa_Zu|^2wdZ < K^2e^{2\nu s}.
\ee
\end{itemize}

\end{definition}

\begin{remark}

Lemma \ref{lem:decomposition} and the regularity of the flow, Proposition \ref{pr:cauchy}, imply that the parameters of Definition \ref{def:trap} are uniquely determined and are in $C^1([s_0,s_1])$. In particular, the renormalisation \fref{def:renormalisationpara} and \fref{def:renormalisationpara2} is indeed well-defined.

\end{remark}

The heart of the paper is the following bootstrap proposition.

\begin{proposition} \lab{pr:bootstrap}

There exist universal constants $K,M,s_0^*\gg 1$ and $0<\nu'\ll\nu\ll 1$ such that the following holds for any $s_0\geq s_0^*$. Any solution which is initially close to the blow-up profile in the sense of Definition \ref{def:ini} is trapped on $[s_0,+\infty)$ in the sense of Definition \ref{def:trap}.

\end{proposition}

Lemma \ref{lem:decomposition}, and a standard continuity argument, imply that for $s_0$ large enough, any solution which is initially close to the blow-up profile in the sense of Definition \ref{def:ini} is trapped in the sense of Definition \ref{def:trap} on some interval $[s_0,s_1]$ with $s_1>s_0$. Letting $s^*>s_0$ be the supremum of times $s_1> s_0$ such that the solution is trapped on $[s_0,s_1]$, the purpose now is to show that $s^*=+\infty$. The strategy is to study the trapped regime via several lemmas and show that the solutions cannot escape from the open set defined by Definition \ref{def:trap}. The proof of Proposition \ref{pr:bootstrap} is then given at the end of this section.\\

\noindent Note that the constants $K$, $M$, $s_0^*$, $\nu'$, $\nu$ and $\eta$ (defined in Lemma \ref{lem:dseL2rho}) will be adjusted during the proof: we will always be able to conclude the proof of the various lemmas by choosing $M$ large enough depending on $K$ and then choosing $s_0^*$ large enough depending on $K$ and $M$. First, note that one has pointwise control of the remainder for trapped solutions.

\begin{lemma} \lab{lem:Linftybd}
There exists $\nu^*>0$ such that for any $K$ and $0<\nu,\nu'<\nu^*$, for any $M$ large enough, a $s_0^*$ exists such that if $u$ is trapped on $[s_0,s_1]$ with $s_0^*\leq s_0$, then for for all $s\in [s_0,s_1]$:
\be \lab{bd:weightedSobolev}
\| \e \|_{L^{\infty}} =\| u \|_{L^{\infty}} \lesssim Kse^{-\left(\frac 14 -\nu \right)s}.
\ee
\end{lemma}

\begin{proof}

First, Sobolev embedding together with \fref{bd:etrap} implies:
$$
\| \e \|_{L^{\infty}(|Z|\leq e^{-s}M^2)}\leq  C(K,M) e^{-(\frac 72-\nu') s}\leq e^{-\left(\frac 14 -\nu \right)s}
$$
for $s_0$ large depending on $K,M$. Let $E:=\{-\pi-a\leq Z \leq -M e^{-s} \}\cup \{ M e^{-s}\leq Z \}$. Then from \fref{bd:w}, we have $w\gtrsim s^{-1}$ and $|A|w\gtrsim s^{-1}$ on $E$, implying:
$$
\| u \|_{L^2(E)}^2 \lesssim s\int_{-\pi-a}^{-Me^{-s}} u^2w+s\int_{Me^{-s}}^{+\infty} u^2w, \ \ \| \pa_Z u \|_{L^2(E)}^2 \lesssim  s\int_{-\pi-a}^{-Me^{-s}} |A\pa_Zu|^2w+s\int_{Me^{-s}}^{+\infty} |A\pa_Zu|^2w.
$$
Therefore, using Agmon's inequality and \fref{bd:etrap2} gives:
$$
\| u \|_{L^{\infty}(E)} \leq C \| u\|_{L^2(E)}^{\frac 12} \left(\| u\|_{L^2(E)}+\|\pa_Z u \|_{L^2(E)}\right)^{\frac 12}\leq C K s^{\frac 12}e^{-\left(\frac 14-\nu \right)s}.
$$
Hence, as for $M$ large enough depending on $K$ the two zones $|Z|\geq Me^{-s}$ and $|Y|\leq M^2$ cover the whole space, there holds $\| u\|_{L^{\infty}}\lesssim K se^{-\left(\frac 14-\nu \right)s}+ e^{-\frac 14 s}\lesssim K se^{-\left(\frac 14-\nu \right)s} $ for $s_0$ large enough.

\end{proof}

%%%%%%%%%%%%%%%%%%%%%%%%%%%%%%%%%%%%%%%%%%%%%%%%%%%%%%%%%%%%%%%
%%%%%%%%%%%%%%%%%%%%%%%%%%%%%%%%%%%%%%%%%%%%%%%%%%%%%%%%%%%%%%%
%%%%%%%%%%%%%%%%%%%%%%%%%%%%%%%%%%%%%%%%%%%%%%%%%%%%%%%%%%%%%%%

\subsection{Analysis near the blow-up point} \lab{subsec:max}

This subsection is devoted to the study of the solution near $y^*$ in parabolic variables \fref{def:renormalisationpara}. This is the most sensitive zone, in which the blow-up parameters are selected. The remainder is dissipated away from this point, until it reaches the outside region $|Z|\gtrsim 1$ where another dynamics takes place (see next subsection). The analysis near the blow-up point is the consequence of the blow-up profile structure, the linear structure Proposition \ref{pr:Ls} and the orthogonality conditions \fref{eq:orthogonalite}. The measure $\rho =ce^{-3Y^2/4} $ decreases very fast because of the transport part of the operator $\mathcal L$ which is unbounded and pushes the characteristics away from the origin. Therefore, the analysis here is poorly affected by the exterior dynamics. From \fref{eq:f}, \fref{id:decomposition vp}, \fref{eq:F1} and \fref{def:a} we infer that $\e$ solves
\be \lab{eq:e}
\left\{ 
\begin{array}{l l }
\e_s+\Ls \e +\tilde \Ls \e +\text{Mod}+NL-\frac{1}{\lb^4\mu^2}\pa_{ZZ}G_1(Z)=0, \\
\e (s,-(\pi+a)\lb^2\mu)=-G_1(-\pi-a).
\end{array}
\right.
\ee
where $\Ls$ is defined by \fref{eq:def L}, the small linear term, the modulation term and the nonlinear term are
\be \lab{def:tildeL}
\tilde \Ls \e:=2\left(1-G_1(Z)\right)\e+\left(\lb^2\mu\pa_{Z}^{-1}G_1(Z)-Y\right)\pa_Y\e+\frac{1}{\lb^2\mu}\pa_Z G_1(Z)\pa_Y^{-1}\e,
\ee
\bee
\text{Mod}(Y) & := & - \frac{\mu_s}{\mu}Z\pa_Z G_1(Z)+\left(\frac{\lb_s}{\lb}-\frac 12 \right)\left((2-Z\pa_Z)G_1(Z)+(2+Y\pa_Y)\e\right)\\
&& +\left( \int_{-(\pi+a)\lb^2\mu}^{0}fdY-\lb y^*_s\right)\left(\frac{1}{\lb^2\mu}\pa_ZG_1+\pa_Y \e \right),
\eee
$$
NL=-\e^2+\pa_Y^{-1}\e\pa_Y \e.
$$
The parameters evolve according to the following dynamics.

\begin{lemma}[Modulation equations] \lab{lem:modulation}
For $\nu$ small enough, and $K,\nu',M$ such that Lemma \ref{lem:Linftybd} holds true, there exists $s_0^*$ such that for solution trapped on $[s_0,s_1]$, with $s_0\geq s_0^*$:
\bea 
\lab{bd:modulation1} & \left|\frac{\lb_s}{\lb}-\frac{1}{2}+\frac{1}{4\lb^4\mu^2} \right| \leq C(K)\left( \lb^{-8}+\lb^{-4}\| \e \|_{L^2_\rho}+\|  \e \|_{L^\infty}\|  \e \|_{L^2_\rho}\right),\\
\lab{bd:modulation2} &\left|\frac{\mu_s}{\mu}-\frac{1}{2\lb^4\mu^2} \right| \leq C(K) \left( \lb^{-8}+\| \e \|_{L^2_\rho}+\lb^4 \| \e \|_{L^\infty_\rho}\|  \e \|_{L^2_\rho}\right),\\
\lab{bd:modulation3} &\left| \int_{-\lb y^*}^{0}fdY-\lb y^*_s\right|\leq C(K)\left(e^{-e^{s}}+\| \e \|_{L^2_\rho}+\lb^4 \| \e \|_{L^\infty} \| \e \|_{L^2_\rho}\right), \\
\lab{bd:modulation4} &\left| a_s+\frac a2 -\int_{-\pi-a}^{-\pi} G_1dZ-\frac{1}{\lb^2\mu}\int_{-\lb y^*}^0 \e dY \right|\leq C(K) \left( \lb^{-4}+\| \e \|_{L^2_\rho}+\lb^4 \|  \e \|_{L^\infty } \|  \e \|_{L^2_\rho}\right).
\eea
and the bound\footnote{There is indeed no constant factor in front of $e^{-\frac{13}{8} s}$}:
\be \label{bd:lossymodulation}
e^{2s}\left|\frac{\lb_s}{\lb}-\frac{1}{2}+\frac{1}{4\lb^4\mu^2} \right|+\left|\frac{\mu_s}{\mu} \right|+\left| \int_{-\lb y^*}^{0}fdY-\lb y^*_s\right| \leq e^{-\frac{13}{8} s}.
\ee

\end{lemma}

To ease the notation, we define:
\be \label{def:coefficientsm}
m_1=\frac{\lb_s}{\lb}-\frac 12, \qquad m_2=\frac{\mu_s}{\mu}, \qquad m_3=\int_{-\lb y^*}^{0}fdY-\lb y^*_s
\ee
Observe that $m_1$ is the difference between the evolution of $\lb$ and the expected self-similar law, while $m_3$ is the difference between the speed of the blow-up point and the value of the transport part of the equation at this point.

\begin{proof}

This is a direct and standard computation using the definition of the geometrical decomposition and the spectral structure of the linearised dynamics. First we differentiate the orthogonality conditions \fref{eq:orthogonalite} for $i=0,1,2$ using the boundary condition \fref{eq:e}:
$$
0=\frac{d}{ds} \left( \int_{-\lb y^*}^{+\infty} \e h_i \rho dY \right)=    -      \frac{d}{ds} (\lb y^*) (h_i \rho)(-\lb y^*)G_1(-\pi-a)+\int_{-\lb y^*}^{+\infty} \e_s h_i\rho dY.
$$
%Since $\lambda y^* \gtrsim e^s$ from 

Thanks to \fref{bd:parameterstrap} and \eqref{def:a}, one has $\lambda y^* \gtrsim e^s$ and therefore $|\rho (\lb y^*)|\leq e^{-e^{\frac 32 s}}$ when $s_0$ is large enough. Hence, as $|\int_{-(\pi+a)\lb^2\mu}^0 fdY|\lesssim \lb^2\mu\lesssim e^s$ from \fref{id:decomposition vp} and \fref{bd:weightedSobolev}, the above identity can be rewritten as:
\be \lab{id:mod interm}
\int_{-\lb y^*}^{+\infty} \e_s h_i\rho dY= O(e^{-e^{s}}(1+|m_1|+|m_3|)).
\ee
We now estimate the contribution of each term when injecting \fref{eq:e} in the above identity.\\

\noindent \textbf{Step 1} \emph{The linear and small linear terms}. Performing integration by parts and thanks to the orthogonality \fref{eq:orthogonalite} and Proposition \ref{pr:Ls}, using the boundary condition \fref{eq:e} and \fref{bd:origin bootstrap2} (this second estimate states boundedness at the boundary, its proof is made later on):
\bea
\lab{bd:mod11}  && \int_{-\lb y^*}^{+\infty} h_i \Ls \e \rho dY = (\pa_Y \e \rho h_i)(-\lb y^*)-(\e \rho \pa_Y h_i)(-\lb y^*)+\int_{-\lb y^*}^{+\infty} \Ls h_i \e \rho dY\\
\non & = & (\pa_Y \e \rho h_i)(-\lb y^*)+(\rho \pa_Y h_i)(-\lb y^*)G_1(-\pi-a) = O(e^{e^{-s}}(1+|\pa_Y \e (-\lb y^*)|))=O(e^{-e^{s}}).
\eea
The small linear term is evaluated as follows. First, using Cauchy-Schwarz, and since $|1-G_1(Z)|\lesssim Z^2\lesssim \lb^{-4}Y^2$, one has:
$$
\left| \int_{-\lb y^*}^{+\infty} h_i (1-G_1(Z))\e \rho dY \right| \lesssim \lb^{-4} \| \e \|_{L^2_{ \rho}}.
$$
Similarly, since $\left|\left(\lb^2\mu\pa_{Z}^{-1}G_1(Z)-Y\right)\right|+|Y|\left|\pa_Y(\left(\lb^2\mu\pa_{Z}^{-1}G_1(Z)-Y)\right)\right|\lesssim \lb^{-4}|Y|^3$:
$$
\left| \int_{-\lb y^*}^{+\infty} h_i \left(\lb^2\mu\pa_{Z}^{-1}G_1(Z)-Y\right)\pa_Y\e \rho dY \right| \lesssim \lb^{-4} \| \e \|_{L^2_{ \rho}}.
$$
Using Cauchy-Schwarz one estimates that
\be \lab{eq:controlenonlocal}
\left| \int_0^Y \e (s,\tilde Y)d\tilde Y\right| \leq \| \e \|_{L^2_{ \rho}} \left(\int_0^Y e^{\frac{3}{4}\tilde Y^2}d\tilde Y \right)^{\frac 12} \lesssim  \| \e \|_{L^2_{\rho}}  \frac{e^{\frac{3Y^2}{8}}}{(1+|Y|)^{\frac 12}}
\ee
which implies the bound, since $|\pa_Z G_1(Z)|\lesssim \lb^{-4}|Y|$:
$$
\left| \int_{-\lb y^*}^{+\infty} h_i \frac{1}{\lb^2\mu}\pa_Z G_1(Z)\pa_Y^{-1}\e \rho dY \right| \lesssim \lb^{-4} \| \e \|_{L^2_{\rho}}.
$$
From \fref{def:tildeL} this gives the bound for the small linear term for $i=1,2,3$:
\be \lab{bd:mod12} 
\left| \int_{-\lb y^*}^{+\infty} h_i \tilde \Ls\e \rho dY \right| \lesssim \lb^{-4} \| \e \|_{L^2_{\rho}}.
\ee
\noindent \textbf{Step 2} \emph{The modulation term}. We first rewrite it performing a Taylor expansion on $G_1$ from \fref{id:G1} near the origin and using \fref{eq:defhi}:
\bea
\lab{id:Mod} \text{Mod}(Y) & = & m_2\left(\frac{1}{\lb^4\mu^2}\left(\frac 16 h_2(Y)+\frac 13 h_0(Y) \right)+\mu^{-4}\lb^{-8}r_2(Y) \right)\\
\non && +m_1\left(2h_0(Y)+\mu^{-4}\lb^{-8}r_1(Y)+(2+Y\pa_Y)\e\right)\\
\non && +m_3\left(-\frac{1}{\lb^4\mu^2}\frac{1}{2\sqrt 3}h_1(Y)+\mu^{-4}\lb^{-8}r_3(Y)+\pa_Y \e \right)
\eea
where $r_1(Y)=\mu^4\lb^{8}((2-Z\pa_Z )G_1-2)$ and $r_2=-\mu^{4}\lb^{8}Z(\pa_Z G_1+Z/2)$ are even functions which are $O(Y^{4})$, and $r_3=\mu^3\lb^6(\pa_Z G_1+Z/2)$ is and odd function that is $O(Y^3)$. We recall that $h_{2i}$ and $h_{2i+1}$ are even and odd functions and form an almost orthogonal family: $\int_{-\lb y^*}^{+\infty} h_ih_j\rho=-\int_{-\infty}^{-\lb y^*} h_ih_j\rho=O(e^{-e^{3s/2}})$. From \fref{eq:orthogonalite} and \fref{eq:defhi}, one has $\int_{-\lb y^*}^{+\infty}p\e=0$ for any polynomial $p$ of degree $2$. Let, for $i=0,1,2$, $\text{Mod}_i:= \int_{-\lb y^*}^{+\infty} h_i \text{Mod} \rho$. Using the previous remarks, \fref{eq:orthogonalite} and the boundary condition \fref{eq:e} we obtain that
\begin{align}
\non \text{Mod}_0 &= m_2 \frac{\| h_0\|_{L^2_\rho}^2+O(\lb^{-4})}{3\lb^4 \mu^2} +m_1 \left(2\| h_0\|_{L^2_\rho}^2+O(\lb^{-8}) +(Y\rho)(-\lb y^*)G_1 (-\pi-a)\right) \\
\non  & \qquad +m_3\left(-\mu^{-4}\lb^{-8} \int_{-\infty}^{-\lb y^*} r_3\rho+\rho(-\lb y^*)G_1(-\pi-a) \right),\\
\lab{bd:modmod1} & = m_2 \left(\frac{\| h_0\|_{L^2_\rho}^2}{3\lb^4 \mu^2}+O(\lb^{-8}) \right) +m_1 \left(2\| h_0\|_{L^2_\rho}^2+O(\lb^{-8}) \right)  +m_3O(e^{-e^s}),
\end{align}
where for the last bound we used the fact that $\lb y^*\gtrsim e^{s}$ and $\rho=Ce^{-\frac{3Y^2}{4}}$; similarly
\begin{align}
\non &\text{Mod}_1 = m_1 \left(O(e^{-e^s})-\mu^{-4}\lb^{-8}\int_{-\infty}^{-\lb y^*} h_1r_1\rho +\frac{(Y^2\rho)}{\sqrt 3} (-\lb y^*)G_1(-\pi-a)+O(\| \e\|_{L^2_{\rho}})\right) \\
\non &  \quad + m_2 \Bigl(O(e^{-e^s})-\mu^{-4}\lb^{-8} \int_{-\infty}^{-\lb y^*} h_1r_1\rho \Bigr) +m_3\Bigl(-\frac{\| h_1\|_{L^2_\rho}^2+O(\lb^{-4})}{2\sqrt 3 \lb^4\mu^2} - \frac{(Y\rho)}{\sqrt 3}(-\lb y^*)G_1(-\pi-a) \Bigr)\\
\lab{bd:modmod2} &\quad \quad =   m_2 O(e^{-e^s})  +m_1O(e^{-e^s}+\| \e\|_{L^2_{\rho}}) -m_3\frac{ \| h_1\|_{L^2_\rho}^2+O(\lb^{-4})}{2\sqrt 3 \lb^2\mu},
\end{align}
\begin{align}
\non  \text{Mod}_2  &= m_2 \left(\frac{\| h_2\|_{L^2_\rho}^2}{6\lb^4 \mu^2}+O(\lb^{-8}) \right)  +m_1 \left(O(\lb^{-8}) +(Yh_2\rho)(-\lb y^*)G_1 (-\pi-a)+O(\| \e \|_{L^2_{\rho}})\right) \\
\non &  \qquad +m_3\left(O(e^{-e^s})-\mu^{-4}\lb^{-8} \int_{-\infty}^{-\lb y^*} h_2r_3\rho+(h_2\rho)(-\lb y^*)G_1(-\pi-a)+O(\| \e \|_{L^2_{\rho}}) \right)\\
\lab{bd:modmod3} & = m_2 \frac{\| h_2\|_{L^2_\rho}^2+O(\lb^{-4})}{6\lb^4 \mu^2} +m_1 O(\lb^{-8}+\| \e \|_{L^2_{\rho}}) +m_3O(e^{-e^s}+\| \e \|_{L^2_{\rho}}).
\end{align}

\noindent \textbf{Step 3} \emph{The nonlinear term}. Since $|h_i|\lesssim (1+Y^2)$ for $i=0,1,2$ we estimate:
$$
\left| \int_{-\lb y^*}^{+\infty} \e^2 h_i \rho dY \right| \lesssim \| \e \|_{L^2_{\rho}}\| \e \|_{L^{\infty}}.
$$
Integrating by parts, using \fref{eq:e} for the boundary term and $|\pa_{Y}^{-1}\e |\leq Y\| \e \|_{L^{\infty}}$ we get:
$$
\left| \int_{-\lb y^*}^{+\infty} h_i \pa_Y \e \pa_Y^{-1}\e \rho dY \right| \lesssim \| \e \|_{L^{\infty}} \| \e \|_{L^2_\rho} +O(e^{-e^s})
$$
Therefore, for $i=0,1,2$:
\be \lab{bd:mod13}
\left| \int_{-\lb y^*}^{+\infty} h_iNL\rho dY \right|\lesssim \| \e \|_{L^2_\rho}\| \e \|_{L^{\infty}}+O(e^{-e^s})
\ee

\noindent \textbf{Step 4} \emph{The error term}. Finally, using a Taylor expansion from \fref{id:G1}:
\be \lab{id:errorterm}
\frac{1}{\lb^4\mu^2}\pa_{ZZ}G_1(Z)=\frac{1}{\lb^4\mu^2}\left(-\left(\frac 12 -\frac{1}{6\lb^4\mu^2}\right)h_0+\frac{1}{12\lb^4\mu^2}h_2+(\pa_{ZZ}G_1+\frac 12 -\frac{Z^2}{4}) \right).
\ee
This gives (since this term is an even function and $h_1$ is an odd function):
\be \lab{bd:mod14}
\int_{-\lb y^*}^{+\infty} \frac{1}{\lb^4\mu^2}\pa_{ZZ}G_1(Z)h_i \rho dY = \left\{
\begin{array}{l l l}
-\frac{1}{\lb^4 \mu^2}\left( \frac 12 -\frac{1}{6\lb^4\mu^2}\right)\| h_0\|_{L^2_\rho}^2+O(\lb^{-12}) & \text{if} \ i=0,\\
O(e^{-e^s}) & \text{if} \ i=1,\\
\frac{1}{12\lb^8 \mu^4} \| h_2\|_{L^2_\rho}^2+O(\lb^{-12}) & \text{if} \ i=2.\\
\end{array}
\right.
\ee
\noindent \textbf{Step 5} \emph{End of the proof}. We collect the above estimates \fref{bd:mod11}, \fref{bd:mod12}, \fref{bd:modmod1}, \fref{bd:modmod2}, \fref{bd:modmod3}, the first in \fref{bd:mod13} and \fref{bd:mod14} and inject them in \fref{id:mod interm} using \fref{eq:e}. One obtains:
\begin{eqnarray*}
&& m_2 \frac{1+O(\lb^{-4})}{3\lb^4 \mu^2}+ m_1 \left(2+O(\lb^{-8}) \right)   +m_3O(e^{-e^s})\\
&& \qquad \qquad = -\frac{1}{\lb^4 \mu^2}\left( \frac 12 -\frac{1}{6\lb^4\mu^2}\right)+O(\lb^{-12})+O(\lb^{-4}\| \e\|_{L^2_{\rho}}+\| \e \|_{L^{\infty}} \| \e \|_{L^2_\rho}), \\
&&  m_2 O(e^{-e^s})  +m_1 O(e^{-e^s}+\| \e\|_{L^2_{\rho}})-\frac{1+O(\lb^{-4})}{\lb^4\mu^2}m_3\\
&& \qquad \qquad  =O(e^{-e^s})+O(\lb^{-4}\| \e\|_{L^2_{\rho}}+\| \e \|_{L^{\infty}} \| \e \|_{L^2_\rho}), \\
&& m_2 \frac{1+O(\lb^{-4})}{6\lb^4 \mu^2}+m_1 O(\lb^{-8}+\| \e \|_{L^2_{\rho}}) +m_3O(e^{-e^s}+\| \e \|_{L^2_{\rho}} )\\
&&\qquad \qquad = \frac{1}{12\lb^8 \mu^4}+O(\lb^{-12})+O(\lb^{-4}\| \e\|_{L^2_{\rho}}+\| \e \|_{L^{\infty}} \|  \e \|_{L^2_\rho}).
\end{eqnarray*}
These three estimates, together with the fact that $\| \e \|_{L^2_{\rho}}\lesssim e^{-7s/2}$ and $\lb \approx e^{s/2}$ obtained from \fref{bd:etrap} and \fref{bd:parameterstrap}, imply \fref{bd:modulation1}, \fref{bd:modulation2} and \fref{bd:modulation3}. The fourth inequality \fref{bd:modulation4} is obtained from \fref{bd:modulation1}, \fref{bd:modulation2}, and \fref{bd:modulation3}, since from \fref{def:a} and $\int_{-\pi}^0 G_1=\pi/2$:
$$
\int_{-\lb y^*}^0 fdY -\lb y^*_s=\lb^2\mu \left[\int_{-\pi-a}^{-\pi}G_1dZ+\frac{1}{\lb^2 \mu}\int_{-(\pi+a)\lb^2\mu}^0 \e dY-a_s-\frac a2 -\left(\left(m_1 \right)+m_2\right)(\pi+a) \right].
$$
Summing \fref{bd:modulation1}, \fref{bd:modulation2} and \fref{bd:modulation3}, we obtain:
$$
\lambda^4\left|\frac{\lb_s}{\lb}-\frac{1}{2}+\frac{1}{4\lb^4\mu^2} \right|+\left|\frac{\mu_s}{\mu} \right|+\left| \int_{-\lb y^*}^{0}fdY-\lb y^*_s\right| \lesssim \lambda^{-4}+\| \e \|_{L^2_\rho}+\lambda^4 \| \e \|_{L^{\infty}}\| \e \|_{L^2_\rho} .
$$
The right-hand side is, from \fref{bd:etrap} and \fref{bd:parameterstrap}, $\leq C(K)(e^{-2s}+e^{-7s/2}+se^{-(3/2+1/4-\nu)s})$ and hence \fref{bd:lossymodulation} if $\nu<1/4$, for $s_0$ large depending on $K$.

\end{proof}

The decay of the remainder $\e$ is encoded by the following Lyapunov functional.

\begin{lemma}[Interior Lyapunov functional] \lab{lem:dseL2rho}

There exist universal $C,\eta^*>0$, such that for any $0<\eta <\eta^*$, the following holds. For $K,\nu,\nu',M$ such that Lemma \ref{lem:Linftybd} holds, there exists $s_0^*$ such that for a solution that is trapped on $[s_0,s_1]$ with $s_0\geq s_0^*$:
\be \lab{bd:lyapunovpara}
\frac{d}{ds}\left(\frac 12 \| \e \|_{L^2_{\rho}}^2\right)+\left(\frac 72 -Ce^{-\eta s}\right) \| \e \|_{L^2_{\rho}}^2+e^{-\eta s} \| \pa_Y \e \|_{L^2_{\rho}}^2 \leq C\| \e \|_{L^2_{\rho}}\lb^{-12}+Ce^{-e^s}.
\ee

\end{lemma}

\begin{proof}

This is a direct computation relying on the spectral gap that absorbs the nonlinear effects, the modulation equations established previously, and the rapid decay of the measure $\rho$. First, from \fref{eq:e} and \fref{bd:lossymodulation}, one computes 
\bea
&& \lab{exp:dse2} \frac{d}{ds} \left( \frac 12 \| \e \|_{L^2_\rho}^2 \right)= \frac 12 \frac{d}{ds} \int_{-\lb y^*}^{+\infty} \e^2 \rho dY \\
\non &=& -  \frac 12(\e^2\rho)(-\lb y^*)\frac{d}{ds}(-\lb y^*)+\int_{-\lb y^*}^{+\infty}\left(- \Ls\e-\tilde \Ls \e-\text{Mod}- \NL+\frac{1}{\lb^4\mu^2}\pa_{ZZ}G_1\right)\e \rho dY\\
\non &=& O(e^{-e^s}(1+\| \pa_Y \e\|_{L^2_\rho}^2))+\int_{-\lb y^*}^{+\infty}\left(- \Ls\e-\tilde \Ls \e-\text{Mod}- \NL+\frac{1}{\lb^4\mu^2}\pa_{ZZ}G_1\right)\e \rho dY.
\eea

\noindent \textbf{Step 1} \emph{The linear term}. First, we claim the dissipative spectral gap estimate
\be \lab{bd:spectralgap}
\int_{-\lb y^*}^{+\infty} |\pa_Y \e |^2\rho dY \geq \frac{9}{2}\left(1-Ce^{-\eta s}\right) \int_{-\lb y^*}^{+\infty} \e^2\rho dY+2e^{-\eta s}\int_{-\lb y^*}^{+\infty} |\pa_Y \e|^2\rho dY-Ce^{-e^s}
\ee
for some universal constant $C>0$. We use analytical results on the whole space $\mathbb R$, with scalar product $\langle u,v \rangle=\int_\mathbb R u v \rho$ (only for the few next lines). Define the extension
$$
\tilde \e:= \left\{ 
\begin{array}{l l}
\e (-\lb y^*) \ \ \text{for} \ Y\leq -\lb y^*,\\
\e (Y) \ \ \text{for} \ Y\geq -\lb y^*.
\end{array}
\right.
$$
Then $\tilde \e \in H^1_\rho$. Define the projection on higher modes
$$
\bar \e:= \tilde \e -\frac{\la \tilde \e,h_0 \ra}{\| h _0\|_{L^2_\rho}^2}h_0-\frac{\la \tilde \e,h_1 \ra}{\| h _1\|_{L^2_\rho}^2}h_1-\frac{\la \tilde \e,h_2 \ra}{\| h _2\|_{L^2_\rho}^2}h_2.
$$
Then from the orthogonality \fref{eq:orthogonalite}, since $\e(-\lb y^*)=-G_1(-\pi-a)$ from the Dirichlet boundary condition, one infers that
$$
\la \tilde \e,h_i \ra = - \frac12\sqrt{\frac3\pi}\int_{-\infty}^{-\lb y^*} h_i G_1(-\pi-a)e^{-\frac{3}{4}Y^2}dY=O(e^{-e^s})
$$
as $\lb y^*\gtrsim e^{s}$. This implies that
$$
\int_{-\lb y^*}^{+\infty} \e^2\rho dY\leq \| \bar \e \|_{L^2_\rho}^2+C e^{-e^s}, \ \ \int_{-\lb y^*}^{+\infty} |\pa_Y\e|^2\rho dY \geq \| \pa_Y\bar \e \|_{L^2_\rho}^2-C e^{-e^s}.
$$
As $\bar \e \in H^1_\rho$ with $\bar \e \perp h_i$ for $i=0,1,2$, one has the spectral gap estimate from \fref{bd:spectralgap}:%\fref{pr:Ls}:
$$
\| \pa_Y \bar \e \|_{L^2_\rho}^2\geq \frac 92 \| \bar \e \|_{L^2_\rho}^2.
$$
The two above estimates imply \fref{bd:spectralgap}. Therefore, the linear term gives from the boundary condition \fref{eq:e} %, \fref{bd:origin bootstrap2} 
and the definition \fref{measure}:
\bee
- \int_{-\lb y^*}^{+\infty} \Ls \e \e \rho dY & = & \int_{-\lb y^*}^{+\infty} \e^2\rho dY - \int_{-\lb y^*}^{+\infty} |\pa_Y \e|^2\rho dY+(\pa_Y \e \rho)(-\lb y^*)G_1(-\pi-a) \\
&\leq & -\left(\frac 72 -C e^{-\eta s}\right)\int_{-\lb y^*}^{+\infty} \e^2\rho dY -2e^{-\eta s}\int_{-\lb y^*}^{+\infty} |\pa_Y \e|^2\rho dY +Ce^{-e^s}.
\eee

\noindent \textbf{Step 2} \emph{The small linear term}. Recall \fref{def:tildeL}. One computes using Poincar\'e \fref{bd:poincare} and the fact that $|G_1(Z)-1|\lesssim \lb^{-4}Y^2$
$$
\left| \int_{-\lb y^*}^{+\infty} (1-G_1(Z))\e^2\rho dY \right| \leq C\lb^{-4} \| \e \|_{H^1_{\rho}}^2.
$$
Next, an integration by parts, the boundary condition \fref{eq:e} together with the fact that $\lb y^*\gtrsim e^s$ gives (note that the boundary term at $Y=+\infty$ is zero from the exponential decay of $\rho$)
\bee
&& \int_{-\lb y^s}^{+\infty} \e \left(\lb^2\mu\pa_{Z}^{-1}G_1(Z)-Y\right)\pa_Y\e \rho dY \\
&= & -\left[(\lb^2\mu \pa_Z^{-1}G_1(Z)-Y)\frac{\e^2}{2}\rho \right](-\lb y^*)-\frac 12 \int_{-\lb y^*}^{+\infty} \e^2 \pa_Y \left((\lb^2\mu \pa_Z^{-1}G_1(Z)-Y)\rho \right)dY \\
&=& O(e^{-e^s})-\frac 12 \int_{-\lb y^*}^{+\infty} \e^2 \pa_Y \left((\lb^2\mu \pa_Z^{-1}G_1(Z)-Y)\rho \right)dY
\eee
Then, one notices that for $|Y|\leq e^{\frac 34 s}$ there holds:
$$
\left| \pa_Y \left((\lb^2\mu \pa_Z^{-1}G_1(Z)-Y)\rho \right) \right| \lesssim \lb^{-4}|Y|^2(1+|Y|)^2\rho \lesssim  e^{-\frac s2}|Y|^2 \rho .
$$
Hence, applying \fref{bd:weightedSobolev}, \fref{bd:poincare}, and splitting in two zones $E=\{|Y|\leq e^{3s/4}\}$ and $E'=[-\lb y^*,+\infty)\backslash E$:
\bee
\left|\int_{-\lb y^s}^{+\infty} \e \left(\lb^2\mu\pa_{Z}^{-1}G_1(Z)-Y\right)\pa_Y\e \rho \right| &\lesssim & e^{-e^s}+ \left|\int_{E} \e^2 \pa_Y \left((\lb^2\mu \pa_Z^{-1}G_1(Z)-Y)\rho \right)\right|\\
&&+ \left|\int_{E'} \e^2 \pa_Y \left((\lb^2\mu \pa_Z^{-1}G_1(Z)-Y)\rho \right)\right| \lesssim e^{-e^s}+e^{-\frac s2}\| \e \|_{H^1_{\rho}}^2.
\eee
For the last term, using \fref{eq:controlenonlocal}, since $|\frac{1}{\lb^2\mu}\pa_Z G_1(Z)|\lesssim \lb^{-4}|Y|$ one has:
\bee
&& \left| \int_{-\lb y^*}^{+\infty} \e \frac{1}{\lb^2\mu}\pa_Z G_1(Z)\pa_Y^{-1}\e \rho dY \right|  \lesssim  \| \e \|_{L^2_\rho} \lb^{-4} \int_{-\lb y^*}^{+\infty} |\e| |Y| \frac{e^{-\frac{3Y^2}{8}}}{(1+|Y|)^{\frac 12}}dY \\
& \lesssim &  \| \e \|_{L^2_\rho} \lb^{-4} \int_{E} |\e| |Y| \frac{e^{-\frac{3Y^2}{8}}}{(1+|Y|)^{\frac 12}}dY+ \| \e \|_{L^2_\rho} \lb^{-4} \int_{E'} |\e| |Y| \frac{e^{-\frac{3Y^2}{8}}}{(1+|Y|)^{\frac 12}}dY \\
& \lesssim &  \| \e \|_{L^2_\rho} \lb^{-4} \| |Y|\e\|_{L^2_\rho} \left(\int_{|Y|\leq e^{3s/4}} \frac{dY}{1+|Y|}\right)^{\frac 12}+O(e^{-e^s})  \lesssim  \| \e \|_{H^1_{\rho}}^2 \lb^{-3}+O(e^{-e^s}) \\
\eee
where we used \fref{bd:poincare} and \fref{bd:weightedSobolev}. Therefore, putting all the above estimates together, as $\lb \approx e^{s/2}$:
$$
\left| \int_{-\lb y^*}^{+\infty} \e \tilde \Ls \e \rho dY \right|\lesssim e^{-e^s}+e^{-\frac s2}\| \e \|_{H^1_{\rho}}^2.
$$

\noindent \textbf{Step 3} \emph{The modulation term}. Recall \fref{def:coefficientsm}. We use the decomposition \fref{id:Mod} and the orthogonality \fref{eq:orthogonalite} to obtain first, with $r_1(Y)=O(Y^4)$, $r_2(Y)=O(Y^4)$ and $r_3(Y)=O(|Y|^3)$:
\bee
\int_{-\lb y^*}^{+\infty} \e \text{Mod} \rho dY &=& \mu^{-4}\lambda^{-8} \int_{-\lb y^*}^{+\infty} \e (m_1r_1+m_2r_2+m_3r_3) \rho dY \\
&& + m_1 \int_{-\lb y^*}^{+\infty} \left((2+Y\pa_Y)\e\right) \e \rho dY+m_3 \int_{-\lb y^*}^{+\infty} \pa_Y \e \e \rho dY.
\eee
For the first line, using Cauchy-Schwarz with \fref{bd:modulation1}, \fref{bd:modulation2}, \fref{bd:modulation3}, \fref{bd:parameterstrap} and \fref{bd:weightedSobolev}:
$$
\left| \mu^{-4}\lambda^{-8} \int_{-\lb y^*}^{+\infty} \e (m_1r_1+m_2r_2+m_3r_3) \rho dY\right|\lesssim \lambda^{-12}\| \e \|_{L^2_\rho}+e^{-2s}\| \e \|_{H^1_\rho}^2.
$$
For the second line, using Poincar\'e \fref{bd:poincare} and \fref{bd:lossymodulation}:
$$
\left| m_1 \int_{-\lb y^*}^{+\infty} \left((2+Y\pa_Y)\e\right) \e \rho dY+m_3 \int_{-\lb y^*}^{+\infty} \pa_Y \e \e \rho dY\right|\lesssim e^{-\frac{13}{8} s}\| \e \|_{H^1_\rho}^2.
$$
The two inequalities above then give:
$$
\left| \int_{-\lb y^*}^{+\infty} \e \text{Mod} \rho dY \right| \lesssim \| \e \|_{L^2_\rho}\lb^{-12}+e^{-\frac{13}{8} s} \| \e \|_{H^1_{\rho}}^2.
$$

\noindent \textbf{Step 4} \emph{The nonlinear term}. A direct $L^{\infty}$ estimate gives
$$
\left| \int_{-\lb y^*}^{+\infty} \e^3 \rho dY \right| \lesssim \| \e \|_{L^{\infty}} \| \e \|_{L^2_\rho}^2 .
$$
For the other nonlinear term one first performs an integration by parts, then a brute force bound for the boundary term, the same estimate as above for the second term, and \fref{bd:poincare}:
\bee
\int_{-\lb y^*}^{+\infty} \e \pa_Y \e \pa_Y^{-1} \e \rho dY & = & -\frac 12 (\e^2\pa_Y^{-1} \e \rho)(-\lb y^*)-\int_{-\lb y^*}^{+\infty} \frac{\e^3}{2}\rho dY-\frac 12 \int_{-\lb y^*}^{+\infty} \e^2 \pa_Y^{-1} \e \pa_Y \rho dY \\
& = & O(e^{-e^s})+O(\| \e \|_{L^{\infty}} \| \e \|_{H^1_{\rho}}^2).
\eee

\noindent \textbf{Step 5} \emph{The error term}. Using the decomposition \fref{id:errorterm}, the orthogonality \fref{eq:orthogonalite} and $|\pa_{ZZ}G_1+\frac 12 -\frac{Z^2}{4}|\lesssim Z^4\approx \lb^{-8} Y^4$ one obtains that:
$$
\left| \int_{-\lb y^*}^{+\infty} \e \frac{1}{\lb^4\mu^2}\pa_{ZZ}G_1(Z) \rho dY \right| = \frac{1}{\lb^4\mu^2} \left|\int_{-\lb y^*}^{+\infty} \e (\pa_{ZZ}G_1+\frac 12 -\frac{Z^2}{4}) \rho dY \right| \lesssim \lb^{-12} \| \e \|_{L^2}.
$$

\noindent \textbf{Step 6} \emph{End of the proof}. Collecting all the estimates of Steps 1, 2, 3, 4 and 5 one finally obtains from \fref{exp:dse2} that:
\bee
\frac{d}{ds}\left(\frac 12 \| \e \|_{L^2_{\rho}}\right) & \leq & -\left(\frac 72 -Ce^{-\eta s}\right) \| \e \|_{L^2_{\rho}}^2-2e^{-\eta s}\| \pa_Y \e \|_{L^2_{\rho}}^2 + C (e^{-\frac s 2}+\| \e \|_{L^{\infty}})\| \e \|_{H^1_{\rho}}^2+C\| \e \|_{L^2_{\rho}}\lb^{-12}+Ce^{-e^s} \\
&\leq& -\left(\frac 72 -Ce^{-\eta s}\right) \| \e \|_{L^2_{\rho}}^2-e^{-\eta s} \| \pa_Y \e \|_{L^2_{\rho}}^2 +C\| \e \|_{L^2_{\rho}}\lb^{-12}+Ce^{-e^s}, \\
\eee
if $\eta$ has been chosen small enough and $s_0$ large enough, where we used \fref{bd:weightedSobolev}.

\end{proof}

%%%%%%%%%%%%%%%%%%%%%%%%%%%%%%%%%%%%%%%%%%%%%%%%%%%%%%%%%%%%%%%
%%%%%%%%%%%%%%%%%%%%%%%%%%%%%%%%%%%%%%%%%%%%%%%%%%%%%%%%%%%%%%%
%%%%%%%%%%%%%%%%%%%%%%%%%%%%%%%%%%%%%%%%%%%%%%%%%%%%%%%%%%%%%%%

\subsection{Analysis outside the blow-up point in the inviscid self-similar zone} \lab{subsec:Z}

This subsection is devoted to the study of the solution outside the blow-up point $y^*(t)$ and we switch to the $Z$ variable \fref{def:renormalisationpara2}. The aim is to find decay for $u$, which receives information from the boundaries $Z=-\pi-a$ and $Z=0$. We first explain the linear estimate which explains the choice of the weight $w$ and then prove full energy estimates. In view of the decomposition \fref{id:decomposition vp} and \fref{eq:F1}, we rewrite \fref{eq:F} as:
\be \lab{eq:u}
\left\{ 
\begin{array}{l l }
 u_s+\mH u -\frac{1}{\lb^4\mu^2}\pa_{ZZ}u+\tH u +NL+\psi=0,\\
u(s,-(\pi+a))=-G_1(-(\pi+a)),
\end{array}\right.
\ee
where the leading order linearised operator is
\be \lab{eq:def mathcalH}
\mH u := \mathcal T\pa_Z u+Vu+\pa_Z^{-1}u\pa_ZG_1,
\ee
with the transport and the potential term being defined by
\be \lab{eq:def T}
\mathcal T(Z):=\left(-\frac Z2 +\pa_Z^{-1}G_1\right)=\left\{ \begin{array}{l l l} 
-\left(\frac Z2 +\frac \pi 2\right)  & \text{for} \  Z\leq -\pi,\\
\frac 12 \sin Z  & \text{for} \ -\pi\leq Z \leq \pi, \\
-\left(\frac Z2 -\frac \pi 2\right)  & \text{for} \  \pi\leq Z,
\end{array} \right.
\ee
\be \lab{eq:def V}
V(Z) :=1-2G_1(Z)= \left\{ \begin{array}{l l l} 
1  & \text{for} \  Z\leq -\pi,\\
-\cos Z  & \text{for} \ -\pi\leq Z \leq \pi, \\
1 & \text{for} \  \pi\leq Z,
\end{array} \right.
\ee
the small linear, the nonlinear term and the error term are given by:
\be \lab{eq:def tH}
\tH u:=m_1(2-Z\pa_Z)u-m_2Z\pa_Z u+m_3'\pa_Z u, \qquad m_3'=\frac{m_3}{\lambda^2 \mu},
\ee
where $m_1$, $m_2$ and $m_3$ are defined in \fref{def:coefficientsm}, and
$$
NL:=-u^2+\pa_Z^{-1}u\pa_Z u,
$$
$$
\psi (s,Z):=-\frac{1}{\lb^4\mu^2}\pa_{ZZ}G_1(Z)+m_1(2-Z\pa_Z)G_1(Z)-m_2Z\pa_Z G_1(Z)+m_3'\pa_Z G_1(Z).
$$
Thanks to \fref{bd:lossymodulation} the parameters $m_1$, $m_2$ and $m_3'$ satisfy:
\be \lab{bd:d}
e^{2s}\left|m_1-\frac{1}{4\lambda^4\mu^2}\right|+|m_2|+e^s |m_3'|\leq e^{-\frac{13}{8} s}.
\ee

%%%%%%%%%%%%%%%%%%%%%%%%%%%%%%%%%%%%%%%%%%%%%%%%%%%%%%%%%%%%%%%
%%%%%%%%%%%%%%%%%%%%%%%%%%%%%%%%%%%%%%%%%%%%%%%%%%%%%%%%%%%%%%%

\subsubsection{Linear analysis} \lab{subsubsec:lin}

We claim that the dynamics of Equation \fref{eq:u} is driven to leading order by the transport and potential terms, and that the nonlocal, viscosity and nonlinear terms are negligible. From a direct check, the eigenvalue problem:
$$
\mathcal T\pa_Z \phi_\beta +V\phi_\beta =\beta \phi_\beta
$$
admits a solution for all $\beta \in \mathbb R$ under the form:
\be \lab{eq:def phinu}
\phi_\beta (Z):= \left\{ \begin{array}{l l l} \phi_\beta^{\text{int}}(Z)&  \text{for} \ Z\in (-\pi,\pi)\backslash \{ 0\}, \\ \phi_\beta^{\text{ext}}(Z) &\text{for} \ Z\in (-\infty,-\pi )\cup (\pi,+\infty), \ea \right. \ \ \ V\phi_\beta +\mathcal T\pa_Z \phi_\beta=\beta \phi_\beta.
\ee
where
$$
\phi_\beta^{\text{int}}=\left(\frac{1-\cos (Z)}{1+\cos (Z)} \right)^{\beta} (\sin Z)^2, \ \  \phi_\beta^{\text{ext}} (Z)=
\left\{ \begin{array}{l l l} 
\left(-(Z+\pi) \right)^{2(1-\beta)} & \text{for} \  Z<-\pi,\\
\left(Z-\pi \right)^{2(1-\beta)} & \text{for} \ Z> \pi.
\end{array} \right.
$$
Note that one has $\phi_\beta (Z)\sim Z^{2(1+\beta)}$ as $Z \rightarrow 0$ and $\phi_\beta (\pi+Z)\sim |Z|^{2(1-\beta)}$ as $Z \rightarrow 0$. The reduced operator $\mathcal T\partial_Z+V$ satisfies the following comparison-type $L^{\infty}$ weighted bound:
$$
\left\Vert \frac{e^{-s(\mathcal T\pa_Z+V)}v_0}{\phi_\beta} \right\Vert_{L^{\infty}}\leq e^{-\beta s} \left\Vert \frac{v_0}{\phi_\beta} \right\Vert_{L^{\infty}}
$$
which can be showed by differentiating along the characteristics. The above bound shows how cancellations near the origin for $u_0$ are crucial for decay since $\phi_\beta$ cancels at the origin for positive $\beta$. Our aim for the full linear problem is to perform a weighted Sobolev energy estimate which mimics the above estimate. We will modify the weight $1/\phi_\beta$ according to three principles: 1) any multiplication by a weight which is decreasing along the vector field $\mathcal T(Z)\pa_Z$ preserves the spectral gap estimate, 2) the nonlocal part can be treated as a perturbation of the transport and potential terms, 3) the viscosity is negligible if one is sufficiently away from the origin. These are the reasons behind the specific choice of $w$ in \fref{eq:def w}. The exponent $1/2$ for the underlying eigenfunction $\phi_{1/2}$ is made optimising two constraints: to optimise the decay in the above inequality, and to minimise the size of the boundary terms in the Lemma below. We claim the following decay, at the linear level, of a Lyapunov functional with weight $w$. We state it on the left of the origin but the analogue holds true on the right as well.

\begin{lemma} \lab{lem:exteriorlinear}

Let $\lb$, $\mu$ and $a$ satisfy \fref{bd:parameterstrap} and $\nu>0$. Assume that $u$ solves
\be \lab{eq:lineairemainorder}
u_s+\mH u-\frac{1}{\lb^4\mu^2}\pa_{ZZ}u=0.
\ee
Let $Z_1:=-(\pi+a)$ and $Z_2:=-Me^{-s}$. Then for any $K>0$, for $M>0$ large enough depending on $K$, and $s_0$ large enough depending on $K,M$, one has the estimate:
\bee
&& \frac{d}{ds} \left(\frac 12 \int_{Z_1}^{Z_2} u^2 wdZ \right) + \left(\frac 12 -\frac{\nu}{4} \right)\int_{Z_1}^{Z_2} u^2 wdZ+\frac{1}{\lb^4\mu^2}\int_{Z_1}^{Z_2} |\pa_Z u|^2wdZ  \\
&\leq & C(K,M) e^{6s}u^2(Z_2)+C(K,M)e^{4s}|\pa_Z u|^2(Z_2)+u^2(Z_1)\left(e^{-\left(\frac 12-\nu\right)s}+|a_s|\right) \\
&&  +C|\pa_Zu|^2(Z_1)e^{-2s}+\frac{Ce^{2s}}{M^2} \left(\int_{Z_2}^{0} |u|dZ\right)\left(\int_{Z_1}^{Z_2} u^2 wdZ \right)^{\frac 12}.
\eee

\end{lemma}

\begin{proof}

One computes first the following identity
\be \lab{eq:interw}
\frac{d}{ds} \left(\frac 12 \int_{Z_1}^{Z_2} u^2 w dZ \right) = \int_{Z_1}^{Z_2} uu_sw dZ+\frac 12 \int_{Z_1}^{Z_2} u^2 w_s dZ+\frac{a_s}{2}(u^2w)(Z_1)+Me^{-s}(u^2w)(Z_2).
\ee
One computes from \fref{eq:def w}:
\be \lab{eq:interw2}
 \int_{Z_1}^{Z_2} u^2 w_s dZ=-\frac 1 s \int_{Z_1}^{Z_2} u^2 w q(Z) dZ\leq 0,
\ee
and we recall that
$$
u_s=-Vu-\mathcal T\pa_Z u +\frac 12\left(\int_0^Z u(s,\tilde Z)d\tilde Z \right)\sin Z \mathds 1_{-\pi \leq Z\leq \pi} +\frac{1}{\lb^4\mu^2}\pa_{ZZ}u.
$$
\textbf{Step 1} \emph{The potential, transport and dissipative effects}. Integrating by parts one finds:
$$
-\int_{Z_1}^{Z_2} u\mathcal T\pa_Z u w dZ =\frac 12 (u^2\mathcal Tw)(Z_1)-\frac 12 (u^2\mathcal Tw)(Z_2)+\int_{Z_1}^{Z_2} u^2 \frac12 \pa_Z(\mathcal Tw) dZ.
$$
One then computes that for $-\pi < Z < 0$, from \fref{eq:def T} and \fref{eq:def phinu}:
\bee
&& \frac 12 \pa_Z (\mathcal Tw) = \frac 14 \pa_Z \left( -\frac{1+\cos Z}{(1-\cos Z)\sin^4(Z)} \frac{4(\pi+Z)^3}{s^{q(Z)}}\right) = \frac 14 \pa_Z \left(-\frac{1}{\phi_{\frac 12}^2} 4(\pi+Z)^3 \frac{1}{s^{q(Z)}}\right) \\
&=& -w \frac{\frac 12 \sin Z \pa_Z\phi_{\frac 12}}{\phi_{\frac 12}}-\frac{1}{\phi_{\frac 12}^2}\frac{3(\pi+Z)^2}{s^{q(Z)}}+\frac{1}{\phi_{\frac 12}^2} \frac{(\pi+Z)^3}{s^{q(Z)}}\ln (s)\pa_Z q  \leq -w \frac{\frac 12 \sin Z \pa_Z\phi_{\frac 12}}{\phi_{\frac 12}},
\eee
and that for $Z\leq -\pi$:
$$
\frac 12 \pa_Z (\mathcal Tw)=-\frac 14 w.
$$
Therefore, one $(-\pi,0)$ one has from \fref{eq:def phinu} the inequality behind the inviscid spectral gap:
$$
-Vu^2 w+u^2\frac 12 \pa_Z (\mathcal Tw)\leq -u^2 w \frac{1}{\phi_{\frac 12}}\left(V\phi_{\frac 12}+\mathcal T \pa_Z\phi_{\frac 12} \right)=-\frac 12 u^2 w,
$$
and on $(-\infty ,-\pi]$ one has from \fref{eq:def V}:
$$
-Vu^2 w+u^2\frac 12 \pa_Z (\mathcal Tw)=-u^2w-\frac 14 w u^2=-\frac 54 w u^2.
$$
Therefore, from the two inequalities above, on the whole ray $(-\infty,0)$ there holds:
$$
-Vu^2 w+u^2\frac 12 \pa_Z (\mathcal Tw)\leq -\frac 12 wu^2.
$$
That is why one has for the part involving the operator $\mathcal T\pa_Z+V$:
$$
\int_{Z_1}^{Z_2} u\left(-Vu-\mathcal T\pa_Z u\right)wdZ  \leq -\frac 12 \int_{Z_1}^{Z_2} u^2 w dZ +\frac 12 (u^2\mathcal Tw)(Z_1)-\frac 12 (u^2\mathcal Tw)(Z_2).
$$
We now turn to the dissipative effects. Integrating by parts
$$
\int_{Z_1}^{Z_2} u \pa_{ZZ}u w = (\frac 12 u^2\pa_Z w-u\pa_Zuw)(Z_1)-(\frac 12 u^2\pa_Z w-u\pa_Zuw)(Z_2)-\int_{Z_1}^{Z_2} |\pa_Z u|^2w +\frac 12 \int_{Z_1}^{Z_2} u^2 w.
$$
The function $\pa_{ZZ}w$, from \fref{eq:def w}, is supported in $(-\pi ,0)$ where one has the bound:
$$
|\pa_{ZZ}w|\lesssim |Z|^{-7} \pa_{ZZ}(s^{-q(Z)})+|Z|^{-8}\pa_Z (s^{-q(Z)})+|Z|^{-9}s^{-q(Z)}\lesssim |Z|^{-9}s^{-q(Z)}(1+Z^2 \ln^2 (s)+|Z| \ln (s))
$$
so that for $s$ large enough depending on $M$, for $Z\leq -Me^{-s}$:
$$
|e^{-2s} \pa_{ZZ}w| \lesssim \frac{w}{M^2}.
$$
From \fref{bd:parameterstrap}, the above identity becomes for $s$ large enough (since $\pa_Zw\geq 0$ near the origin):
\bee
 \frac{1}{\lb^4\mu^2} \int_{Z_1}^{Z_2} u \pa_{ZZ}u w dZ &\leq & Ce^{-2s}|\frac 12 u^2\pa_Z w-u\pa_Zuw|(Z_1)+Ce^{-2s}|u\pa_Zuw|(Z_2) \\
&&-\frac{1}{\lb^4\mu^2}\int_{Z_1}^{Z_2} |\pa_Z u|^2wdZ+\frac{C(K)}{M^2} \int_{Z_1}^{Z_2} u^2 wdZ.
\eee
At this point one has proved that for the operator $\mathcal T\pa_Z +V-\lb^{-4}\mu^{-2}\pa_{ZZ}$:
\bea
\non && \int_{Z_1}^{Z_2} u\left(-Vu-\mathcal T\pa_Z u +\frac{1}{\lb^4\mu^2}\pa_{ZZ}u \right)w+\frac 12 \int_{Z_1}^{Z_2} u^2 w_s+\frac{a_s}{2}(u^2w)(Z_1)+Me^{-s}(u^2w)(Z_2) \\
\non &\leq &-\frac 12 \int_{Z_1}^{Z_2} u^2 w-\frac{1}{\lb^4\mu^2}\int_{Z_1}^{Z_2} |\pa_Z u|^2w +\frac{C(K)}{M^2} \int_{Z_1}^{Z_2} u^2w+\frac{a_s}{2}(u^2w)(Z_1)+Me^{-s}(u^2w)(Z_2)\\
\non && +\frac 12 (u^2\mathcal Tw)(Z_1)-\frac 12 (u^2\mathcal Tw)(Z_2) +Ce^{-2s}|\frac 12 u^2\pa_Z w-u\pa_Zuw|(Z_1)+C(K,M) e^{-2s}|u\pa_Zuw|(Z_2)\\
\non &\leq &-\frac 12 \int_{Z_1}^{Z_2} u^2 w-\frac{1}{\lb^4\mu^2}\int_{Z_1}^{Z_2} |\pa_Z u|^2w +\frac{C(K)}{M^2} \int_{Z_1}^{Z_2} u^2w+C(K,M) e^{6s}u^2(Z_2)\\
\lab{bd:mainout}  && +C(K,M) e^{4s}|\pa_Z u|^2(Z_2)+Cu^2(Z_1)\left(e^{-\left(\frac 12-\nu\right)s}+|a_s|\right)+C|\pa_Zu|^2(Z_1)e^{-2s},
\eea
where we used \fref{bd:parameterstrap}, the fact that $|w(Z_2)|\lesssim Z_2^{-7}\lesssim e^{7s}$, $|w(Z_1)|\lesssim 1$, $|\mathcal T(Z_1)|\lesssim |\pi+Z_1|\lesssim |a|\lesssim e^{-(1/2-\nu)s}$, $|\mathcal T(Z_2)|\lesssim |Z_2|\lesssim e^{-s}$, $|\pa_Z w(Z_1)|\lesssim 1$.\\

\noindent \textbf{Step 2} \emph{The nonlocal term}. Using Cauchy-Schwarz one has for $Z\in (-\pi,0)$:
\be \lab{bd:paz-1u}
\left| \int_0^Z u(s,\tilde Z)d\tilde Z \right|\leq \int_{Z_2}^{0}|u|dZ+\left(\int_{Z_1}^{Z_2} u^2 wdZ \right)^{\frac 12} \left(\int_Z^0 w^{-1}(s,\tilde Z)d\tilde Z \right)^{\frac 12}.
\ee
One computes that for $Z\in (-\pi,0)$:
$$
|w^{-1}(s,Z)|\lesssim |Z|^7 s^{q( Z)}=| Z|^7 e^{q(Z)\ln s}
$$
from what we infer from the assumptions on $q$ in Subsection \ref{subsec:wq}:
\be \lab{bd:paZ-1w-1}
\int_Z^0 w^{-1}(s,\tilde Z)d\tilde Z\lesssim \int_Z^0 |\tilde Z|^7 e^{q(\tilde Z)\ln s}d\tilde Z \lesssim |Z|^7 \int_Z^0 \frac{1}{\ln s \pa_Z q}\frac{d}{dZ}(e^{q(Z)\ln s})\lesssim \frac{|Z|^7}{|\pi+Z|\ln s}e^{q(Z)\ln s}.
\ee
Therefore:
$$
\left(\int_Z^0 w^{-1}(s,\tilde Z )d\tilde Z\right) \sin^2 Z \lesssim \frac{|Z|^{9}|\pi+Z|}{\ln s}s^{q(Z) }
$$
which produces
$$
\int_{Z_1}^{Z_2} \left(\int_Z^0 w^{-1}(s,\tilde Z)d\tilde Z\right)\sin^2 Z\mathds 1_{-\pi \leq Z \leq 0} wdZ \lesssim \int_{-\pi}^{0} \frac{|Z|^{9}|\pi+Z|}{\ln s}s^{q(Z) } \frac{1}{|Z|^7}\frac{1}{s^{q(Z)}}dZ\lesssim \frac{1}{\ln s}.
$$
One also has
$$
\int_{Z_2}^{Z_1} \sin^2 Z\mathds 1_{0\leq Z \leq \pi} wdZ \approx \int_{Z_2}^{Z_1} \frac{dZ}{|Z|^5s^{q(Z)}} \lesssim \frac{e^{4s}}{M^4}.
$$
Thus the contribution of the nonlocal term is estimated as follows:
\bee
&& \left| \int_{Z_2}^{Z_1} u \left(\int_0^Z u(s,\tilde Z)d\tilde Z\right)\sin Z \mathds 1_{-\pi \leq Z \leq 0} wdZ \right| \\
& \lesssim & \left| \int_{Z_2}^{Z_1} u \left(\int_{Z_2}^{0} u \right)\sin Z \mathds 1_{-\pi \leq Z \leq 0} wdZ \right|+\left| \int_{Z_2}^{Z_1} u \left(\int_{Z_2}^Z u\right)\sin Z \mathds 1_{-\pi\leq Z \leq 0} wdZ \right|  \\
&\lesssim & \frac{1}{\ln s}\int_{Z_2}^{Z_1} u^2 wdZ+\frac{e^{2s}}{M^2} \left(\int_{Z_2}^{0} |u|dZ\right)\left(\int_{Z_2}^{Z_1} u^2 wdZ \right)^{\frac 12}.
\eee
\textbf{Step 3} \emph{End of the proof}. The above identity, \fref{bd:mainout},  and \fref{eq:interw} finally yield
\bee
&& \frac{d}{ds} \left(\frac 12 \int_{Z_1}^{Z_2} u^2 w \right) \\
&\leq & -\frac 12 \int_{Z_1}^{Z_2} u^2 w-\frac{1}{\lb^4\mu^2}\int_{Z_1}^{Z_2} |\pa_Z u|^2w +\frac{C(K)}{M^2} \int_{Z_1}^{Z_2} u^2w+C(K,M) e^{6s}u^2(Z_2)+C(K,M) e^{4s}|\pa_Z u|^2(Z_2)\\
  && +Cu^2(Z_1)\left(e^{-\left(\frac 12-\nu\right)s}+|a_s|\right)+C|\pa_Zu|^2(Z_1)e^{-2s}+ \frac{1}{\ln s}\int_{Z_2}^{Z_1} u^2 w+\frac{e^{2s}}{M^2} \left(\int_{Z_2}^{0} |u|\right)\left(\int_{Z_2}^{Z_1} u^2 w \right)^{\frac 12},\\
 &\leq & \left(-\frac 12 +\frac{C(K)}{M^2}+\frac{C}{\ln s}\right)\int_{Z_1}^{Z_2} u^2 w +C(K,M)e^{6s}u^2(Z_2)+C(K,M) e^{4s}|\pa_Z u|^2(Z_2) \\
&&  +Cu^2(Z_1)\left(e^{-\left(\frac 12-\nu\right)s}+|a_s|\right)+C|\pa_Zu|^2(Z_1)e^{-2s}+\frac{Ce^{2s}}{M^2} \left(\int_{Z_2}^0 |u|\right)\left(\int_{Z_1}^{Z_2} u^2 w \right)^{\frac 12}-\frac{1}{\lb^4\mu^2}\int_{Z_1}^{Z_2} |\pa_Z u|^2w,
\eee
which ends the proof of the Lemma for $M$ and $s_0$ large enough.

\end{proof}

%%%%%%%%%%%%%%%%%%%%%%%%%%%%%%%%%%%%%%%%%%%%%%%%%%%%%%%%%%%%%%%
%%%%%%%%%%%%%%%%%%%%%%%%%%%%%%%%%%%%%%%%%%%%%%%%%%%%%%%%%%%%%%%

\subsubsection{Exterior Lyapunov estimates}

We now study the functional \fref{lem:exteriorlinear} for the the full problem. First, let us estimate the function at the boundaries, $Z_1=-\pi-a$ and $Z_2=-Me^{-s}$. From \fref{bd:etrap} and Sobolev near the maximum:
\be \lab{eq:estimationuz2}
u^2(Z_2)=\e^2(-M\lb^2 \mu e^{-s})\leq C \| \e \|_{H^2(|Y|\leq M^2)}\leq Ce^{-(7-2\nu')s}, \ \ (\pa_Z u)^2(Z_2)\leq Ce^{-(5-2\nu')s}.
\ee
From the boundary condition \fref{eq:F}, the decomposition \fref{id:decomposition vp}, \fref{id:F1} and \fref{bd:parameterstrap}, at the origin in original variables:
\be \lab{eq:estimationuz1}
u^2(Z_1)=G_1^2(-\pi-a)\leq C a^4 \leq C e^{-\left(2-8\nu \right)s}.
\ee
Finally, from \fref{bd:origin bootstrap2}, \fref{id:decomposition vp} and \fref{bd:parameterstrap}:
\be \lab{eq:estimationpazuz1}
|\pa_Z u(Z_1)|\leq |\pa_Z F(Z_1)| +|\pa_Z G_1(Z_1)|\leq \lb^{-1} \mu |\pa_y \xi (0)|+C|a|\leq Ce^{-\left(\frac 12 -2\nu \right)s}.
\ee
One has the following energy estimate for the function in $Z$ variable outside the maximum.

\begin{lemma}[Exterior Lyapunov functional on the left] \lab{lem:exteleft1}
There exists $\nu^*>0$ such that for any $K>0$, $0<\nu,\nu'\leq \nu^*$, an $M^*>0$ exists such that for $M\geq M^*$, there exists $s_0^*$ and $C(K,M)$ such that if the solution is trapped on $[s_0,s_1]$ with $s_0\geq s_0^*$ the following inequality holds true:
\bea
\lab{eq:exteriorenergyidentity1}  &&\frac{d}{ds} \left(\frac 12 \int_{Z_1}^{Z_2} u^2 wdZ \right)+\left(\frac 12-\frac{\nu}{2}  \right)\int_{Z_1}^{Z_2} u^2 wdZ   \\
\non &\leq & C(K,M) \left(e^{6s}u^2(Z_2)+e^{4s}|\pa_Z u|^2(Z_2)+ \left( \int_{Z_1}^{Z_2} u^2 wdZ \right)^{\frac 12} e^{-\frac 58 s}+e^{-(2+\frac 16)s} \right)
\eea

\end{lemma}

\begin{proof}

One first computes from \fref{eq:u} the identity
\be \lab{eq:expout}
\frac{d}{ds} \left(\frac 12 \int_{Z_1}^{Z_2} u^2 w \right) = \int_{Z_1}^{Z_2} u(-\mH u +\frac{\pa_{ZZ}u}{\lb^4\mu^2}-\tH u -NL-\psi)w+ \int_{Z_1}^{Z_2} \frac{u^2}{2} w_s+\frac{a_s}{2}(u^2w)(Z_1)+\frac{Me^{-s}}{2}(u^2w)(Z_2).
\ee

\noindent \textbf{Step 1} \emph{The leading order linear terms}. From \fref{eq:estimationuz1}, \fref{bd:as} and \fref{eq:estimationpazuz1}:
$$
u^2(Z_1)\left(e^{-\left(\frac 12-\nu\right)s}+|a_s|\right)+|\pa_Zu|^2(Z_1)e^{-2s}\lesssim e^{(\frac 52 -10\nu)s}+e^{(5-4\nu)s}\lesssim e^{(\frac 52 -10\nu)s}.
$$
From \fref{def:renormalisationpara2}, \fref{bd:parameterstrap} and \fref{bd:etrap}, as $\lb^2e^{-s}\mu M\approx 1$:
$$
e^{2s} \int_{Z_2}^0 |u|dZ=\frac{e^{2s}}{\lb^2\mu}\int_{-\lb^2e^{-s}\mu M}^0 |\e |dY\lesssim e^s \| \e \|_{L^2_\rho}\lesssim e^{-\frac 52 s}.
$$
We now apply Lemma \ref{lem:exteriorlinear} and inject the two above inequalities:
\bea
\non && \int_{Z_1}^{Z_2} u(-\mH u +\frac{1}{\lb^4\mu^2}\pa_{ZZ}u)w+\frac 12 \int_{Z_1}^{Z_2} u^2 w_s+\frac{a_s}{2}(u^2w)(Z_1)+\frac{Me^{-s}}{2}(u^2w)(Z_2) \\
\non &\leq & \left(-\frac 12 +\frac{\nu}{4} \right)\int_{Z_1}^{Z_2} u^2 w +C(K,M) e^{6s}u^2(Z_2)+C(K,M) e^{4s}|\pa_Z u|^2(Z_2)+e^{-(\frac 52 -10\nu)s} \\
\lab{bd:expout1} &&  +e^{-\frac 52 s}\left(\int_{Z_1}^{Z_2} u^2 w \right)^{\frac 12}-\frac{1}{\lb^4\mu^2}\int_{Z_1}^{Z_2} |\pa_Z u|^2w.
\eea

\noindent \textbf{Step 2} \emph{The small linear term}. Recall \fref{eq:def tH}, then
$$
- \int_{Z_1}^{Z_2} u\tH uw dZ=- \int_{Z_1}^{Z_2} u\left(m_3'\pa_Z u+m_1(2-Z\pa_Z)u-m_2 Z\pa_Z u \right)wdZ
$$
Integrating by parts, one has:
$$
\int_{Z_1}^{Z_2} u \pa_Z u w dZ=\frac 12(u^2w)(Z_2)-\frac 12 (u^2w)(Z_1)-\frac 12 \int_{Z_1}^{Z_2}u^2\pa_Z w dZ.
$$
One has that $\pa_Z w$ is supported on $(-\pi,0)$, with for $|Z|\gtrsim e^{-s}$:
$$
|\pa_Z w|\lesssim |Z|^8s^{q(Z)}(1+|Z|\ln s)\lesssim e^s w.
$$
Therefore, since $w(Z_1)\lesssim 1$ and $w(Z_2)\lesssim e^{7s}$, using \fref{bd:d}:
$$
\left| \int_{Z_1}^{Z_2} u m_3'\pa_Z u w dZ\right|\lesssim e^{\frac{35}{8} s}u^2(Z_2)+e^{-\frac{21}{8} s} u^2(Z_1)+e^{-\frac{13}{8} s} \int_{Z_1}^{Z_2} u^2wdZ.
$$
The same strategy applies for the other term, and as$|\pa_Z (Zw)|\lesssim e^{s/2}w$, this gives using \fref{bd:d}:
$$
\left| \int_{Z_1}^{Z_2} u(m_1(2-Z\pa_Z)u-m_2 Z\pa_Z u)w \right| \lesssim  e^{\frac{35}{8} s} u^2(Z_2)+e^{-\frac{13}{8} s} u^2(Z_1)+e^{-\frac 98 s} \int_{Z_1}^{Z_2} u^2w.
$$
In conclusion one has for the small linear term, using \fref{eq:estimationuz1}, \fref{bd:etrap2} and \fref{eq:estimationuz2} as $0<\nu'\ll \nu$:
\bea \non
\left| \int_{Z_1}^{Z_2} u\tH uw dZ \right| & \lesssim &  e^{\frac{35}{8} s} u^2(Z_2)+e^{-\frac{13}{8} s} u^2(Z_1)+e^{-\frac 98s} \int_{Z_1}^{Z_2} u^2w \\
\lab{eq:bdtildeH} & \lesssim &  e^{5 s} u^2(Z_2)+e^{-\left(\frac{29}{8}-8\nu\right) s}+e^{-s} \int_{Z_1}^{Z_2} u^2w .
\eea

\noindent \textbf{Step 3} \emph{The nonlinear term}. For the nonlinear term one recalls the identity:
$$
\int_{Z_1}^{Z_2} uNLw dZ =\int_{Z_1}^{Z_2} u(-u^2+\pa_Z^{-1}u \pa_Z u)w dZ.
$$
The first term is estimated in brute force:
$$
\left| \int_{Z_1}^{Z_2} u^3wdZ \right| \leq \| u\|_{L^{\infty}} \int_{Z_1}^{Z_2} u^2wdZ.
$$
For the second, we integrate by parts and use the brute force estimate $|\pa_Z^{-1}u|\leq |Z|\| u \|_{L^{\infty}}$:
\bee
\int_{Z_1}^{Z_2} u \pa_Z^{-1}u\pa_Z u w &= &\frac 12 ( u^2 \pa_Z^{-1}u w)(Z_2)-\frac 12 ( u^2 \pa_Z^{-1}u w)(Z_1)-\frac 12 \int_{Z_1}^{Z_2} u^3w -\frac 12 \int_{Z_1}^{Z_2} \pa_Z^{-1}u \pa_Z w u^2 \\
\eee
since $|Z\pa_Z w|\lesssim \ln (s) w$. In conclusion, the contribution of the nonlinear term is, using \fref{bd:weightedSobolev} and \fref{eq:estimationuz1} for $s_0$ large enough:
\bea
\non \left| \int_{Z_1}^{Z_2} uNLw dZ  \right| & \lesssim & \| u \|_{L^{\infty}}e^{6s}u^2(Z_2)+\| u\|_{L^{\infty}} u^2(Z_1)+\ln (s)\| u \|_{L^{\infty}} \int_{Z_1}^{Z_2} w u^2dZ\\
\lab{bd:expout2} & \lesssim & e^{(6-\frac 18)s}u^2(Z_2)+e^{-(2+\frac 14-9\nu)s}+\frac{\nu}{4} \int_{Z_1}^{Z_2} w u^2dZ.
\eea

\noindent \textbf{Step 4} \emph{The error term}. Recall \fref{def:coefficientsm}. One has that $\psi$ is supported on $[-\pi,\pi]$, with the estimate from \fref{id:G1}
\bee
|\psi (s,Z)| & = & \left|-\frac{1}{\lb^4\mu^2}\pa_{ZZ}G_1(Z)+m_1 (2-Z\pa_Z)G_1(Z)-m_2Z\pa_Z G_1(Z)+m_3'\pa_Z G_1(Z)\right|\\
&\lesssim & \left| m_1 +\frac{1}{4\lb^4 \mu^2}\right| +Z^2\left(\frac{1}{\lb^4}+ |m_1|+\left| m_2 \right| \right)+|m_3'||Z|.
\eee
Since $w\lesssim |Z|^{-7}$ one has, using \fref{bd:d} and \fref{bd:parameterstrap}, for $s_0$ large:
\bee
\int_{Z_1}^{Z_2} \psi^2 w dZ & \lesssim  & e^{6s}\left| m_1 +\frac{1}{4\lb^4 \mu^2}\right|^2 +e^{2s} \left(\frac{1}{\lb^4}+ \left| \frac{\lb_s}{\lb}-\frac 12\right|+\left|\frac{\mu_s}{\mu}\right| \right)^2+e^{4s}|m_3'|^2\\
& \lesssim  & e^{-\left(\frac{13}{4}-2\right)s}+C(K)e^{-2s}\lesssim e^{-\frac 54 s}.
\eee
By Cauchy Schwarz, one has proved that for the error term:
\be \lab{bd:expout3} 
\left| \int_{Z_1}^{Z_2} u\psi w \right| \lesssim \left( \int_{Z_1}^{Z_2} u^2 w \right)^{\frac 12} e^{-\frac 58 s}.
\ee

\noindent \textbf{Step 5} \emph{End of the proof}. Collecting the previous estimates \fref{bd:expout1}, \fref{eq:bdtildeH}, \fref{bd:expout2} and \fref{bd:expout3} and injecting them in \fref{eq:expout} yields the desired energy estimate \fref{eq:exteriorenergyidentity1}.

\end{proof}

A similar energy estimate also holds for the adapted derivative of $u$, $A\pa_Z u$ where $A$ is defined by \fref{eq:def A}, at the left of the origin. This vector field is chosen because its commutator with $\mathcal T$ vanishes for $Z\in [-\pi/2,\pi/2]$, and has a good sign for $|Z|>\pi /2$. Before stating the estimate, let us investigate the size of the boundary terms. From Sobolev embedding and \fref{bd:etrap}, since $|A|\sim |Z|$ and $|\pa_Z A|\lesssim 1$ near the origin:
\be \lab{bd:apazuboundary1}
|A\pa_Z u|^2(Z_2)\leq |Y\pa_Y \e|^2(-M\lb^2 \mu e^{-s})\leq C \| \e \|_{H^2(|Y|\leq M^2)}^2 \leq C(K,M) e^{-(7-2\nu')s},
\ee
\bea
\non (\pa_Z (A\pa_Z u))^2(Z_2) & \leq & (|\pa_Z u|^2+|Z\pa_{ZZ}u|^2)(Z_2)\\
 \lab{bd:apazuboundary2} &\leq & \lb^4\mu^2(|\pa_Y \e|^2+|Y\pa_{YY}\e|^2)(-M\lb^2 \mu e^{-s})\leq C(K,M) e^{-(5-2\nu')s}.
\eea
Since $A= 1$ near $-\pi$, from \fref{eq:estimationpazuz1}:
\be \lab{bd:apazuboundary3}
|A\pa_Z u|^2(Z_1)\leq C |\pa_Z u|^2(Z_1)\leq Ce^{-(1-4\nu)s}.
\ee
Now we write $\pa_Z (A\pa_Z u)=A\pa_{ZZ}u$ since $|\pa_Z A(-\pi-a)|=0$. Since $\pa_{yy}\xi (0)=0$ from the boundary condition in the equation \fref{1DPrandtl}, \fref{id:decomposition vp} and \fref{bd:parameterstrap} imply:
\be \lab{bd:apazuboundary4}
|\pa_Z (A\pa_Z u)(Z_1) |=|\pa_{ZZ}u (Z_1)|=|\pa_{ZZ}(F-G_1)(Z_1)|\leq |\lb^2\mu^2\pa_{yy}\xi (0)|+|\pa_{ZZ}G_1|(-\pi-a)\leq \frac 12.
\ee
We perform the same weighted energy estimate outside the maximum for $A\pa_Z u$ as we did for $u$.

\begin{lemma}[Exterior Lyapunov functional on the left for the derivative]  \lab{lem:exteleft2}

Let $Z_1=-\pi-a$, $Z_2=-Me^{-s}$ and $v=A\pa_Z u$. There exists $\nu^*>0$ such that for any $K>0$ and $0<\nu,\nu'<\nu^*$, there exists $M^*>0$ for any $M\geq M^*$, there exists $s_0^*$ such that if the solution is trapped on $[s_0,s_1]$ with $s_0\geq s_0^*$:
\be \lab{eq:exteriorenergyidentity2}
 \frac{d}{ds} \left(\frac 12 \int_{Z_1}^{Z_2} v^2 wdZ \right)-\frac \nu2 \int_{Z_1}^{Z_2} v^2 wdZ+\frac{1}{2\lb^4\mu^2}\int_{Z_1}^{Z_2} |\pa_Z v|^2w dZ\leq e^{-\frac 14 s}
\ee

\end{lemma}

\begin{proof}

In this proof, the constant $C$ might depend on $K$ and $M$. One first computes the evolution equation for $v=A\pa_Z u$ from \fref{eq:u}:
\bea 
\non 0&=&v_s+(\mathcal T\pa_Z+V)v+\frac{A\pa_Z\mathcal T-\mathcal T\pa_Z A}{A}v-\frac{1}{\lb^4\mu^2}\left(\pa_{ZZ}v+[A\pa_Z,\pa_{ZZ}]u\right)+\tH v+[A\pa_Z,\tH]u,\\
\lab{eq:v} &&+\tilde{NL}+A\pa_Z\psi+Au\pa_Z G_1+\pa_Z^{-1}uA\pa_{ZZ}G_1
\eea
where
$$
\tilde{NL}=-\left(2u+\pa_Z^{-1}u \frac{\pa_ZA}{A}\right)v+\pa_{Z}^{-1}u \pa_Zv.
$$
First, one has the following identity for the energy estimate:
\be \lab{eq:expv}
\frac{d}{ds} \left(\frac 12 \int_{Z_1}^{Z_2} v^2 wdZ \right) = \int_{Z_1}^{Z_2} vv_swdZ+\frac 12 \int_{Z_1}^{Z_2} v^2 w_sdZ+\frac{a_s}{2}(v^2w)(Z_1)+\frac{Me^{-s}}{2}(v^2w)(Z_2).
\ee

\noindent \textbf{Step 1} \emph{The leading order linear terms}. From \fref{bd:mainout}, injecting \fref{bd:apazuboundary1}, \fref{bd:apazuboundary2}, \fref{bd:apazuboundary3}, \fref{bd:apazuboundary4}:
\bea
 \non && \int_{Z_1}^{Z_2} v\left(-Vv-\mathcal T\pa_Z v +\frac{1}{\lb^4\mu^2}\pa_{ZZ}v \right)w+\frac 12 \int_{Z_1}^{Z_2} u^2 w_s+\frac{a_s}{2}(v^2w)(Z_1)+\frac{Me^{-s}}{2}(v^2w)(Z_2) \\
 \non &\leq &-\frac 12 \int_{Z_1}^{Z_2} v^2 w-\frac{1}{\lb^4\mu^2}\int_{Z_1}^{Z_2} |\pa_Z v|^2w +\frac{\nu}{4} \int_{Z_1}^{Z_2} v^2w+C e^{6s}v^2(Z_2)+Ce^{4s}|\pa_Z v|^2(Z_2)\\
\non && +Cv^2(Z_1)e^{-\left(\frac 12-\nu\right)s}+C|\pa_Zv|^2(Z_1)e^{-2s}\\
\lab{eq:lineairev} &\leq &-\frac 12 \int_{Z_1}^{Z_2} v^2 w-\frac{1}{\lb^4\mu^2}\int_{Z_1}^{Z_2} |\pa_Z v|^2w +\frac{\nu}{4} \int_{Z_1}^{Z_2} v^2w+Ce^{-(1-2\nu')s}.
\eea
Then, for the commutator with $A$ and the transport $\mathcal T$, a direct computation shows, since $A=2\mathcal T$ for $|Z|\leq \pi/2$, and $A=-1$ for $Z\leq -\pi/2$, that for all for $Z\leq 0$:
$$
\frac{A\pa_Z\mathcal T-\mathcal T\pa_Z A}{A}=\pa_Z \mathcal T \mathds 1_{Z\leq -\frac \pi 2}\geq -\frac 12 \mathds 1_{Z\leq -\frac \pi 2}
$$
which implies:
\be \lab{bd:estimationcommutateurAT}
- \int_{Z_1}^{Z_2} v \frac{A\pa_Z \mathcal T-\mathcal T\pa_Z A}{A} v w\leq \frac 12 \int_{Z_1}^{Z_2} v^2w.
\ee

\noindent \textbf{Step 2} \emph{The small linear term and other commutators}. For the small linear term, from \fref{eq:bdtildeH}, injecting \fref{bd:apazuboundary1}, \fref{bd:apazuboundary3} and \fref{bd:etrap2}, for $s_0$ large enough:
\bea
\non  \left| \int_{Z_1}^{Z_2} v\tH vwdZ \right| & \lesssim &e^{\frac{35}{8} s} v^2(Z_2)+e^{-\frac{13}{8} s} v^2(Z_1)+e^{-\frac 98s} \int_{Z_1}^{Z_2} v^2w \\
\lab{bd:commutateurtildeH} & \leq & C(K,M)\left(e^{-(\frac{21}{8} -2\nu')s}+e^{-(\frac{21}{8} -4\nu)s}+e^{-(\frac 98-2\nu)s}\right) \ \leq \ e^{-s}.
\eea
Next, we turn to the commutator with the dissipative term. one has
$$
[A\pa_Z,\pa_{ZZ}]u=\left(- \frac{\pa_{ZZ}A}{A}+\frac{2(\pa_Z A)^2}{A^2}\right)v-2\frac{\pa_Z A}{A}\pa_Z v.
$$
Since, for $Z\geq Me^{-s}$:
$$
\left| \frac{\pa_{ZZ}A}{A}\right|+\left|\frac{(\pa_Z A)^2}{A^2}\right|\leq \frac{C}{Z^2}\leq \frac{Ce^{2s}}{M^2},
$$
one has for the first term that:
$$
\left| \frac{1}{\lb^4\mu^2}\int_{Z_1}^{Z_2} v^2\left(- \frac{\pa_{ZZ}A}{A}v+\frac{2(\pa_Z A)^2}{A^2}\right) wdZ\right| \leq \frac{C(K)}{M^2}  \int_{Z_1}^{Z_2} v^2wdZ.
$$
For the second term, one first integrates by parts:
\bee
-\int_{Z_1}^{Z_2} 2v \frac{\pa_Z A}{A}\pa_Z v wdZ &=&(v^2\frac{\pa_Z A}{A}w)(Z_1)-(v^2\frac{\pa_Z A}{A}w)(Z_2)+\int_{Z_1}^{Z_2} v^2 \pa_Z\left(\frac{\pa_Z A}{A} w\right)dZ\\
&= &-(v^2\frac{\pa_Z A}{A}w)(Z_2)+\int_{Z_1}^{Z_2} v^2 \pa_Z\left(\frac{\pa_Z A}{A} w\right)dZ
\eee
since $\pa_Z A (Z_1)=0$. From a direct inspection:
$$
\left| \pa_Z\left(\frac{\pa_Z A}{A} w\right)\right|\leq \frac{Cw}{Z^2}\leq \frac{Ce^{2s}}{M^2}w.
$$
Therefore:
$$
\left| \frac{1}{\lb^4\mu^2} \int_{Z_1}^{Z_2} 2v \frac{\pa_Z A}{A}\pa_Z v wdZ\right| \leq C e^{6s}v^2(Z_2)+\frac{C(K)}{M^2}\int_{Z_1}^{Z_2} v^2wdZ.
$$
One has proved that for the commutator with the dissipative term, for $M$ large enough depending on $K$, using \fref{bd:apazuboundary1}:
\be \lab{bd:commutateurvpazz}
\left| \frac{1}{\lb^4\mu^2} \int_{Z_1}^{Z_2} v [A\pa_Z,\pa_{ZZ}]u w\right| \leq C e^{6s}v^2(Z_2)+\frac{C(K)}{M^2}\int_{Z_1}^{Z_2} v^2w\leq Ce^{-(1-2\nu')s}+\frac{\nu}{8}\int_{Z_1}^{Z_2} v^2w.
\ee
Next, one computes that the commutator with the small linear term is:
$$
[A\pa_Z,\tilde H]u=\left(-m_3'\frac{\pa_Z A}{A}-m_1\left(1-\frac{Z\pa_Z A}{A}\right)+m_2 \frac{Z\pa_Z A}{A} \right)v.
$$
Since $|\pa_Z A/A|\lesssim 1/Z \lesssim e^s$ for $|Z|\geq Me^{-s}$, this implies using \fref{bd:d}:
\be \lab{bd:commutateurtildeH2}
\left| \int_{Z_1}^{Z_2} v[A\pa_Z,\tilde H]uw \right| \lesssim \left(|m_1|+|m_2|+e^sm_3' \right)\int_{Z_1}^{Z_2} v^2w\lesssim e^{-\frac{13}{8} s}\int_{Z_1}^{Z_2} v^2w.
\ee

\noindent \textbf{Step 3} \emph{The nonlinear term}. Since $|\pa_Z A/A|\lesssim 1/Z$ one has:
$$
\left| \int_{Z_1}^{Z_2} v\left(u+\pa_Z^{-1}u \frac{\pa_ZA}{A}\right)vw dZ \right| \lesssim \| u\|_{L^{\infty}} \int_{Z_1}^{Z_2} v^2wdZ.
$$
For the other term, an integration by parts gives:
\bee
\left| \int_{Z_1}^{Z_2} v\pa_{Z}^{-1}u \pa_ZvwdZ\right| & = & \left| \frac 12 (\pa_Z^{-1}uv^2w)(Z_1)-\frac 12 (\pa_Z^{-1}uv^2w)(Z_2)+\int_{Z_1}^{Z_2} v^2\pa_Z(\pa_Z^{-1}u w)dZ \right| \\
&\lesssim & \|u\|_{L^{\infty}}v^2(Z_1)+\|u\|_{L^{\infty}}e^{6s}v^2(Z_2)+\log (s)\| u\|_{L^{\infty}} \int_{Z_1}^{Z_2} v^2wdZ,
\eee
where we used the fact that $|\pa_Zw|\lesssim \log (s)Z^{-1}w$. One has then showed that for the nonlinear term, using \fref{bd:weightedSobolev}, \fref{bd:apazuboundary1} and \fref{bd:apazuboundary3}, as $0<\nu'\ll \nu$:
\bea
\non \left| \int_{Z_1}^{Z_2} v \tilde{NL}wdZ \right| & \lesssim & \|u\|_{L^{\infty}}v^2(Z_1)+\|u\|_{L^{\infty}}e^{6s}v^2(Z_2)+\log (s)\| u\|_{L^{\infty}} \int_{Z_1}^{Z_2} v^2wdZ\\
\lab{bd:estimationapazunonlinear}& \lesssim & e^{-(1+\frac 14 -5\nu)s}+\frac{\nu}{8} \int_{Z_1}^{Z_2} v^2wdZ.
\eea

\noindent \textbf{Step 4} \emph{The error term}. Recall \fref{def:coefficientsm}. For the error one first computes, since $|A|\lesssim |Z|$ for $|Z|\leq \pi$ with $A(-\pi)=-1$, and since $\pa_{ZZ}G_1$ has limit $0$ and $1/2$ on the left and on the right of $-\pi$ respectively:
\bee
A\pa_Z\psi (s,Z) & = & A\pa_Z\left(-\frac{1}{\lb^4\mu^2}\pa_{ZZ}G_1(Z)+m_1(2-Z\pa_Z)G_1(Z)-m_2Z\pa_Z G_1+m_3' \pa_Z G_1(Z)\right)\\
&= & \frac 12 \delta_{\{Z=-\pi\}}+O\left( Z^2\left(\frac{1}{\lb^4}+ \left| m_1 \right|+ \left| m_2\right| \right)+|m_3'||Z|\right).
\eee
Since $w\lesssim |Z|^7$ one has, using \fref{bd:d}:
$$
\int_{Z_1}^{Z_2} |O\left( Z^2\left(\frac{1}{\lb^4}+ \left| m_1 \right|+ \left| m_2 \right| \right)+|m_3' |Z\right)|^2 wdZ  \lesssim   e^{2s} \left(\frac{1}{\lb^4}+ \left| m_1\right|+ \left| m_2 \right| \right)^2+e^{4s}|m_3'|^2 \lesssim  e^{-\frac 54 s}
$$
For the Dirac term, either one has $a<0$ and then $-\pi< Z_1$ in which case there is nothing to estimate since
$$
\int_{Z_1}^{Z_2} v\delta_{\{Z=-\pi \}}dZ =0.
$$
Otherwise, if $Z_1\leq-\pi$, we use Sobolev embedding (since $w\approx s^{-1}$ near $-\pi$) to find:
\bee
\frac{1}{\lb^4\mu^2} \int_{Z_1}^{Z_2} v\delta_{\{ Z=-\pi \}} w&=& \frac{1}{\lb^4\mu^2}w(-\pi)v(-\pi) \leq \frac{C}{\lb^4\mu^2} \left(\left(\int_{Z_1}^{Z_2} v^2 w\right)^{\frac 12}+\left(\int_{Z_1}^{Z_2} (\pa_Zv)^2w \right)^{\frac 12}\right) \\
&\leq & \frac{C}{\lb^4\mu^2} \left(\int_{Z_1}^{Z_2} v^2 w\right)^{\frac 12} +\frac{C \kappa}{\lb^4\mu^2} \int_{Z_1}^{Z_2} (\pa_Zv)^2w+\frac{C}{\kappa \lb^4}.
\eee
Using Cauchy-Schwarz, one has then showed that for the error term, in both cases $Z_1\leq \pi$ or $Z_1>\pi$, for $\kappa$ small enough, using  \fref{bd:parameterstrap} and \fref{bd:etrap2} for the last inequality:
\bea
\non \left| \int_{Z_1}^{Z_2} v A \pa_Z \psi w \right| & \lesssim & \left(e^{-\frac 58 s}+\frac{C}{\lambda^4\mu^2} \right)\left(\int_{Z_1}^{Z_2} v^2 w\right)^{\frac 12} +\frac{1}{2\lb^4\mu^2} \int_{Z_1}^{Z_2} (\pa_Zv)^2w+\frac{C}{\kappa \lb^4} \\
\lab{bd:estimationapzuerror}& \lesssim & e^{-(\frac 58 -\nu)s}+\frac{1}{2\lb^4\mu^2} \int_{Z_1}^{Z_2} (\pa_Zv)^2w
\eea

\noindent \textbf{Step 5} \emph{The remaining lower order terms}. One has from \fref{bd:etrap2} that for the first one:
\be \lab{bd:estimationapzulowerorder1}
\left| \int_{Z_1}^{Z_2} v  Au\pa_Z G_1 wdZ\right| \lesssim \left(\int_{Z_1}^{Z_2} u^2 wdZ\right)^{\frac 12}\left(\int_{Z_1}^{Z_2} v^2 wdZ\right)^{\frac 12}\lesssim e^{-(\frac 12 -2\nu)s}
\ee
since $A\pa_Z G_1$ is bounded. For the last term, from \fref{bd:paz-1u} one has:
\bee
\left| \pa_Z^{-1}uA\pa_{ZZ}G_1 \right| & \lesssim & \left( \int_{Z_2}^0 |u|d\tilde Z\right) |Z|\mathds 1_{-\pi \leq Z\leq 0}+\left(\int_{Z_1}^{Z_2} u^2 wd\tilde Z \right)^{\frac 12} \left(\int_Z^0 w^{-1}d\tilde Z \right)^{\frac 12}|Z|\mathds 1_{-\pi \leq Z\leq 0} \\
& \lesssim & \left( \int_{Z_2}^{0}|u|d\tilde Z\right) |Z|\mathds 1_{-\pi \leq Z\leq 0}+\sqrt s\left(\int_{Z_1}^{Z_2} u^2 wd\tilde Z \right)^{\frac 12} |Z|^5\mathds 1_{0\leq Z \leq \pi} \\
\eee
where we used the fact that $w\approx |Z|^{-7} s^{-q(Z)}$ for $-\pi \leq Z< 0$, and that $q$ is maximal at $-\pi$ with $q(-\pi)=1$. One then computes that
$$
\int_{Z_1}^{Z_2} Z^2wdZ\lesssim \int_{Z_1}^{Z_2} Z^{-5}dZ\lesssim e^{4s}, \ \ \int_{Z_1}^{Z_2} Z^{10} wdZ\lesssim1.
$$
Therefore:
$$
\int_{Z_1}^{Z_2} |\pa_Z^{-1}uA\pa_{ZZ}G_1|^2 wdZ \lesssim e^{4s}\left( \int_{Z_2}^{0}|u|dZ\right)^2+\int_{Z_1}^{Z_2} u^2 wdZ
$$
which, by Cauchy-Schwarz, gives for the last lower order term, using \fref{bd:etrap} and \fref{bd:etrap2}:
\bee
&& \left| \int_{Z_1}^{Z_2} v \pa_Z^{-1}uA\pa_{ZZ}G_1 wdZ \right| \lesssim \left( \int_{Z_1}^{Z_2} v^2 wdZ \right)^{\frac 12}\left(e^{2s}\left( \int_{Z_2}^{0}|u|dZ\right)+s\left(\int_{Z_1}^{Z_2} u^2 w dZ\right)^{\frac 12}\right)\\
& \lesssim & e^{\nu s} \left(e^{s}\left( \int_{M^2}^{0}|\e|dY\right)+se^{-(\frac 12 -\nu)s}\right)\lesssim e^{\nu s} \left(e^{s}e^{-\frac 72 s}+se^{-(\frac 12 -\nu)s}\right)\lesssim e^{-(\frac 12 -2\nu)s}.
\eee

\noindent \textbf{Step 6} \emph{End of the proof}. In conclusion, from the identities \fref{eq:v}, \fref{eq:expv}, collecting the estimates \fref{eq:lineairev}, \fref{bd:estimationcommutateurAT}, \fref{bd:commutateurvpazz}, \fref{bd:commutateurtildeH}, \fref{bd:commutateurtildeH2}, \fref{bd:estimationapazunonlinear}, \fref{bd:estimationapzuerror}, \fref{bd:estimationapzulowerorder1} and the above inequality:
\bee
&& \frac{d}{ds} \left(\frac 12 \int_{Z_1}^{Z_2} v^2 wdZ \right) \\
&\leq &-\frac 12 \int_{Z_1}^{Z_2} v^2 wdZ-\frac{1}{\lb^4\mu^2}\int_{Z_1}^{Z_2} |\pa_Z v|^2wdZ +\frac{\nu}{4} \int_{Z_1}^{Z_2} v^2wdZ+Ce^{-(1-2\nu')s}+\frac 12 \int_{Z_1}^{Z_2} v^2wdZ\\
&&+e^{-s}+Ce^{-(1-2\nu)s}+\frac{\nu}{8}\int_{Z_1}^{Z_2} v^2wdZ+Ce^{-\frac{13}{8} s} \int_{Z_1}^{Z_2} v^2 wdZ+Ce^{-(1+\frac 14 -5\nu)s}+\frac{\nu}{8} \int_{Z_1}^{Z_2} v^2w \\
&&+Ce^{-(\frac58 -\nu)s}+\frac{1}{2\lb^4\mu^2} \int_{Z_1}^{Z_2} (\pa_Zv)^2wdZ+Ce^{-(\frac 12 -2\nu)s}\\
&\leq & \frac \nu2 \int_{Z_1}^{Z_2} v^2 wdZ-\frac{1}{2\lb^4\mu^2}\int_{Z_1}^{Z_2} |\pa_Z v|^2wdZ +C(K,M) e^{-(\frac 12 -2\nu)s}
\eee
which is the desired differential inequality \fref{eq:exteriorenergyidentity2} for $\nu$ small enough and $s_0$ large enough.

\end{proof}

The very same analysis can be done at the right of the origin. The analogues of Lemmas \ref{lem:exteleft1} and \ref{lem:exteleft2} hold and their proofs are exactly the same.

\begin{lemma}[Exterior Lyapunov functionals on the right] \lab{lem:exteright}

Let $Z_3=Me^{-s}$. There exists $\nu^*>0$ such that for any $K>0$, $0<\nu,\nu'\leq \nu^*$, an $M^*>0$ exists such that for $M\geq M^*$, there exists $s_0^*$ and $C(K,M)$ such that if the solution is trapped on $[s_0,s_1]$ with $s_0\geq s_0^*$ the following inequality holds true, with $v=A\pa_Z u$:
\bea
\lab{eq:exteriorenergyidentity3}  &&\frac{d}{ds} \left(\frac 12 \int_{Z_3}^{+\infty} u^2 w dZ\right)+\left(\frac 12-\frac{\nu}{2}  \right)\int_{Z_3}^{+\infty} u^2 wdZ   \\
\non &\leq & C(K,M)\left(e^{6s}u^2(Z_3)+e^{4s}|\pa_Z u|^2(Z_3)+e^{-(2+\frac 16)s} + \left( \int_{Z_3}^{+\infty} u^2 wdZ \right)^{\frac 12} e^{-\frac 58 s}\right)
\eea

\be \lab{eq:exteriorenergyidentity4}
 \frac{d}{ds} \left(\frac 12 \int_{Z_3}^{+\infty} v^2 wdZ \right)-\frac \nu2 \int_{Z_3}^{+\infty} v^2 wdZ+\frac{1}{2\lb^4\mu^2}\int_{Z_3}^{+\infty} |\pa_Z v|^2wdZ \leq e^{-\frac 14 s}
\ee

\end{lemma}

\begin{proof}

The proof of Lemma \ref{lem:exteright} follows exactly the same lines as the proof of Lemmas \ref{lem:exteleft1} and \ref{lem:exteleft2}, since everything is symmetric except the boundary condition, and we safely skip it. The only difference is  then that in this case the only boundary terms are coming from $Z_3$.

\end{proof}

%%%%%%%%%%%%%%%%%%%%%%%%%%%%%%%%%%%%%%%%%%%%%%%%%%%%%%%%%%%%%%%
%%%%%%%%%%%%%%%%%%%%%%%%%%%%%%%%%%%%%%%%%%%%%%%%%%%%%%%%%%%%%%%
%%%%%%%%%%%%%%%%%%%%%%%%%%%%%%%%%%%%%%%%%%%%%%%%%%%%%%%%%%%%%%%

\subsection{Analysis close to the origin}

This subsection is devoted to the analysis of the solution in original variables, on compact sets and in particular close to the origin. Since the blow-up happens at infinity, eventually the nonlinear effects become weak and the solution stays regular. We state it in a perturbative way and track precisely the constants, so that this can be used both to derive uniform estimates at the origin, and to derive the asymptotics \fref{eq:blowupprofile} for the profile at blow-up time.

\begin{lemma}[No blow-up on compact sets] \lab{lem:noblowup}

Let $0\leq s_0\leq s_1$, $b>0$, $N,L,L'\geq 1$, $q\in 2\mathbb N$. Assume that $s$ is given by \fref{def:renormalisationpara} with $\lambda$ satisfying \fref{bd:parameterstrap}. Let $\xi$ solve \fref{1DPrandtl} on $[0,t(s_1)]\times [0,2N]$, with $\xi \in C^3([0,t(s_1)]\times [0,2N])$, and such that the following properties hold:
$$
\xi_0(t(s_0))=by^2+\tilde \xi (t(s_0)), \ \ \| \tilde \xi (t(s_0))\|_{L^{\infty}([0,2N])}\leq L, \ \ \| \pa_y \tilde \xi (t(s_0))\|_{L^{2}([0,2N])}\leq L'.
$$
and for all $t\in [t(s_0),t(s_1)]$:
$$
\| \xi(t) \|_{L^{\infty}([0,2N])}\leq e^{\left(1-\frac{1}{8}\right)s}, \ \ \| \pa_y \xi (t)\|_{L^{2}([0,2N])}\leq e^{s},
$$
then, writing $\xi=by^2+\tilde \xi$, for all $t\in [t(s_0),t(s_1)]$:
$$
\| \tilde \xi \|_{L^{q}([0,N])}\lesssim LN^{\frac 1q}+N^{2+\frac 1q}e^{-\frac{s_0}{16}}, \ \ \| \pa_y \tilde \xi \|_{L^{2}([0,N])}\lesssim L'+N^{\frac 32}e^{-\frac{s_0}{8q}}.
$$

\end{lemma}

\begin{corollary} \lab{lab:cor}

There exists a universal $C>0$ such that for any $K,\nu,\nu',M$, there exists $s_0^*$ such that if the solution is trapped on $[s_0,s_1]$ with $s_0\geq s_0^*$, we have for all $t\in [t(s_0),t(s_1)]$:
\be \lab{bd:origin bootstrap2}
\| \xi(t,\cdot) \|_{W^{1,\infty}([0,1/2])}\leq C. 
\ee

\end{corollary}

\begin{proof}[Proof of Corollary \ref{lab:cor}]

From \fref{bd:weightedSobolev}, \fref{bd:parameterstrap} and \fref{def:renormalisationpara2} we infer that for $s_0$ large enough one has for all $s\in [s_0,s_1]$:
$$
\| \xi \|_{L^{\infty}([0,2])} \leq e^{\left(1-\frac 18\right)s}.
$$
Hence one obtains from Lemma \ref{lem:noblowup}, using \fref{bd:eini3}, that for all $t\in [0,t(s_1)]$:
$$
\|  \xi \|_{L^{q}([0,1])}\lesssim 1, \ \ \| \pa_y \xi \|_{L^{2}([0,1])}\lesssim 1.
$$
The desired bound \fref{bd:origin bootstrap2} then follows from a standard parabolic regularity result. We do not prove it here and refer to the proof of Lemma \ref{lem:bootstrapcompactY} for a similar strategy.

\end{proof}

\begin{proof}[Proof of Lemma \ref{lem:noblowup}]

The proof relies on a standard localized bootstrap argument similar to \cite{GK}. The fact that we performed such an argument close to the anticipated profile at blow-up time is inspired by \cite{HV2,MZ2}.\\

\noindent \textbf{Step 1} \emph{The bootstrap procedure}. Let $1<\alpha_1<2$, $0<\kappa<1$ with $\kappa\neq 1-1/(16q)$, $L_1=LN^{\frac 1q}+N^{2+\frac 1q}e^{-\frac{s_0}{16}}$ and assume that for $t\in [t(s_0),t(s_1)]$ one has the bound
\be \lab{eq:originhp ubt}
\int_{y\leq 2N} |\tilde \xi |^qdy \leq L_1^q e^{q(1-\kappa )s}.
\ee
We claim that then for all $t\in [0,t(s_1)]$ one has the bound:
\be \lab{eq:originhp ubtgain}
\int_{y\leq \alpha_1 N} |\tilde \xi |^qdy \lesssim \left\{ \ba{l l l} L_1^q e^{q(1-\kappa-\frac{1}{16q} )s} & \text{if} \ \kappa<1-1/(8q) \\ L_1^q & \text{if} \ 1-1/(8q)<\kappa. \ea \right.
\ee
We now prove this claim. We write $\xi =by^2+\tilde \xi$. Then $\tilde \xi$ solves:
$$
\tilde \xi_t-\pa_{yy}\tilde \xi+\pa_y^{-1}\xi \pa_y \tilde \xi -\xi^2+2b\pa_y^{-1}\xi y-2b=0, \ \ \tilde \xi (t,0)=0.
$$
Let $0<\alpha \ll1$ and $\chi$ be a smooth cut-off function, with $\chi(y)=1$ for $y\leq 1+\alpha$ and $\chi (y)=0$ for $y\geq 1+2\alpha$, set $\chi_1=\chi\left(\frac{y}{\alpha_1 N}\right)$ and let $v:=\chi_1 \xi$. Then $v$ solves:
$$
v_t-\pa_{yy}v+\pa_{y}^{-1}\xi \pa_y v+2\pa_y \chi_1 \pa_y \tilde \xi-\chi_1 \xi^2+2b \pa_y^{-1} \xi \chi_1 y -2b\chi_1+\pa_{yy}\chi_1 \tilde \xi-\pa_y^{-1}\xi \pa_y \chi_1 \tilde \xi=0.
$$
One then has the following identity for an $L^q$ energy estimate:
\bee
0&=&\frac{d}{dt}\left(\frac{1}{q}\int v^q dy \right)+(q-1)\int v^{q-2}|\pa_y v|^2dy\\
&&+\int v^{q-1}\left(\pa_{y}^{-1}\xi \pa_y v+2\pa_y \chi_1 \pa_y \tilde \xi-\chi_1 \xi^2+2b \pa_y^{-1} \xi \chi_1 y -2b\chi_1+\pa_{yy}\chi_1 \tilde \xi-\pa_y^{-1}\xi \pa_y \chi_1 \tilde \xi\right)dy.
\eee
We now estimate all terms. For the first one, an integration by parts gives, using $|v|\lesssim |\tilde \xi|$:
$$
\left| \int v^{q-1} \pa_{y}^{-1}\xi \pa_y vdy\right| =\frac{1}{q}\left| \int v^q \xi dy \right|\lesssim \| \xi \|_{L^{\infty}([0,2N])} \int_{y\leq 2N} |\tilde \xi|^{q}dy \lesssim L_1^q e^{\left(q(1-\kappa)+1-\frac 18 \right)s}
$$
For the second one, integrating by parts, applying H\"older and Young inequality and $|v|\lesssim |\tilde \xi|$:
$$
\left| \int v^{q-1} \pa_y \chi_1 \pa_y \tilde \xi \right|\leq \frac{1}{2} \int |\pa_y v|^2v^{q-2}+C \int_{y\leq 2N} |\tilde \xi|^q\leq  \frac{1}{2} \int |\pa_y v|^2v^{q-2}+CL_1^q e^{q(1-\kappa) s}.
$$
For the third term, since $|v|\lesssim \tilde \xi$ and $\xi^2 \lesssim  |\xi|(|\tilde \xi|+y^{2}|)$ there holds from H\"older and \fref{eq:originhp ubt}:
\bee
\left| \int v^{q-1} \chi_1 \xi^2 \right| & \lesssim & \| \xi \|_{L^{\infty}([0,2N])} \int_{y\leq 2N} \tilde \xi^q+\| \xi \|_{L^{\infty}([0,2N])} \left( \int_{y\leq 2N} y^{2q}dy\right)^{\frac{1}{q}}\left( \int_{y\leq 2N} |\tilde \xi |^{q} \right)^{1-\frac{1}{q}} \\
&\lesssim & L_1^q e^{\left(q(1-\kappa)+1-\frac 18 \right)s}+e^{(1-\frac 18)s}N^{2+\frac 1q}L_1^{q-1}e^{(q-1)(1-\kappa)s}\lesssim L_1^{q}e^{\left(q(1-\kappa)+1-\frac{1}{16}\right)s}\\
\eee
since $e^{-\frac{1}{16}}N^{2+\frac 1q}\leq L_1$. For the fourth term, since $|\pa_y^{-1} \xi y |\leq \| \xi \|_{L^{\infty}([0,2N])}y^2$ and $|v|\lesssim |\tilde \xi|$:
$$
\left| \int v^{q-1}\pa_y^{-1} \xi \chi_1 y \right|\leq\| \xi \|_{L^{\infty}([0,2N])} \left( \int_{y\leq 2N} y^{2q}dy\right)^{\frac{1}{q}}\left( \int_{y\leq 2N} |\tilde \xi |^{q} \right)^{1-\frac{1}{q}}\lesssim L_1^{q}e^{\left(q(1-\kappa)+1-\frac{1}{16}\right)s}.
$$
For the the next two terms:
$$
\left| \int v^{q-1}\left( -2b\chi_1+\pa_{yy}\chi_1 \tilde \xi \right)dy \right|\lesssim \int_{y\leq 2N}\tilde \xi^qdy \lesssim L_1^q e^{q(1-\kappa)s}.
$$
Finally, for the last term, as $\pa_y \chi_1\lesssim N^{-1}$, one has $|\pa_y^{-1}\xi \pa_y \chi_1|\lesssim \| \xi \|_{L^{\infty}([0,2N])}$ and:
$$
\left| \int v^{q-1} \pa_y^{-1}\xi \pa_y \chi_1 \tilde \xi dy \right|\leq \| \xi \|_{L^{\infty}([0,2N])} \int_{y\leq 2N} \tilde \xi^q dy \lesssim L_1^q e^{\left(q(1-\kappa)+1-\frac 18 \right)s}.
$$
Collecting all the above estimates gives:
$$
\frac{d}{dt}\left( \int v^q dy \right)\lesssim L_1^q e^{\left(q(1-\kappa)+1-\frac{1}{16} \right)s}.
$$
We reintegrate with time the above identity, using the relation $ds/dt=\lb^2\approx e^{s}$ from \fref{bd:parameterstrap}:
\bee
\int v^q & \lesssim & \int |\tilde \xi (s_0)|^q+ L_1^q \int_{s_0}^{s} e^{\left(q(1-\kappa)-\frac{1}{16} \right)s'}ds' \\
&\lesssim & \left\{ \ba{l l l} L^q N+ L_1^q e^{\left(q(1-\kappa)-\frac{1}{16} \right)s} & \text{if} \ \kappa<1-\frac{1}{16q}, \\ 
L^q N+ L_1^q e^{\left(q(1-\kappa)-\frac{1}{16} \right)s_0} & \text{if} \ \kappa>1-\frac{1}{16q},
\ea \right. \lesssim \left\{ \ba{l l l}  L_1^q e^{\left(q(1-\kappa)-\frac{1}{16} \right)s} & \text{if} \ \kappa<1-\frac{1}{16q}, \\ 
L_1^q & \text{if} \ \kappa>1-\frac{1}{16q},
\ea \right. \\
\eee
since $L_1=LN^{\frac 1q}+N^{2+\frac 1q}e^{-\frac{s_0}{8q}}$ (the case $\kappa =1-1/(16q)$ produces a harmless log which can be avoided by choosing slightly different parameters without affecting the result). This ends the proof of \fref{eq:originhp ubtgain} and of the claim.\\

\noindent \textbf{Step 2} \emph{Uniform in time $L^q$ bound}. We iterate Step 1 for a sequence of intervals $[0,\alpha_1N]$,...,$[0,\alpha_kN]$ and parameter $\kappa_1,...,\kappa_k$. Note that this is possible from the initial bounds. At each iteration, if one is not in the second case the gain in \fref{eq:originhp ubtgain} is $\kappa_i=\kappa_{i-1}+1/(16q)$. Hence we only need a finite number of iterations depending on the choice of $q$ to reach the second case, yielding:
$$
\int_{y\leq N} |\tilde \xi |^qdy \lesssim L_1^q =L^qN+N^{2q+1}e^{-\frac{s_0}{16}}.
$$

\noindent \textbf{Step 3}  \emph{The bootstrap procedure for the derivative}. Let $1<\alpha_1<2$, $0\leq \kappa<2$ with $\kappa\neq 2-1/8$, $L_1=L'+N^{3/2}e^{-\frac{s_0}{8q}}$ and assume that for $t\in [t(s_0),t(s_1)]$ one has the bound
\be \lab{eq:originhp ubt2}
\int_{y\leq 2N} |\pa_y \tilde \xi |^2dy \leq L_1^2 e^{(2-\kappa )s}.
\ee
We claim that then for all $t\in [0,t(s_1)]$ one has the bound:
\be \lab{eq:originhp ubtgain2}
\int_{y\leq \alpha_1 N} |\tilde \xi |^2dy \lesssim \left\{ \ba{l l l} L_1^2 e^{(2-\kappa-\frac{1}{8} )s} & \text{if} \ \kappa<2-1/8 \\ L_1^q & \text{if} \ 2-1/8<\kappa. \ea \right.
\ee
We now prove this claim. Let $\zeta:=\pa_y \xi$. Then it solves:
$$
\zeta_t-\xi \zeta+\pa_y^{-1}\xi \pa_y \zeta -\pa_{yy}\zeta=0.
$$
We write $\zeta =h+\tilde \zeta$ with $h$ smooth such that $h=2by $ for $y\geq 1$, $h(0)=h'(0)=h''(0)=0$. Then $\tilde \zeta$ solves:
$$
\tilde \zeta_t-\pa_{yy}\tilde \zeta+\pa_y^{-1}\xi \pa_y \tilde \zeta -\xi \zeta+\pa_y^{-1}\xi \pa_y h-\pa_{yy}h =0, \ \ \pa_y\tilde \zeta (t,0)=0.
$$
Let $0<\alpha \ll1$ and $\chi$ be a smooth cut-off function, with $\chi(y)=1$ for $y\leq 1+\alpha$ and $\chi (y)=0$ for $y\geq 1+2\alpha$, set $\chi_1=\chi\left(\frac{y}{\alpha_1 N}\right)$ and let $v:=\chi_1 \tilde \zeta$. Then $v$ solves:
$$
v_t-\pa_{yy}v+\pa_{y}^{-1}\xi \pa_y v+2\pa_y \chi_1 \pa_y \tilde \zeta-\chi_1 \xi \zeta+\pa_y^{-1} \xi \chi_1 \pa_y h -2b\pa_{yy}h+\pa_{yy}\chi_1 \tilde \zeta-\pa_y^{-1}\xi \pa_y \chi_1 \tilde \zeta=0.
$$
An $L^2$ energy estimate then writes:
$$
\frac{d}{dt}\left(\frac{1}{2}\int v^q \right)+\int |\pa_y v|^2+\int v\left(\pa_{y}^{-1}\xi \pa_y v+2\pa_y \chi_1 \pa_y \tilde \zeta-\chi_1 \xi \zeta+\pa_y^{-1} \xi \chi_1 \pa_y h -2b\pa_{yy}h+\pa_{yy}\chi_1 \tilde \zeta-\pa_y^{-1}\xi \pa_y \chi_1 \tilde \zeta \right)=0.
$$
We now estimate all terms. For the first one, an integration by parts gives, using $|v|\lesssim |\tilde \zeta|$:
$$
\left| \int v \pa_{y}^{-1}\xi \pa_y v dy\right| =\frac{1}{2}\left| \int v^2 \xi dy \right|\lesssim \| \xi \|_{L^{\infty}([0,2N])} \int_{y\leq 2N} |\tilde \zeta |^{2}dy \lesssim L_1^2 e^{\left(2-\kappa+1-\frac 18 \right)s}
$$
For the second one, integrating by parts, applying H\"older, Young inequality and $|v|\lesssim |\tilde \zeta|$:
$$
\left| \int v \pa_y \chi_1 \pa_y \tilde \zeta dy \right|\leq \frac{1}{2} \int |\pa_y v|^2dy+C \int_{y\leq 2N} |\tilde \zeta|^2dy \leq  \frac{1}{2} \int |\pa_y v|^2dy+CL_1^2 e^{(2-\kappa) s}.
$$
For the third term, since $|v \xi \zeta |\lesssim |\tilde \zeta|^2 |\xi|+y|\xi|$ there holds:
$$
\left| \int v \chi_1 \xi \zeta \right|\lesssim \| \xi \|_{L^{\infty}([0,2N])} \int_{y\leq 2N} \tilde \zeta^2+\| \xi \|_{L^{\infty}([0,2N])} \int_{y\leq 2N} y^{2}\lesssim L_1^2 e^{\left(2-\kappa+1-\frac 18 \right)s}+N^{3}e^{(1-\frac 18)s}.
$$
Similarly for the forth term, since $|\pa_y^{-1} \xi \pa_y h |\leq \| \xi \|_{L^{\infty}([0,2N])}y$ and $|v|\lesssim |\tilde \zeta|$:
$$
\left| \int v \pa_y^{-1} \xi \chi_1 \pa_y h \right|\leq \| \xi \|_{L^{\infty}([0,2N])} \int_{y\leq 2N} \tilde \zeta^2+\| \xi \|_{L^{\infty}([0,2N])} \int_{y\leq 2N} y^{2}\lesssim L_1^2 e^{\left(2-\kappa+1-\frac 18 \right)s}+N^{3}e^{(1-\frac 18)s}.
$$
Finally, for the the next two terms:
$$
\left| \int v \left( -\pa_{yy}h \chi_1+\pa_{yy}\chi_1 \tilde \zeta \right)dy \right|\lesssim \int_{y\leq 2N}\tilde \zeta^2dy \lesssim L_1^2 e^{(2-\kappa)s}.
$$
Finally, for the last term, as $\pa_y \chi_1\lesssim N^{-1}$, one has $|\pa_y^{-1}\xi \pa_y \chi_1|\lesssim \| \xi \|_{L^{\infty}([0,2N])}$ and:
$$
\left| \int v \pa_y^{-1}\xi \pa_y \chi_1 \tilde \zeta \right|\leq \| \xi \|_{L^{\infty}([0,2N])} \int_{y\leq 2N} \tilde \zeta^2+\| \xi \|_{L^{\infty}([0,2N])} \int_{y\leq 2N} y^{2}dy\lesssim L_1^2 e^{\left(2-\kappa+1-\frac 18 \right)s}+N^{3}e^{(1-\frac 18)s}.
$$
Collecting all the above estimates gives:
$$
\frac{d}{dt}\left( \int v^2 dy \right)\lesssim L_1^2 e^{\left(2-\kappa+1-\frac 18 \right)s}+N^{3}e^{(1-\frac 18)s}.
$$
We reintegrate with time the above identity, using the relation $ds/dt=\lb^2\approx e^{s}$ from \fref{bd:parameterstrap}:
\bee
\int v^2 & \lesssim & \int |\tilde \zeta (s_0)|^2+ L_1^2 \int_{s_0}^{s} e^{\left(2-\kappa-\frac 18 \right)s'}ds'+N^{3}\int_{s_0}^s e^{-\frac 18 s} \\
&\lesssim & \left\{ \ba{l l l} L^{'2}+ L_1^2 e^{\left(2-\kappa-\frac 18 \right)s}+N^{3}  e^{-\frac 18 s_0} & \text{if} \ \kappa<2-1/8, \\ 
L^{'2}+ L_1^2 e^{\left(2-\kappa-\frac 18 \right)s_0}+N^{3}  e^{-\frac 18 s_0} & \text{if} \ \kappa>2-1/8,
\ea \right. \lesssim \left\{ \ba{l l l}  L_1^2 e^{\left(2-\kappa -\frac 18 \right)s} & \text{if} \ \kappa<2-1/8, \\ 
L_1^2 & \text{if} \ \kappa>2-1/8,
\ea \right. \\
\eee
since $L_1=L'+N^{3}e^{-\frac{s_0}{8}}$. This ends the proof of \fref{eq:originhp ubtgain2} and of the claim.\\

\noindent \textbf{Step 4} \emph{Uniform in time $L^2$ bound for the derivative}. Again, as in Step 2, we iterate Step 3 for a finite sequence of intervals $[0,\alpha_1N]$,...,$[0,\alpha_kN]$ and finally obtain:
$$
\int_{y\leq N} |\pa_y \tilde \xi |^2dy \lesssim L_1^2 =L^{'2}+N^{3}e^{-\frac{s_0}{8}}.
$$

\end{proof}

%%%%%%%%%%%%%%%%%%%%%%%%%%%%%%%%%%%%%%%%%%%%%%%%%%%%%%%%%%%%%%%%%%%%%%%%
%%%%%%%%%%%%%%%%%%%%%%%%%%%%%%%%%%%%%%%%%%%%%%%%%%%%%%%%%%%%%%%%%%%%%%%%
%%%%%%%%%%%%%%%%%%%%%%%%%%%%%%%%%%%%%%%%%%%%%%%%%%%%%%%%%%%%%%%%%%%%%%%%

\subsection{End of the proof of Proposition \ref{pr:bootstrap} and proof of Theorem \ref{th:main}} \lab{subsec:prbootstrap}

In this subsection we reintegrate over time the modulation equations and the various energy estimates, to show that the various upper bounds describing the bootstrap cannot be saturated. We first reintegrate the modulation equations and Lyapunov functionals.

\begin{lemma} \label{lem:reintegrationbootstrap}

There exists $\nu^*>0$ such that for any $\nu<\nu^*$, for $\nu'$ small enough and then for $\eta$ small enough, for any $K,M$ such that Lemmas \ref{lem:Linftybd}, \ref{lem:exteleft1}, \ref{lem:exteleft2} and \ref{lem:exteright} hold true, the following holds for $s_0$ large enough. For a solution that is trapped on $[s_0,s_1]$, at time $s\in [s_0,s_1]$:
\be \lab{bd:int e}
\| \e \|_{L^2_\rho}^2 \leq 2 e^{-\frac 72 s}, \ \ \int_{s_0}^s e^{(7-\eta)\tilde s}\| \pa_Y \e(\tilde s) \|_{L^2_\rho}^2d\tilde s\leq 2,
\ee
\be \lab{bd:boostrap improved parameters}
\frac{1}{2e} \leq \mu \leq 2e, \ \ \frac 14 e^{\frac s 2}\leq \lb \leq \frac 94 e^{\frac s2}, \ \ |a|\leq 2e^{-\left(\frac 12 -2\nu \right)s},
\ee
\be \lab{bd:boostrap improved parameters2}
\mu=\mu_{\infty}(1+O(e^{-s})), \ \ \lb=e^{\frac s 2}\tilde \lb_{\infty}(1+O(e^{-2s})),
\ee
\be \lab{bd:int e2}
\int_{Z_1}^{Z_2} u^2 wdZ +\int_{Z_3}^{+\infty} u^2 wdZ \leq 4e^{-(1-2\nu)s} , \ \ \int_{Z_1}^{Z_2} |A\pa_Zu|^2 wdZ +\int_{Z_3}^{+\infty} |A\pa_Zu|^2 wdZ \leq 4e^{2\nu s} .
\ee

\end{lemma}

\begin{proof}

\noindent \textbf{Step 1} \emph{Interior Lyapunov functional and energy dissipation}. We rewrite \fref{bd:lyapunovpara} as:
$$
\frac{d}{ds}\left(e^{7s} \| \e \|_{L^2_{\rho}}^2\right)+e^{(7-\eta)s}\| \pa_Y\e\|_{L^2_\rho}^2 \leq Ce^{(7-\eta)s}\|\e\|_{L^2_\rho}^2+Ce^{7s}\| \e \|_{L^2_\rho}\lb^{-12}+Ce^{7s-e^s}.
$$
Injecting the bounds \fref{bd:parameterstrap} and \fref{bd:etrap} and integrating in time using \fref{bd:eini} gives:
$$
e^{7s}\| \e \|_{L^2_\rho}^2 -1+\int_{s_0}^s e^{(7-\eta)\tilde s}\| \pa_Y \e (\tilde s)\|_{L^2_\rho}^2\leq \int_{s_0}^s \left(C(K)e^{-\eta \tilde s}+C(K)e^{-\frac 12 \tilde s}+Ce^{7\tilde s-e^{\tilde s}}\right)d\tilde s\leq 1
$$
for $s_0$ large enough depending on $K$, which implies the desired estimates \fref{bd:int e}.\\

\noindent \textbf{Step 2} \emph{Law for $\mu$}. We integrate over time the inequality \fref{bd:lossymodulation}, implying for $s_0$ large enough:
$$
|\log \mu (s)-\log (\mu(s_0))|\leq  \int_{s_0}^s e^{-\frac{13}{8}\tilde s}d\tilde s \leq 1
$$
which using \fref{bd:parametersini} gives indeed $(2e)^{-1}\leq \mu \leq 2e$ and if the solution is trapped for all times:
\begin{eqnarray*}
\mu(s)&=&\mu(s_0)\exp \left( \int_{s_0}^s O(e^{-\frac{13}{8}\tilde s})d\tilde s\right)=\mu(s_0)\exp \left( (\int_{s_0}^\infty-\int_s^\infty)   O(e^{-\frac{13}{8}\tilde s})d\tilde s\right)\\
&=&\mu_{\infty}(1+O(e^{-\frac{13}{8}s})) 
\end{eqnarray*}
where we have set $\mu_\infty:=\mu(s_0)\exp\left(\int_{s_0}^\infty O(e^{-(\frac{13}{8}\tilde s)})d\tilde s\right)$.\\
\noindent \textbf{Step 3} \emph{Law for $\lb$}. We rewrite as in Step 2 the equation for $\lambda $ in \fref{bd:lossymodulation} using \fref{bd:parameterstrap}:
\be \lab{bd:lbreintegration}
\left|\frac{\lb_s}{\lb}-\frac 12 \right| \leq C(K)e^{-2s}.
\ee
This can be written alternatively as $\left| \frac{d}{ds}(e^{-\frac s2}\lb)\right| \leq C(K)e^{-5s/2}$ which when reintegrated over time using \fref{bd:parametersini} gives:
$$
|e^{-\frac s2}\lb-e^{-\frac{s_0}{2}}\lb (s_0)|\leq C(K)\int_{s_0}^s e^{-5\tilde s/2}d\tilde s
$$
which with \fref{bd:parametersini} yields $1/4\leq e^{-s/2}\lb \leq 9/4$ for $s_0$ large enough, implying the bound for $\lb$ in \fref{bd:boostrap improved parameters}. If the solution is trapped for all times this gives:
\begin{eqnarray*}
\lb &=& e^{\frac s2} \left(e^{-\frac{s_0}{2}}\lb_0+\int_{s_0}^s O(e^{-5\tilde s /2})d\tilde s \right)=e^{\frac s2} \left(e^{-\frac{s_0}{2}}\lb_0+(\int_{s_0}^\infty-\int_s^\infty)  O(e^{-5\tilde s /2})d\tilde s \right)  \\
&=&e^{\frac s 2}\tilde \lb_{\infty}(1+O(e^{-\frac 52 s}))
\end{eqnarray*}
where we have set $\tilde \lambda_\infty=e^{-\frac{s_0}{2}}\lb_0+\int_{s_0}^\infty  O(e^{-5\tilde s /2})d\tilde s$.

\noindent \textbf{Step 4} \emph{Law for $a$}. We rewrite the equation for $a$ in \fref{bd:modulation4} and inject the bounds \fref{bd:parameterstrap}, \fref{bd:etrap} and \fref{bd:weightedSobolev}, using that $G(-\pi+Z)=O(Z^2)$ as $Z\rightarrow 0$:
\bee
\left|\frac{d}{ds} (e^{\frac s2}a)\right| & \lesssim & e^{\frac s2}\left(\left|\int_{-\pi-a}^{-\pi}G_1dZ \right|+\left|\int_{-\pi-a}^0 udZ\right|+\lambda^{-4}+\| \e \|_{L^2_\rho}+\lambda^4\| \e \|_{L^\infty}\|\e \|_{L^2_\rho}\right) \\
&\lesssim & e^{\frac s2}|a|^3+e^{\frac s2}\left|\int_{-\pi-a}^{-Me^{-s}} udZ+\int_{-Me^{-s}}^0udZ \right|+C(K)(e^{-\frac 32 s}+e^{-3s}+se^{-(1+\frac 14-\nu)s})\\
&\lesssim & e^{-(1-6\nu)s}+e^{\frac s2}\left( s\left(\int_{-\pi-a}^{-Me^{-s}} wu^2dZ\right)^{\frac 12}+e^{-s}\int_{CM}^0 |\e|dY \right) \\
&\lesssim& e^{-(1-6\nu)s}+e^{\frac s2}\left( se^{-\left(\frac 12 -\nu \right)s}+e^{-\frac 92 s} \right) \ \lesssim  \ s e^{\nu s},\\
\eee
for $s_0$ large enough. This implies in particular the following bound for $a_s$ using \fref{bd:parameterstrap}:
\be \lab{bd:as}
|a_s|\lesssim e^{-(\frac 12 -2\nu)s}.
\ee
Reintegrating over time this estimate gives using \fref{bd:parametersini}:
$$
|a|=e^{-\frac s2}\left| a_0e^{\frac s2}+\int_{s_0}^s O(\tilde s e^{\nu \tilde s})d\tilde s \right| \leq 2 e^{-\left(\frac 12-2\nu \right)s}
$$

\noindent \textbf{Step 5} \emph{Exterior energy functionals}. We inject in \fref{eq:exteriorenergyidentity1} the bounds \fref{bd:etrap2} and \fref{eq:estimationuz2}:
\bee
 &&\frac{d}{ds} \left(e^{(1-\nu)s} \int_{Z_1}^{Z_2} u^2 w \right)+\left(\frac 12-\frac{\nu}{2}  \right)\int_{Z_1}^{Z_2} u^2 w   \\
 &\leq & C(K,M) e^{(1-\nu)s}\left(e^{6s}u^2(Z_2)+e^{4s}|\pa_Z u|^2(Z_2)+e^{-(2+\frac 16)s} + \left( \int_{Z_1}^{Z_2} u^2 w \right)^{\frac 12} e^{-\frac 58 s}\right)\\
&\lesssim & C(K,M) \left(e^{(2\nu'-\nu)s}+e^{-(1+\frac 16 - \nu)s}+e^{-(\frac 18 -2\nu)s} \right) \ \lesssim \ C(K,M) e^{(2\nu'-\nu)s}
\eee
where the $e^{-(\nu-2\nu') s}$ is the worst term, due to the boundary condition at $Z_2$. Indeed, we optimised the weight $w$ to match the exterior decay with the interior decay, hence the choice of $\beta =1/2$ for the eigenfunction \fref{eq:def phinu} in the weight \fref{eq:def w}. Reintegrating in time the above identity using \fref{bd:int e} and \fref{bd:eini2} yields since $0<\eta \ll \nu'\ll \nu \ll 1$, for $s_0$ large enough:
\bee
\int_{Z_1}^{Z_2} u^2 w & \leq & e^{-(1-\nu)s}\left[e^{(1-\nu)s_0}\int_{Z_1(s_0)}^{Z_2(s_0)} u^2w+C(K,M) \int_{s_0}^s e^{-(\nu-2\nu') s}  \right]\\
&\leq & e^{-(1-2\nu)s} \left(e^{\nu (s_0-s)}+C(K,M)e^{-\nu s} \right)\leq 2e^{-(1-2\nu)s}.
\eee
The differential inequality on the right \fref{eq:exteriorenergyidentity3} can be reintegrated with time the very same way, giving $\int_{Z_3}^{+\infty} u^2 w \leq 2e^{-(1-2\nu)s}$. These two bounds imply the first bound in \fref{bd:int e2}. We now turn to the derivative. We write \fref{eq:exteriorenergyidentity2} as
$$
\left| \frac{d}{ds} \left(e^{-\nu s}\int_{Z_1}^{Z_2} |A\pa_Z u|^2 wdZ \right)\right| \leq e^{-\frac 14 s}
$$
Note that compared to the differential inequality for $u$, the above identity for $A\pa_Z u$ is better. Indeed the fact that $A\sim Z$ near the origin improves the control of the boundary term at $Z_2$, and $A\pa_Z$ kills the worst component of the error near the origin. Reintegrating in time the above identity using \fref{bd:int e} and \fref{bd:eini2} yields:
\bee
\int_{Z_1}^{Z_2} |A\pa_Z u|^2 &\leq & e^{2\nu s} \left(e^{-\nu s}e^{\nu s_0}\int_{Z_1(s_0)}^{Z_2(s_0)} |A\pa_Z u(s_0)|^2+Ce^{-\nu s}\int_{s_0}^s e^{-\frac 14 \tilde s}d\tilde s\right)\\
& \leq & e^{2\nu s}(e^{\nu (s_0-s)}+Ce^{-\nu s}) \leq 2e^{2\nu s}.
\eee
The same bound can also be proved the same way for the derivative at the right of the origin, implying the last bound in \fref{bd:int e2}.

\end{proof}

We now bootstrap the last bound and control $\e$ on $[-M^2,M^2]$ using parabolic regularity.

\begin{lemma} \lab{lem:bootstrapcompactY}

There exists $\nu^*>0$ such that for any $\nu<\nu^*$, then for $\nu'$ small enough, for $K,M$ such that Lemma \ref{lem:reintegrationbootstrap} holds true, for a solution that is trapped $[s_0,s_1]$ for $s_0$ large enough:
\be \lab{bd:boostrap ecompact}
\| \e(s_1) \|_{H^3(|Y|\leq M^2)}^2\leq 10 e^{-(7-\nu')s_1}.
\ee

\end{lemma}

\begin{proof}

The proof is a classical use of parabolic regularity: $\e$ evolves according to a parabolic equation, its size and the size of the forcing terms are precisely $e^{-7s/2}$, hence this bound propagates for higher order derivatives due to the smoothing effect of the heat kernel. In this proof, the constants $C$ might depend on $M$ and $K$ unless explicitly mentioned. We rewrite \fref{eq:e} as:
$$
\e_s-\pa_{YY}\e+ \tilde V  \e+\tilde{\mathcal T} \pa_Y \e=\mathcal F,
$$
where
$$
\tilde V:= (2\frac{\lb_s}{\lb}-2G_1-\e), \ \ \tilde{\mathcal T}:= \frac{\lb_s}{\lb}Y+\int_{(-\pi-a)\lb^2\mu}^0 f -\lb y^*_s+\lb^2\mu \pa_Z^{-1}G_1+\pa_Y^{-1}\e,
$$
\bee
\mathcal F &:=& \left(m_2-\frac{1}{2\lb^4\mu^2}\right)Z\pa_Z G_1+\left(m_1+\frac{1}{4\lb^4\mu^2}\right)(2-Z\pa_Z)G_1+m_3 \frac{1}{\lb^2\mu}\pa_Z G_1\\
&&-\frac{1}{\lb^4\mu^2}\left(\pa_{ZZ}G_1+\frac 14 Z\pa_Z G_1 +\frac 12 G_1 \right).
\eee
Note that from \fref{bd:lossymodulation}, \fref{lem:Linftybd} and \fref{bd:etrap} one has for a universal $C>0$:
\be \lab{eq:bd tildeT}
\| \tilde{\mathcal T}\|_{W^{1,\infty}(|Y|\leq M^3)}+\| \tilde V\|_{W^{1,\infty}(|Y|\leq M^3)}\leq C
\ee
We now let $\e^1:=\pa_Y \e$. It solves:
\be \lab{id:eq paYe compact}
\e_s^1-\pa_{YY}\e^1+ (\tilde V+\pa_Y \tilde{\mathcal T})  \e^1+\tilde{\mathcal T} \pa_Y \e^1=-\pa_Y \tilde V \e+\pa_Y \mathcal F.
\ee
Let $M^2<M_1<M_2<M^3$, and $\chi $ be a cut-off function with $\chi=1$ for $Y\leq M_1$ and $\chi=0$ for $Y\geq M_2$ and let $v=\chi \e^1$. Then $v$ solves:
$$
v_s-\pa_{YY}v+(\tilde V +\pa_Y \tilde{\mathcal T})v+\tilde{\mathcal T}\pa_Y v=-\pa_{YY}\chi \e^1-2\pa_Y \chi \pa_Y \e^1-\tilde{\mathcal T} \pa_Y \chi \e^1-\chi \pa_Y \tilde V \e+\chi \mathcal F.
$$
We then perform a standard energy estimate:
\bee
\frac{d}{ds} \left(\frac 12 \int v^2dY \right)+\int |\pa_Y v|^2dY & = & \int \left(-\pa_{YY}\chi \e^1-2\pa_Y \chi \pa_Y \e^1-\tilde{\mathcal T} \pa_Y \chi \e^1-\chi \pa_Y \tilde V \e+\chi \mathcal F \right)vdY \\
&&-\int \left((\tilde V +\pa_Y \tilde{\mathcal T})v+\tilde{\mathcal T}\pa_Y v\right)vdY.
\eee
Let $0<\kappa \ll 1$, integrating by parts and using Young inequality one finds since $|v|\lesssim \e^1$:
$$
\left| \int \left(-\pa_{YY}\e^1-2\pa_Y \chi \pa_Y \e^1\right)vdY \right|  \leq  \frac{C}{\kappa} \int_{|Y|\leq M_3}|\e^1|^2+\kappa C \int |\pa_Y v|^2dY \leq  C \| \pa_Y \e \|_{L^2_\rho}^2+ \frac 14  \int |\pa_Y v|^2dY
$$
for $\kappa$ small enough. Similarly, integrating by parts, using Young inequality, \fref{bd:etrap} and \fref{eq:bd tildeT}:
\bee
\left| \int  \tilde{\mathcal T} \pa_Y \chi \e^1v\right| & = & \left| \int  \tilde{\mathcal T} \pa_Y \chi \pa_Y \e v\right|\leq \frac 14 \int |\pa_Y v|^2+C\| \tilde{\mathcal T}\|_{W^{1,\infty}(|Y|\leq M_2)}\int_{|Y|\leq M_2} \e^2+C\int_{|Y|\leq M_2} |\e^1|^2 \\
&\leq &\frac 14 \int |\pa_Y v|^2+C\| \e \|_{L^2_\rho}^2+C\| \pa_Y \e \|_{L^2_\rho}^2 \leq \frac 14 \int |\pa_Y v|^2+Ce^{-7s}+C\| \pa_Y \e \|_{L^2_\rho}^2.
\eee
Next, from Cauchy-Schwarz, \fref{bd:etrap} and \fref{eq:bd tildeT}, and Young inequality:
$$
\left| \int \chi \pa_Y \tilde V \e v\right| \leq  C \| \pa_Y \tilde V \|_{L^{\infty}(|Y|\leq M^3)} \| v \|_{L^2} \| \e \|_{L^2(|Y|\leq M^3)}\leq Ce^{-\frac72 s}\| v \|_{L^2} \leq Ce^{-7s}+C\| v \|_{L^2}^2
$$
For the error, we recall the cancellation $\pa_{ZZ}G_1+\frac 14 Z\pa_Z G_1 +\frac 12 G_1 =O(|Z|^4)$ and $|\pa_Z G_1|=O(|Z|)$ as $Z\rightarrow 0$, which implies using \fref{bd:d} that:
$$
\int \chi^2 \mathcal F^2dY \leq Ce^{-\frac{29}{4}s}
$$
which by Cauchy-Schwarz and Young yields:
$$
\left| \int \chi \mathcal  F v dY \right| \leq Ce^{-\frac{29}{8}}\| v \|_{L^2}\leq Ce^{-\frac{29}{4}s}+C\| v \|_{L^2}^2.
$$
Performing an integration by parts and using \fref{eq:bd tildeT}:
$$
\left| \int \left((\tilde V +\pa_Y \tilde{\mathcal T})v+\tilde{\mathcal T}\pa_Y v\right)v \right| \leq \| v \|_{L^2}^2 (\| \tilde V\|_{W^{1,\infty}(|Y|\leq M^3)}+\| \tilde{\mathcal T}\|_{W^{1,\infty}(|Y|\leq M^3)})\leq C \| v\|_{L^2}^2.
$$
Let $0<\eta \ll \nu_1 \ll \nu'$. Collecting all the estimates above, and since $|v|\lesssim |\e^1|$ one has the energy identity:
\bee
\frac{d}{ds} \left( e^{(7-\nu_1)s} \int v^2 \right)+\frac 12 e^{(7-\nu_1)s}\int |\pa_Y v|^2 & \leq  & Ce^{(7-\nu_1)s}\| \pa_Y \e \|_{L^2_\rho}^2+Ce^{-\nu_1 s}
\eee
Reintegrated with time, using \fref{bd:int e} and \fref{bd:eini} this gives for $s_0$ large enough:
$$
e^{(7-\nu_1)s}\int v^2dY+\frac 12 \int_{s_0}^s e^{(7-\nu_1)s'}\int |\pa_Y v|^2dYds' \leq e^{(7-\nu_1)s_0}\int v_0^2dY+1\leq 2.
$$
Therefore, $\| v(\tilde s)\|_{L^2}\leq 2e^{-(\frac 72 -\nu_1)\tilde s}$. One has then proved the following pointwise bound for $\pa_Y \e$ and integrated bound for $\pa_{YY} \e$:
$$
\forall s \in [s_0,s_1], \ \ \int_{|Y|\leq M_1} |\pa_Y \e |^2dY\leq 10 e^{-(\frac 72 -\nu_1)s}, \ \ \text{and} \ \  \int_{s_0}^s e^{(7-\nu_1)s'}\int_{|Y|\leq M_1} |\pa_{YY} \e|^2dYds' \leq 2.
$$
Let now $M^2<M_4<M_3<M_1$. We claim that we can differentiate equation \fref{id:eq paYe compact} and, with the exact same arguments, obtain the analogue of the above estimates for $\pa_{YY} \e$, with an exponent $\nu_2$ such that $\nu_1 \ll \nu_2 \ll \nu'$. Indeed, the only crucial arguments to derive the above bounds were the pointwise in time boundedness \fref{bd:etrap} of $\| \e \|_{L^2_\rho}$ and the dissipation estimate \fref{bd:int e} for $\| \pa_Y \e\|_{L^2_\rho}$, and we just obtained the analogues for $\pa_Y \e$ so that the same strategy can be applied. Then, another iteration yields the analogue of the above bounds for $\pa_Y^{(3)}\e$ for $|Y|\leq M_4$ for an exponent $\nu_2\ll \nu_3 \ll \nu'$, which ends the proof of the Lemma.

\end{proof}

All the bounds of the bootstrap and the modulation equations have been investigated previously. We can now end the proof of Proposition \ref{pr:bootstrap}.

\begin{proof}[Proof of Proposition \ref{pr:bootstrap}]

Let an initial datum satisfy the properties of Definition \ref{def:ini} at time $s_0$. Let $\tilde s$ be the supremum of times such that the solution is trapped on $[s_0,\tilde s]$. Assume by contradiction that $\tilde s<+\infty$. Then from the local well-posedness Proposition \ref{pr:cauchy} and the blow-up criterion \fref{id:blowupcriterion}, the solution can be extended beyond the time $\tilde s$. Hence, from the definition of $\tilde s$ and Definition \ref{def:trap} and a continuity argument, one of the inequalities \fref{bd:parameterstrap}, \fref{bd:etrap} or \fref{bd:etrap2} must be an equality at time $\tilde s$. This is however impossible for $K$ large enough from \fref{bd:int e}, \fref{bd:boostrap improved parameters}, \fref{bd:int e2} and \fref{bd:boostrap ecompact}, which is desired contradiction. Hence $\tilde s=+\infty$ which proves Proposition \ref{pr:bootstrap}.

\end{proof}

Theorem \ref{th:main} is a direct consequence of Proposition \ref{pr:bootstrap} and we can now give its proof.

\begin{proof}[Proof of Theorem \ref{th:main}]

For an initial datum of the form \fref{id:condtion initiale}, let $s_0=2\ln (\lb_0^2)$. Then for $\epsilon (\lb_0)>0$ small enough, thanks to the smoothing effect of the equation, see Proposition \ref{pr:cauchy}, $\tilde \xi_0$ is instantaneously regularised, and $\xi (t^*)$ is initially trapped in the sense of Definition \ref{def:ini}. Applying Proposition \ref{pr:bootstrap}, the solution is then trapped for all times in the sense of Definition \ref{def:trap}. Since $ds/dt=\lb^2$ and $\lb$ satisfies \fref{bd:boostrap improved parameters2}:
$$
\frac{dt}{ds}= e^{-s}\tilde \lb_{\infty}^{-2}(1+O(e^{-2s})).
$$
Reintegrating the above equation, there exists $T>0$ such that:
$$
T-t=e^{-s}\tilde \lb_{\infty}^{-2}(1+O(e^{-2s})).
$$
This implies $e^{-s}=\lb^2_{\infty} (T-t)+O((T-t)^3)$. The identities \fref{th:bd para} are then consequences of \fref{bd:boostrap improved parameters2}. From \fref{bd:weightedSobolev}, $\tilde x(t,y)=u(s,Z)$ and \fref{id:decomposition vp} one infers:
$$
\| \tilde \xi \|_{L^{\infty}} = \lb^2 \| u \|_{L^{\infty}}\lesssim e^{-s}e^{-\left(\frac 14 -\nu \right)s}\leq C (T-t)^{1-\frac 18}
$$
which proves \fref{th:bd xi}. We now investigate the existence and asymptotic behaviour of the blow-up profile at time $T$. The existence of a limit $\xi (t,y)\rightarrow \xi^*(y)$ as $t\uparrow T$ follows from Lemma \ref{lem:noblowup} and a standard parabolic bootstrap argument. We now use more carefully Lemma \ref{lem:noblowup} to find the asymptotic of the profile at blow-up time. For $y^*\geq e^{(\frac 12 -\frac{1}{16})s_0}$ we define the following adapted time, which now depends on the point that we consider:
$$
s_0(y^*)=\left(\frac{1}{2}-\frac{1}{16} \right)^{-1}\log (y)=\log (y^\alpha), \ \ \alpha:=\left(\frac{1}{2}-\frac{1}{16} \right)^{-1}=\frac{16}{7}, \ \ \text{so} \ \text{that} \ y^*=e^{\left(\frac 12 -\frac{1}{16}\right)s_0(y)}.
$$
For $s\geq s_0(y)$, for $y \in [0,2y^*]$, one has
$$
Z(y)=\frac{ y-y^*}{\lb \mu}=-\pi-a+\frac{ y}{\lb \mu}=-\pi+O( e^{-\frac{s_0}{16}}).
$$
Therefore one can apply the Taylor expansion of $G_1$ near the origin for $s_0$ large enough. Using \fref{bd:boostrap improved parameters},\fref{bd:boostrap improved parameters2} and \fref{bd:weightedSobolev}, for $s\geq s_0(y^*)$:
$$
\lb^2(s_0)G_1(Z(y))=\frac{1}{4}\left(-a+\frac{y}{\lb \mu}\right)^2 \lb^2+\lb^2O\left( \left|-a+\frac{y}{\lb \mu}\right|^4\right)=\frac{y^2}{4\mu_{\infty}^2}+O(y^{*2-\frac{1}{16}})\leq e^{\left(1-\frac 18 \right)s_0}\leq e^{\left(1-\frac 18 \right)s},
$$
$$
|\lb^2(s_0)u(s_0,Z(y))|\leq C \lb^2(s_0) e^{-\frac 16 s_0}\leq C e^{(1-\frac 16)s_0}=Cy^{*\frac{5\alpha}{6}}=Cy^{*2-\frac{2}{21}},
$$
The two above identities imply that, writing $\xi= \frac{y^2}{4\mu_{\infty}^2}+\tilde \xi$, at time $s_0^*$ on $[0,y^*]$:
$$
\xi (t(s_0(y)),y)= \frac{y^2}{4\mu_{\infty}^2}+O(y^{*2-\frac{1}{16}}), \ \ \text{i.e.} \ \ \| \tilde \xi (s_0(y^*))\|_{L^{\infty}([0,2y^*])}\leq Cy^{*2-\frac{1}{16}},
$$
and that for $s\geq s_0(y^*)$:
$$
\| \xi \|_{L^{\infty}([0,2y^*])}\lesssim e^{(1-\frac 18)s}.
$$
Moreover, from \fref{bd:etrap2}, changing variables:
$$
\| \pa_y(\lb^2u(s,Z(y)))\|_{L^2([0,2y^*])}\lesssim \lb^{\frac 32} \| \pa_Z u(s,Z)\|_{L^2([0,2y^*/\lb])}\lesssim e^{\frac 34 s}se^{2\nu}\leq e^s,
$$
$$
\| \pa_y(\lb^2G_1(s,Z(y)))\|_{L^2([0,2y^*])}\lesssim \lb^{\frac 32} \| \pa_Z G_1(s,Z)\|_{L^2([0,2y^*/\lb])}\lesssim e^{\frac 34 s}\leq e^s,
$$
$$
\| \pa_y(y^2)\|_{L^2([0,2y^*])}\lesssim \lb^{\frac 32} \| \pa_Z G_1(s,Z)\|_{L^2([0,2y^*/\lb])}\lesssim y^{*\frac 32}\leq e^s,
$$
for $s_0$ large enough, so that for $s\geq s_0(y^*)$:
$$
\| \pa_y \tilde \xi \|_{L^2([0,2y ^*])}\leq e^s.
$$
We apply Lemma \ref{lem:noblowup} and obtain that for all $t\geq t(s_0(y^*))$:
$$
\| \tilde \xi \|_{L^{q}([0,y^*])}\lesssim y^{*2-\frac{1}{16}}y^{*\frac 1q}+y^{*2+\frac 1q}y^{*-\frac{\alpha}{16}}\lesssim y^{*2-\frac{1}{16}+\frac 1q}
$$
and for some fixed constant $c>0$:
$$
\| \pa_y \tilde \xi \|_{L^{2}([0,y^*])}\lesssim y^{*\alpha}+y^{*\frac 32} e^{-\frac{s_0}{8q}}\lesssim y^{*c}.
$$
We apply the following interpolated Sobolev inequality:
$$
\| h\|_{L^{\infty}}\lesssim \| h\|_{L^{q}}^{1-\frac{2}{q+2}}\| \pa_y h\|_{L^{2}}^{\frac{2}{q+2}},
$$
yielding that for all $t\geq t(s_0(y^*))$:
$$
\| \tilde \xi \|_{L^{\infty}([0,y^*])}\lesssim y^{*2-\frac{1}{16}+\frac{c}{q}}\lesssim y^{*2-\frac{1}{32}}
$$
for $q$ large enough. Therefore, since this remains true at the limit at time $T$ one has showed that for $y^*\geq e^{(\frac 12 -\frac{1}{16})s_0}$:
$$
\xi (y^*) =\frac{y^{*2}}{4\mu_{\infty}^2}+O(y^{*2-\frac{1}{32}})
$$
which ends the proof of \fref{eq:blowupprofile}.

\end{proof}

\subsection{Localised initial data} \lab{subsec:propadd}

We prove here Proposition \ref{pr:stability weighted}. It is obtained from the analysis of the previous subsections by controlling an additional weighted norm. We introduce $Z^*=M e^{-s }$.

\begin{lemma} \lab{lem:additionallinftybounds}

Fix any $\nu,\nu',K$ such that Lemma \ref{lem:Linftybd} holds true, and assume a solution is trapped on $[s_0,+\infty)$. Then for $M$ large enough, there exists $\delta_0 >0$ such that for $s_0$ large enough:
\begin{itemize}
\item[(i)] If $ \ \sup_{Z \geq \pi } |F(s_0,Z)|Z ^2\leq \delta_0 \ $ then $ \ \sup_{Z \geq \pi } |F(s,Z)| Z^2 \leq e^{-\frac{1}{8}s}$ for $s\geq s_0$.
\item[(ii)] If $ \ \sup_{Z\geq -(\pi+a(s_0)) } |\pa_Z u(s_0,Z)|\langle Z \rangle^{2} \leq \delta_0\ $ then $\ \sup_{Z \geq -(\pi-a(s_0)) } |\pa_{Z} u (s,Z)|\langle Z \rangle^2 \leq e^{-\frac{1}{50}s}$ for $s\geq s_0$.
\end{itemize}
\end{lemma}

\begin{proof}[Proof of Proposition \ref{pr:stability weighted}]

Let $\| f \|_*=\sup_{Z \geq \pi } |f(Z)|Z ^2+\sup_{Z\geq -(\pi+a(s_0)) } |\pa_Z f(Z)|\langle Z \rangle^{2}$. Let an initial datum $\xi_0\in \mathcal B$ for \fref{1DPrandtl} satisfy the conditions of Proposition \ref{pr:bootstrap} and $\|u(s_0)\|_*\leq \frac{\delta_0}{2}$. Consider the open set of initial datum satisfying $\| \tilde \xi_0-\xi_0\|_{\mathcal B}< \bar \delta$ with corresponding variable $\tilde u$. Then, $\tilde \xi_0$ satisfies the conditions of Proposition \ref{pr:bootstrap} and $\| \tilde u(s_0)\|_*\leq \delta$. Hence $\tilde \xi$ satisfies the conclusions of Theorem \ref{th:main} from its proof done in subsection \ref{subsec:prbootstrap}, and the bounds \fref{bd:lowerglobal} and \fref{id:lowernearboundary} are then consequences of Lemma \ref{lem:additionallinftybounds} and of \fref{bd:origin bootstrap2}.

\end{proof}

\begin{proof}[Proof of Lemma \ref{lem:additionallinftybounds}]

Recall \fref{def:coefficientsm} and \fref{eq:def tH}. Let $m_4=\frac{1}{\lambda^4\mu^2}$. Note that one has the following estimates, from \fref{bd:boostrap improved parameters}, \fref{bd:boostrap improved parameters2}, \fref{bd:lossymodulation}, \fref{bd:etrap} (using Sobolev embedding), for $s_0$ large enough:
\be \label{newpreliminarybound1}
m_1=O(e^{-\frac 32 s}), \qquad m_2=O(e^{-\frac 32 s}), \qquad m_3'=O(e^{-\frac 52 s}), \qquad m_4=O(e^{-2s}),
\ee
\be \label{newpreliminarybound2}
\|u\|_{L^{\infty}}\leq e^{-\frac{s}{8}}, \qquad \| u\|_{L^{\infty}[-Z^*,Z^*]}+e^{-s}\| \pa_{Z}u\|_{L^{\infty}[-Z^*,Z^*]} \leq e^{-3s }.
\ee

\textbf{Step 1} \emph{Proof of (i)}. We rewrite the equation \fref{eq:F} as $(\pa_{s}+\mathcal M) F=0$, with the elliptic operator $\mathcal M=2\frac{\lambda_s}{\lambda}-F+(\pa_{Z}^{-1}F-\frac{\lambda_s}{\lambda}Z-m_2Z+m_3' )\pa_{Z} -m_4 \pa_{ZZ}$. We compute that $F'=e^{-\frac 18 s}Z^{-2}$ is a supersolution on $[\pi,+\infty)$. Indeed, from \fref{newpreliminarybound1}, \fref{newpreliminarybound2}, and since $G_1(Z)=0$ for $Z\geq \pi$ and $\int_0^\pi G_1=\pi/2$:
\begin{align*}
& Z^{2}e^{\frac 18 s}(\pa_{s}+\mathcal M)F'  = -\frac 18 +1-u+2\left(\frac{Z}{2}-\frac{\pi}{2}-\pa_{Z}^{-1}u \right)Z^{-1}+4m_1+2m_2-\frac{2m_3'}{Z}-\frac{6m_4}{Z^4}\\
&\qquad \qquad = \frac{7}{8}+(Z-\pi)Z^{-1}+o_{s_0\rightarrow \infty}(1) \ > \ 0
\end{align*}
where the $o()$ is uniform for $(s,Z)\in [s_0,\infty)\times [\pi,\infty)$. At the boundary $|F(s,\pi)|\leq F'(s,\pi)$ for all $s\geq s_0$ for $s_0$ large enough from \fref{newpreliminarybound2}. (i) is then a consequence of parabolic comparison principle.

\textbf{Step 2} \emph{Proof of (ii)}. Let $\Omega_1=[-\pi+a,-Z^*]\cup [Z^*,+\infty)$ and $\Omega_2=[-Z^*,Z^*]$. For $\chi$ a smooth cut-off with $\chi(y)=1$ for $y\leq -1$ and $\chi(y)=0$ for $y\geq 0$, we define $\chi^*(s,Z)= \chi(e^{s/2}(Z-\pi)) \chi(e^{s/2}(-Z-\pi))$. After smoothing the profile near the the points $\pm \pi$, we decompose $\partial_ZF$ as follows:
$$
\pa_{Z} F= \chi^* \pa_{Z} G_1+\bar{F}.
$$
Since $\pa_{Z} u=(\chi^*-1)\pa_{Z} G_1+\bar{F}$, recalling \fref{id:G1}, it is sufficient to prove (ii) for $\bar{F}$ in order to prove it for $\pa_{Z} u$. On $\Omega_1$ we have from \fref{eq:F}, \fref{bd:lossymodulation}, \fref{newpreliminarybound1} and \fref{newpreliminarybound2} that $(\pa_{s} +\bar{\mathcal M}+\tilde{\mathcal M})\bar{F}=\mathcal E$ where:
$$
\bar{\mathcal M }=\frac 12 -G_1+\mathcal T\pa_{Z}, \qquad  \mathcal E=-(\pa_{s}+\bar{\mathcal M}+\tilde{\mathcal M})(\chi^* \pa_{Z} G_1),
$$
\begin{align*}
\tilde{\mathcal M} & = m_1 \left(2-2Z\pa_Z\right)-m_2Z\pa_Z-u+\left(\pa_Z^{-1}u +m_3' \right)\pa_{Z}-m_4 \pa_{ZZ} \\
& \qquad = \ O(e^{-\frac 18 s})+\left(O\left(e^{-\frac 18 s} |Z| \right)\right)\pa_{Z}+(O(e^{-2s}))\pa_{ZZ}.
\end{align*}
We claim that there exists a smooth positive function $\tilde{w}$ on $\mathbb R\backslash \{0\}$ such that:
\be \label{id:technicalweight}
\tilde{w} (Z)=\left|\frac{\sin \left(\frac{Z}2\right)}{\cos \left(\frac{Z} 2\right)}\right|^{\frac{1}{10}}|\sin (Z)| \quad \mbox{for } |Z|\leq \frac{\pi}{2}, \quad \tilde{w}(Z)=\frac{1}{Z^2} \quad \mbox{for }|Z| \geq \pi+1, \quad \mbox{and} \quad \bar{\mathcal M}\tilde{w} \geq \frac{1}{20}\tilde{w}.
\ee
We relegate the proof of this fact to Step 3. We now show that $\bar{F}'(s,Z)=e^{-\frac{1}{50} s}\tilde{w}(Z)$ is a supersolution on $\Omega_1$. Note first that $|\pa_{Z}^i\tilde{w}|\lesssim Z^{-i}\tilde{w} $ for $i=1,2$. Hence on $\Omega_1$:
$$
\tilde{\mathcal M}\tilde{w} =O(e^{-\frac 18 s}\tilde{w} )+O\left(e^{-\frac 18 s} |Z| \right)O(|Z|^{-1}\tilde{w} )+O(e^{-2s})O(|Z|^{-2}\tilde{w})=O((e^{-\frac 18 s}+M^{-2})\tilde{w}),
$$
as $|Z|\geq Me^{-s}$. This and \fref{id:technicalweight} imply that on $\Omega_1$, for $M$ large enough and then $s_0$ large enough:
$$
(\pa_s+\bar{\mathcal M}+\tilde{\mathcal M})\bar F' \geq \frac{1}{50} \bar F'.
$$
Next, since $\bar{\mathcal M }(\pa_{Z} G_1)=0$, from \fref{id:G1} we obtain that $(\pa_{s}+\bar{\mathcal M})(\chi^* \pa_{Z} G_1)=O(e^{-\frac s 2})$ and that this has its support inside $[-\pi,-\pi+e^{-\frac s2}]\cup [\pi-e^{-\frac s2},\pi]$. From \fref{id:G1} and \fref{newpreliminarybound2} we also obtain that $\tilde{\mathcal M}(\chi^* \pa_{Z} G_1)=O(e^{-\frac 18 s}|Z|)$ and that this has its support inside $[-\pi,\pi]$. Therefore, since $\tilde{w}(Z)\sim |Z|^{1+\frac{1}{10}}$ as $Z\rightarrow 0$, and since $e^{-\frac 18 s}\leq e^{-\frac{1}{40}s}|Z|^{\frac{1}{10}}M^{-\frac{1}{10}}$ on $\Omega_1$, we get that on $\Omega_1$:
$$
|\mathcal E|\leq \frac{1}{100}\bar{F}'
$$
for $s_0$ large enough. The two above inequalities imply that $(\pa_{s}+\bar{\mathcal M}+\tilde{\mathcal M})\bar{F}'-\mathcal E\geq 0$ on $\Omega_1$. At the boundary, $|\bar{F}(s,\pm Z^*)|\leq \bar{F}'(s,\pm Z^*)$ from \fref{newpreliminarybound2}, $|\bar{F}(s,-\pi-a)|\leq Ce^{-s/2}\leq \bar{F}'(s,-\pi-a)$ from \fref{bd:origin bootstrap2} and $|\bar F(s_0)|\leq F'(s_0)$ for $\delta_0$ small enough. Hence $|\bar{F}|\leq \bar{F}'$ on $\Omega_1$ from parabolic comparison (using similarly $-\bar{F}'$ as a subsolution). This bound on $\Omega_1$ and the bound \fref{newpreliminarybound2} on $\Omega_2$ show (ii).

\textbf{Step 3} \emph{Existence proof for \fref{id:technicalweight}}. For example, we choose $\tilde{w}(Z)=\left|\frac{\sin \left(\frac{Z}{2}\right)}{\cos \left(\frac{Z}{2}\right)}\right|^{\frac{1}{10}}|\sin (Z)|\bar{w}(Z)$ on $(0,\pi)$, with $\bar{w}(Z)=1$ for $0<Z\leq \frac \pi2$ so that on $(0,\pi)$:
$$
\bar{\mathcal M}\tilde w=\frac{1}{20} \tilde w+\frac 12\left|\frac{\sin \left(\frac{Z}{2}\right)}{\cos \left(\frac{Z}{2}\right)}\right|^{\frac{1}{10}}|\sin (Z)|(\sin Z) \pa_Z \bar{w}.
$$
We choose $\bar{w}>0$ an increasing function of $Z$ such that $\tilde w$ is smooth on $(0,\pi]$ all the way up to $\pi$ with $\tilde w(\pi)=1$, and then $\bar{\mathcal M}\tilde w\geq \frac{1}{20} \tilde w$ on $(0,\pi]$ from the above identity. Next, on $[\pi,+\infty)$ we choose $\tilde w$ to be any smooth extension that is a non-increasing function of $Z$ with $\tilde w (Z)=Z^{-2}$ for $Z\geq \pi+1$. Then $\bar{\mathcal M}\tilde w=\frac 12 \tilde w-\frac 12(Z-\pi)\pa_{Z}\tilde w\geq \frac 12 \tilde w$ on $[\pi,+\infty)$. Hence the desired properties hold on $(0,\infty)$. We finally extend $\tilde w $ to $(-\infty,0)$ by even symmetry.

\end{proof}

\section{Application to the two-dimensional Prandtl system} \label{sec:analytic}

Here we prove Theorem \ref{th:analytic}. Recalling that $\xi_i$ is defined by \fref{def:xii}, we introduce:
\be \label{def:xiik}
\xi_{i,k}(t,y):=\pa_y^k \xi_{i}(t,y)
\ee
(i.e. $\xi_{i,k}=\pa_x^{2i+1} \pa_y^k u_{|x=0}$). First, we apply Proposition \ref{lwp:pr:analytic} using \fref{bd:analytichypothesis}, and get that there exist $T_0,\tau_0',C_0''>0$, and $(\xi_i)_{i\geq 0}\in C([0,T_0]\times [0,\infty))$ with $(\xi_i)_{i\geq 0}\in C^\infty((0,T_0]\times [0,\infty))$ a classical solution to \fref{eq:evolutionxii} on $(0,T_0]$ such that:
\begin{align} \label{bd:analytichypothesis3}
& |\xi_i(t,y) |  \leq C_0'' \tau_0^{'-2i-1}(2i+1)!  \langle y \rangle^{-2} \qquad \quad \mbox{for all }(t,y)\in [0,T_0]\times [0,\infty),\\
& \label{bd:analytichypothesis2} |\xi_{i,k}(T_0,y)|\leq C_0'' \tau^{'-2i-k-1}_0 (2i+k+1)! \langle y \rangle^{-2} \qquad \quad \mbox{for all } y\in  [0,\infty).
\end{align}
Thus now our aim is to control $(\xi_{i})_{i\geq 0}$ from $T_0$ up to time $T$.

We first establish linear estimates in Subsection \ref{subsec:linearanalytic}, then study $\xi_{1,0}$ in Subsection \ref{subsec:paxxx}, then all remaining derivatives in Subsection \ref{subsec:higheranalytic}. Theorem \ref{th:analytic} is proved in Subsubsection \ref{subsubsec:analytic}. Throughout this section, we assume that all the hypotheses of Theorem \ref{th:analytic}, \fref{bd:analytichypothesis3} and \fref{bd:analytichypothesis2} hold true. In particular the parameters $T,T_0,C_0,C'_0,C_0'',\iota,\tau_0,\tau_0'$ are independent of all other forthcoming parameters, since fixed a priori. We shall denote by $C$ a constant that may change values from line to line, but that depends solely on these parameters. Since the precise value of $\mu$ will never play a role, we assume:
$$
\mu=1
$$
without loss of generality. We perform the following renormalizations:
$$
\sfs=-\ln (T-t), \qquad \sfs_0=-\ln (T-T_0), \qquad z=(T-t)^{\frac 12}y-\pi, 
$$
$$
\xi(t,y)=(T-t)^{-1}\sfF(\sfs,z), \qquad \sfF(\sfs,z)=G_1(z)+\sfu (\sfs,z), \qquad \xi_{i,k}(t,y)=(T-t)^{\frac k2-3i-1}F_{i,k}(\sfs,z),
$$
and will use the following notation:
$$
(\partial^{-1}_zf)(z)=\int_{0}^z f, \qquad (\partial^{-1} f)(z)=\int_{-\pi}^z f
$$
so that $\pa^{-1}f=\int_{-\pi}^0 f+\pa_z^{-1}f$. The evolution equation for $F_{i,k}$ for $i+k\geq 1$ is from \fref{eq:evolutionxii}:
\begin{align}
\label{eq:Fikevo} &\pa_s F_{i,k}+\mathcal L_{i,k}F_{i,k}=\delta_{k\neq 0}\delta_{i\neq 0}(2i+1)F_{0,k+1} \pa^{-1} F_{i,0}+\sum_{j=1}^{i-1}\binom{2i+1}{2j} (\pa^{-1}F_{j,0})F_{i-j,k+1}\\
\nonumber &\qquad \qquad \quad \qquad \qquad-\sum_{(j,l)\in E^1_{i,k}} \binom{2i+1}{2j+1}\binom{k}{l} F_{j,l} F_{i-j,k-l}+\sum_{(j,l)\in E^2_{i,k} } \binom{2i+1}{2j}\binom{k}{l+1} F_{j,l}F_{i-j,k-l} 
\end{align}
where the sums run over indices in the sets $E^1_{i,k}=\{0\leq j\leq i, \ 0\leq l\leq k \}\backslash \{(0,0),(i,k)\}$ and $E^2_{i,k}=\{0\leq j\leq i, \ 0\leq l\leq k-1\}\backslash \{(0,0)\}$, with Kronecker notation $\delta_{p\neq 0}=0$ if $p=0$ and $\delta_{p\neq 0}=1$ if $p\geq 1$ and similarly for $\delta_{p=0}$, and where the linearised operator is:
\begin{align}
\nonumber \mathcal L_{i,k}F_{i,k}=&-\left(\frac k2-3i-1+(2i-k+2)\sfF\right) F_{i,k}-\left(\frac z 2-\pa^{-1}\sfF+\frac \pi2\right) \pa_z F_{i,k}-e^{-2s} \pa_{zz}F_{i,k}\\
\label{id:defmathcalLik}&+\delta_{k= 0}\delta_{i\neq 0}(2i+1)\pa_z \sfF \pa^{-1}F_{i,k}.
\end{align}
Note that above, from the assumptions of Theorem \ref{th:analytic}, $\sfF$ is well-defined for all times $\sfs \geq \sfs_0$. The quantity $ \mathcal L_{i,k}F_{i,k}$ is then linear with respect to $F_{i,k}$, but the coefficients of $ \mathcal L_{i,k}$ involve $\sfF$. We introduce the weight $\sfw$ and its associated weighted space:
$$
\sfw(z)=\left\{ \begin{array}{l l} 1 \qquad \mbox{for }-\pi\leq z \leq \pi,\\
\langle z-\pi \rangle^{-2} \qquad \mbox{for }z\geq \pi \end{array} \right.  \qquad \| f \|_{L^{\infty}_\sfw}:=\sup_{z\geq -\pi} \ \frac{|f(z)|}{\sfw(z)}.
$$
so that hypothesis (i) in Theorem \ref{th:analytic} implies, with $0<\iota\leq 1/8$ without loss of generality:
\be \label{bd:decayassumptionpaxxx}
\| \sfu \|_{L^{\infty}_\sfw}+\|\pa_z\sfu \|_{L^{\infty}_\sfw}\leq C_0e^{-\iota \sfs}.
\ee

\subsection{Linear bounds} \label{subsec:linearanalytic}

The semi-group generated by the linear part is denoted by $S_{i,k}$. That is, we write $v(\sfs)=S_{i,k}(\sfs_1,\sfs)(v_0)$ for the solution $v$ on $[-\pi,+\infty)$ of:
\be \lab{eq:linearFik}
\pa_{\sfs} v+\mathcal L_{i,k}v=0, \qquad v(\sfs,-\pi)=0, \qquad v(\sfs_1,z)=v_0.
\ee

\begin{proposition} \label{pr:linearanalytic}

Assume hypothesis (i) in Theorem \ref{th:analytic} holds, and set for any $\eta>0$ the constants:
\be \label{def:cik}
c_{i,k}:=  \max \left(-i-\frac k2 +1\ ,\ \frac k2-3i-1 \right)+\eta \langle i \rangle .
\ee
Then, there exists $C,\sfK>0$ depending on $\eta$, $T$, $C_0$ and $\iota$ but independent of $i$ and $k$ such that the following bound holds true for any $i+k\geq 1$ and $\sfs_2\geq \sfs_1 \geq \sfs_0$:
\be \label{bd:lineaire}
\left\| S_{i,k}(\sfs_1,\sfs_2)(v_0)\right\|_{L^{\infty}_\sfw}\leq C e^{c_{i,k}(\sfs_2-\sfs_1)}\left(\frac{(1+e^{-\frac{\iota}{2}\sfs_1})e^{\sfK e^{-\sfs_1}}}{(1+e^{-\frac{\iota}{2}\sfs_2})e^{\sfK e^{-\sfs_2}}}\right)^{a_{i,k}} \left\| v_0\right\|_{L^{\infty}_\sfw}
\ee
where $a_{i,k}$ is defined by \fref{def:bik}. Moreover, one can take $C=1$ in the inequality above if $k\geq 1$.

\end{proposition}

\begin{remark}

In the right-hand side of \fref{bd:lineaire}, $e^{c_{i,k}(\sfs_2-\sfs_1)}$ is the sharp leading factor. The factor $(1+e^{-\frac{\iota}{2}\sfs})$ controls lower order terms close to the blow-up time. The blow-up time $T$ is arbitrary, hence there is a transient regime between $t=0$ and a time close to $T$, in which $\xi$ has not yet entered its asymptotic regime described by (i) in Theorem \ref{th:analytic} (i.e. $\tilde \xi$ may be large). The $e^{Ke^{-\sfs}}$ factor controls the solution in this transient regime. Together with the exponent $a_{i,k}$, they could have been chosen differently, but such formulation will be easier to use in the sequel.

\end{remark}

To prove Proposition \ref{pr:linearanalytic}, when $k\geq 1$ we decompose $\mathcal L_{i,k}$ as:
$$
\mathcal L_{i,k}=\mathcal L_{i,k}'+\tilde{\mathcal L}_{i,k}'
$$
where the leading order and lower order linear operators are ($\mathcal T$ being defined in \fref{eq:def T}):
\be \lab{def:Lik'tildeLik'}
\mathcal L'_{i,k}=3i+1-\frac k2-(2i+2-k)G_1+\mathcal T(z)\pa_z, \qquad  \tilde{\mathcal L}'_{i,k}=-(2i+2-k)\sfu +(\pa^{-1} \sfu)\pa_z -e^{-2s} \pa_{zz}.
\ee
For $k=0$, there is a nonlocal term in \fref{id:defmathcalLik}. We will write $\pa_z\sfF\pa^{-1}=\pa_z\sfF \int_{-\pi}^0+\pa_z\sfF\pa_z^{-1}$ as the sum of a projection onto $\pa_z \sfF$ and of a nonlocal term that we treat perturbatively. $\pa_z \sfF$ is indeed a stable eigenfunction at leading order for $\mathcal L_{i,0}$, as the next Lemma will show. Let us define
\be \lab{def:chi*}
\chi^*(\sfs,z)= \chi\left(e^{\frac{\sfs}{2}}(z-\pi)\right) \chi \left(e^{\frac{\sfs}{2}}(-z-\pi)\right)
\ee
where $\chi$ is a smooth cut-off with $\chi(y)=1$ for $y\leq -1$ and $\chi(y)=0$ for $y\geq 0$. Then

\begin{lemma} \label{lem:linearpaxxx2}
For any $i\geq 1$, there holds the identity for $\pa_z G_1$:
\be \label{id:eigenfunctionpaZG1}
\left(\mathcal L_{i,0}'+(2i+1)\pa_z G_1\pa^{-1}\right)\pa_z G_1=\left(3i+\frac 12 \right)\pa_z G_1.
\ee
Moreover, assume \fref{bd:decayassumptionpaxxx} and set $\phi (\sfs,z)= \chi^*(\sfs,z)\pa_z G_1(z)$, with $\chi^*$ given by \fref{def:chi*}. Then there holds for all $i\geq 1$ (the constant in the $O()$ being universal and uniform):
\be \label{bd:tildeR}
\sfR_i:=\pa_{\sfs} \phi +(\mathcal L_{i,0}+(2i+1)\pa_z G_1\pa^{-1} )\phi-\left(3i+\frac 12 \right) \phi=O(ie^{-\iota \sfs})
\ee
and that $\sfR_i$ has compact support on $[-\pi,\pi]$.

\end{lemma}

\begin{proof}

Differentiating equation %the fact that $G_1$ solves 
\fref{eq:F1} yields the identity:
$$
\frac 12 \pa_z G_1-G_1\pa_z G_1+\left(-\frac z2 +\pa_{z}^{-1}G_1 \right)\pa_{zz} G_1=0.
$$
In turn, this directly implies \fref{id:eigenfunctionpaZG1} as $\int_{-\pi}^0G_1=\pi/2$. Using \fref{id:eigenfunctionpaZG1} and \fref{bd:decayassumptionpaxxx}, we next make the following computation which proves \fref{bd:tildeR} :
\bea
\nonumber | \sfR_i| &=&\biggl| \pa_{\sfs} \phi+ \left(3i+1-(2i+2)\sfF\right) (\phi-\pa_z G_1)-\left(\frac z 2-\pa_z^{-1}\sfF+\frac \pi2 \right) \pa_z (\phi-\pa_z G_1)-e^{-2\sfs} \pa_{zz}\phi\\
\nonumber &&\quad +(2i+1)(\pa_z\sfF) \pa^{-1}(\phi-\pa_z G_1)+\left(3i+\frac 12 \right)(\phi-\pa_z G_1)\\
\nonumber &&\quad -(2i+2)\sfu \pa_z G_1+(\pa^{-1}\sfu )\pa_z \pa_z G_1+(2i+1)(\pa_z \sfu)\pa_z^{-1}\pa_z G_1 \biggr|\\
\nonumber &\lesssim & e^{-\frac \sfs2}+ie^{-\frac{\sfs}2}+e^{-\frac{\sfs}{2}}+e^{-\frac 32 \sfs}+ i e^{-\sfs}+ ie^{-\frac \sfs2}+ie^{-\iota \sfs}+e^{-\iota \sfs}+ie^{-\iota \sfs}.
\eea

\end{proof}

In the case $k=0$ we thus decompose a solution $v$ of \fref{eq:linearFik} with $\phi$ defined in Lemma \ref{lem:linearpaxxx2}:
\be \label{id:eqlineardecomp}
v(\sfs,z)=b (\sfs)\phi(\sfs,z) +v'(\sfs,z), \qquad \mbox{with} \qquad \left\{ \begin{array}{l l} \frac{d b}{d\sfs}=-(3i+\frac 12)b-(2i+1)\int_{-\pi}^0  v', \\ b(\sfs_1)=0. \end{array} \right.
\ee
We obtain the following evolution equation for $v'$ using \fref{id:defmathcalLik} and Lemma \ref{lem:linearpaxxx2}: 
\be \label{id:eqlineardecomp2}
(\pa_{\sfs}+\mathcal L_{i,0}'+\tilde{\mathcal L}_{i,0}'+\hat{\mathcal L}_{i,0}')v'=b\sfR_i
\ee
where the leading order and lower order elliptic operators $\mathcal L_{i,0}'$ and $\tilde{\mathcal L}_{i,0}'$ are given by \fref{def:Lik'tildeLik'} with $k=0$, and where the nonlocal operator is:
$$
\hat{\mathcal L}'_{i,0}v'=(2i+1)\pa_z G_1\pa_z^{-1}v'+(2i+1)\pa_z \sfu \pa^{-1}v' +(2i+1)(1-\chi^*)\pa_z G_1\int_{-\pi}^0 v',
$$
We first study the dynamics generated by $\mathcal L_{i,k}'+\tilde{\mathcal L}_{i,k}'$. We write $\tilde v(\sfs)=\tilde S_{i,k}(\sfs_1,\sfs)(\tilde v_0)$ for the solution $\tilde v$ on $[-\pi,+\infty)$ of:
\be \lab{eq:linearFiktilde}
\pa_{\sfs} \tilde v+\mathcal L_{i,k}'\tilde v+\tilde{\mathcal L}_{i,k}'\tilde v=0, \qquad \tilde v(\sfs,-\pi)=0, \qquad \tilde v(\sfs_1,z)=\tilde v_0.
\ee
Let $\tilde{\sfw}:[-\pi,\infty)\rightarrow (0,\infty)$ be a function that satisfies all the following properties (note that it is possible to construct explicitly such a weight $\tilde{\sfw}$ for any $\eta>0$):
\begin{itemize}
\item[(i)] $\tilde{\sfw}$ is $C^2$. $\tilde{\sfw}$ is non-increasing on $[-\pi,0]$, $\tilde{\sfw} (0)=1$, and $\tilde{\sfw}$ is non-decreasing on $[\pi,2\pi]$.
 $\tilde{\sfw}(z)=\tilde{\sfw} (\pi)\langle z-\pi \rangle^{-2}$ for $z\geq \pi$.
\item[(ii)] For all $z \in [-\pi,\pi]$, $\left|\pa_z^{-1} \tilde{\sfw}(z) \right|\leq \eta^2 \tilde{\sfw}(z)$.
\end{itemize}

We introduce the space $\| f \|_{L^{\infty}_{\tilde\sfw}}=\sup_{z\geq -\pi} \ |f(z)|\tilde{\sfw}^{-1}(z)$.

\begin{lemma} \label{lem:premilinearylinear}

For any $\eta>0$, there exist $\sfs^*$ and $\sfK>0$ such that for all $i+k\geq 1$ and $\sfs_2\geq \sfs_1\geq \sfs_0$:
\be
\lab{bd:lineairepreliminaire2}  \| \tilde S_{i,k}(\sfs_1,\sfs_2)(\tilde v)\|_{X}\leq e^{c_{i,k}(\sfs_2-\sfs_1)}\left(\frac{(1+e^{-\frac{\iota}{2}\sfs_1})e^{\sfK e^{-\sfs_1}}}{(1+e^{-\frac{\iota}{2}\sfs_2})e^{\sfK e^{-\sfs_2}}}\right)^{a_{i,k}} \| \tilde v\|_{X} \qquad \mbox{if }k\geq 1 \mbox{ and }i+k\geq2,
\ee
\be \lab{bd:lineairepreliminaire} \| \tilde S_{i,k}(\sfs_1,\sfs_2)(\tilde v)\|_{X}\leq e^{c_{i,k}(\sfs_2-\sfs_1)}\left(\frac{1+e^{-\frac{\iota}{2}\sfs_1}}{1+e^{-\frac{\iota}{2}\sfs_2}}\right)^{a_{i,k}} \| \tilde v\|_{X}, \qquad \mbox{if }\sfs_2\geq \sfs_1\geq \sfs^*,
\ee
where $X$ denotes either $L^{\infty}_{\tilde{\sfw}}$ or $L^{\infty}_\sfw$, and $c_{i,k}$ and $a_{i,k}$ are defined in \fref{def:cik} and \fref{def:bik}.
\end{lemma}

\begin{proof}

Let $\mathbf{w}$ denote either $\sfw$ or $\tilde{\sfw}$, and let $h(\sfs)$ denote either $1$ or $e^{-a_{i,k}\sfK e^{-\sfs}}$. We prove that
\be \lab{def:barf}
\sfS (\sfs,z):=e^{c_{i,k} \sfs}(1+e^{-\frac{\iota}{2}\sfs})^{-a_{i,k}}h(\sfs)  \mathbf{w} (z)
\ee
is a supersolution for the parabolic operator $\pa_{\sfs}+\mathcal L_{i,k}'+\tilde{\mathcal L}_{i,k}'$. We compute:
\begin{align}
\lab{id:intersupersol}R & := \pa_{\sfs} \sfS - \left(\frac k2-3i-1+(2i-k+2)\sfF\right) \sfS-\left(\frac z 2-\pa^{-1}\sfF-\frac \pi2\right) \pa_z \sfS-e^{-2s} \pa_{zz}\sfS\\
\nonumber &\quad = \sfS\left[c_{i,k}+\frac{\iota a_{i,k}}{2}e^{-\frac \iota2\sfs}+\frac{\pa_{\sfs}h}{h}-\frac k2+3i+1-(2i-k+2)\sfF+\left(\mathcal T(z)+\pa^{-1}\sfu\right) \frac{\pa_z\mathbf{w}}{\mathbf{w}}-e^{-2\sfs} \frac{\pa_{zz}\mathbf{w}}{\mathbf{w}} \right].
\end{align}
On the interval $[\pi,+\infty)$, using that $G_1(z)=0$, $\pa_{\sfs}h\geq 0$ and $\mathcal T(z)=(z-\pi)/2$ and \fref{bd:decayassumptionpaxxx}:
$$
R\geq \sfS \left[c_{i,k}+\frac{\iota a_{i,k}}{2}e^{-\frac \iota2\sfs}-\frac k2+3i+1+O(\langle i+k\rangle e^{-\iota \sfs})-\left(\frac{z-\pi}{2}+O(\langle z \rangle e^{-\iota \sfs})\right) \frac{\pa_z\mathbf{w}}{\mathbf{w}}-e^{-2\sfs} \frac{\pa_{zz}\mathbf{w}}{\mathbf{w}} \right].
$$
We have $c_{i,k}-\frac k2+3i+1\geq \eta \langle i \rangle$. On $[\pi,+\infty)$, $\pa_z \sfw \leq 0$ and $|\pa_z^j \sfw| \lesssim \sfw \langle z \rangle^{-j}$ for $j=1,2$. Hence:
$$
R\geq \sfS \left[ \eta \langle i \rangle+\frac{\iota a_{i,k}}{2}e^{-\frac \iota2\sfs}-C\langle i+k \rangle e^{-\iota \sfs}-Ce^{-2\sfs} \langle z\rangle^{-2} \right]>0
$$
for $\sfs$ larger than some $\sfs^*$ depending on $\eta$, $C_0$, $\iota$ and $T$, but independent of $i$ and $k$.

On $[-\pi,\pi]$, we have $0\leq G_1\leq 1$ so that $c_{i,k}-\frac k2+3i+1-(2i-k+2)G_1\geq \eta \langle i \rangle$ from \fref{def:cik}. As $\mathcal T(z)$ is nonpositive on $[-\pi,0]$ and nonnegative on $[0,\pi]$, and $\mathbf{w}$ is non-increasing on $[-\pi,0]$ and non-decreasing on $[0,\pi]$ we get $\mathcal T(z)\pa_z \mathbf{w}\geq 0$. Hence from \fref{id:intersupersol}, using \fref{bd:decayassumptionpaxxx} and $\pa_{\sfs}h\geq 0$:
\begin{align*}
R & \geq  \sfS\left[\eta \langle i \rangle+\frac{\iota a_{i,k}}{2}e^{-\frac \iota2\sfs}-(2i-k+2)\sfu+\pa^{-1}\sfu \frac{\pa_z\mathbf{w}}{\mathbf{w}}-e^{-2\sfs} \frac{\pa_{zz}\mathbf{w}}{\mathbf{w}} \right]\\
&\qquad \geq \sfS\left[\eta \langle i \rangle+\frac{\iota a_{i,k}}{2}e^{-\frac \iota2\sfs}-C\langle i+k\rangle e^{-\iota \sfs}-Ce^{-2\sfs} \right] \ > \ 0
\end{align*}
for $\sfs$ large enough as $\langle i\rangle +a_{i,k}\gtrsim \langle i+k\rangle$. We conclude that $R>0$ on $[-\pi,+\infty)$ for $\sfs$ large enough. Hence there exists $\sfs^*$ such that $\sfS$ is a supersolution for $\sfs\geq \sfs^*$. The bound \fref{bd:lineairepreliminaire} is then a consequence of the maximum principle. Assume now $i+k\geq 2$ and $h(\sfs)=e^{-a_{i,k}\sfK e^{-\sfs}}$. We proved that $\sfS$ is a supersolution for $\sfs\geq \sfs^*$. For $\sfs^*\geq \sfs\geq \sfs_0$ we have that $\frac{\pa_{\sfs}h}{h}=Ka_{i,k}e^{-\sfs}\geq cK\langle i+k\rangle$ for some $c>0$ depending on $s^*$, since $a_{i,k}\gtrsim \langle i+k\rangle$. All other terms in \fref{id:intersupersol} are $O(\langle i+k\rangle)$ with some uniform constant, hence for all $\sfs\in [\sfs_0,\sfs^*]$ and $z\geq -\pi$:
$$
R \geq \sfS\left[\frac{\pa_{\sfs}h}{h}-C\langle i+k\rangle \right]\geq \sfS \left[cK\langle i+k\rangle-C\langle i+k\rangle\right]>0
$$
for $\sfK$ large enough depending on $\sfs^*$. Hence $\sfS$ is a supersolution for all $\sfs\geq \sfs_0$, proving \fref{bd:lineairepreliminaire2} applying the maximum principle.

\end{proof}

We now study the dynamics of \fref{id:eqlineardecomp} for $k=0$, setting $b=0$. We write $\hat v(\sfs)=\hat S_{i,0}(\sfs_1,\sfs)(\hat v_0)$ for the solution $\hat v$ on $[-\pi,+\infty)$ of:
\be \lab{eq:linearFik2}
\pa_{\sfs} \hat v+\mathcal L_{i,0}'\hat v+\tilde{\mathcal L}_{i,0}'\hat v+\hat{\mathcal L}_{i,0}'\hat v=0, \qquad \hat v(\sfs,-\pi)=0, \qquad \hat v(\sfs_1,z)=\hat v_0.
\ee

\begin{lemma}

For any $\eta>0$, if $\sfs_2\geq \sfs_1$ are large enough, there holds for all $i\geq 1$:
\be \lab{bd:lineairepreliminaire3}
\| \hat S_{i,0}(\sfs_1,\sfs_2)(\hat v)\|_{L^{\infty}_\sfw}\leq C(\eta)e^{c_{i,0}(\sfs_2-\sfs_1)}\left(\frac{1+e^{-\frac{\iota}{2}\sfs_1}}{1+e^{-\frac{\iota}{2}\sfs_2}}\right)^{a_{i,0}} \| \hat v\|_{L^{\infty}_\sfw}
\ee
where $c_{i,0}$ and $a_{i,0}$ are defined in \fref{def:cik} and \fref{def:bik}.

\end{lemma}

\begin{proof}

We reason with a parameter $\eta'>0$, and let $\tilde{\sfw}'=\tilde{\sfw}[\eta']$ and $c_{i,0}'=c_{i,0}[\eta']$. Using the assumption (ii) on $\tilde{\sfw}'$, the bound \fref{bd:decayassumptionpaxxx} and that $G_1$ vanishes outside $[-\pi,\pi]$ we get the bound, if $\eta'$ has been chosen small enough, and then $\sfs$ large enough:
$$
\| \hat{\mathcal L}'_{i,0} \hat v\|_{L^{\infty}_{\tilde{\sfw}'}}\leq \left(Ci \eta^{'2}+C(\eta') i e^{-\iota \sfs}+C(\eta') i e^{-\frac{\sfs}{2}}\right)\| \hat v \|_{L^{\infty}_{\tilde{\sfw}'}}\leq \eta' i \| \hat v \|_{L^{\infty}_{\tilde{\sfw}'}}.
$$
Duhamel gives the identity $\hat S_{i,0}(\sfs_1,\sfs_2)(\hat v)=\tilde S_{i,0}(\sfs_1,\sfs_2)(\hat v)-\int_{\sfs_1}^{\sfs_2}\tilde S_{i,0}(\sfs,\sfs_2)(\hat{\mathcal L}'_{i,0}(S_{i,0}(\sfs_1,\sfs)(\hat v)))d\sfs$. Set $\Phi(\sfs)=\left(\frac{1+e^{-\frac{\iota}{2}\sfs}}{1+e^{-\frac{\iota}{2}\sfs_1}}\right)^{a_{i,0}}\| \hat S_{i,0}(\sfs_1,\sfs)(\hat v)\|_{L^{\infty}_{\tilde \sfw'}}$. The above bound and \fref{bd:lineairepreliminaire} imply:
$$
\Phi(\sfs_2) \leq   e^{c_{i,0}'(\sfs_2-\sfs_1)} \| \hat v\|_{L^{\infty}_{\tilde{\sfw}'}}+i \eta' \int_{\sfs_1}^{\sfs_2} e^{c_{i,0}'(\sfs_2-\sfs)} \Phi(\sfs)d\sfs.
$$
Gronwall then gives $\Phi(\sfs)\leq  e^{(c_{i,0}'+i\eta')(\sfs_2-\sfs_1)} \| \hat v\|_{L^{\infty}_{\tilde{\sfw}'}}$. This proves the Lemma, upon noticing that $c_{i,k}[\eta']+i\eta'\leq c_{i,k}[\eta] $ for $\eta'$ small enough, and that the weights $\tilde{\sfw}'$ and $\sfw$ are equivalent.

\end{proof}

\begin{proof}[Proof of Proposition \ref{pr:linearanalytic}]

\textbf{Step 1} We claim that for any $\eta>0$, there exists $\sfs^*,C'>0$ such that:
\be \label{bd:lineaireblowuptime}
\left\| S_{i,k}(\sfs_1,\sfs_2)(v_0)\right\|_{L^{\infty}_\sfw}\leq C' e^{c_{i,k}(\sfs_2-\sfs_1)}\left(\frac{1+e^{-\frac{\iota}{2}\sfs_1}}{1+e^{-\frac{\iota}{2}\sfs_2}}\right)^{a_{i,k}} \left\| v_0\right\|_{L^{\infty}_\sfw} \qquad \mbox{for }\sfs_2\geq \sfs_1\geq \sfs^*,
\ee
for all $i+k\geq 1$, and that one can take $C'=1$ if $k\geq 1$. For $k\geq 1$, this is Lemma \ref{bd:lineairepreliminaire}. So we only need to prove the above inequality for $k=0$. Let $\eta'>0$, and $c_{i,0}'=c_{i,0}[\eta']$. Recall \fref{id:eqlineardecomp} and \fref{id:eqlineardecomp2}. We write by Duhamel $b(\sfs)=(2i+1)e^{-(3 i+\frac 12) \sfs}\int_{\sfs_1}^{\sfs}e^{(3i+\frac 12) \sfs'}\int_{-\pi}^0v'(\sfs')d\sfs'$ and $v'(\sfs)=\hat S_{i,0}(\sfs_1,\sfs)(v_0)+\int_{\sfs_1}^{\sfs} b(\sfs')\hat S_{i,0}(\sfs',\sfs)(\sfR_i(\sfs'))d\sfs'$. Recall that \fref{bd:tildeR} gives $\| \sfR_i\|_{L^{\infty}_\sfw}\leq Cie^{-\iota \sfs}$. We take $\sfs_1$ large enough, and apply the linear estimate \fref{bd:lineairepreliminaire3} with parameter $\eta'$ to get:
\bea
\nonumber &|b(\sfs)|\leq Ci \int_{\sfs_1}^{\sfs}e^{-\left(3i+\frac12\right) (\sfs-\sfs')}\| v'(\sfs ')\|_{L^{\infty}_\sfw}d\sfs', \\
\nonumber & \| v'(\sfs)\|_{L^{\infty}_\sfw}\leq C e^{c_{i,k}'(\sfs-\sfs_1)}\left(\frac{1+e^{-\frac{\iota}{2}\sfs_1}}{1+e^{-\frac{\iota}{2}\sfs}}\right)^{a_{i,k}} \left\| v_0\right\|_{L^{\infty}_\sfw}+Ci \int_{\sfs_1}^{\sfs} e^{c_{i,k}'(\sfs-\sfs')}\left(\frac{1+e^{-\frac{\iota}{2}\sfs'}}{1+e^{-\frac{\iota}{2}\sfs}}\right)^{a_{i,k}} e^{-\iota \sfs'}|b(\sfs')|d\sfs'
\eea
Consider the function $\Phi(\sfs)=\left(\frac{1+e^{-\frac{\iota}{2}\sfs}}{1+e^{-\frac{\iota}{2}\sfs_1}}\right)^{a_{i,k}}(\| v'(\sfs)\|_{L^{\infty}_\sfw}+\eta' |b(\sfs)|)$. Take $\sfs_1$ large enough so that $e^{-\iota \sfs'}\leq \eta^{'2}$ for $\sfs'\geq \sfs_1$. Note that for all $i+k\geq1$ one has $c_{i,k}'\geq -3i-\frac 12$. Hence $\Phi$ satisfies the integral inequality $\Phi(\sfs)\leq C e^{c_{i,k}'(\sfs-\sfs_1)} \left\| v_0\right\|_{L^{\infty}_\sfw}+Ci \eta' \int_{\sfs_1}^{\sfs}e^{c_{i,k}'(\sfs-\sfs')}\Phi(\sfs')d\sfs'$. Hence $\Phi(\sfs)\leq C e^{(c_{i,k}'+Ci \eta')(\sfs-\sfs_1)} \left\| v_0\right\|_{L^{\infty}_\sfw}$ by Gronwall Lemma. This shows \fref{bd:lineaireblowuptime}, taking $\eta'$ small so that $c_{i,k}[\eta']+Ci\eta'\leq c_{i,k}[\eta] $.

\textbf{Step 2} Fix $\eta>0$, $\sfs^*$ as in Step 1, and $\sfs^*\geq \sfs_2\geq \sfs_1\geq -\ln T$. The control of \fref{eq:linearFik} on $[-\ln T,s^*]$ is direct since this is a linear equation with bounded coefficients, over a \emph{finite} interval. Indeed, the functions $|\sfF|,|\pa_z \sfF|\leq C \sfw$ are uniformly bounded from \fref{bd:decayassumptionpaxxx}. Then, \fref{eq:linearFik} is a linear parabolic equation, with variable coefficients in the elliptic part that are uniformly bounded by $C\langle i+k\rangle$, and with a nonlocal operator $v\mapsto \delta_{i\neq 0}(2i+1)\pa_z \sfF \pa^{-1}v$ that is bounded from $L^{\infty}_{\tilde{\sfw}}$ onto $L^{\infty}_{\tilde \sfw}$ with operator norm $\leq C \langle i \rangle$. As a result, we have a classical linear bound using a standard Gronwall argument: there exists $C>0$ depending on $T$, $s^*$, $\iota$ and $C_0$ such that $ \left\| S_{i,k}(\sfs_1,\sfs_2)(v_0)\right\|_{L^{\infty}_\sfw}\leq e^{C\langle i+k\rangle (\sfs_2-\sfs_1)} \left\| v_0 \right\|_{L^{\infty}_\sfw}$. Since for $\sfK',C''$ large enough, $e^{C\langle i+k\rangle (\sfs_2-\sfs_1)}\leq C''e^{a_{i,k}\sfK'(e^{-\sfs_1}-e^{-\sfs_2})}$ uniformly for $\sfs^*\geq \sfs_2\geq \sfs_1\geq -\ln T$, we get:
$$
\left\| S_{i,k}(\sfs_1,\sfs_2)(v_0)\right\|_{L^{\infty}_\sfw}\leq C''e^{a_{i,k}\sfK'(e^{-\sfs_1}-e^{-\sfs_2})} \left\| v_0 \right\|_{L^{\infty}_\sfw}.
$$
The above estimate on $[-\ln T,\sfs^*]$ and \fref{bd:lineaireblowuptime} on $[\sfs^*,\infty)$ directly imply \fref{bd:lineaire} for all $\sfs_2\geq \sfs_1\geq -\ln T$ (upon using them after writing $S_{i,k}(\sfs_1,\sfs_2)=S_{i,k}(\sfs^*,\sfs_2)\circ S_{i,k}(\sfs_1,\sfs^*)$ if $\sfs_1\leq \sfs^*\leq \sfs_2$, and up to taking $\sfK,C>0$ large enough depending on $\sfK',C',C'',\sfs^*,T$).

\end{proof}

\subsection{Control of the third order tangential derivative on the axis} \lab{subsec:paxxx}

We first control $\xi_{1,0}$ (equivalently, $F_{1,0}$). This is because the growth of this function as $t\rightarrow T$ will be responsible for the $(T-t)^{7/4}$ bound of the radius of analyticity, and because the bound below is critical for the linearised analysis due to the presence of a non-trivial kernel, thus requiring a more careful treatment.

\begin{proposition} \lab{pr:paxxx}

Assume hypothesis (i) in Theorem \ref{th:analytic} and $\| \xi_{1,0}(0)\|_{L^{\infty}(\langle y\rangle^{-2})}<+\infty$. Then the solution $\xi_{1,0}$ of \fref{eq:Fikevo} is defined for all $t\in [0,T)$. Moreover, there exists $C_2>0$ such that for all $t\in [0,T)$ and $\sfs\geq \sfs_0$:
\be \label{bd:paxxxoptimal}
\| \xi_{1,0}(t)\|_{L^{\infty}([0,\frac 14])}\leq C_2, \qquad \mbox{and}\qquad \| F_{1,0}(\sfs)\|_{L^{\infty}_\sfw}\leq C_2.
\ee

\end{proposition}

The proof is decomposed in several steps, and Proposition \ref{pr:paxxx} is proved at the end of this subsection. The existence up to time $T$ is straightforward, since $F_{1,0}$ solves a linear equation:
\be \label{id:evoF10}
(\pa_{\sfs}+\mathcal L_{1,0}) F_{1,0}=0 \qquad \Leftrightarrow \qquad (\pa_{\sfs}+\mathcal N+\tilde{\mathcal N}-e^{-2s} \pa_{zz}) F_{1,0}=0,
\ee
where the leading and lower order linear operators are:
$$
\mathcal N F_{1,0}=4(1-G_{1}(z))F_{1,0}+\mathcal T(z)\pa_zF_{1,0}+3 \pa_z G_1(z) \pa^{-1}F_{1,0},
$$
$$
\tilde{\mathcal N}F_{1,0}=-4\sfu F_{1,0}+\pa^{-1}\sfu \pa_zF_{1,0}+3 \pa_z \sfu \pa^{-1}F_{1,0} .
$$
Applying Proposition \ref{pr:linearanalytic}, for any $\eta>0$, as $c_{1,0}=\langle 1 \rangle \eta$ and $a_{1,0}=0$, we obtain that $\| F_{1,0}(\sfs)\|_{L^{\infty}}\leq C e^{\langle 1 \rangle \eta \sfs}$. By taking a smaller $\eta$ in this inequality and $\sfs$ large enough we get:
\be \label{bd:lossypaxxx}
\| F_{1,0}(\sfs)\|_{L^{\infty}(\mathsf w)}\leq e^{\eta \sfs} \qquad \mbox{for any }\eta>0,  \mbox{ for } \sfs \mbox{ large enough depending on }\eta.
\ee
This is almost the second bound in \fref{bd:paxxxoptimal} we wish to prove. The problem with improving to $\eta=0$ above is the presence of a nontrivial kernel.

\begin{lemma}[\cite{CGM2}, (vi) in Proposition 6] \label{lem:linearpaxxx}

There exists a $C^1$ solution $Q$ to $\mathcal N Q=0$ on $[-\pi,+\infty)$ that has the following properties. The support of $Q$ is $[-\pi,\pi]$, and $Q$ restricted to $[-\pi,\pi]$ is smooth. $Q$ is positive on $(-\pi,\pi)$ with $Q(0)=1$. There exist two positive constants $c$ and $c'$ such that $Q(z)\sim c(z+\pi)^8$ as $z\downarrow -\pi$ and $Q(z)\sim c'(\pi-z)$ as $z \uparrow \pi$.

\end{lemma}

To improve \fref{bd:lossypaxxx}, we prove boundedness in a parabolic neighborhood of a particular characteristics of the transport operator, and extend this local bound to a global one.

\begin{lemma} \label{lem:char}

There exists a solution $z^*(\sfs)$ of $\pa_{\sfs}z^*=-\frac{z^*}{2}+\int_{-\pi}^{z^*} \sfF(\sfs,z)dz-\frac{\pi}{2}$ such that:
\be \lab{bd:zstarpaxxx}
|z^*|\leq C e^{-\iota \sfs}.
\ee

\end{lemma}

\begin{proof}

For an initial time $\sfs_1\geq \sfs_0$, using $\int_{-\pi}^0G_1=\pi/2$ and \fref{bd:decayassumptionpaxxx} the ODE becomes:
\be \label{paxxx:id:oderecast}
\pa_{\sfs}z^*=-\frac{z^*}{2}+\int_{0}^{z^*}G_1 +\int_{-\pi}^{z^*} \sfu = \frac 12 z^*+O(z^{*3})+O(e^{-\iota \sfs}), \qquad z^*(\sfs_1)=z^*_0.
\ee
Consider for $M>0$ the sets $I^-_{M,\sfs_1}$ and $I^+_{M,\sfs_1}$ defined by
$$
I^\pm_{M,\sfs_1}=\{|z^*_0|\leq Me^{-\iota \sfs_1}, \ \exists \sfs_2\geq\sfs_1, \ |z^*(\sfs)|< Me^{-\iota \sfs} \ \mbox{for } \sfs_1\leq \sfs<\sfs_2 \mbox{ and }z^*(\sfs_2)=\pm Me^{-\iota \sfs_2}\}.
$$
Then for $M$ large enough and then for $\sfs_1$ large enough the following holds true. For $z^*_0\in I^\pm_{M,\sfs_1}$, at time $\sfs_2$ there holds using \fref{paxxx:id:oderecast}: $\pa_s (|e^{\iota \sfs}z^*|)(\sfs_2)=\frac M2+O(M^3e^{-2\iota \sfs_2})+O(1)>0$. This inequality, by continuity of the flow of the ODE \fref{paxxx:id:oderecast}, implies that both $I^-_{M,\sfs_1}$ and $I^+_{M,\sfs_1}$ are open in $[-Me^{-\iota \sfs_1},Me^{-\iota \sfs_1}]$. They are moreover disjoints by definition, and non-empty as they contain $-Me^{-\iota \sfs_1}$ and $Me^{-\iota \sfs_1}$ respectively. Hence, by connectedness, there exists $z^*_0\in[-Me^{-\iota \sfs_1},Me^{-\iota \sfs_1}]$ with $z^*_0 \notin I^-_{M,\sfs_1} \cup I^+_{M,\sfs_1}$. The solution to \fref{paxxx:id:oderecast} with data $z^*_0$ at time $\sfs_1$ then satisfies the conclusions of the Lemma by definition of $I^-_{M,\sfs_1}$ and $I^+_{M,\sfs_1}$.

\end{proof}

\begin{lemma} \label{lem:controlchar}

Let $z^*$ satisfy the conclusion of Lemma \ref{lem:char}. Then there exists $d\in \mathbb R$ such that for any $\sfM>0$ and $0<\iota'<\iota$, for all large enough $\sfs$ and all $z\in [z^*-\sfM e^{-s},z^*+\sfM e^{-s}]$:
\be \label{bd:Linftynearzstar}
|F_{1,0}(\sfs,z)-d|\leq e^{-\iota' \sfs}.
\ee

\end{lemma}

\begin{proof}
We switch toward the following parabolic variables:
\be \label{id:varparapaxxx}
\sfY= \frac{z-z^*}{T-t}, \qquad F_{1,0}(\sfs,z)= f_{1,0}(\sfs,\sfY), \qquad \sfY^*= \frac{\pi+z^*}{(T-t)}.
\ee
Then $f_{1,0}$ solves the following equation on $[-\sfY^*,+\infty)$ with Dirichlet boundary condition:
$$
\pa_{\sfs} f_{1,0}+\left(\frac{\sfY}{2}+\pa_{\sfY}^{-1}\sfF(\sfs,z)\right)\pa_{\sfY} f_{1,0}-\pa_{\sfY\sfY} f_{1,0}=(4\sfF-4) F_{1,0}-3(\int_{-\pi}^z F_{1,0})\pa_z \sfF,
$$
the right-hand side being a function of the space variable $z$. Set $d(\sfs)=\int_{\mathbb R} \tilde \chi f_{1,0}\rho(\sfY)d\sfY$, where $\rho(\sfY)=e^{-3\sfY^2/4}$ and $\tilde \chi(\sfs,\sfY)=\chi(e^{-\frac{\sfs}{2}}\sfY)$ for $\chi$ a smooth cut-off, $\chi(y)=1$ for $|y|\leq 1$ and $\chi(y)=0$ for $|y|\geq 0$. Notice that the support of $\tilde \chi$ is strictly inside $[-\sfY^*,+\infty)$ for $\sfs$ large, justifying that the integral for $d$ is on $\mathbb R$. Note that $\pa_{\sfY}\rho(\sfY)=-\frac 32 \sfY \rho(\sfY)$. Then integrating by parts, using the exponential decay of $\rho$ and \fref{bd:lossypaxxx} to upper bound by $Ce^{-e^{\sfs}}$ all boundary terms due to $\tilde \chi$:
\be \lab{id:interpaxxx}
\pa_{\sfs} d= \int_{\mathbb R} \tilde \chi \left( \left( 4\sfF(\sfs,z)-4+\rho^{-1} \pa_{\sfY}\left(\rho \pa_{\sfY}^{-1}(\sfF(\sfs,z)-1)\right) \right)F_{1,0}-3(\int_{-\pi}^z F_{1,0})\pa_z \sfF \right)\rho d\sfY+O(e^{-e^{\sfs}}).
\ee
Above, we note that, using \fref{id:varparapaxxx}, \fref{bd:zstarpaxxx} and \fref{bd:decayassumptionpaxxx}:
\begin{align}
\lab{bd:interpaxxx1} |\sfF(\sfs,z)-1|\leq |G_1(z)-G_1(z^*)|+|G_1(z^*)-1|+|\sfu|\leq C\sfY e^{-\sfs}+Ce^{-\iota \sfs},\\
\lab{bd:interpaxxx2} |\pa_{z}\sfF(\sfs,z)|\leq |\pa_zG_1(z)-\pa_zG_1(z^*)|+|\pa_zG_1(z^*)|+|\pa_z\sfu|\leq C\sfY e^{-\sfs}+Ce^{-\iota \sfs}.
\end{align}
Let $0<\iota''<\iota$, and $\sfs$ large so that $|F_{1,0}|\leq e^{-\iota''\sfs}$ from \fref{bd:lossypaxxx}. Injecting \fref{bd:interpaxxx1} and \fref{bd:interpaxxx2} in \fref{id:interpaxxx} gives $|\pa_{\sfs}d|\leq Ce^{-(\iota-\iota'')\sfs}$. Hence the existence of $d_{\infty}\in \mathbb R$ with $|d-d_{\infty}|\leq Ce^{-(\iota-\iota'')\sfs}$. Set now $\bar f_{1,0}=f_{1,0}-d$. It solves the equation:
$$
\pa_{\sfs} \bar f_{1,0}+\left(\frac{\sfY}{2}+\pa_{\sfY}^{-1}\sfF(\sfs,z) \right)\pa_{\sfY}\bar f_{1,0}-\pa_{\sfY\sfY} \bar f_{1,0}=(4\sfF-4) F_{1,0}-3(\int_{-\pi}^z F_{1,0})\pa_z \sfF-\pa_s d.
$$
We compute the following energy estimate by integrating by parts, using the exponential decay of $\rho$, that $d$ is bounded and \fref{bd:lossypaxxx} to upper bound all boundary terms due to $\tilde \chi$ by $Ce^{-e^{\sfs}}$:
\begin{align*}
&\frac{d}{d\sfs} \frac{1}{2}\left(\int_{\mathbb R} |\tilde \chi \bar f_{1,0}|^2\rho \right)=-\int_{\mathbb R} |\pa_{\sfY}(\tilde \chi f_{1,0})|^2\rho+\int_{\mathbb R} \tilde \chi^2 \pa_{\sfY}(\rho \pa_{\sfY}^{-1}(\sfF(\sfs,z)-1)) |\bar f_{1,0}|^2\\
&\qquad \qquad \qquad \qquad \qquad +\int_{\mathbb R} \tilde \chi^2 \bar f_{1,0}\left((4\sfF-4) F_{1,0}-3(\int_{-\pi}^z F_{1,0})\pa_z \sfF-\pa_s d \right)\rho +O(e^{-e^{\sfs}}).
\end{align*}
Above, since $\int_{\mathbb R}\tilde \chi \bar f_{1,0}\rho=0$, we obtain the coercivity $\int_{\mathbb R} |\pa_{\sfY}(\tilde \chi f_{1,0})|^2\rho\geq \frac 32 \int_{\mathbb R} |\tilde \chi f_{1,0}|^2\rho $ from Proposition \ref{pr:Ls} (note that $1=h_0$). Bounding the remaining terms by using that $d$ is bounded, that $|\bar f_{1,0}|=|f_{1,0}-d|\leq Ce^{\iota'' \sfs}$, \fref{bd:lossypaxxx}, \fref{bd:interpaxxx1} and \fref{bd:interpaxxx2} we get:
$$
\frac{d}{d\sfs} \frac{1}{2}\left(\int_{\mathbb R} |\tilde \chi \bar f_{1,0}|^2\rho \right)\leq -\frac{3}{2}\int_{\mathbb R} |\tilde \chi \bar f_{1,0}|^2\rho+Ce^{-(\iota-2\iota'')\sfs}.
$$
Reintegrating this inequality gives $\int_{\mathbb R} |\tilde \chi \bar f_{1,0}|^2\rho \leq Ce^{-(\iota-2\iota'')\sfs}$. Hence $\| \bar f_{1,0}\|_{L^2([-2\sfM,2\sfM])}\leq Ce^{-(\iota-2\iota'')\sfs}$. A standard application of parabolic regularisation gives that $\| \bar f_{1,0}\|_{L^{\infty}([-\sfM,\sfM])}\leq Ce^{-(\iota-2\iota'')\sfs}\leq e^{-\iota' \sfs}$ (upon choosing $\iota''$ small depending on $\iota'$, and then $\sfs$ large). This and the bound on $d$ show the Lemma upon renaming $d_\infty$ by $d$.

\end{proof}

\begin{lemma} \label{lem:cvpaxxxtokernel}

Let $(z^*,d)$ be given by Lemmas \ref{lem:char} and \ref{lem:controlchar}, and $Q$ be defined in Lemma \ref{lem:linearpaxxx}. Then:
$$
\lim_{\sfs \rightarrow +\infty} \| F_{1,0}(\sfs,z)-dQ(z)\|_{L^{\infty}_\sfw}=0,
$$

\end{lemma}

\begin{proof}

We regularise the element of the kernel $Q$ near the points $\pm \pi$ and decompose:
$$
F_{1,0}(\sfs,z)=d\chi^*(\sfs,z)Q(z)+\bar F_{1,0}(\sfs,z).
$$
where $\chi^*$ was defined in \fref{def:chi*}. Then $\bar F_{1,0}$ solves:
$$
(\pa_{\sfs}+\mathcal N''+\hat{\mathcal N}''-\gamma_{2}\pa_{zz})\bar F_{1,0}=\sfE (\sfs,z),
$$
where the elliptic linear operator, the nonlocal linear operator and the error are:
$$
\mathcal N''=4(1-G_1(z))+\left(\pa_{z}^{-1}G_1+\int_{-\pi}^z \sfu -\frac z2\right)\pa_z , \qquad \hat{\mathcal N}''\bar F_{1,0}=3\pa_z G_1 \pa_{z}^{-1}\bar F_{1,0}
$$
$$
\sfE=-d\left((\pa_{\sfs}+\tilde{\mathcal N}-e^{-2s} \pa_{zz})(\chi^* Q)+\mathcal N((\chi^*-1)Q)\right)+4\sfu \bar{F}_{1,0}-3\pa_z u \int_{-\pi}^z F_{1,0}
$$
Since $d$ is bounded, from the asymptotic behaviour of $Q$ near $\pm \pi$ in Lemma \ref{lem:linearpaxxx} and \fref{bd:decayassumptionpaxxx} we get:
\be \label{bd:sfE}
\| \sfE \|_{L^{\infty}_\sfw}\leq e^{-\frac{3\iota}{4}\sfs}.
\ee
We introduce the domain $\bar \Omega=[-\pi,z^*-\sfM e^{-s}]\cup [z^*+\sfM e^{-s},+\infty]$ and let $\sfs_1$ be large.

\textbf{Step 1} Let $\bar{\sfw}:(-\infty,\infty)\rightarrow (0,\infty)$ be a function that satisfies all the following properties (note that it is possible to construct explicitly such a weight $\bar{\sfw}$ for any $\iota>0$, along the very same lines as in Subsection \ref{subsubsec:lin}):
\begin{itemize}
\item[(i)] $\bar{\sfw}$ is an even $C^2$ solution of the differential inequality $\mathcal N \bar{\sfw}\geq \frac{\iota}{4}\bar w$ on $\mathbb R\backslash \{0\}$.
\item[(ii)] $\bar{\sfw}(z)= |z|^{\frac{\iota}{2}}$ for $|z|$ small enough and $\bar{\sfw}(z)=\bar{\sfw} (\pi)\langle Z-\pi \rangle^{-2}$ for $z\geq \pi$.
\item[(iii)] For all $z \in [-\pi,\pi]$, $\left|\pa_z^{-1} \bar{\sfw}(z) \right|\leq \iota^2\bar{\sfw}(z)$.  
\end{itemize}
Then we claim that there exists $\sfM>0$ such that for $\sfs$ large enough, $\bar \sfS(\sfs,z)=e^{-\frac{\iota}{8}\bar s}\bar w(z-z^*)$ satisfies $(\pa_s+\mathcal N''-e^{-2s}\pa_{zz})\bar \sfS\geq \frac{\iota}{10}\bar \sfS$ on $\bar \Omega$. This is a direct computation. We indeed compute using the evolution equation for $z^*$ that, for $\sfs$ large enough:
\begin{align*}
&(\pa_s+\mathcal N''-e^{-2s}\pa_{zz})\bar \sfS \\
 =& -\frac{\iota}{8}\bar \sfS +e^{-\frac{\iota}{4}\sfs}\left( 4(1-G_1(z-z^*) \bar{\sfw}(z-z^*)+\left(\int_{z^*}^z G_1(\tilde z-z^*)d\tilde z-\frac{z-z^*}{2}\right)\pa_z \bar{\sfw} (z-z^*)\right)\\
 & -e^{-2s} \pa_{zz}\bar \sfS+\int_{z^*}^z u \pa_z \bar \sfS+4(G_1(z-z^*)-G_1(z))\bar \sfS +\int_{z^*}^z (G_1(\tilde z)-G_1(\tilde z-z^*))d\tilde z \pa_z \bar \sfS\\
\geq & -\frac{\iota}{8}\bar \sfS+\frac{\iota}{4}\bar \sfS +O(\sfM^{-2}\bar \sfS)+O(e^{-\iota \sfs}|\bar \sfS|) \ \geq \frac{\iota}{10}\bar \sfS
\end{align*}
where we used property (i) for the second term, $e^{-2s}|\pa_{zz}\bar \sfS|\lesssim e^{-2s}|z-z^*|^{-2}\bar \sfS\lesssim \sfM^{-2}\bar \sfS$ for $|z-z^*|\geq \sfM e^{-s}$ for the third one, and \fref{bd:decayassumptionpaxxx} and \fref{bd:zstarpaxxx} for the remaining terms.

\textbf{Step 2}. Let $\bar F_{1,0}^1$ solve $(\pa_s+\mathcal N'' -e^{-2s} \pa_{zz})\bar F_{1,0}^1=\sfE$ on $\bar \Omega$, with  boundary conditions $\bar F_{1,0}^1(-\pi)=0$, $\bar F_{1,0}^1(z^*\pm \sfM e^{-s})=\bar F_{1,0}(z^*\pm \sfM e^{-s})$ and $\bar F^1_{1,0}(\sfs_1)=\bar F_{1,0}(\sfs_1)$. From the behaviour of $\bar{\sfw}$ near $0$, there holds that $\bar \sfS\geq C(\sfM)e^{-\frac{\iota}{2}\sfs}\sfw$ uniformly on $\bar \Omega$ for a constant $C(\sfM)>0$. Hence from Step 1 and \fref{bd:sfE} we get $(\pa_s+\mathcal N'' -e^{-2s} \pa_{zz})\bar \sfS \geq |\sfE|$. Moreover, $\bar \sfS\geq |\bar F_{1,0}|$ at the boundary $0$ and $z^*\pm Me^{-s}$ from \fref{bd:Linftynearzstar}. Hence from parabolic comparison, for some $C(\sfs_1)>0$, for $\sfs \geq \sfs_1$ large:
$$
|\bar F_{1,0}^1|\leq C(\sfs_1)|\bar \sfS|\leq C(\sfs_1)e^{-\frac{\iota}{8}\sfs}\bar \sfw.
$$

\textbf{Step 3}. Let $\bar F_{1,0}=\bar F_{1,0}^1+\bar F_{1,0}^2$. Then $\bar F_{1,0}^2$ solves $(\pa_s+\mathcal N'' +\hat{\mathcal N}''-e^{-2s} \pa_{zz})\bar F_{1,0}^2=-\bar{\mathcal N}'' \bar F_{1,0}^1$ on $\bar \Omega$ with Dirichlet boundary conditions and zero initial datum at time $\sfs_1$. From Step 1, the solution to $(\pa_s+\mathcal N'' -e^{-2s} \pa_{zz})v=0$ on $\bar \Omega$ with Dirichlet boundary conditions satisfies $\|v(\sfs) \|_{L^{\infty}(\bar \sfw)}  \leq e^{-\frac{\iota}{8}(\sfs-\sfs_2)}\| v(\sfs_2)\|_{L^{\infty}(\bar \sfw)}$ for $\sfs\geq \sfs_2$ large enough. Hence, making the exact same reasoning as in Step 1 in the proof of Proposition \ref{pr:linearanalytic}, one obtains that the solution to $(\pa_s+\mathcal N'' +\bar{\mathcal N}''-e^{-2s} \pa_{zz})u=0$ on $\bar \Omega$ with Dirichlet boundary conditions satisfies $\| u(\sfs) \|_{L^{\infty}(\bar \sfw)} \leq e^{-\frac{\iota}{16}(\sfs-\sfs_2)}\| u(\sfs_0)\|_{L^{\infty}(\bar \sfw)}$ for $\iota$ small enough. This linear bound and the bound for $\bar F_{1,0}^1$ obtained in Step 2 show that for $\sfs \geq \sfs_1$:
$$
|\bar F_{1,0}^2|\leq C|\bar \sfS|\leq Ce^{-\frac{\iota}{16}\sfs}\bar \sfw
$$
for $\sfs_1$ large enough. This bound and the one from Step 2 prove the lemma since $\bar F_{1,0}=\bar F_{1,0}^1+\bar F_{1,0}^2$, and since $\bar{\sfw}\leq C \sfw$ for some constant $C>0$.

\end{proof}

We can now end the proof of Proposition \ref{pr:paxxx}

\begin{proof}[Proof of Proposition \ref{pr:paxxx}]

The second bound is a direct consequence of Lemma \ref{lem:cvpaxxxtokernel}. As for the first one, from \fref{id:lowernearboundary} we get that $\xi_{1,0}$ solves a linear parabolic equation on $[0,1/2]$, with variable coefficients involving $\xi$ and $\pa_y \xi$, but which are uniformly bounded on $[0,T]\times [0,1/2]$ from \fref{id:lowernearboundary}. Hence this first bound is obtained via a standard parabolic bootstrap.

\end{proof}

\subsection{Control of higher order derivatives} \label{subsec:higheranalytic}

\subsubsection{The analytic norm and formal explanations} \label{subsubsec:analytic}

We will use in this subsection the bound \fref{bd:paxxxoptimal} obtained for the third order tangential derivative and the constant $C_2$ is now considered as a universal constant. Our aim is to control the following semi norm for derivatives in an analytical setting (we recall that $i$ denotes $2i+1$ derivatives in the tangential variable):
$$
\sup_{i+k\geq 2} \ \tau^{a_{i,k}} \hat \tau^{a_{i,k}} \bar \tau^{b_{i,k}} \frac{\langle i+k\rangle^{3}}{(2i+1+k)!}\left\| F_{i,k}(\sfs) \right\|_{L^{\infty}_\sfw}
$$
where $0<\bar \tau,\hat \tau\leq 1$ are constants, $3$ is a correction exponent\footnote{There is no need to optimise the value of the exponent $3$.}, the other exponents are
\be \lab{def:bik}
a_{i,k}=\left\{ \begin{array}{l l} 0 \qquad \mbox{if } i+k\leq 1,\\ i+k-\frac 74\qquad  \mbox{otherwise} ,\end{array} \right. \quad b_{i,k}=\left\{ \begin{array}{l l} 0 \qquad \mbox{if }i=0,\\ 2i-1 \qquad  \mbox{if } k\geq 1\mbox{ and }i\geq 1,\\ 2i-2 \qquad  \mbox{if } k=0 \mbox{ and }i\geq 1, \end{array} \right.
\ee
and for $\sfK>0$ a constant such that Proposition \ref{pr:linearanalytic} holds true:
\be \label{def:tau}
\tau(s)=e^{-\frac 12 \sfs}\left(1+e^{-\frac{\iota}{2} \sfs} \right)^2e^{\sfK e^{-\sfs}}=(T-t)^{\frac1 2}\tilde \tau, \qquad \tilde \tau(t)=\left(1+(T-t)^{\frac{\iota}{2}} \right)^2e^{\sfK (T-t)}.
\ee
Let us now explain formally why the above semi-norm will remain bounded, and the role of the parameters. For this, let us only keep the term with $j=1$ in the first line of \fref{eq:Fikevo}:
$$
\pa_s F_{i,k}+\mathcal L_{i,k}F_{i,k}= \binom{2i+1}{2} (\pa^{-1}F_{1,0})F_{i-1,k+1}+...
$$
Since $\pa^{-1}F_{1,0}=O(1)$ from Proposition \ref{pr:paxxx}, this means that the evolution of $F_{i,k}$ has a forcing term that is $O(F_{i-1,k+1})$, that that of $F_{i-1,k+1}$ has a $O(F_{i-2,k+2})$ forcing and so on. At the end of this chain, we see that $F_{0,k+i}$ is somehow forcing the evolution of $F_{i,k}$. The optimal bound on the linear evolution of $F_{0,k+i}$ is $F_{0,k+i}=O(e^{-\frac{1}{2}(k+i)\sfs+O(1)})$ from Proposition \ref{pr:linearanalytic}. Hence we get formally that $F_{i,k}=O(e^{-\frac{1}{2}(k+i)\sfs+O(1)})$ and in particular $F_{i,0}=O(e^{-\frac{i}{2}\sfs+O(1)})$. Back in original variables, this gives that $\xi_{i,0}=O((T-t)^{-\frac{3}{2}(2i+1)}F_{i,0})=O((T-t)^{-\frac 72i+O(1)})$, hence a radius of analyticity of $(T-t)^{7/4}$ in the $x$ direction.

We separate the radius of analyticity in three parts that will play different roles. $\tau^{a_{i,k}}$ is the time dependent part: it encodes the above expected temporal bound, and is compatible with the linear estimates of Proposition \ref{pr:linearanalytic}. $\hat \tau$ is the constant part, and taking it small enough allows to control the second line in \fref{eq:Fikevo}. $\bar \tau^{b_{i,k}}$ gives a different estimate for $\pa_x$ and $\pa_y$ derivatives ; this anisotropy in the norm allows to control the first line in \fref{eq:Fikevo}.

Finally, let us mention that certain short time analytical results as \cite{KV} only require the control of a finite number of $\pa_y$ derivatives, relying on parabolic regularising effects. However, here the viscosity is negligible as $s\rightarrow \infty$, and $\pa_y$ derivatives are forcing $\pa_x$ derivatives as explained above, requiring us to control an infinite number of $\pa_y$ derivatives.\\

The heart of the analysis is to control the analytic norm using a bootstrap argument. We introduce the weight $\omega(t,y)=\sfw (y\sqrt{T-t}-\pi)$ and the space $\| f\|_{L^{\infty}_\omega}=\sup_{y\geq 0} |f(y)|\omega^{-1}(s,y)$.

\begin{definition} \label{def:trappedanalytic}
Let $\sfL,\hat \tau,\bar \tau,\sfK,3>0$. We say for $T_0< t_1<T$ that $u$ is in the analytic trap on $[T_0,t_1]$ if $(\xi_{i,k})_{i,k\in\mathbb N}$ is a $C^\infty$ solution of \fref{eq:evolutionxii} on $[T_0,t_1]\times [0,\infty)$ such that, initially:
\be \label{bd:indainitialtrap}
\| \xi_{i,k}(T_0,\cdot )\|_{L^{\infty}_\omega}\leq (T-T_0)^{-\frac 72 i}\hat \tau^{-a_{i,k}}\bar \tau^{-b_{i,k}} \tilde \tau^{-a_{i,k}} \frac{(2i+k+1)!}{\langle i +k\rangle^3} \qquad \mbox{for }i+k\geq 2,
\ee
and for all $t\in [T_0,t_1]$, setting $\tilde L=L/(2(T-T_0)^{\frac 18})$:
\be \label{bd:indatrap}
\|\xi_{i,k}(t,\cdot)\|_{L^{\infty}_\omega}\leq \sfL(T-t)^{-\frac 72 i-\frac 18}\hat \tau^{-a_{i,k}}\bar \tau^{-b_{i,k}} \tilde \tau^{-a_{i,k}} \frac{(2i+k+1)!}{\langle i +k\rangle^3}\qquad \mbox{for }i+k\geq 2,
\ee
\be \label{bd:boundarytrap}
|\xi_{i,k}(t,0)|\leq \tilde{\sfL}(T-t)^{-\frac 72 i}\hat \tau^{-a_{i,k}}\bar \tau^{-b_{i,k}} \tilde \tau^{-a_{i,k}} \frac{(2i+k+1)!}{\langle i +k\rangle^3} \qquad \mbox{for }i+k\geq 0.
\ee
\end{definition}

\begin{proposition} \label{pr:analyticbootstrap}

For any constants $T_0$, $T$, $\mu$, $\iota$, $C_0$, $C'_0$ and $\tau_0$ in the hypotheses of Theorem \ref{th:analytic}, and $C_2$ in the inequality \fref{bd:paxxxoptimal}, there exists $\sfL^*,\sfK>0$ such that for any $\sfL\geq \sfL^*$, there exist $\hat \tau^*,\bar \tau^*>0$ such that, for any $0<\bar \tau<\bar \tau^*$ and $0<\hat \tau<\hat \tau^*$, if $u$ is in the analytic trap on $[T_0,t_1]$ in the sense of the previous definition, then at time $t_1$ for all $i+k\geq 2$:
\be \label{bd:indatrapimproved}
\|\xi_{i,k}(t_1,\cdot)\|_{L^{\infty}_\sfw}\leq \frac{3}{4}\sfL(T-t)^{-\frac 72 i-\frac 18}\hat \tau^{-a_{i,k}}\bar \tau^{-b_{i,k}} \tilde \tau^{-a_{i,k}} \frac{(2i+k+1)!}{\langle i +k\rangle^3},
\ee
\be \label{bd:boundarytrapimproved}
|\xi_{i,k}(t_1,0)|\leq \frac 12 \tilde \sfL(T-t)^{-\frac 72 i}\hat \tau^{-a_{i,k}}\bar \tau^{-b_{i,k}} \tilde \tau^{-a_{i,k}} \frac{(2i+k+1)!}{\langle i +k\rangle^3}.
\ee
\end{proposition}

\begin{proof}[Proof of Proposition \ref{pr:analyticbootstrap}]
The inequality \fref{bd:indatrapimproved} is proved in Proposition \ref{pr:nonlinearrenormalisedanalytic}. The inequality \fref{bd:boundarytrapimproved} for $k$ even is proved in Corollary \ref{cor:improvedboundaryeven} and for $k$ odd in Lemma \ref{lem:improvedboundaryodd}.
\end{proof}

\begin{remark}

The bounds \fref{bd:indatrapimproved} and \fref{bd:boundarytrapimproved} improve \fref{bd:indatrap} and \fref{bd:boundarytrap} by $<1$ factors. This is used to prove Theorem \ref{th:analytic} as follows: for a solution starting in the analytic trap (Definition \ref{def:trappedanalytic}), then the bounds \fref{bd:indatrap} and \fref{bd:boundarytrap} can never be saturated, showing that the solution remains in this trap up to the blow-up time $T$.

The bound \fref{bd:indatrapimproved}, since valid up to the blow-up time, is used to prove Theorem 3. The bounds \fref{bd:indatrap} and \fref{bd:boundarytrapimproved} however are only used as additional estimates to prove Proposition \ref{pr:analyticbootstrap}.

\end{remark}

The above Proposition directly implies Theorem \ref{th:analytic}.

\begin{proof}[Proof of the Theorem \ref{th:analytic}]

\underline{Proof of (ii)}. Fix all constants in Definition \ref{def:trappedanalytic} such that Proposition \ref{pr:analyticbootstrap} holds true. Then \fref{bd:indainitialtrap} is satisfied because of \fref{bd:analytichypothesis2}, by choosing possibly a smaller coefficient $\hat \tau$.

Let now $\ell\in \mathbb N$ and define $t^*(\ell)$ as the supremum of times $t_1> T_0$ such that $(\xi_{i,k})_{i+k\leq \ell}$ satisfies the estimates \fref{bd:indatrap} and \fref{bd:boundarytrap} in $[T_0,t_1]$. Assume $t^*(\ell)<T$ by contradiction. Notice that $(\xi_{i,k})_{i+k\leq \ell}$ solves a closed system of equations of the form $\pa_t \xi_{i,k}-\pa_{yy}\xi_{i,k}-\pa_y^{-1}\xi \pa_y \xi_{i,k}=f_{i,k}((\xi_{i',k'})_{i'+k'\leq \ell})$ from \fref{eq:evolutionxii}. Notice that, as a consequence, the proof of the bounds \fref{bd:indatrapimproved} and \fref{bd:boundarytrapimproved} for $i+k\leq \ell$, only relies on the use of the bounds \fref{bd:indatrap} and \fref{bd:boundarytrap} for $i+k\leq \ell$. Thus, by definition of $t^*(\ell)$, the bounds \fref{bd:indatrapimproved} and \fref{bd:boundarytrapimproved} hold true for $i+k\leq \ell$ at any time $t_1<t^*$, hence at time $t^*$ as well by continuity. By a continuity argument and because of propagation of regularity, using that \eqref{bd:indatrapimproved} and \fref{bd:boundarytrapimproved} strictly improve \fref{bd:indatrap} and \fref{bd:boundarytrap}, $(\xi_{i,k})_{i+k\leq \ell}$ satisfies \fref{bd:indatrap} and \fref{bd:boundarytrap} on $[t^*,t^*+\delta]$ for some $\delta>0$, contradicting the definition of $t^*$.

Hence $t^*(\ell)=T$. Letting $\ell \rightarrow \infty$, we get that, $(\xi_{i,k})_{i,k\geq 0}$, on $[0,T)$, satisfies \fref{bd:indatrap} and \fref{bd:boundarytrap}. Thus, $(\xi_{i,0})_{i\geq 0}$ satisfy \fref{bd:analytiresult} and \fref{bd:analytiresult2}, as consequences of \fref{bd:analytichypothesis3}, \fref{bd:paxxxoptimal} and \fref{bd:indatrapimproved}.\\

\noindent \underline{Proof of (i)}. We now define $u(t,x,y)=\sum_{i=0}^\infty x^{2i+1}\frac{\xi_{i,0}(t,y)}{(2i+1)!}$. The convergence in $E_{T,\tau^*}$ for $\tau^*$ independent of time small enough is a direct consequence of \fref{bd:analytiresult}. Moreover, the trace of all $x$ derivatives of $u$ on the vertical axis $\{x=0\}$ solves the corresponding trace of Prandtl's equations. Hence $u$ solves Prandtl's equations on $E_{T,\tau^*}$ by uniqueness of analytic extensions.\\

\noindent \underline{Proof of (iii)}. We set $\tau$ to be constant in time on $[T-\delta',T]$ and then bound for $x=\pm \tau(T-t)^{7/4}$ for $T-\delta' \leq t \leq T$, using \fref{bd:lowerglobal}, \fref{bd:analytiresult} and \fref{bd:analytiresult2}:
\begin{align*}
&|u| \leq  \tau (T-t)^{\frac 74}|\xi_{0}(t,y)|+\tau^3 (T-t)^{\frac{21}{4}}\frac{|\xi_{1}(t,y)|}{6}+ \sum_{i=2}^{\infty} \tau^{2i+1}(T-t)^{\frac{7(2i+1)}{4}}\frac{|\xi_{i}(t,y)|}{(2i+1)!}\\
&\qquad \leq  \tau (T-t)^{\frac 74}\frac{(1+o_{\delta \rightarrow 0}(1))}{T-t}+C_1\tau^3(T-t)^{\frac{5}{4}}+C_1 \sum_{i=2}^{\infty} \frac{\tau^{2i+1}}{\tau_1^{2i+1}}(T-t)^{\frac{13}{8}}\ <\ \frac 74 \tau (T-t)^{\frac 34}
\end{align*}
if $\tau$ and $\delta' $ have been chosen small enough. This shows \fref{id:causality} on $[T-\delta,T)$. Since on $[0,T-\delta']$, $\xi_0$ and $\xi_1$ remain bounded, it suffices to take $\tau$ decreasing fast enough on $[0,T-\delta]$ to obtain \fref{id:causality} on this interval.
\end{proof}

We turn to the proof of Proposition \ref{pr:analyticbootstrap}. We use the following throughout this section. For any $\sfK>0$, one has $0\leq \tau \hat \tau \leq1$ on $[T_0,T)$ for $\hat \tau$ small enough. The function $\tilde \tau$ satisfies:
\be \label{id:proptildetau}
\tilde \tau \geq 1 \quad \mbox{and} \quad \tilde \tau \mbox{ is decreasing on } [T_0,T].
\ee
We shall use the following properties of the exponents for any $i,i',k,k'\geq 0$:
\be \label{bd:exponents3}
a_{i,k}\leq a_{i',k'}, \quad \mbox{and} \quad b_{i,k}\leq b_{i',k'} \qquad \mbox{if }i\leq i' \mbox{ and }k\leq k',
\ee
\be \label{bd:exponents1}
a_{i,k}+a_{i',k'}\leq a_{i+i',k+k'}, \quad b_{i,k}+b_{i',k'}\leq b_{i+i',k+k'},
\ee
\be \label{bd:exponents6}
a_{i,k}\leq a_{i,k+1}-\frac 14 \quad \mbox{if }i+k\geq 1, \qquad a_{i,k}= a_{i,k+1}-1\quad \mbox{if }i+k\geq 2
\ee
\be \label{bd:exponents2}
a_{i,k}+a_{i',k'}\leq a_{i+i',k+k'}-\frac 14 \quad \mbox{if }i+k\geq 1\mbox{ and }i'+k'\geq 1
\ee
and if $i\geq 1$,
\be \label{bd:exponents4}
a_{i,0}+a_{i',k+1}\leq a_{i+i',k}, \quad \mbox{ and if moreover }k\geq1 \mbox{ or }i'\geq 1 \mbox{ then }b_{i,0}+b_{i',k+1}\leq b_{i+i',k}-1.
\ee

\subsubsection{Analytic control at the boundary}

The aim of this subsubsection is to prove \fref{bd:boundarytrapimproved}. We rely on the fact that the control of $\pa_{y}^{2m}$ derivatives is similar to that of $\pa_t^m$ ones for parabolic equations, these latter having the advantage of maintaining Dirichlet boundary conditions. However, this equivalence degenerates as one approaches the blow-up time $T$. We need to exploit two gains coming from the fact that near the boundary one is away from the blow-up zone: first the bound \fref{id:lowernearboundary}, and then the fact that $\pa_y^{-1}$ lose a $(T-t)^{-1/2}$ factor for $y\sim (T-t)^{-1/2}$ but not for $y=O(1)$.

\begin{lemma}[Improved estimates at the boundary for even derivatives] \label{lem:analyticboundary1}

For any $\sfL,\sfK,\bar \tau^*>0$, there exists $\hat \tau^*>0$ such that the following holds true for any $0<\bar \tau\leq \bar \tau^*$ and $0<\hat \tau\leq \hat \tau^*$. Assume \fref{id:lowernearboundary}, \fref{bd:paxxxoptimal} and (c) in Definition \ref{def:trappedanalytic}. Then for any $m\geq 1$, for any $i$:
\be \label{id:decompoevenboundary}
\pa_t^m \xi_{i,0}=\pa_{y}^{2m}\xi_{i,0}+\xi_{i}^m
\ee
where for some universal $C>0$, for all $k$ and $t\in [T_0,t_1]$:
\be \label{bd:relationpatpay}
| \pa_y^k \xi_{i}^m (t,0)|\leq C \tilde{\sfL} (T-t)^{-\frac 72 i} \bar \tau^{-b_{i,k+2m-2}} \hat \tau^{-a_{i,k+2m-2}}\tilde \tau^{-a_{i,k+2m-2}} \frac{(2i+k+2m+1)!}{\langle i+k+m\rangle^{3}},
\ee
with the convention that for all $i$, $a_{i,k}=b_{i,k}=0$ for $k=-1$ and $k=-2$.
\end{lemma}

\begin{corollary} \label{cor:improvedboundaryeven}

With the same hypotheses, for a universal $C>0$, for any $i,k\geq 0$ with $k$ even:
\be \label{bd:improvedboundaryeven1}
|\xi_{i,k}(t,0)|\leq \hat \tau \tilde \tau \sfL (T-t)^{-\frac 72 i} \bar \tau^{-b_{i,k}} \hat \tau^{-a_{i,k}}\tilde \tau^{-a_{i,k}} \frac{(2i+k+1)!}{\langle i+k\rangle^{3}},
\ee
\be \label{bd:improvedboundaryeven2}
|\pa_t \xi_{i,k}(t,0)|\leq \sfL (T-t)^{-\frac 72 i} \bar \tau^{-b_{i,k}} \hat \tau^{-a_{i,k}}\tilde \tau^{-a_{i,k}} \frac{(2i+k+3)!}{\langle i+k\rangle^{3}}.
\ee

\end{corollary}

\begin{proof}[Proof of Corollary \ref{cor:improvedboundaryeven}]

The boundary condition $u_{|y=0}=0$ implies \fref{bd:improvedboundaryeven1} for all $i\geq 0$ for $k=0$. By time differentiation, and from the equation \fref{eq:evolutionxii} for $\xi_{0,0}$ one obtains $\xi_{0,2}(t,0)=0$ hence \fref{bd:improvedboundaryeven1} for $(i,k)=(0,2)$. By differentiation again, $\pa_t^k \xi_{i,0}(t,0)=0$ for all $k$, hence $\pa_{y}^{2k}\xi_{i,0}(t,0)=-\xi_{i}^k(t,0)$. So \fref{bd:improvedboundaryeven1} is then a direct consequence of \fref{bd:relationpatpay}, since $a_{i,k-2}\leq a_{i,k}-1$ for any $i$ if $k\geq 2$ and $(i,k)\neq (0,2)$. Next, we write $\pa_t \xi_{i,2k}=\xi_{i,2k+2}+\pa_y^{2k}\xi_i^1$, so that $\pa_t \xi_{i,2k}(t,0)=-\xi_{i}^{k+1}(t,0)+\pa_y^{2k}\xi_i^1(t,0)$ at the boundary, and \fref{bd:improvedboundaryeven2} is again obtained from \fref{bd:relationpatpay}.

\end{proof}

\begin{proof}[Proof of Lemma \ref{lem:analyticboundary1}]
We introduce the notation:
$$
\xi_{i,k,m}=\pa_t^m \xi_{i,k}.
$$
By induction on the equation \fref{eq:evolutionxii} we obtain the recurrence identity:
\bee
\pa_t^m \xi_{i,0}=\pa_y^{2m}\xi_{i,0}+\sum_{n=0}^{m-1}\pa_{y}^{2m-2n-2}\pa_t^{n}\left(-\sum_{j=0}^{i}  \binom{2i+1}{2j+1} \xi_{j,0} \xi_{i-j,0}+\sum_{j=0}^{i}\binom{2i+1}{2j} (\pa_y^{-1}\xi_{j,0})\xi_{i-j,1} \right).
\eee
We now reason by induction on $m\geq 0$ to prove \fref{bd:relationpatpay}. For $m=0$ the bound is trivial since $\xi^0_i=0$ for all $i$. We now assume the desired bound holds true for all $m'\leq m$, for all $i$ and $k$. Note that if $m\geq 1$ then $a_{i,k+2m-2}\leq a_{i,k+2m}-1$ so that $\hat \tau^{-a_{i,k+2m-2}}\tilde \tau^{-a_{i,k+2m-2}}\leq  (\hat \tau \tilde \tau) \hat \tau^{-a_{i,k+2m}}\tilde \tau^{-a_{i,k+2m}}$. Note also that if $m=0$ then $\xi^m_i=\xi^0_i=0$. In particular, the identity \fref{id:decompoevenboundary} and the bounds \fref{bd:boundarytrap} and \fref{bd:relationpatpay} give for $m'\leq m$ for $\hat \tau$ small enough:
\be \label{bd:xiikm}
| \xi_{i,k,m'} (t,0)| \leq 2 \tilde{\sfL} (T-t)^{-\frac 72 i} \bar \tau^{-b_{i,k+2m}}\hat \tau^{-a_{i,k+2m}}\tilde \tau^{-a_{i,k+2m}}  \frac{(2i+k+2m+1)!}{\langle i+k+m\rangle^3}
\ee
To prove the desired bound for $m+1$ we first obtain the following identity from the recurrence identity using Leibniz rule and the fact that $\pa_y^{-1}$ terms vanish at the boundary (with the convention that $\binom{a}{b}=0$ if $b>a$):
\begin{align}
\nonumber \pa_y^k \xi_{i}^{m+1}(t,0)& = \sum_{n=0}^{m}\sum_{j=0}^{i}  \sum_{p=0}^n\sum_{l=0}^{k+2m-2n}  \binom{n}{p} \xi_{j,l,p}(t,0) \xi_{i-j,k+2m-2n-l,n-p}(t,0) \\
\label{analytic:idpaykxiim+1}& \qquad \qquad \left(\binom{2i+1}{2j}\binom{k+2m-2n}{l+1}-\binom{2i+1}{2j+1}\binom{k+2m-2n}{l}\right)
\end{align}
Note that in the sum, if $(j,l,p)\in \{(0,0,0),(i,2m-2n+k,n)\}$ then the term is zero because $\xi_{0,0,0}(t,0)=u_x(t,0)=0$ from the Prandtl boundary condition $u_{|y=0}=0$. Therefore we assume $(j,l,p)\notin \{(0,0,0),(i,2m-2n+k,n)\}$ without loss of generality. Introducing $r=2j+l+2p+1$ we bound using \fref{bd:xiikm}:
\begin{align}
\nonumber & |  \xi_{j,l,p} (t,0)\xi_{i-j,2m-2n+k-l,n-p}(t,0)|\\
\nonumber \leq & C\tilde{\sfL}^2 (T-t)^{-\frac 72 j-\frac 72 (i-j)} \bar \tau^{-b_{j,2p+l}-b_{i-j,2(m-p)+k-l}}\hat \tau^{-a_{j,2p+l}-a_{i-j,2(m-p)+k-l}}\tilde \tau^{-a_{j,2p+l}-a_{i-j,2(m-p)+k-l}} \\
\nonumber &  \frac{(2j+2p+l+1)!}{\langle j+p+l\rangle^3}\frac{(2(i-j)+2(m-p)+k-l+1)!}{\langle i-j+m-p+k-l\rangle^3}\\
\label{analytic:bd:paykxiim+1inter} \leq & \tilde{\sfL} (T-t)^{-\frac 72 i} \bar \tau^{-b_{i,2m+k}}\hat \tau^{-a_{i,2m+k}}\tilde \tau^{-a_{i,k+2m}}  \frac{r!}{\langle r\rangle^3}\frac{(2i+2m+k+2-r)!}{\langle 2i+k+2m+1-r\rangle^3},
\end{align}
where in the last bound we used \fref{bd:exponents1} and \fref{bd:exponents2} for the exponents, and $C\tilde{\sfL}\tilde \tau \hat \tau \leq 1$ for $\hat \tau$ small enough. We recall the estimate for some universal $C>0$:
$$
\sum_{r=0}^{2i+2m+k+1}\binom{2i+2m+k+1}{r}\frac{r!}{\langle r\rangle^3}\frac{(2i+k+2m+2-r)!}{\langle 2i+k+2m+1-r\rangle^3}\leq C \frac{(2i+k+2m+2)!}{\langle i+k+m\rangle^3}.
$$
Using the inequality $\binom{n}{p}\leq \binom{2n}{2p}$, \fref{bd:technicalcombinatorial3} with $(A_1,A_2,A_3,r_2)=(2i+1,k+2m-2n,2n,2p+2j+l+1)$ and the above inequality:
\begin{align}
\label{bd:combinatorics2} & \sum_{j=0}^{i}  \sum_{p=0}^n\sum_{l=0}^{k+2m-2n} \frac{r!}{\langle r\rangle^3}\frac{(2i+2m+k+2-r)!}{\langle 2i+k+2m+1-r\rangle^3} \binom{n}{p} \\
\nonumber &\qquad \qquad \qquad \qquad \left| \binom{2i+1}{2j}\binom{k+2m-2n}{l+1}-\binom{2i+1}{2j+1}\binom{k+2m-2n}{l} \right|\\
\nonumber &\quad \qquad \quad \leq   \sum_{r=0}^{2i+2m+k+1}\binom{2i+2m+k+1}{r}\frac{r!}{\langle r\rangle^3}\frac{(2i+k+2m+2-r)!}{\langle 2i+k+2m+1-r\rangle^3}\leq   C \frac{(2i+k+2m+2)!}{\langle i+k+m\rangle^3}.
\end{align}
Injecting \fref{analytic:bd:paykxiim+1inter} in the identity \fref{analytic:idpaykxiim+1}, then using \fref{bd:combinatorics2} and the inequality $\sum_{n=0}^m 1\leq \langle 2i+k+2m+3\rangle$, we upper bound:
$$
|\pa_y^k \xi_{i}^{m+1}(t,0) | \leq  C\tilde{\sfL} (T-t)^{-\frac 72 i} \bar \tau^{-b_{i,k+2m}}\hat \tau^{-a_{i,k+2m}}\tilde \tau^{-a_{i,k+2m}}  \frac{(2i+k+2(m+1)+1)!}{\langle i+k+m\rangle^3}.
$$
Thus \fref{bd:relationpatpay} holds true for $m+1$, for any $i$ and $k$. It holds true for any $i,k,m$ by induction.

\end{proof}

\begin{lemma}[Improved estimates at the boundary for odd derivatives] \label{lem:improvedboundaryodd}
Assume that bounds \fref{id:lowernearboundary}, \fref{bd:paxxxoptimal}, \fref{bd:indatrap}, \fref{bd:improvedboundaryeven1} and \fref{bd:improvedboundaryeven2} are satisfied. Then for any $\sfL>0$, there exists $\hat \tau^*>0$ small enough such that for all $0<\hat \tau \leq \hat \tau^*$ for all $k$ odd (with $k\geq 3$ if $i=0$) and $t\in [T_0,t_1]$:
$$
| \xi_{i,k}(t,0)|\leq  \frac{\tilde{\sfL}}{2} \tilde \tau^{-a_{i,k}}(T-t)^{-\frac{7}{2}i} \bar \tau^{-b_{i,k}}\hat \tau^{-a_{i,k}} \frac{(2i+k+1)!}{\langle i+k\rangle^3},
$$
\end{lemma}

\begin{proof}

Assume $k$ is even, with $k\geq 2$ if $i=0$. Let $\chi:[0,\infty)\rightarrow \mathbb R$ be a smooth cut-off function with $\chi (y)=1$ for $|y|\leq 1/8$ and $\chi(y)=0$ for $|y|\geq 1/4$. Set $\zeta_{i,k}=\chi \xi_{i,k}$. Then from the equation \fref{eq:evolutionxii} we infer the evolution equation of $\zeta_{i,k}$:
\bee
\pa_t \zeta_{i,k}-\pa_{yy}\zeta_{i,k}&=& \underbrace{-\sum_{j=0}^{i}\sum_{l=0}^k \binom{2i+1}{2j+1}\binom{k}{l} \chi \xi_{j,l} \xi_{i-j,k-l}+\sum_{j=0}^{i}\sum_{l=0}^{k-1} \binom{2i+1}{2j}\binom{k}{l+1} \chi \xi_{j,l}\xi_{i-j,k-l}}_{I}\\
&&+\underbrace{\sum_{j=1}^{i} \binom{2i+1}{2j} \chi (\pa_y^{-1}\xi_{j,0})\xi_{i-j,k+1}}_{II}+\underbrace{\chi (\pa_y^{-1}\xi_{0,0})\xi_{i,k+1}-2\pa_{y}\chi \xi_{i,k+1}-\pa_{yy}\chi \xi_{i,k}}_{III}.
\eee
We decompose $\zeta_{i,k}(t,y)=\xi_{i,k}(t,0)\chi(y)+\eta_{i,k}(t,y)+\eta_{i,k}'(t,y)$ where
\begin{align*}
& \left\{ \begin{array}{l l} \pa_t \eta_{i,k}-\pa_{yy}\eta_{i,k}=\xi_{i,k}(t,0)\pa_{yy}\chi-\pa_t\xi_{i,k}(t,0) \chi,\\ \eta_{i,k}(T_0,y)=\chi(y)(\xi_{i,k}(T_0,y)-\xi_{i,k}(T_0,0)), \quad \eta_{i,k}(t,0)=0,\end{array} \right., \\
& \left\{ \begin{array}{l l} \pa_t \eta_{i,k}'-\pa_{yy}\eta_{i,k}'=I+II+III,\\ \eta_{i,k}'(T_0,y)=0, \quad \eta_{i,k}'(t,0)=0. \end{array} \right.
\end{align*}
\underline{The first term $\eta_{i,k}$}. Recall \fref{id:defsymmetryodd}, and that $\eta$ is given by the representation formula \fref{id:representationformulaheat}. For the first part, as $\pa_{yy}\chi=0$ on $[0,1/8]$, $\overline{\pa_{yy}\chi}$ is a smooth function so that from \fref{bd:improvedboundaryeven1} and \fref{id:proptildetau}:
\bee
&&\| \pa_y \int_{T_0}^t K_{t-t'}*\left(\xi_{i,k}(t',0)\overline{\pa_{yy}\chi}\right)dt'\|_{L^{\infty}}= \| \int_{T_0}^t \xi_{i,k}(t',0)K_{t-t'}*\left(\pa_y\overline{\pa_{yy}\chi}\right)dt'\|_{L^{\infty}} \\
&&\leq C \| \xi_{i,k}(\cdot,0)\|_{L^{\infty}([T_0,t])}\leq  C\hat \tau \tilde \tau \sfL (T-t)^{-\frac 72 i} \bar \tau^{-b_{i,k}} \hat \tau^{-a_{i,k}}\tilde \tau^{-a_{i,k}} \frac{(2i+k+1)!}{\langle i+k\rangle^{3}}.
\eee
For the second part, we let $\bar t=\max (t-\langle i+k\rangle^{-2},T_0)$, decompose the time integral and integrate by parts:
\bee
\int_{T_0}^t \pa_{t'}\xi_{i,k}(t',0)  K_{t-t'}* \bar \chi dt'&=& \xi_{i,k}(\bar t,0) K_{t-\bar t} *\bar \chi-\xi_{i,k}(T_0,0) K_{t-T_0} *\bar \chi\\
&&+\int_{T_0}^{\bar t}\xi_{i,k}(t',0) \pa_tK_{t-t'}* \bar \chi+\int_{\bar t}^t \pa_{t'}\xi_{i,k}(t',0) K_{t-t'} \bar \chi.
\eee
We estimate the first line. It is zero if $\bar t=T_0$ so we assume $t>T_0+\langle i+k\rangle^{-2}$. For the first term on the first line we have using \fref{bd:heatanalytic2} and \fref{bd:improvedboundaryeven1}:
$$
\left|  \pa_y \left(\xi_{i,k}(\bar t,0) K_{\langle i+k\rangle^{-2}} *\bar \chi\right)\right|  \leq C \langle i+k\rangle |\xi_{i,k}(\bar t,0)|\leq C\hat \tau \tilde \tau \sfL (T-t)^{-\frac 72 i} \bar \tau^{-b_{i,k}} \hat \tau^{-a_{i,k}}\tilde \tau^{-a_{i,k}} \frac{(2i+k+2)!}{\langle i+k\rangle^{3}}.
$$
The second term on the first line enjoys the same estimate. For the first term on the second line, using \fref{bd:heatanalytic2}, \fref{bd:improvedboundaryeven1} and \fref{id:proptildetau}:
\bee
&& \left| \pa_y \int_{T_0}^{\bar t}\xi_{i,k}(t',0) \pa_tK_{t-t'}* \bar \chi dt'\right| \lesssim \int_{T_0}^{\bar t} \frac{1}{(t-t')^{\frac 32}}|\xi_{i,k}(t',0) |dt'\\
&\lesssim & \langle i +k\rangle \| \xi_{i,k}(\cdot ,0)\|_{L^{\infty}([T_0,t])}\leq C \sfL (T-t)^{-\frac 72 i} \bar \tau^{-b_{i,k}} \hat \tau^{-a_{i,k}+1}\tilde \tau^{-a_{i,k}} \frac{(2i+k+2)!}{\langle i+k\rangle^{3}}.
\eee
For the second term on the second line, from \fref{bd:paxxxoptimal}, \fref{bd:improvedboundaryeven2}, \fref{id:proptildetau} and \fref{bd:heatanalytic2}:
\bee
&& \left|\pa_y\int_{\bar t}^t \pa_{t'}\xi_{i,k}(t',0) K_{t-t'} \bar \chi dt'\right| \leq C\int_{\max(t-\langle i+k\rangle^{-2},T_0)}^t \frac{|\pa_{t'}\xi_{i,k}(t',0)|}{\sqrt{t-t'}}dt'\\
&\leq& C\langle i+k\rangle^{-1}\| \pa_t\xi_{i,k}(\cdot ,0)\|_{L^{\infty}[T_0,t]}\leq C \sfL (T-t)^{-\frac 72 i} \bar \tau^{-b_{i,k}} \hat \tau^{-a_{i,k}}\tilde \tau^{-a_{i,k}} \frac{(2i+k+2)!}{\langle i+k\rangle^{3}}.
\eee
From \fref{bd:exponents1}, \fref{bd:exponents6}, and the initial bound \fref{bd:indainitialtrap}, the collection of above estimates imply:
\be \label{bd:analyticetaboundary}
\| \pa_y \eta_{i,k}\|_{L^{\infty}}\leq (1+C(\hat \tau \tilde \tau)^{\frac 14}\sfL) (T-t)^{-\frac 72 i} \bar \tau^{-b_{i,k+1}} \hat \tau^{-a_{i,k+1}}\tilde \tau^{-a_{i,k+1}} \frac{(2i+k+2)!}{\langle i+k\rangle^{3}}.
\ee

\noindent \underline{The second term $\eta_{i,k}'$}. For $I$ we first have the bound from \fref{id:lowernearboundary}, \fref{bd:paxxxoptimal} and \fref{bd:indatrap}:
\begin{eqnarray*}
 \| \xi_{j,l} \xi_{i-j,k-l}\|_{L^{\infty}([0,1/4])} &\leq& (T-t)^{-\frac{7}{2}j-\frac{1}{8}-\frac{7}{2}(i-j)-\frac{1}{8}} \sfL^2\bar \tau^{-b_{j,l}-b_{i-j,k-l}} \hat \tau^{-a_{j,l}-a_{i-j,k-l}}\tilde \tau^{-a_{j,l}-a_{i-j,k-l}}\\
&&\cdot \frac{(2j+l+1)!}{\langle j+l\rangle^3} \frac{(2i-2j+k-l+1)!}{\langle i-j+k-l\rangle^3}\\
& \leq& (T-t)^{-\frac{7}{2}i-\frac{1}{4}} \sfL^2 \bar \tau^{-b_{i,k}} \tilde \tau^{-a_{i,k}}\hat \tau^{-a_{i,k}} \frac{(r)!}{\langle r\rangle^3} \frac{(2i+k+2-r)!}{\langle 2i+k+1-r\rangle^3}
\end{eqnarray*}
where we used \fref{bd:exponents1} and set $r=2j+l+1$. Therefore, using the bound \fref{bd:combinatorics2} with $n=m=0$:
$$
\| I \|_{L^{\infty}([0,1/4])}\leq (T-t)^{-\frac{7}{2}i-\frac{1}{4}} \sfL^2 \bar \tau^{-b_{i,k}} \hat \tau^{-a_{i,k}}\tilde \tau^{-a_{i,k}}\frac{(2i+k+2)!}{\langle i+k\rangle^3}.
$$
We turn to $II$ and estimate from the bounds \fref{id:lowernearboundary}, \fref{bd:paxxxoptimal} and \fref{bd:indatrap}, the inequalities \fref{bd:exponents4} for the exponents since $j\geq 1$ in the sum:
\begin{eqnarray*}
 \| (\pa_y^{-1}\xi_{j,0})\xi_{i-j,k+1}\|_{L^{\infty}([0,1/4])} & \leq  & (T-t)^{-\frac{7}{2}j-\frac{1}{8}-\frac{7}{2}(i-j)-\frac{1}{8}} \sfL^2 \bar \tau^{-b_{j,0}-b_{i-j,k+1}} \hat \tau^{-a_{j,0}-a_{i-j,k+1}}\\
 && \cdot \tilde \tau^{-a_{j,0}-a_{i-j,k+1}} \frac{(2j+1)!}{\langle j\rangle^3} \frac{(2i-2j+k+2)!}{\langle i-j+k+1\rangle^3}\\
&\leq &C (T-t)^{-\frac{7}{2}i-\frac{1}{4}} \sfL^2 \bar \tau^{-b_{i,k}}\hat \tau^{-a_{i,k}}\tilde \tau^{-a_{i,k}} \frac{(2j+1)!}{\langle j\rangle^3} \frac{(2i+k+2-2j)!}{\langle i-j+k\rangle^3}.
\end{eqnarray*}
As $ \sum_{j=0}^{i}\binom{2i+1}{2j} \frac{(2j+1)!}{\langle j\rangle^3} \frac{(2i+k+2-2j)!}{\langle i-j+k\rangle^3}\leq C\frac{(2i+k+2)!}{\langle i+k\rangle^3} $, we conclude that $II$ enjoys the same estimate as $I$, namely:
$$
\| II \|_{L^{\infty}([0,1/4])}\leq C (T-t)^{-\frac{7}{2}i-\frac{1}{4}} \sfL^2 \bar \tau^{-b_{i,k}}\hat \tau^{-a_{i,k}}\tilde \tau^{-a_{i,k}} \frac{(2i+k+2)!}{\langle i+k\rangle^3}.
$$
Therefore using \fref{id:proptildetau} and \fref{bd:heatanalytic2}:
\begin{eqnarray}
\nonumber \left| \pa_y \int_{T_0}^t K_{t-t'}*\overline{(I+II)}dt'\right| &\leq &C\sfL^2 \bar \tau^{-b_{i,k}}\hat \tau^{-a_{i,k}} \frac{(2i+k+2)!}{\langle i+k\rangle^3} \int_{T_0}^t \frac{1}{\sqrt{t-t'}} (T-t')^{-\frac{7}{2}i-\frac{1}{4}} \tilde \tau^{-a_{i,k}}(t') dt'\\
 \label{bd:analyticeta'boundary} &\leq & C (\hat \tau \tilde \tau)^{\frac 14}\sfL^2(T-t)^{-\frac 72 i} \bar \tau^{-b_{i,k+1}}\hat \tau^{-a_{i,k+1}}\tilde \tau^{-a_{i,k+1}}(t) \frac{(2i+k+2)!}{\langle i+k\rangle^3}
\end{eqnarray}
where we used \fref{bd:exponents1} and \fref{bd:exponents6}. We turn to III. Let $r_0>0$ to be fixed small in a universal way. Let $\chi_0$ be a smooth function on $[0,\infty)$ such that $\chi_0(y)=1$ on $[0,r_0]$ and $\chi_0(y)=0$ for $y\geq 2r_0$. We decompose $\chi\pa_y^{-1}\xi_{0,0}\xi_{i,k+1}=\chi_0\pa_y^{-1}\xi_{0,0}\xi_{i,k+1}+(\chi-\chi_0)\pa_y^{-1}\xi_{0,0}\xi_{i,k+1}$. Since from \fref{id:lowernearboundary} we have $|\pa_y^{-1}\xi_{0,0}|\leq Cy$, then from \fref{bd:indatrap}, \fref{id:proptildetau} and \fref{bd:heatanalytic2}:
\bee
\left|\pa_y \int_{T_0}^t K_{t-t'}* (\overline{\chi_0\pa_y^{-1}\xi_{0,0}\xi_{i,k+1}})dt'\right| & \leq & Cr_0\sfL \bar \tau^{-b_{i,k+1}}\hat \tau^{-a_{i,k+1}} \frac{(2i+k+2)!}{\langle i+k\rangle^3} \int_{T_0}^t \frac{\tilde \tau^{-a_{i,k+1}}(t')(T-t')^{-\frac{7}{2}i-\frac{1}{8}}dt'}{\sqrt{t-t'}}\\
&\leq & Cr_0\sfL (T-t)^{-\frac{7}{2}i} \bar \tau^{-b_{i,k+1}}\hat \tau^{-a_{i,k+1}} \tilde \tau^{-a_{i,k+1}} \frac{(2i+k+2)!}{\langle i+k\rangle^3}.
\eee
For the other term we write:
$$
(\chi-\chi_0)\pa_y^{-1}\xi_{0,0}\xi_{i,k+1}=\pa_y\left((\chi-\chi_0)\pa_y^{-1}\xi_{0,0}\xi_{i,k}\right)+\pa_y\chi_0\pa_y^{-1}\xi_{0,0}\xi_{i,k}+(\chi-\chi_0)\xi_{0,0}\xi_{i,k}.
$$
Notice that all terms are supported away from the origin, at distance $r_0$ from it. We have thus using \fref{bd:heatanalytic2}, \fref{id:lowernearboundary}, \fref{bd:indatrap}, \fref{id:proptildetau} and integration by parts:
\bee
&&\left| \int_{T_0}^t \pa_y\left( K_{t-s}*\overline{(\chi-\chi_0)\pa_y^{-1}\xi_{0,0}\xi_{i,k+1}}\right)(0) ds \right|  \leq  C(r_0) \int_{T_0}^t \| \xi_{i,k}\|_{L^{\infty}}ds\\
&& \qquad \qquad \qquad \qquad \qquad \qquad \leq  (\hat \tau \tilde \tau)^{\frac 14}C(r_0)\sfL (T-t)^{-\frac{7}{2}i} \tilde \tau^{-a_{i,k+1}}  \bar \tau^{-b_{i,k+1}}\hat \tau^{-a_{i,k+1}} \frac{(2i+k+1)!}{\langle i+k\rangle^3} 
\eee
where we used \fref{bd:exponents1} and \fref{bd:exponents6}. The other remaining terms in $III$ can be treated the very same way. Hence:
\be \label{bd:analyticeta'boundary2}
\left| \int_{T_0}^t \pa_y\left( K_{t-s}*\overline{III}\right)(0) ds \right| \leq (C(r_0)(\hat \tau \tilde \tau)^{\frac 14}+Cr_0) \sfL\bar \tau^{-b_{i,k+1}}\hat \tau^{-a_{i,k+1}} \frac{(2i+k+1)!}{\langle i+k\rangle^3} \tilde \tau^{-a_{i,k+1}} (T-t)^{-\frac{7}{2}i}.
\ee

\underline{Conclusion} Gathering the estimates \fref{bd:analyticetaboundary}, \fref{bd:analyticeta'boundary}, \fref{bd:analyticeta'boundary2}, we have proved that:
$$
|\pa_y \xi_{i,k}(0)|\leq (\sfL^{-1}+C(r_0)(\hat \tau \tilde \tau)^{\frac 14}+Cr_0+C(\hat \tau \tilde \tau)^{\frac 14} \sfL) \sfL \bar \tau^{-b_{i,k+1}}\hat \tau^{-a_{i,k+1}} \frac{(2i+k+1)!}{\langle i+k\rangle^3} \tilde \tau^{-a_{i,k+1}} (T-t)^{-\frac{7}{2}i}
$$
which is the desired estimate upon taking $\sfL \geq 2$, $r_0>0$ small enough in a universal way, and then $\hat \tau$ small enough depending on $r_0$, $\sfK$, $\sfL$ and $T$.

\end{proof}

\subsubsection{Analytic analysis in the blow-up zone}

The aim of this subsubsection is to prove \fref{bd:indatrapimproved}. Note that \fref{bd:indatrap} is equivalent to for $i+k\geq 2$:
\be \label{bd:indarenormalisedtrap}
\|F_{i,k}(\sfs)\|_{L^{\infty}_\sfw}\leq \sfL\tau^{-a_{i,k}}\bar \tau^{-b_{i,k}} \tilde \tau^{-a_{i,k}} \frac{(2i+k+1)!}{\langle i +k\rangle^3}.
\ee
Recall the evolution equation \fref{eq:Fikevo} for $F_{i,k}$, and Proposition \ref{pr:linearanalytic} for the linear evolution.

\begin{proposition} \label{pr:nonlinearrenormalisedanalytic}

Assume (i) and (ii) in Theorem \ref{th:analytic}, and that \fref{bd:indainitialtrap}, \fref{bd:boundarytrap} and \fref{bd:indarenormalisedtrap} hold on $[\sfs_0,\sfs_1]$. Then there holds:
\be \label{bd:indarenormalisedtrapimproved}
\|F_{i,k}(\sfs_1)\|_{L^{\infty}_\sfw}\leq \frac 34 \sfL \tau^{-a_{i,k}}\bar \tau^{-b_{i,k}} \tilde \tau^{-a_{i,k}} \frac{(2i+k+1)!}{\langle i +k\rangle^3}.
\ee

\end{proposition}

\begin{proof}

Note that combining the assumption \fref{bd:lowerglobal} in Theorem \ref{th:analytic}, \fref{bd:paxxxoptimal} and \fref{bd:indarenormalisedtrap}, we get that for all $i+k\geq 1$:
\be \label{bd:estimationlow}
\|F_{i,k}(\sfs)\|_{L^{\infty}_\sfw}\leq \sfL\tau^{-a_{i,k}}\bar \tau^{-b_{i,k}} \tilde \tau^{-a_{i,k}} \frac{(2i+k+1)!}{\langle i +k\rangle^3}.
\ee
if $\sfL$ has been chosen large enough. We fix $i+k\geq 2$ and recall $\sfs_0=-\ln (T-T_0)$.

\textbf{Step 1} \emph{The case $k\geq 1$}. Assume $k\geq 1$. We write from \fref{eq:Fikevo}, where $S_{i,k}$ is the semigroup \fref{eq:linearFik}:
$$
F_{i,k}(\sfs_1)=\tilde F_{i,k}(\sfs_1)+\int_{\sfs_0}^{\sfs_1} S_{i,k}(\sfs,\sfs_1)\left(I+II\right)d\sfs,
$$
where $\tilde F_{i,k}$ solves the free evolution with the same boundary conditions than $F_{i,k}$:
$$
\pa_{\sfs}\tilde F_{i,k}+\mathcal L_{i,k}F_{i,k}=0, \quad \tilde F_{i,k}(\sfs,-\pi)=F_{i,k}(\sfs,-\pi), \quad \tilde F_{i,k}(\sfs_0,z)=F_{i,k}(\sfs_0,z),
$$
and the second term is obtained via Duhamel formula with forcing terms:
\begin{align*}
& I=\sum_{j=1}^{i}\binom{2i+1}{2j} (\pa^{-1}F_{j,0})F_{i-j,k+1}, \\
& II=-\sum_{ E^1_{i,k}} \binom{2i+1}{2j+1}\binom{k}{l} F_{j,l} F_{i-j,k-l}+\sum_{ E^2_{i,k} } \binom{2i+1}{2j}\binom{k}{l+1} F_{j,l}F_{i-j,k-l}.
\end{align*}

\noindent \underline{The free evolution term}. Let $e_{i,k}=\frac{\sfL}{2} \frac{(2i+k+1)!}{\langle i+k\rangle^3}\hat \tau^{-a_{i,k}}\bar \tau^{-b_{i,k}} $ and define $\tilde{\sfS}(\sfs,z):=e_{i,k}\tau^{-a_{i,k}}(\sfs)\sfw (z)$. Then notice that $\sfS=e_{i,k} (1+e^{-\frac{\iota}{2}\sfs})^{-a_{i,k}}e^{(\frac{a_{i,k}}{2}-c_{i,k}) \sfs}\sfS$ where $\sfS$ was defined in \fref{def:barf} (with $h(\sfs)=e^{-\sfK a_{i,k}e^{-\sfs}}$). We compute for $i+k\geq 2$, from the definitions \fref{def:cik} and \fref{def:bik} of $c_{i,k}$ and $a_{i,k}$:
$$
c_{i,k}-\frac{a_{i,k}}{2} = \max \left(-\frac 32 i-k+\frac{15}{8},-\frac 72i-\frac 18 \right)+\eta \langle i \rangle.
$$
Therefore, there exists $\eta^*>0$ and $c>0$ independent of $i$ and $k$ such that for all $0<\eta\leq \eta^*$ and $i+k\geq 2$:
\be \label{bd:encadrementdecay}
-\frac{1}{c}\langle i \rangle\leq c_{i,k}-\frac{a_{i,k}}{2}\leq -c\langle i \rangle<0.
\ee
As a result, since $\sfS$ was proved to be a supersolution for $\pa_{\sfs}+\mathcal L_{i,k}$ in the proof of Lemma \ref{lem:premilinearylinear} for all $\sfs\geq \sfs_0$, then $\tilde{\sfS}$ is also a supersolution for $\pa_{\sfs}+\mathcal L_{i,k}$. At the boundary $\{\sfs=\sfs_0\}$ or $\{z=-\pi\}$ there holds $|\tilde F_{i,k}|\leq \tilde{\sfS}$ from \fref{bd:boundarytrapimproved} (proved in the previous subsubsection) and \fref{bd:indainitialtrap}. Hence we deduce that $|\tilde F_{i,k}|\leq \tilde{\sfS}$ for all $\sfs\geq \sfs_0$ and $z\geq -\pi$ by maximum principle. Hence the bound:
\be \label{bd:nonlinearrenormalised1}
\|\tilde F_{i,k}(\sfs_1)\|_{L^{\infty}_\sfw}\leq \frac 12 \sfL \tau^{-a_{i,k}}(\sfs_1)\bar \tau^{-b_{i,k}} \tilde \tau^{-a_{i,k}} \frac{(2i+k+1)!}{\langle i +k\rangle^3}.
\ee

\noindent \underline{The first term $I$}. Note that this term is zero if $i=0$ so we assume $i\geq 1$. We have using that $\| \pa^{-1}fg\|_{L^{\infty}_\sfw}\lesssim \| f \|_{L^{\infty}_\sfw} \| g \|_{L^{\infty}_\sfw}$ and \fref{bd:estimationlow}:
$$
\left\| (\pa^{-1}F_{j,0})F_{i-j,k+1}\right\|_{L^{\infty}_\sfw}\lesssim \sfL^2 \bar \tau^{-b_{i,k}+1}\tau^{-a_{i,k}} \hat \tau^{-a_{i,k}}\frac{(2j+1)!}{\langle j\rangle^3}\frac{(2i-2j+k+2)!}{\langle i-j+k\rangle^3}.
$$
where we used \fref{bd:exponents4} as $j\geq 1$ and $k\geq 1$. We then compute:
$$
 \frac{(2i+1)!}{(2i-2j+1)!} \frac{(2i-2j+k+2)!}{\langle i-j+k\rangle^3} = \frac{(2i+1+k)!}{\langle i+k\rangle^3}\frac{\langle i+k\rangle^{3}}{\langle i-j+k\rangle^3} \frac{(2i+1)!(2i-2j+k+2)!}{(2i-2j+1)!(2i+k+1)!}.
$$
Using that for $m=0,...,2j-1$: $(2i+1-m)/(2i+k+1-m)\leq 1$ we get:
$$
\frac{(2i+1)!(2i-2j+k+2)!}{(2i-2j+1)!(2i+k+1)!}=\frac{(2i+1)\times ...\times (2i-2j+2)}{(2i+k+1)\times...\times (2i-2j+k+3)}\lesssim \langle i-j\rangle.
$$
As a result:
$$
\frac{\langle i+k\rangle^3}{(2i+1+k)!} \sum_{j=1}^{i-1} \frac{(2i+1)!}{(2j)!(2i-2j+1)!} \frac{(2j+1)!}{\langle j\rangle^3}\frac{(2i-2j+k+2)!}{\langle i-j+k\rangle^3}\lesssim \sum_{j=1}^{i-1} \frac{\langle i+k\rangle^3 \langle i-j\rangle }{\langle j \rangle^{2}\langle i-j+k\rangle^3}\lesssim \langle i \rangle.
$$
Hence the bound:
$$
\| I \|_{L^{\infty}_\sfw}\leq (C\bar \tau \sfL)\sfL   \bar \tau^{-b_{i,k}}\tau^{-a_{i,k}}\hat \tau^{-a_{i,k}} \langle i \rangle \frac{(2i+1+k)!}{\langle i+k\rangle^3} .
$$
Using this and \fref{bd:lineaire} we find that for the first term, for a universal constant $C>0$:
$$
\| \int_{\sfs_0}^{\sfs_1}S_{i,k}(\sfs,\sfs_1)(I)d\sfs \|_{L^{\infty}_\sfw}\leq(C\bar \tau \sfL)\sfL  \bar \tau^{-b_{i,k}}   \hat \tau^{-a_{i,k}}  \langle i \rangle  \frac{(2i+1+k)!}{\langle i+k\rangle^3} \int_{\sfs_0}^{\sfs_1} p(\sfs,\sfs_1)\tau^{-a_{i,k}}(\sfs)d\sfs
$$
where we introduced the notation $p(\sfs,\sfs_1)=e^{c_{i,k}(\sfs_1-\sfs)}\left(\frac{(1+e^{-\frac{\iota}{2}\sfs})e^{\sfK e^{-\sfs}}}{(1+e^{-\frac{\iota}{2}\sfs_1})e^{\sfK e^{-\sfs_1}}}\right)^{a_{i,k}}$ so that:
$$
p(\sfs,\sfs_1)\tau^{-a_{i,k}}(\sfs)=e^{-a_{i,k}\sfK e^{-\sfs_1}}e^{c_{i,k}\sfs_1}(1+e^{-\frac \iota2 \sfs_1})^{-a_{i,k}}e^{\left(\frac{a_{i,k}}{2}-c_{i,k}\right)\sfs}(1+e^{-\frac \iota2 \sfs})^{-a_{i,k}}
$$
We compute using \fref{def:tau} and integration by parts that for $\eta$ small independently of $i$ and $k$:
\bee
&& \int_{\sfs_0}^{\sfs_1} e^{\left(\frac{a_{i,k}}{2}-c_{i,k}\right)\sfs}(1+e^{-\frac \iota2 \sfs})^{-a_{i,k}}d\sfs=  \int_{\sfs_0}^{\sfs_1} \pa_s \left(\frac{e^{\left(\frac{a_{i,k}}{2}-c_{i,k}\right)s}}{\frac{a_{i,k}}{2}-c_{i,k}} \right)\left(1+e^{-\frac \iota 2 s} \right)^{-a_{i,k}}ds\\
&\leq &\frac{e^{\frac{a_{i,k}}{2}(s-\sfs_1)}}{\frac{a_{i,k}}{2}-c_{i,k}}\left(1+e^{-\frac \iota 2 \sfs_1}\right)^{-a_{i,k}}-\frac{\frac \iota 2 a_{i,k}}{\frac{a_{i,k}}{2}-c_{i,k}}e^{c_{i,k}\sfs_1} \int_{\sfs_0}^{\sfs_1}e^{\left(\frac{a_{i,k}}{2}-c_{i,k}-\frac \iota 2\right)s}\left( 1+e^{-\frac \iota 2 s} \right)^{-a_{i,k}-1}ds.
\eee
Using the identity above and \fref{bd:encadrementdecay}, we infer that there exists $C>0$ depending on $T$ and $T_0$ (since $\sfs_0=-\ln (T-T_0)$) and $\iota>0$ such that:
\be \label{bd:integrationtau}
\int_{\sfs_0}^{\sfs_1} p(\sfs,\sfs_1)\tau^{-a_{i,k}}(\sfs)d\sfs+\frac{\langle i+k\rangle}{\langle i \rangle} \int_{\sfs_0}^{\sfs_1} p(\sfs,\sfs_1)\tau^{-a_{i,k}}(\sfs)e^{-\frac \iota 2\sfs}d\sfs\leq \frac{C}{\langle i \rangle}\tau^{-a_{i,k}}(\sfs_1).
\ee
In particular, we obtain finally that for $\bar \tau$ small enough depending only on $\sfL$ and $T$:
\be \label{bd:nonlinearrenormalised2}
\| \int_{\sfs_0}^{\sfs_1}S_{i,k}(\sfs,\sfs_1)(I)ds \|_{L^{\infty}_\sfw} \leq \frac{\sfL}{8} \bar \tau^{-b_{i,k}}\tau^{-a_{i,k}}(\sfs_1)\hat \tau^{-a_{i,k}} \frac{(2i+1+k)!}{\langle i+k\rangle^3}.
\ee

\noindent \underline{The second term $II$}. We compute that for this term from \fref{bd:estimationlow}:
\bee
 \| F_{j,l}\|_{L^{\infty}_\sfw} \| F_{i-j,k-l} \|_{L^{\infty}_\sfw} \leq (\hat \tau \tau)^{\frac 14}\sfL^2 \tau^{-a_{i,k}}\bar \tau^{-b_{i,k}}\hat \tau^{-a_{i,k}} \frac{(2j+\ell+1)!}{\langle j +\ell \rangle^3} \frac{(2i-2j+k-\ell+1)!}{\langle i-j+k-\ell \rangle^3}.
\eee
where we used \fref{bd:exponents1} and \fref{bd:exponents2} as $l+j\geq 1$ and $i-j+k-l\geq 1$ in the sums. We use the identity \fref{bd:technicalcombinatorial2} with $(A_1,A_2,r_1)=(2i+1,k,2j+l+1)$ to obtain:
\bee
\| II\|_{L^{\infty}_\sfw} &\leq & (\hat \tau  \tau)^{\frac 14}\sfL^2 \hat \tau^{-a_{i,k}}\tau^{-a_{i,k}}\bar \tau^{-b_{i,k}} \sum_{r=2}^{2i-1+k}\binom{2i+1+k}{r} \frac{r!}{\langle r \rangle^3} \frac{(2i+k-r+2)!}{\langle 2 i+k+1-r \rangle^3}.
\eee
Hence, since $\sum_{r=0}^{2i+1+k} \langle r \rangle^{-3} \langle 2i+k+1-r\rangle^{-2}\langle i+k\rangle^{3}\lesssim \langle i+k\rangle$ and $\tau=e^{-\frac 12s}\tilde \tau$ we get:
$$
\| II \|_{L^{\infty}_\sfw}\leq C \langle i+k\rangle (\hat \tau \tilde \tau)^{\frac 14}e^{-\frac 1 8 s}\sfL^2 \hat \tau^{-a_{i,k}} \tau^{-a_{i,k}}\bar \tau^{-b_{i,k}} \frac{(2i+1+k)!}{\langle i+k\rangle^3}.
$$
From the above bound and the linear estimate \fref{bd:lineaire}:
$$
\| \int_{\sfs_0}^{\sfs_1}S_{i,k}(\sfs,\sfs_1)(II)d\sfs \|_{L^{\infty}_\sfw}\leq C(\hat \tau \tilde \tau)^{\frac 14} \sfL^2 \hat \tau^{-a_{i,k}} \bar \tau^{-b_{i,k}}  \frac{(2i+1+k)!}{\langle i+k\rangle^3} \langle k+i \rangle\int_{\sfs_0}^{\sfs_1} p(\sfs,\sfs_1)\tau^{-a_{i,k}}(\sfs)e^{-\frac 18 \sfs}d\sfs.
$$
Using \fref{bd:integrationtau}, for $\hat \tau$ small enough depending on $\sfL$, $\sfK$ and $T$, we obtain from the above identity:
\be \label{bd:nonlinearrenormalised3}
\|  \int_{\sfs_0}^{\sfs_1}S_{i,k}(\sfs,\sfs_1)(II)ds \|_{L^{\infty}_\sfw}\leq   \frac{\sfL}{8} \bar \tau^{-b_{i,k}}\tau^{-a_{i,k}}(\sfs_1)\hat \tau^{-a_{i,k}} \frac{(2i+1+k)!}{\langle i+k\rangle^3} .
\ee

\noindent \underline{End of the proof} Summing the three estimates \fref{bd:nonlinearrenormalised3}, \fref{bd:nonlinearrenormalised3} and \fref{bd:nonlinearrenormalised3} shows \fref{bd:indarenormalisedtrapimproved}.\\

\textbf{Step 2.} The case $k=0$. Note that $i\geq 2$ since $i+k\geq 2$. This case can be treated almost exactly the same way. We just point out the minor modifications.

For the free evolution $\tilde F_{i,0}$, it now satisfies the Dirichlet boundary condition at $z=-\pi$ as $F_{i,0}$ does so. To estimate it, we use the linear estimate \fref{bd:lineaire} and the initial datum estimate \fref{bd:indainitialtrap}. The resulting bound is acceptable if $\sfL$ has been taken large enough depending solely on the universal constant $C>0$ in \fref{bd:lineaire}.

Next, the forcing terms $I$ and $II$ are treated the exact same way. Note that for $I$ the sum is performed only on $j\in \{1,i-1\}$ since the term corresponding to $i$ is zero for $k=0$ from \fref{eq:Fikevo}. The estimate \fref{bd:exponents4} is still valid in this case, and so $I$ is estimated the same way. There are no changes to make to treat $II$. This concludes the proof of the proposition.

\end{proof}

%%%%%%%%%%%%%%%%%%%%%%%%%%%%%%%%%%%%%%%%%%%%%%%%%%%%%%%%%%%%%%%%%
%%%%%%%%%%%%%%%%%%%%%%%%%%%%%%%%%%%%%%%%%%%%%%%%%%%%%%%%%%%%%%%%%
%%%%%%%%%%%%%%%%%%%%%%%%%%%%%%%%%%%%%%%%%%%%%%%%%%%%%%%%%%%%%%%%%

\section{Analyticity in the transverse variable close to the axis} \label{sec:regularisation}

Here we show that solutions $\vec \xi=(\xi_i)_{i\geq 0}$ to system \eqref{eq:evolutionxii}, rewritten as:
\be \label{id:systeme}
\left\{ \begin{array}{l l} \pa_t\xi_i=\pa_{yy}\xi_i+H_i(\vec \xi, \vec \xi)+ J_i(\vec \xi,\vec \xi) , \\ \xi_i(0,y)=\xi^0_i(y), \\ \xi_i(t,0)=0,  \end{array} \right. \qquad i\in \mathbb N, \ y\in [0,\infty), \ t>0,
\ee
with
$$
H_i(\vec \xi, \vec \xi')= - \sum_{j=0}^i \binom{2i+1}{2j+1}  \xi_j \xi_{i-j}', \quad J_i(\vec \xi, \vec \xi')=\sum_{j=0}^i \binom{2i+1}{2j} (\pa_y^{-1}\xi_j)\pa_y \xi_{i-j}',
$$
are instantaneously regularized for $t>0$ and become analytic in $y$, up to the boundary $y=0$. Solutions with only bounded initial data will be understood in an integral sense, and will be classical solutions for $t>0$. Indeed, there is a representation formula for solutions to
\be \label{id:linearheatdirichlet}
\left\{ \begin{array}{l l} \pa_t\phi=\pa_{yy} \phi, \\ \phi(0,y)=\phi^0(y), \\ \phi(t,0)=0,  \end{array} \right.  \qquad y\in [0,\infty), \ t>0.
\ee
Given $f$ a real valued function on $[0,\infty)$, let $\bar f$ denote its extension to $\mathbb R$ by odd symmetry:
\be \label{id:defsymmetryodd}
\bar f (y)=\left\{ \begin{array}{l l} f(y) \quad \mbox{for }y\geq 0, \\ - f(-y) \quad \mbox{for }y< 0. \end{array}\right.
\ee
Then the solution $\phi(t)=S(t)\phi^0$ to \eqref{id:linearheatdirichlet} is given by ($K_t$ being defined in \eqref{eq:defheatkernel})
\be \label{id:representationlinearheat}
(S(t)\phi^0)(y)=\int_{-\infty}^\infty K_t(y-\tilde y) \bar \phi^0(\tilde y)d\tilde y,\qquad y>0.
\ee
We shall therefore look for solutions to \eqref{id:systeme} in the following integral sense, using Duhamel's formula:
\be \label{id:systemesolution}
\xi_i(t,y)=S(t)\xi_i^0+\int_0^t S(t-t') (H_i(\vec \xi(t'), \vec \xi(t'))+ J_i(\vec \xi(t'),\vec \xi(t')))dt'.
\ee
Throughout this section, $\omega$ denotes the weight
$$
\omega (y)=\langle y \rangle^{-2}
$$
and we introduce the weighted $L^\infty$ spaces for $\phi:[0,\infty)\rightarrow \mathbb R$ or $\phi:\mathbb R\rightarrow \mathbb R$ respectively:
$$
\| \phi \|_{L^\infty_\omega}=\sup_{y\geq 0} \frac{|\phi(y)|}{\omega(y)} \ \mbox{ or } \ \| \phi\|_{L^\infty_\omega}=\sup_{y\in \mathbb R} \frac{|\phi(y)|}{\omega(y)}.
$$
For $\tilde \tau>0$, we introduce the weighted (in $y$) analytic space (in $x$, recalling that $\xi_i$ stands the trace of $\pa_x^{2i+1}u$ on the axis) 
$$
\| \vec \xi^0 \|_{X^0}=\sup_{i \in \mathbb N} \ \frac{\tilde \tau^{2i+1} }{(2i+1)!} \| \xi^0_i\|_{L^\infty_\omega}.
$$

The main result of this section is the following proposition.

\begin{proposition} \label{lwp:pr:analytic}

Let $\tilde \tau>0$ and assume $\| \vec \xi^0 \|_{X^0}<\infty$. Then there exists $T_0>0$ and $\vec \xi$ a solution to \eqref{id:systeme} on $[0,T_0]$ in the sense of \eqref{id:systemesolution} such that for each $i$, $\xi_i \in C([0,T_0],L^\infty_\omega)$. Moreover, there holds:
\begin{itemize}
\item[(i)] \emph{Immediate regularization up to the boundary:} For each $i$, $\xi_i \in C^\infty ((0,T_0]\times [0,\infty))$, and $\vec \xi$ is a classical solution to \fref{id:systeme} on $(0,T_0]\times [0,\infty)$.
\item[(ii)] \emph{Analytic bounds:} There exist $C,\tau>0$ such that for all $i\geq 0$ and $(t,y)\in [0,T_0]\times [0,\infty)$:
\be \label{lwp:bd:analyticy2}
 |\xi_i(t,y) |  \leq C \tau^{-2i-1}(2i+1)!  \langle y \rangle^{-2}
\ee
For each $\tilde T\in (0,T_0)$, there exists $\bar C,\bar \tau>0$ such that for all $i,n\geq0$ and $(t,y) \in [\tilde T,T_0]\times [0,\infty)$:
\be \label{lwp:bd:analyticy}
 |\pa_y^n \xi_i(t,y) |  \leq C \bar \tau^{-2i-1-n}(2i+n+1)!  \langle y \rangle^{-2}
\ee
\end{itemize}

\end{proposition}

\begin{proof}

Proposition \ref{lwp:pr:analytic} is a direct consequence of the results of Lemmas \ref{lwp:lemGevreytime} and \ref{lwp:lem:elliptic}.

\end{proof}

\begin{remark}

We believe our proof of Proposition \ref{lwp:pr:analytic} could be adapted to show instantaneous analytic in $y$ regularization for solutions to the Prandtl system \fref{2DPrandtl}, for data that are everywhere $x$-analytic, and without the oddness in $x$ assumption.

\end{remark}

To ease notation, from now on and throughout this section, in the estimates we will use quantities of the form $(2i+n)!$ instead of $(2i+n+1)!$, and $\tau^i$ instead of $\tau^{2i+1}$. These are equivalent, up to changing certain constants that will appear by a fixed factor, which is harmless for the analysis. We write $T$ instead of $T_0$ for commodity, so $T$ here is not the blow-up time.

\subsection{Strategy of the proof of Proposition \ref{lwp:pr:analytic}}

For small times, we approximate the solution to \eqref{id:systeme} by the linear solution to \eqref{id:systemelinear}, showing that it undergoes a parabolic regularization like the linear solution does. We proceed as follows.

\begin{itemize}
\item We construct the solution through a Picard approximation scheme \eqref{lwp:id:picardscheme}. At each iterative step, the scheme preserves $C^\infty$ differentiability in $t$ but not necessarily the $C^\infty$ differentiability in $y$ due to boundary effects. That is why we first obtain the $C^\infty$ in time regularity.
\item This $C^\infty$ in $t$ regularity is measured in Gevrey-$2$ spaces. Indeed, first the system of homogeneous linear heat equations \eqref{id:systemelinear} regularizes the initial data, making it analytical in time, and so Gevrey-$\alpha$ for all $\alpha\geq 1$, see Lemma \ref{lem:estimatelinearheat2}. Second, for the inhomogeneous linear system \eqref{id:equationheatinhomogeneous} with a source term that has an analyticity radius $\sqrt t$ (singular at initial time), we show continuity in Gevrey-$2$ regularity, see Lemma \ref{lem:estimatelineaduhamel}.
\item Once a Gevrey-$2$ in $t$ solution is obtained, we get its analyticity in $y$ in Lemma \ref{lwp:lem:elliptic} by elliptic regularity techniques applied to equation \eqref{id:systeme}.
\end{itemize}

\subsection{Regularization for the system of homogeneous heat equations}

Our strategy is to approximate, for small times $t>0$, solutions to \eqref{id:systeme} by solutions to the system of linear heat equations:
\be \label{id:systemelinear}
\left\{ \begin{array}{l l} \pa_t \xi_i=\pa_{yy}\xi_i , \\ \xi_i(0,y)=\xi^0_i(y), \\ \xi_i(t,0)=0,  \end{array} \right. \qquad i\geq 0, \  y\in [0,\infty), \ t\in (0,T].
\ee
Standard regularization estimates for \eqref{id:systemelinear} rely on the above formula and on the standard heat kernel estimates given in Lemma \ref{lwp:lem:heatkernel} in Appendix \ref{ap:sec:heat}.

\begin{lemma}[Estimates for the system of homogeneous heat equations] \label{lem:estimatelinearheat2}

Let $\tau_0$ be given by Lemma \ref{lwp:lem:heatkernel}. There exists $C>0$ such that for each $\tilde \tau>0$, for $0<\tau < \min(\tilde \tau^2/2,\tau_0)$, given $\vec \xi^0$ satisfying $\| \vec \xi^0\|_{X^0}<\infty$, the solution $\vec \xi=S(t)\vec \xi^0$ to \eqref{id:systemelinear} satisfies for all $i\in \mathbb N$, $t\in (0,1]$ and $y\in [0,\infty)$:
$$
 |\pa_t^k \xi_i(t,y) |+\sqrt t |\pa_t^k \pa_y \xi_i(t,y) | \leq C\omega(y) (2i+k)! t^{-k} \tau^{-i-k}\| \vec \xi^0\|_{X^0}.
$$

\end{lemma}

\begin{proof}

Recall from \eqref{id:representationlinearheat} that $\xi_i(t)=K_t*\bar \xi_i^0$ where $\bar{\xi}_i^0$ is defined by \eqref{id:defsymmetryodd}. Differentiating, using \eqref{bd:heatanalytic2}, then \eqref{bd:heatanalytic21}, and then $a!b!\leq (a+b)!$, we obtain for any $i\in \mathbb N$, $m=0,1$, $k\in \mathbb N$, $t\in (0,1]$ and $y\geq 0$:
\begin{align*}
|\pa_t^k \pa_y^m \xi_i(t,y)| & \leq C k! t^{-k-\frac m 2} \tau_0^{-k} \int_{\tilde y\in \mathbb R} |\bar{\xi}_0^i(\tilde y)| K_{\kappa t} (y-\tilde y)d\tilde y\\
& \leq C k! t^{-k-\frac m 2} \tau_0^{-k}(2i+1)!\tilde \tau^{-2i} \omega (y) \| \vec \xi^0\|_{X^0} \\
 & \leq C\omega (y)  t^{-k-\frac m 2} \tau^{-i-k} (2i+k)!   \| \vec \xi^0\|_{X^0} (i+1) \frac{\tau^{i+k}}{\tau_0^k\tilde \tau^{2i}}.
\end{align*}
This proves the Lemma, as $(i+1) \frac{\tau^{i+k}}{\tau_0^k\tilde \tau^{2i}}$ is uniformly bounded since $0<\tau < \min(\tilde \tau^2/2,\tau_0)$.

\end{proof}

\subsection{Estimates for the system of inhomogeneous heat equations}

Nonlinear terms in \eqref{id:systeme} will be considered as forcing terms for a linear inhomogeneous heat equation. That is why here we study solutions to
\be \label{id:equationheatinhomogeneous}
\left\{ \begin{array}{l l} \phi_t=\pa_{yy}\phi+f , \\ \phi(0,y)=0, \\ \phi(t,0)=0,  \end{array} \right. \qquad  y\in [0,\infty), \ t\in (0,T].
\ee
We will formulate estimates in particular function spaces, in order to be able to apply them to \eqref{id:systeme} later on. Namely we introduce for $\tau$ defined by:
\be \label{bd:decreasetau}
\frac{\pa_t \tau}{\tau}=-\frac{1}{\sqrt{T}\sqrt{t}}, \qquad (\mbox{i.e.} \ \tau(t)=\tau(0)e^{-2\sqrt{\frac{t}{T}}}).
\ee
For $i\in \mathbb N $ and a Sobolev correction exponent $\alpha$ (that we will take equal\footnote{The exact value $2$ is not relevant. It only needs to be large enough for the inequality \eqref{bd:technicalcombinatorial} to hold true.} to $2$), define the coefficients
$$
\Lambda_{i,k,\alpha}(t)=t^{-k}\tau^{-i-k}(2i+2k)!\langle i+k\rangle^{-\alpha}.
$$
For measurable functions $u$ such that, for each $y\in [0,\infty)$, the function $t\mapsto u(t,y)$ is $C^\infty$ with respect to $t\in (0,T]$, we introduce the Gevrey $2$ in time norms:
\begin{align}
\label{lwp:iddefXi}& \| \phi \|_{X_{T,\alpha}^{i}([0,\infty))}=  \sup_{t\in (0,T], \ k,i\in \mathbb N} \Lambda_{i,k,\alpha}^{-1}(t) \ \| \pa_t^k \phi(t)\|_{L^\infty_\omega}, \\
\label{lwp:iddefYi}& \| \phi \|_{Y_{T,\alpha}^{i}([0,\infty))}=  \sup_{t\in (0,T], \ k,i\in \mathbb N}  t^{\frac 12}\Lambda_{i,k,\alpha}^{-1}(t) \ \| \pa_t^k \phi(t)\|_{L^\infty_\omega}, \\
\label{lwp:iddefZi} & \| \phi \|_{ Z_{T,2}^{i}([0,\infty))}= \| \phi \|_{X_{T,2}^{i}([0,\infty))}+ \| \pa_y \phi \|_{Y_{T,1}^{i}([0,\infty))}.
\end{align}
To ease notation, we write:
$$
\Lambda_{i,k}=\Lambda_{i,k,2}.
$$

\begin{lemma}[Estimates for the inhomogeneous heat equation]  \label{lem:estimatelineaduhamel}

There exist $C>0$, $\tau^*>0$, such that for any $0<T\leq 1$, $\tau$ satisfying \eqref{bd:decreasetau} with $0<\tau(0)\leq \tau^*$, the following holds true. Let $\Theta$ be the mapping which to $f$ associates the solution $\phi=\Theta(f)$ to \eqref{id:equationheatinhomogeneous}. Then $\Theta$ satisfies the continuity estimate:
\be \label{bd:inhomogeneoustimeanalytic}
\| \Theta (f) \|_{Z_{T,2}^{i}} \leq C \sqrt{T}  \| f \|_{Y_{T,1}^{i}}
\ee

\end{lemma}

\begin{proof}

Pick $m\in \{0,1\}$. The solution to \eqref{id:equationheatinhomogeneous} is given by the following formula
\be \label{id:representationformulaheat}
u(t)=\int_0^t K_{t-t'}*h(t')dt', \qquad h(t)=\bar f(t).
\ee
Recall that if $\varphi \in L^\infty_\omega$, then $\bar \varphi \in L^\infty_\omega$ with same norm. For $k\in \mathbb N$, we let $\theta_k=1-\frac{1}{2+k}$ and decompose:
\be \label{id:lwpdecompositionu}
 u(t)=\underbrace{\int_0^{\theta_k t} ( K_{t-t'})*h(t')dt'}_{=u_1}+\underbrace{\int_{\theta_k t}^t ( K_{t-t'})* h(t')dt'}_{=u_2}.
\ee

\noindent \textbf{Step 1}. \emph{Estimates for $u_1$}. By a direct computation:
\be \label{id:lwpdecompositionu1}
\pa^k_t \pa_y^m u_1=\theta_k \sum_{p=0}^{k-1}\pa_t^{k-1-p}[\pa_y^m\pa_t^p K_{(1-\theta_k)t}*h(\theta_k t)]+\int_0^{\theta_k t} (\pa_y^m  \pa_t^k K_{t-t'})* h(t')dt'.
\ee
For the first term in \eqref{id:lwpdecompositionu1}, we can assume $k\geq 1$, and using the Leibniz identity:
\be \label{id:lwpdecompositionu122}
\pa_t^{k-1-p}[\pa_y^m\pa_t^p K_{(1-\theta_k )t}*h(\theta_k t)]= \sum_{l=0}^{k-1-p} \binom{k-1-p}{l}(\pa_t^l[\pa_y^m\pa_t^p K_{(1-\theta_k )t}]* \pa_t^{k-1-p-l}[h(\theta_k t)] ).
\ee
Using \eqref{bd:heatanalytic2} and that $(1-\theta_k)^{-p-m/2}=(2+k)^{p+m/2}\leq Ck^{p+m/2} $ with $C$ independent of $k$ and $p$ as $p\leq k-1$, we obtain that for $t\in (0,T]$:
\begin{align}
\nonumber  |(\pa_t^l[\pa_y^m\pa_t^p K_{(1-\theta_k )t}] | &= (1-\theta_k )^{l} | ( \pa_y^m\pa_t^{l+p} K)_{(1-\theta_k )t} |  \leq C\tau_0^{-l-p} t^{-l-p-\frac m 2} (1-\theta_k)^{-p-\frac m2} (l+p)! K_{(1-\theta_k)\kappa t} \\
\label{bd:heatanalyticadd1} &\leq C\tau_0^{-l-p} t^{-l-p-\frac m 2} k^{p+\frac m2} (l+p)! K_{(1-\theta_k)\kappa t}.
\end{align}
Using the definition of $Y_{T,1}^i$, $\theta_k^{-1/2}\lesssim 1$ and then the inequality $(2i+2k-2-2p-2l)!\lesssim (2i+2k-2p-2l)!\langle i+k-p-l\rangle^{-2}$:
\begin{align}
\nonumber & | \pa_t^{k-1-p-l}[h(\theta_k t)] )| = \theta_k^{k-1-p-l} | (\pa_t^{k-1-p-l} h) (\theta_k t) | \ \leq \theta_k^{-\frac 12}t^{-\frac 12} \Lambda_{i,k-1-p-l,1}(t)  \| f\|_{Y_{T,1}^i} \omega \\
\label{bd:heatanalyticadd2} &\qquad \qquad \lesssim (2i+2k-2p-2l)! \langle i+k-p-l\rangle^{-3} t^{-k+p+l+\frac 12} \tau^{-i-k+p+l+1} \| f\|_{Y_{T,1}^i} \omega .
\end{align}
Combining \eqref{bd:heatanalytic21}, \eqref{bd:heatanalyticadd1} and \eqref{bd:heatanalyticadd2}, choosing $\tau(0)\leq \tau_0/2$ so that $\tau\leq \tau_0/2$ from \eqref{bd:decreasetau}:
\begin{align*}
& \|  (\pa_t^l[\pa_y^m\pa_t^p K_{(1-\theta_k )t}]* \pa_t^{k-1-p-l}[h(\theta_k t)] ) \|_{L^\infty_\omega} \\
& \leq C \sqrt T \| f\|_{Y_{T,1}^i} \tau^{-i-k+1}t^{- k-\frac m 2} \langle i+k-p-l\rangle^{-3}k^{\frac m2} 2^{-l-p} (2i+2k-2p-2l)!(l+p)!k^p.
\end{align*}
We now estimate using $a!b!\leq (a+b)!$, then $(l!)^{-1}\leq 1$, then that for each $p\leq k-1$ and $l\leq k-1-p$ one has $2i+2k-p-l+1\geq k-1-p$ and $2i+2k-p+1\geq k$:
\begin{align*}
& \binom{k-1-p}{l} (2i+2k-2p-2l)!(l+p)!k^p \leq \frac{(k-1-p)...(k-p-l)}{l!}(2i+2k-p-l)!k^{p} \\
&\qquad \leq \frac{k-1-p}{2i+2k-p} \ ... \ \frac{k-p-l}{2i+2k-p-l+1}(2i+2k-p)!k^{p} \leq (2i+2k-p)!k^p \  \leq \ (2i+2k)!.
\end{align*}
Injecting the two above inequalities in \eqref{id:lwpdecompositionu122} we obtain:
\begin{align} 
\nonumber & \| \pa_t^{k-1-p}[\pa_y^m\pa_t^p K_{(1-\theta_k )t}*h(\theta_k t)] \|_{L^\infty_\omega}\\
\nonumber & \leq C \sqrt T \| f\|_{Y_{T,1}^i} \tau^{-i-k+1}t^{- k-\frac m 2}(2i+2k)! k^{\frac m2}\sum_{p=0}^{k-1}\sum_{l=0}^{k-1-p} \langle i+k-p-l\rangle^{-3} 2^{-l-p}  \\
\label{id:lwpdecompositionu1bound1} & \leq C \sqrt{T} \langle i+k\rangle^{-\frac 12} \| f\|_{Y_{T,1}^i} \tau t^{-\frac m2} \Lambda_{i,k}.
\end{align}
where we used $\sum_{p=0}^{k-1}\sum_{l=0}^{k-1-p} \langle i+k-p-l\rangle^{-3} 2^{-l-p}\leq C \langle i+k\rangle^{-3}$ for $C$ independent of $i,k$.

For the second term in \eqref{id:lwpdecompositionu1}, we estimate using \eqref{bd:heatanalytic2} and \eqref{bd:heatanalytic21}, then \eqref{bd:decreasetau}:
\begin{align*}
\|\pa_t^k \pa_y^m K_{t-t'}* h(t')\|_{L^\infty_\omega} & \leq Ck! (t-t')^{-k-\frac m 2} \tau_0^{-k}  \| f \|_{L^\infty_\omega} \\
&\leq Ck! (2i)!  \langle i \rangle^{-1} (\frac{\tau(0)}{\tau_0})^k \| f\|_{Y_{T,1}^i} (t-t')^{-k-\frac m 2}t^{'-\frac 12}\tau^{-i-k}.
\end{align*}
We estimate the following time integral if $i=k=0$:
$$
\int_0^{\theta_k t} (t-t')^{-k-\frac m 2} t^{'-\frac 12}\tau^{-i-k}(t')dt' =\int_0^{\theta_k t} (t-t')^{-\frac m 2} t^{'-\frac 12}dt' \lesssim \sqrt{T}t^{-\frac m 2},
$$
and using \eqref{bd:decreasetau} if $i+k\geq 1$:
\begin{align*}
& \int_0^{\theta_k t} (t-t')^{-k-\frac m 2} t^{'-\frac 12}\tau^{-i-k}(t')dt' \ \leq \ (t-t\theta_k)^{-k-\frac m 2} \int_0^t t^{'-\frac 12}\tau^{-i-k}(t')dt'  \\
& = t^{-k-\frac m 2}(1-\theta_k)^{-k-\frac m 2} \frac{\sqrt{T}}{i+k} \int_0^t \pa_{t'} (\tau^{-i-k})(t')dt' \ \leq \ t^{-k-\frac m 2}(1-\theta_k)^{-k-\frac m 2} \frac{\sqrt{T}}{i+k} \tau^{-i-k}.
\end{align*}
Hence, using $\langle i \rangle^{-1} \langle i+k\rangle^{-1}(\frac{\tau(0)}{\tau_0})^k(1-\theta_k)^{-\frac 12}\lesssim(\frac{\tau(0)}{\tau_0})^{\frac k2} \langle i+k\rangle^{-2}$ if $\tau(0)$ is small enough:
$$
\|  \int_0^{\theta_k t} \pa_y^{m}\pa_t^k K_{t-t'}* h(t')dt'  \|_{L^\infty_\omega}\leq  C k! (2i)! \langle i +k \rangle^{-2} t^{-k-\frac m 2} (1-\theta_k)^{-k}\tau^{-i-k} (\frac{\tau(0)}{\tau_0})^{\frac k2} \sqrt{T}  \| f\|_{Y_{T,1}^{i}}.
$$
We estimate using Stirling's formula that $\frac{k!}{(2k)!}\sim \frac{1}{\sqrt{2}}(\frac e4)^{k}k^{-k}$, so that since $1-\theta_k=\frac{1}{k+2}$:
$$
k!(1-\theta_k)^{-k} \sim (2k)! \frac{1}{\sqrt{2}}(\frac e4)^{k} (\frac{k+2}{k})^k \sim (2k)! \frac{e^2}{\sqrt{2}}(\frac e4)^{k} \qquad \mbox{as }k\to \infty,
$$
and hence, for $\tau(0)/\tau_0$ small enough, using $(2k)!(2i)!\leq (2i+2k)!$ we upper bound:
\be \label{id:lwpdecompositionu1bound2}
\|  \int_0^{\theta_k t} \pa_y^m \pa_t^k K_{t-t'}* h(t')dt'v  \|_{L^\infty_\omega}\leq  C \sqrt T \| f\|_{Y_{T,1}^{i}}  t^{-\frac m 2} \Lambda_{i,k}
\ee
Combining \eqref{id:lwpdecompositionu1bound1} and \eqref{id:lwpdecompositionu1bound2}, we get that $u_1$ satisfies the bound:
\be \label{id:lwpdecompositionu1bound}
\| u_1\|_{Z_{T,2}^{i}} \leq C\sqrt{T} \| f\|_{Y_{T,1}^{i}}.
\ee

\noindent \textbf{Step 2.} \emph{Estimate for $u_2$}. We differentiate with respect to time and then integrate by parts to find:
\begin{align*}
\pa_t u_2 & =h(t)-\theta_k K_{(1-\theta_k)t}*h(\theta_k t)+\int_{\theta_k t}^t \pa_t [K_{t-t'}]*h(t')dt' \\
& =h(t)-\theta_k K_{(1-\theta_k)t}*h(\theta_k t)-\int_{\theta_k t}^t \pa_{t'} [K_{t-t'}]*h (t')dt' \\
& = (1-\theta_k) K_{(1-\theta_k)t}*h(\theta_k t)+\int_{\theta_k t}^t  K_{t-t'}* (\pa_{t}h)(t')dt' .
\end{align*}
Iterating the above computation, we find the following identity for all $k\in \mathbb N$:
\be \label{id:lwpdecompositionu2}
\pa_t^k \pa_y^m u_2  = \sum_{p=0}^{k-1}\pa_t^{k-1-p}[ (1-\theta_k) \pa_y^m K_{(1-\theta_k)t}*(\pa_t^p h)(\theta_k t)]+\int_{\theta_k t}^t  \pa_y^m K_{t-t'}* (\pa_{t}^k h)(t')dt' .
\ee
The first term in \eqref{id:lwpdecompositionu2} is estimated as the first term in \eqref{id:lwpdecompositionu1} in Step 1. Namely, using Leibniz:
\be \label{heatanalytic:add:inter10}
\pa_t^{k-1-p}[ \pa_y^m K_{(1-\theta_k)t}*(\pa_t^p h)(\theta_k t)]=\sum_{l=0}^{k-p-1} \binom{k-p-1}{l} \pa_t^{l}[ \pa_y^m K_{(1-\theta_k)t}]*\pa_t^{k-1-p-l}[(\pa_t^p h)(\theta_k t)].
\ee
Using \eqref{bd:heatanalyticadd1} with $p=0$ one has:
$$
|\pa_t^{l}[ \pa_y^m K_{(1-\theta_k)t}]|\leq C \tau_0^{-l}t^{-l-\frac m2}k^{\frac m2}l! K_{(1-\theta_k)\kappa t}.
$$
Using $\pa_t^{k-1-p-l}[(\pa_t^p h)(\theta_k t)]=\theta_k^{-p}\pa_t^{k-1-l}[( h)(\theta_k t)]$, then \eqref{bd:heatanalyticadd2} (with $p+l$ replaced by $l$), and $\theta_k^{-p}\leq \theta_k^{-k}=(1-1/(k+2))^{-k}\lesssim 1$ one has:
$$
|\pa_t^{k-1-p-l}[(\pa_t^p h)(\theta_k t)]|\leq C (2i+2k-2l)! \langle i+k-l\rangle^{-3}t^{-k+l+\frac 12} \tau^{-i-k+l+1} \| f\|_{Y^i_{T,1}}\omega
$$
Choosing $\tau\leq \tau_0/2$ and using \eqref{bd:heatanalytic21}, the two inequalities above give:
\begin{align}
\nonumber & \|  \pa_t^{l}[ \pa_y^m K_{(1-\theta_k)t}]*\pa_t^{k-1-p-l}[(\pa_t^p h)(\theta_k t)] \|_{L^\infty_\omega} \\
\label{bd:heatanalyticadd10} & \leq C \sqrt T \| f\|_{Y_{T,1}^i} \tau^{-i-k+1}t^{- k-\frac m 2} \langle i+k-l\rangle^{-3}k^{\frac m2} 2^{-l} (2i+2k-2l)!l!.
\end{align}
Using $a!b!\leq (a+b)!$, then $(l!)^{-1}\leq 1$, and then that for $p\leq k-1$ and $l\leq k-1-p$ there holds $k-1-p\leq 2i+2k-l$:
\begin{align}
\nonumber & \binom{k-1-p}{l} (2i+2k-2l)!l! \  \leq \  \frac{(k-1-p)...(k-p-l)}{l!} (2i+2k-l)! \\
 \label{bd:heatanalyticadd11}& \qquad  \qquad  \qquad \qquad \qquad  \leq \frac{k-1-p}{2i+2k} \ ... \ \frac{k-p-l}{2i+2k-l+1} (2i+2k)! \ \leq  \ (2i+2k)!.
\end{align}
Injecting \eqref{bd:heatanalyticadd10} and \eqref{bd:heatanalyticadd11} in \eqref{heatanalytic:add:inter10}, using that $\sum_{p=0}^{k-1}\sum_{l=0}^{k-1-p}\langle i+k-l\rangle^{-3} 2^{-l}\lesssim k \langle i+k\rangle^{-3}$, that $1-\theta_k\leq k^{-1}$ and that $k^{m/2}\leq \langle i+k\rangle^{1/2}$, one finds that for the first term in \eqref{id:lwpdecompositionu2}:
\be \label{id:lwpdecompositionu2bound1}
\|  \sum_{p=0}^{k-1}\pa_t^{k-1-l}[ (1-\theta_k) \pa_y^m K_{(1-\theta_k)t}*(\pa_t^p h)(\theta_k t)] \|_{L^\infty_\omega} \leq C\tau  \sqrt{T} \| f \|_{Y_{T,1}^{i}} \langle k+i\rangle^{-\frac 12} t^{-\frac m2} \Lambda_{i,k}.
\ee

For the second term in \eqref{id:lwpdecompositionu2} we first estimate by \eqref{bd:heatanalytic21}:
$$
\| \pa_y^m  K_{t-t'}* (\pa_{t}^kh )(t') \|_{L^\infty_\omega} \lesssim (t-t')^{-\frac m 2}  \|  (\pa_{t}^k h)(t') \|_{L^\infty_\omega} \lesssim (t-t')^{-\frac m 2}\tau^{-i-k} (2k+2i)!\langle i+k\rangle^{-1} t^{'- k-\frac 12} \| f\|_{Y_{T,1}^{i}}
$$
and hence, since $\theta_k^{-k}=(1-\frac{1}{2+k})^{-k}\rightarrow e$ as $k\rightarrow \infty$:
\begin{align*}
\| \int_{\theta_k t}^t  \pa_y^m K_{t-t'}* (\pa_{t}^kh)(t')dt' \|_{L^\infty_\omega} & \leq C (2k+2i)!\langle k+i\rangle^{-1}\| f\|_{Y_{T,1}^{i}} \int_{\theta_k t}^t (t-t')^{-\frac m 2}t^{'- k-\frac 12}  \tau^{-i-k}(t') dt'\\
&\leq C (2k+2i)!\langle k+i\rangle^{-1}\| f\|_{Y_{T,1}^{i}} t^{-k} \int_{\theta_k t}^t (t-t')^{-\frac m 2} t^{-\frac 12}\tau^{-i-k}(t') dt'.
\end{align*}
We now estimate the above integral. For $i=k=0$, we have $\int_{\theta_k t}^t (t-t')^{-\frac m 2}t^{'-\frac 12} \tau^{-i-k}(t') dt'\lesssim \sqrt{T}t^{-\frac m 2} $. For $i+k\geq 1$, for $m=0$, using \eqref{bd:decreasetau}:
$$
\int_{\theta_k t}^t t^{'-\frac 12} \tau^{-i-k}(t') dt'=\frac{\sqrt T}{i+k} \int_{\theta_k t}^t \pa_{t'}( \tau^{-i-k})(t') dt' \leq \frac{\sqrt{T}}{i+k} \tau^{i+k}(t),
$$
while for $m=1$, we get using $\tau(t')\geq \tau(t)$ for $t'\leq t$:
$$
\int_{\theta_k t}^t (t-t')^{-\frac 12} t^{'-\frac 12}\tau^{-i-k}(t') dt' \leq \tau^{-i-k}(t) \int_{\theta_k t}^t (t-t')^{-\frac 12} t^{'-\frac 12}dt'\leq  \tau^{-i-k} \int_0^1 (1-\sigma)^{-\frac 12}\sigma^{-\frac 12}d\sigma \lesssim   \tau^{-i-k} .
$$
Combining the four above inequalities, one ends up with
\be \label{id:lwpdecompositionu2bound2}
\| \int_{\theta_k t}^t  \pa_y^m K_{t-t'}* (\pa_{t}^kh)(t')dt' \|_{L^\infty_\omega} \leq C \sqrt{T} t^{-\frac m2}\langle i+k\rangle^m \Lambda_{i,k} \| f\|_{Y_{T,1}^{i}}.
\ee
Therefore, summing \eqref{id:lwpdecompositionu2bound1} and \eqref{id:lwpdecompositionu2bound2} we obtain that for the $u_2$ term:
\be \label{id:lwpdecompositionu2bound}
\| u_2\|_{Z_{T,2}^{i}} \leq C\sqrt{T} \| f\|_{Y_{T,1}^{i}}.
\ee

\noindent \underline{Conclusion.} Injecting \eqref{id:lwpdecompositionu1bound} and \eqref{id:lwpdecompositionu2bound} in \eqref{id:lwpdecompositionu} shows the bound \eqref{bd:inhomogeneoustimeanalytic}.

\end{proof}

\subsection{Bilinear estimates}

We estimate here the quadratic terms in \eqref{id:systeme}. We introduce the $X_{T,2}^i$ and $Z_{T,2}^i$ based (defined by \eqref{lwp:iddefXi} and \eqref{lwp:iddefZi}) vector spaces:
$$
\| \vec \xi \|_{X_{T,2}([0,\infty))}= \sup_{i\in \mathbb N} \ \| \xi_i \|_{X_{T,2}^{i}([0,\infty))}, \qquad \| \vec \xi \|_{ Z_{T,2}([0,\infty))}= \sup_{i\in \mathbb N} \ \| \xi_i \|_{ Z_{T,2}^{i}([0,\infty))}.
$$
The following Lemma states that $H_i$ and $J_i$ both loose a derivative in a combinatorics sense (a $\langle i+k\rangle$ factor). Moreover, as $\pa_y$ derivatives are regularized each with a $t^{-1/2}$ factor, $J_i$ looses an additional $t^{-1/2}$ factor.

\begin{lemma}[Bilinear estimates]

There holds the following estimates for $C>0$ independent of $\tau$ and $i$:
\be \label{bd:bilinearXT22}
\| H_i(\vec \xi,\vec \xi') \|_{X_{T,1}^i} \leq C \| \vec \xi \|_{X_{T,2}} \| \vec \xi' \|_{X_{T,2}}, \qquad \|  J_i(\vec \xi,\vec \xi') \|_{Y_{T,1}^{i}} \leq C \| \vec \xi \|_{X_{T,2}} \| \vec \xi' \|_{Z_{T,2}}.
\ee

\end{lemma}

\begin{proof}

We write $H_i=H_i(\vec \xi,\vec \xi')$ and $J_i=J_i(\vec \xi,\vec \xi')$ for simplicity. Then, by Leibniz rule for differentiation, for $t\in (0,T]$:
$$
|\pa_t^{k}H_i| \leq  \sum_{j=0}^i \sum_{l=0}^k \binom{2i+1}{2j+1}  \binom{k}{l} \ | \pa_t^l \xi_j \ \pa_t^{k-l} \xi_{i-j}'|.
$$
We introduce $r=2j+2l$ and bound using the definition of the $X_{T,2}$ norms:
\begin{align*}
 \|  \pa_t^l \xi_j \ \pa_t^{k-l} \xi_{i-j}' \|_{L^\infty_\omega} &\leq \Lambda_{j,l}  \| \vec \xi \|_{X_{T,2}} \Lambda_{i-j,k-l}  \| \vec \xi' \|_{X_{T,2}} \\
& \lesssim \| \vec \xi \|_{X_{T,2}}  \| \vec \xi' \|_{X_{T,2}} \tau^{-i-k}t^{-k} r!(2i+2k-r)! \langle r\rangle^{-2}\langle 2i+2k-r \rangle^{-2}.
\end{align*}
Therefore, using that $\binom{k}{l}\leq \binom{2k}{2l}$ and \eqref{bd:technicalcombinatorial2} with $A_1=2i+1$, $A_2=2k$ and $r_1=r+1$:
\begin{align*}
\| \pa_t^{k}H_i\|_{L^\infty_\omega}&  \lesssim t^{-k} \tau^{-i-k} \| \vec \xi \|_{X_{T,2}}  \| \vec \xi' \|_{X_{T,2}}  \sum_{j=0}^i \sum_{l=0}^k \binom{2i+1}{2j+1}  \binom{k}{l} r!(2i+2k-r)! \langle r\rangle^{-2}\langle 2i+2k-r \rangle^{-2} \\
 & \lesssim t^{-k} \tau^{-i-k}  \| \vec \xi \|_{X_{T,2}}  \| \vec \xi' \|_{X_{T,2}}  \sum_{r=0}^{2i+2k+1} \binom{2i+2k+1}{r+1} r!(2i+2k-r)! \langle r\rangle^{-2}\langle 2i+2k-r \rangle^{-2} \\
& \lesssim t^{-k} \tau^{-i-k}  (2i+2k+1)!\| \vec \xi \|_{X_{T,2}}  \| \vec \xi' \|_{X_{T,2}}  \sum_{r=0}^{2i+2k+1} \langle r\rangle^{-3}\langle 2i+2k-r \rangle^{-2} \\
& \lesssim t^{-k} \tau^{-i-k}  (2i+2k)!  \langle i+k \rangle^{-1} \| \vec \xi \|_{X_{T,2}}  \| \vec \xi' \|_{X_{T,2}} \ \lesssim \ \Lambda_{i,k} \langle i+k\rangle  \| \vec \xi \|_{X_{T,2}}  \| \vec \xi' \|_{X_{T,2}} ,
\end{align*}
where we used \eqref{bd:technicalcombinatorial} with $K=2i+2k+1$. This precisely implies the first inequality in \eqref{bd:bilinearXT22}.

To prove the second inequality in \eqref{bd:bilinearXT22}, we first write:
$$
\pa_t^k  J_i=\sum_{l=0}^k\sum_{j=0}^i \binom{k}{l} \binom{2i+1}{2j} (\pa_y^{-1}\pa_t^l \xi_j) \pa_y \pa_t^{k-l}\xi_{i-j}'.
$$
We then bound, using the definition of the $X_{T,2}$ and $Z_{T,2}$ norms, that $\int_0^\infty \omega(y)dy<\infty$, and introducing $r=2l+2j$:
\begin{align*}
 \| (\pa_y^{-1} \pa_t^l \xi_j) \ \pa_y \pa_t^{k-l} \xi_{i-j}' \|_{L^\infty_\omega} & \lesssim  \Lambda_{j,l} \| \vec \xi \|_{X_{T,2}} \Lambda_{i-j,k-l} \langle i-j+k-l\rangle t^{-\frac 12} \| \vec \xi' \|_{Z_{T,2}} \\
& \lesssim \| \vec \xi \|_{X_{T,2}}  \| \vec \xi' \|_{Z_{T,2}} \tau^{-i-k}t^{-k-\frac 12} r!(2i+2k-r)! \langle r\rangle^{-2}\langle 2i+2k-r \rangle^{-1},
\end{align*}
so that, using $\binom{k}{l}\leq \binom{2k}{2l}$:
$$
\|\pa_t^k  J_i\|_{L^\infty_\omega}  \lesssim  \| \vec \xi \|_{X_{T,2}}  \| \vec \xi' \|_{Z_{T,2}} \tau^{-i-k}t^{-k-\frac 12} \sum_{l=0}^k\sum_{j=0}^i \binom{2k}{2l} \binom{2i+1}{2j} r!(2i+2k-r)! \langle r\rangle^{-2}\langle 2i+2k-r \rangle^{-1}.
$$
Using \eqref{bd:technicalcombinatorial2} with $(A_1,A_2,r_1)=(2i+1,2k,r)$, and then \eqref{bd:technicalcombinatorial} with $K=2i+2k+1$:
\begin{align*}
& \sum_{l=0}^k\sum_{j=0}^i \binom{2k}{2l} \binom{2i+1}{2j} r!(2i+2k-r)! \langle r\rangle^{-2}\langle 2i+2k-r \rangle^{-1}\\
 &\quad \leq \sum_{r=0}^{2i+2k} \binom{2i+2k+1}{r} r !(2i+2k-r)! \langle r\rangle^{-2}\langle 2i+2k-r \rangle^{-1} \\
 &\qquad =\sum_{r=0}^{2i+2k} (2i+2k+1)! (2i+2k-r+1)^{-1} \langle r\rangle^{-2}\langle 2i+2k-r \rangle^{-1} \ \lesssim \ (2i+2k)! \langle i+k\rangle^{-1}.
\end{align*}
Combining the two above inequalities shows $\|\pa_t^k  J_i\|_{L^\infty_\omega}  \leq C \Lambda_{i,k} t^{-\frac 12} \langle i+k\rangle  \| \vec \xi \|_{X_{T,2}}  \| \vec \xi' \|_{Z_{T,2}} $ which is precisely the second inequality in \eqref{bd:bilinearXT22}.

\end{proof}

\subsection{Obtention of Gevrey-$2$ in time regularity by Picard iteration scheme}

The Lemma below shows analytic regularization for solutions to \eqref{id:systemelinear}.

\begin{lemma} \label{lwp:lemGevreytime}

For any $\tilde \tau>0$, there exists $\tau(0)>0$ such that the following holds. For any $\vec \xi^0$ satisfying $\| \vec \xi^0\|_{X^0}<\infty$, then there exists a $T>0$ and a solution $\vec \xi$ to \eqref{id:systeme} in the sense of \eqref{id:systemesolution} such that for each $i$, $\xi_i \in C([0,T],L^\infty_\omega)$. Moreover, there holds:
$$
\| \vec \xi \|_{Z_{T,2}}<\infty.
$$

\end{lemma}

\begin{proof}

Let $\vec \zeta=S(t)\vec \xi^0$. We look for a solution under the form $\vec \xi=\vec \zeta+\vec{\tilde \zeta}$. We consider the mapping $\Phi$ which to $\vec v \in Z_{T,2}$ associates the unique solution $\vec w=\Phi(\vec v)$ to:
\be \label{lwp:id:picardscheme}
\left\{ \begin{array}{l l} \pa_tw_i=\pa_{yy}w_i+H_i(\vec \zeta+\vec v, \vec \zeta+\vec v)+ J_i(\vec \zeta+\vec v,\vec \zeta+\vec v) , \\ w_i(0,y)=0, \\ w_i(t,0)=0,  \end{array} \right. \qquad i\in \mathbb N, \ y\in [0,\infty), \ t\in [0,T], 
\ee
By Lemma \ref{lem:estimatelinearheat2} and \eqref{bd:bilinearXT22}, we have for all $i\in \mathbb N$ that
\begin{align*}
\| H_i(\vec \zeta+\vec v, \vec \zeta+\vec v) \|_{Y_{T,1}} & \leq \sqrt{T}\| H_i(\vec \zeta+\vec v, \vec \zeta+\vec v) \|_{X_{T,1}} \lesssim \sqrt{T}(\| \vec \zeta \|_{X_{T,2}}+\| \vec v \|_{X_{T,2}})^2 \\
& \lesssim  \sqrt{T}(C+\| \vec v \|_{Z_{T,2}})^2
\end{align*}
where $C$ is independent of $T$, and similarly for $\vec{v'}\in Z_{T,2}$, since $H_i$ is bilinear:
\begin{align*}
\| H_i(\vec \zeta+\vec v, \vec \zeta+\vec v)-H_i(\vec \zeta+\vec v', \vec \zeta+\vec v')  \|_{Y_{T,1}} &= \| H_i(\vec \zeta+\vec v, \vec v-\vec{v'})+H_i(\vec v-\vec v', \vec \zeta+\vec v')  \|_{Y_{T,1}}  \\
& \lesssim  \sqrt{T}\| \vec v-\vec{v'} \|_{Z_{T,2}} \left(C+\| \vec v \|_{Z_{T,2}}+\| \vec{v'} \|_{Z_{T,2}} \right).
\end{align*}
Similarly, using again Lemma \ref{lem:estimatelinearheat2} and \eqref{bd:bilinearXT22}, we have:
\begin{align*}
\| J_i(\vec \zeta+\vec v, \vec \zeta+\vec v) \|_{Y_{T,1}} & \lesssim (\| \vec \zeta \|_{Z_{T,2}}+\| \vec v\|_{Z_{T,2}})^2\lesssim (C+\| \vec v\|_{Z_{T,2}})^2
\end{align*}
In addition, using the fact that $J_i$ is bilinear
\begin{align*}
\| J_i(\vec \zeta+\vec v, \vec \zeta+\vec v)-J_i(\vec \zeta+\vec v', \vec \zeta+\vec v')  \|_{Y_{T,1}} &= \| J_i(\vec \zeta+\vec v, \vec v-\vec{v'})+J_i(\vec v-\vec v', \vec \zeta+\vec v')  \|_{Y_{T,1}}  \\
& \lesssim  \| \vec v-\vec{v'} \|_{Z_{T,2}} \left(C+\| \vec v \|_{Z_{T,2}}+\| \vec{v'} \|_{Z_{T,2}} \right).
\end{align*}
Therefore, thanks to \eqref{bd:inhomogeneoustimeanalytic} we deduce from the above estimates that:
$$
\| \Phi(\vec v)\|_{Z_{T,2}}\lesssim \sqrt{T}(C+\| \vec v \|_{Z_{T,2}})^2, \quad \| \Phi(\vec v)-\Phi(\vec{v'})\|_{Z_{T,2}}\lesssim \sqrt{T}  \| \vec v-\vec{v'} \|_{Z_{T,2}} \left(C+\| \vec v \|_{Z_{T,2}}+\| \vec{v'} \|_{Z_{T,2}} \right).
$$
Thus, there exists $T$ small enough such that $\Phi$ is a contraction on the unit ball of $Z_{T,2}$. Hence $\Phi$ possesses a unique fixed point $\vec{\tilde \zeta}$ by Banach-Picard fixed point Theorem. Then $\vec \xi=\vec{\zeta}+\vec{\tilde \zeta}$ solves the system \eqref{id:systeme} on $(0,T]$, and belongs to $Z_{T,2}$.

\end{proof}

\subsection{Instantaneous analytic regularization in the transverse variable}

Thanks to Lemma \eqref{lwp:lemGevreytime}, for any $0<t_0<T$, the solution is Gevrey-$2$ in time on $[t_0,T]$, with a radius of analyticity that is now \emph{bounded from below uniformly on $[t_0,T]$}. In this subsection $\tau$ is then now independent of time. Analyticity in the $y$ variable is given by the following Lemma.

\begin{lemma} \label{lwp:lem:elliptic}

Let $\tau>0$ and assume that $\vec \xi$ is a smooth in space and time on $[t_0,T]\times [0,\infty)$ solution to \eqref{id:systeme} such that for all $k,i\in \mathbb N$ and $m=0,1$:
\be \label{bd:hpanalytic}
\| \pa_t^k \pa_y^m \xi_i \|_{L^\infty([t_0,T],L^\infty_\omega )}\leq C \tau^{-i-k}(2i+2k+m)!\langle i+k+m \rangle^{-2}.
\ee
Then, there exists $\tau'>0$ such that for all $m,i\in \mathbb N$:
$$
\| \pa_y^m \xi_i \|_{L^\infty([t_0,T], L^\infty_\omega)}\leq C \tau^{'-i-m}(2i+m)!\langle i+k+m \rangle^{-2}.
$$

\end{lemma}

\begin{proof}

To shorten notation, we shall write $L^\infty$ for $L^\infty([t_0,T],L^\infty_\omega [0,\infty))$. In the proof $\tilde C$ denotes a constant that is independent of the other parameters, whose value may change from one line to another. We prove the following bound by induction on $m\in \mathbb N$:
\be \label{bd:analyticintermediate}
\| \pa_t^{k} \pa_y^n \xi_i \|_{L^\infty}\leq C\Lambda_{i,k,n} \qquad \mbox{for all } 0\leq n\leq m \mbox{ and } i,k\in \mathbb N.
\ee
where we introduce the notation:
$$
\Lambda_{i,k,m}=\tau^{-k-i} \tau^{'-m}(2i+2k+m)!\langle i+k+m\rangle^{-2}.
$$
Note that, if $0<\tau'\leq 1$, then \eqref{bd:analyticintermediate} is true for $m=1$ by the hypothesis \eqref{bd:hpanalytic} of the Lemma. We now assume it is true for $m-1\geq 1$, and aim at proving it is true for $m$.\\

\noindent \textbf{Step 1}. \emph{If $m=2p$ is even.} By induction on $p$, using \eqref{id:systeme} we get the following formula:
$$
\pa_y^{m}\xi_i =\pa_t^{p} \xi_i -\sum_{q=0}^{p-1} \pa_t^q \pa_y^{2(p-1-q)}(H_i+J_i),
$$
and hence, for all $k\in \mathbb N$:
\be \label{lwp:id:decompositionttoy}
\pa_t^k \pa_y^{m}\xi_i =\pa_t^{p+k} \xi_i -\sum_{q=0}^{p-1} \pa_t^{k+q} \pa_y^{2(p-1-q)}(H_i+J_i).
\ee
We bound the first term in the right hand side of \eqref{lwp:id:decompositionttoy} using \eqref{bd:hpanalytic}:
\be \label{lwp:bd:firzttern} 
\| \pa_t^{p+k} \xi_i \|_{L^\infty} \leq C \tau^{-i-p-k}(2i+2p+2k)!\langle i+k+p\rangle^{-2} \leq \tilde C C (\frac{\tau^{'}}{\sqrt{\tau}})^m \Lambda_{i,k,m} .
\ee
For the second and third terms in \eqref{lwp:id:decompositionttoy}, using Leibniz formula for differentiation:
\begin{align}
\label{lwp:id:Fttoy}\pa_t^{k+q} \pa_y^{2(p-1-q)}H_i &= -\sum_{l=0}^{k+q} \sum_{n=0}^{2(p-1-q)} \sum_{j=0}^i \binom{k+q}{l} \binom{2(p-1-q)}{n} \binom{2i+1}{2j+1} \ \pa_t^l\pa_y^n \xi_j \ \pa_{t}^{k+q-l} \pa_y^{2(p-1-q)-n} \xi_{i-j} ,\\
\nonumber  \pa_t^{k+q} \pa_y^{2(p-1-q)}J_i &= \underbrace{\sum_{l=0}^{k+q} \sum_{n=0}^{2(p-1-q)-1} \sum_{j=0}^i \binom{k+q}{l} \binom{2(p-1-q)}{n+1} \binom{2i+1}{2j} \ \pa_t^l\pa_y^{n} \xi_j \ \pa_{t}^{k+q-l} \pa_y^{2(p-1-q)-n} \xi_{i-j}}_{=I} \\
\label{lwp:id:Gttoy} &+\underbrace{ \sum_{l=0}^{k+q} \sum_{j=0}^i \binom{k+q}{l} \binom{2i+1}{2j} \ \pa_y^{-1}(\pa_t^l \xi_j) \ \pa_{t}^{k+q-l} \pa_y^{2(p-1-q)+1} \xi_{i-j}}_{=II}.
\end{align}
We next bound using \eqref{bd:analyticintermediate}, introducing $r=2j+2l+n$, and using the inequality $\langle k+q-l+2(p-1-q)-n+i-j\rangle^{-2}\lesssim \langle 2i+2k+m-r\rangle^{-2}$ given the range of the parameters $l$, $n$ and $j$ in the sum:
\begin{align*}
& \| \pa_t^l\pa_y^n \xi_j \pa_{t}^{k+q-l} \pa_y^{2(p-1-q)-n} \xi_{i-j}  \|_{L^\infty} \ \leq  \ C^2 \Lambda_{j,l,n} \Lambda_{i-j,k+q-l,2(p-1-q)-n} \\
 &\qquad  \leq \ C^2 \tau^{-l-j} \tau^{'-n}(2j+2l+n)! \langle j+l+n\rangle^{-2} \times \tau^{-(k+q-l)-(i-j)} \tau^{'-(2(p-1-q)-n)} \\
 &\qquad \qquad  (2(i-j)+2(k+q-l)+2(p-1-q)-n)! \langle 2i+2k+m-r\rangle^{-2} \\
&\qquad= \ C^2 \tau^{'2}(\frac{\tau^{'2}}{\tau})^q \tau^{-i-k}\tau^{'-m} r!  (2i+2k+m-r-2)!\langle r \rangle^{-2} \langle 2i+2k+m-r\rangle^{-2}.
\end{align*}
Using $\binom{k+q}{l} \leq \binom{2k+2q}{2l}$ and \eqref{bd:technicalcombinatorial3} with $(A_1,A_2,A_3,r_2)=(2k+2q,2(p-1-q),2i+1,r+1)$:
\begin{align*}
& \sum_{l=0}^{k+q} \sum_{n=0}^{2(p-1-q)} \sum_{j=0}^i \left(\binom{k+q}{l} \binom{2(p-1-q)}{n} \binom{2i+1}{2j+1} +\binom{k+q}{l} \binom{2(p-1-q)}{n+1} \binom{2i+1}{2j}\right) \\
& \qquad \qquad \qquad \qquad \qquad  r!  (2i+2k+m-r-2)! \langle r \rangle^{-2} \langle 2i+2k+m-r\rangle^{-2} \\
& \leq  \sum_{r=0}^{2i+2k+m} \binom{2i+2k+m-1}{r+1}  r!  (2i+2k+m-r-2)! \langle r \rangle^{-2} \langle 2i+2k+m-r\rangle^{-2} \\
& = \sum_{r=0}^{2i+2k+m}  (2i+2k+m-1)!(r+1)^{-1}\langle r \rangle^{-2} \langle 2i+2k+m-r\rangle^{-2} \\
&\lesssim  (2i+2k+m-1)! \langle i+k+m\rangle^{-2}
\end{align*}
where we used \eqref{bd:technicalcombinatorial} with $K=2i+2k+m$. Combining the three inequalities above, \eqref{lwp:id:Fttoy} and \eqref{lwp:id:Gttoy} one finds:
\be \label{lwp:bd:interFandG1} 
\| \pa_t^{k+q} \pa_y^{2(p-1-q)}H_i +I  \|_{L^\infty} \leq \tilde C C^2  \tau^{'2}(\frac{\tau^{'2}}{\tau})^q  \langle i+k+m\rangle^{-1} \Lambda_{i,k,m}.
\ee

For the second term in $J_i$, using \eqref{bd:analyticintermediate}, $\int_0^\infty \omega(y)dy<\infty$, letting $r=2j+2l$, and using that $\langle i-j+ k+q-l +2(p-1-q)+1 \rangle^{-2}\lesssim \langle 2i+2k+m-r \rangle^{-2} $ for the range of parameters $l$ and $j$ in the sum:
\begin{align}
\nonumber &  \|   \pa_y^{-1}(\pa_t^l \xi_j) \ \pa_{t}^{k+q-l} \pa_y^{2(p-1-q)+1} \xi_{i-j} \|_{L^\infty_\omega} \ \leq \ \tilde C\| \pa_t^l \xi_j \|_{L^\infty_\omega} \| \pa_{t}^{k+q-l} \pa_y^{2(p-1-q)+1} \xi_{i-j}\|_{L^\infty_\omega} \\
\nonumber &\qquad \leq \tilde C C \tau^{-l-j} (2j+2l )! \langle j+l \rangle^{-2} \times C \tau^{-(k+q-l)-(i-j)}\\
\nonumber &\qquad \qquad  \tau^{'-(2(p-1-q)+1)}(2(i-j)+2(k+q-l)+2(p-1-q)+1)! \langle 2i+2k+m-r\rangle^{-2} \\
\label{lwp:bd:interGII} &\qquad = \tilde C C^2 \tau' (\frac{\tau^{'2}}{\tau})^q  \tau^{-i-k} \tau^{'-m}  r!(2i+2k+m-r-1)! \langle r\rangle^{-2}\langle 2 i+2k+m-r\rangle^{-2}.
\end{align}
Using $\binom{k+q}{l}\leq \binom{2k+2q}{2l}$ and \eqref{bd:technicalcombinatorial2} with $(A_1,A_2,r_1)=(2k+2q,2i+1,r)$ we have:
\begin{align*}
& \sum_{l=0}^{k+q} \sum_{j=0}^i \binom{k+q}{l} \binom{2i+1}{2j}  r!(2i+2k+m-r-1)! \langle r\rangle^{-2}\langle i+k+m-r\rangle^{-2} \\
\leq& \sum_{r=0}^{2k+2q+2i} \binom{2k+2q+2i+1}{r}    r!(2i+2k+m-r-1)! \langle r\rangle^{-2}\langle i+k+m-r\rangle^{-2} \\
\leq& (2k+2i+m-1)!  \sum_{r=0}^{2k+2q+2i} \frac{(2k+2q+2i+1)!}{(2k+2i+m-1)!} \frac{(2i+2k+m-1-r)!}{(2k+2q+2i+1-r)!} \langle r\rangle^{-2}\langle 2 i+2k+m-r\rangle^{-2} \\
\end{align*}
We estimate:
\begin{align*}
 \frac{(2k+2q+2i+1)!}{(2k+2i+m-1)!} \frac{(2i+2k+m-1-r)!}{(2k+2q+2i+1-r)!}  & = \frac{2i+2k+2p-1-r}{2i+2k+2p-1} \ ... \ \frac{2i+2k+2q+2-r}{2i+2k+2q+2}\leq 1.
\end{align*}
Combining the two above inequalities and \eqref{bd:technicalcombinatorial} with $K=2k+2q+2i$ we obtain:
$$
 \sum_{l=0}^{k+q} \sum_{j=0}^i \binom{k+q}{l} \binom{2i+1}{2j}  r!(2i+2k+m-1-r)! \langle r\rangle^{-2}\langle i+k+m-r\rangle^{-2}  \leq \tilde C(2k+2i+m)!\langle i+k+m \rangle^{-3}.
$$
Combining \eqref{lwp:bd:interGII} and the above inequality, we get for the second term in \eqref{lwp:id:Gttoy}:
\be \label{lwp:bd:interG2} 
\| II \|_{L^\infty}\leq \tilde C C^2  \tau' (\frac{\tau^{'2}}{\tau})^q \langle i+k+m \rangle^{-1} \Lambda_{i,k,m}
\ee
Combining \eqref{lwp:bd:interFandG1} and \eqref{lwp:bd:interG2} we get for $\tau'$ small enough:
$$
\| \pa_t^{k+q} \pa_y^{2(p-1-q)}(H_i+J_i)  \|_{L^\infty} \leq  \tilde C C^2  \tau' (\frac{\tau^{'2}}{\tau})^q \langle i+k+m \rangle^{-1} \Lambda_{i,k,m}.
$$
Therefore, for $\tau^{'2}\leq \frac{\tau}{2}$, we have $\sum_{q=0}^{p-1} (\frac{\tau^{'2}}{\tau})^q\leq 2$ and so from the above identity:
$$
\| \sum_{q=0}^{p-1} \pa_t^{k+q} \pa_y^{2(p-1-q)}(H_i+J_i) \|_{L^\infty}  \leq   \tilde C C^2  \tau' \langle i+k+m \rangle^{-1} \Lambda_{i,k,m}.
$$
Injecting the above inequality and \eqref{lwp:bd:firzttern} in \eqref{lwp:id:decompositionttoy} gives:
$$
\| \pa_t^{k}\pa_y^{m}\xi_i \|_{L^\infty}  \leq  C \Lambda_{i,k,m} \left(  C\tilde C\tau' +\tilde C (\frac{\tau^{'}}{\sqrt{\tau}})^m \right)  \leq  C \Lambda_{i,k,m} 
$$
for $\tau'$ small enough, since $m\geq 2$. Therefore, \eqref{bd:analyticintermediate} is true for $m$.\\

\noindent \textbf{Step 2} \emph{If $m=2p+1$ is even.} By the formula of Step 1 we obtain
$$
\pa_t^k \pa_y^{m}\xi_i =\pa_t^{p+k} \pa_y \xi_i -\sum_{q=0}^{p-1} \pa_t^{k+q} \pa_y^{2(p-1-q)+1}(H_i+J_i).
$$
and the same computations show the desired result. We omit the details.

\end{proof}

%%%%%%%%%%%%%%%%%%%%%%%%%%%%%%%%%%%%%%%%%%%%%%%%%%%%%%%%%%%%%%%%%
%%%%%%%%%%%%%%%%%%%%%%%%%%%%%%%%%%%%%%%%%%%%%%%%%%%%%%%%%%%%%%%%%
%%%%%%%%%%%%%%%%%%%%%%%%%%%%%%%%%%%%%%%%%%%%%%%%%%%%%%%%%%%%%%%%%
%%%%%%%%%%%%%%%%%%%%%%%%%%%%%%%%%%%%%%%%%%%%%%%%%%%%%%%%%%%%%%%%%

\begin{appendix}

\section{Functional analysis}

\begin{lemma}

There exists $C>0$ such that for all $-\infty\leq Y_0<0$ %\in \mathbb R$ 
and $\e :(Y_0,+\infty)\rightarrow \mathbb R$ with $\e\in H^1_\rho$, the following inequality holds true:%{\text{loc}}((Y_0,+\infty))$. Then:
\be \lab{bd:poincare}
\int_{Y_0}^{+\infty} Y^2\e^2e^{-\frac{3Y^2}{4}}dY \leq C \| \e \|_{H^1_\rho}^2
\ee

\end{lemma}

\begin{proof}

Let first $Y_0=-\infty$. For $\e \in C^{\infty}_c(\mathbb R)$, by integrating by parts:
$$
\frac43\int_{\mathbb R} \e \pa_Y \e Y e^{-\frac{3Y^2}{4}}dY+\frac23\int_{\mathbb R} \e^2  e^{-\frac{3Y^2}{4}}dY=\int_{\mathbb R} \e^2 Y^2  e^{-\frac{3Y^2}{4}}dY.
$$
From Cauchy-Schwarz and Young inequalities, $4|\int \e \pa_Y \e Y e^{-\frac{3Y^2}{4}}|\leq 1/2 \int Y^2\e^2 e^{-\frac{3Y^2}{4}}+8 \int |\pa_Y\e|^2 e^{-\frac{3Y^2}{4}}$ and we infer from the above identity that:
$$
\int_{\mathbb R} \e^2 Y^2  e^{-\frac{3Y^2}{4}}dY\leq 4\int_{\mathbb R}\e^2  e^{-\frac{3Y^2}{4}}dY+\frac{16}5 \int_{\mathbb R} |\pa_Y \e |^2 e^{-\frac{3Y^2}{4}}dY.
$$
By density, this proves \fref{bd:poincare} for all $\e \in H^1_\rho$ in case $Y_0=-\infty$. For $-\infty<Y_0< 0$, define the even extension: $\tilde \e(Y)=\e (Y)$ for $Y\geq Y_0$ and $\tilde \e(Y)=\e (2Y_0-Y)$ for $Y<Y_0$, and $\tilde Y_0=-\infty$. Then $\| \e \|_{H^1_{\rho,Y_0}}^2\leq \| \tilde \e \|_{H^1_{\rho,\tilde Y_0}}^2\leq 2\|  \e \|_{H^1_{\rho,Y_0}}^2$, where the second inequality holds since $\rho(Y)\leq \rho(2Y_0-Y)$ for $Y\leq Y_0$. Applying \fref{bd:poincare} for $\tilde \e$ with $\tilde Y_0=\infty$ then implies \fref{bd:poincare} for $\e$ with $Y_0$.

\end{proof}

%%%%%%%%%%%%%%%%%%%%%%%%%%%%%%%%%%%%%%%%%%%%%%%%%%%%%%%%%%%%%%%%%
%%%%%%%%%%%%%%%%%%%%%%%%%%%%%%%%%%%%%%%%%%%%%%%%%%%%%%%%%%%%%%%%%
%%%%%%%%%%%%%%%%%%%%%%%%%%%%%%%%%%%%%%%%%%%%%%%%%%%%%%%%%%%%%%%%%
%%%%%%%%%%%%%%%%%%%%%%%%%%%%%%%%%%%%%%%%%%%%%%%%%%%%%%%%%%%%%%%%%

\section{Geometrical decomposition} \lab{ap:decomposition}

\begin{proof}[Proof of Lemma \ref{lem:decomposition}]

The proof relies on a classical use of the implicit function theorem, preceded by a renormalization procedure to obtain a result which is uniformly valid for all $\lambda$ large enough. Define the mapping
$$
\Phi : (\e,\lb,\mu,\tilde Y_0) \mapsto \lb_0^4 \left(\la \tilde \e,h_0\ra_\rho,\la \tilde \e,h_1\ra_\rho,\la \tilde \e,h_2\ra_\rho \right),
$$
where $\langle u,v\rangle_\rho=\int_{Y_0-\tilde Y_0}^{+\infty}uv\rho$ and, for $Y\geq Y_0-\tilde Y_0$:
$$
\tilde \e (Y)=G_1\left(\frac{Y+\tilde Y_0}{\lb_0^2}\right)-(1+\lb_0^{-4}\lb)^2G_1\left(\frac{Y}{\lb_0^2(1+\lb_0^{-4}\lb)^2\mu}\right)+\frac{\e(Y+\tilde Y_0)}{\lb_0^4}.
$$
$\Phi$ is a $C^2$ mapping on $L^2_\rho \times (-\lb_0^4,+\infty)\times (0,+\infty)\times \mathbb R$. Moreover, one computes that its differential at $(0,0,1,0)$ is, where $\langle u,v\rangle=\int_{Y\geq Y_0}uv\rho$:
\bee
&&J\Phi (0,0,1,0)+O(e^{-\lb_0^2}) \\
&=& \begin{pmatrix} \la \cdot,h_0 \ra & \la -2G_1\left(\frac{Y}{\lb_0^2}\right)+2\frac{Y}{\lb_0^2}\pa_Z G_1 \left(\frac{Y}{\lb_0^2}\right),h_0\ra & \lb_0^2 \la Y \pa_Z G_1 \left(\frac{Y}{\lb_0^2}\right),h_0\ra & \lb_0^2 \la \pa_Z G_1 \left(\frac{Y}{\lb_0^2}\right),h_0\ra \\
\la \cdot,h_1 \ra & \la -2G_1\left(\frac{Y}{\lb_0^2}\right)+2\frac{Y}{\lb_0^2}\pa_Z G_1 \left(\frac{Y}{\lb_0^2}\right),h_1\ra  & \lb_0^2 \la Y\pa_Z G_1 \left(\frac{Y}{\lb_0^2}\right),h_1\ra & \lb_0^2 \la \pa_Z G_1 \left(\frac{Y}{\lb_0^2}\right),h_1\ra \\
\la \cdot,h_2 \ra & \la -2G_1\left(\frac{Y}{\lb_0^2}\right)+2\frac{Y}{\lb_0^2}\pa_Z G_1 \left(\frac{Y}{\lb_0^2}\right),h_2\ra & \lb_0^2 \la Y \pa_Z G_1 \left(\frac{Y}{\lb_0^2}\right),h_2\ra & \lb_0^2 \la \pa_Z G_1 \left(\frac{Y}{\lb_0^2}\right),h_2\ra \end{pmatrix}
\eee
where the $O(e^{-\lb_0^2})$ comes from the boundary terms. Using the Taylor expansion of $G_1$ one has:
$$
-2G_1\left(\frac{Y}{\lb_0^2}\right)+2\frac{Y}{\lb_0^2}\pa_Z G_1 \left(\frac{Y}{\lb_0^2}\right)=-2-\frac{Y^2}{2\lb_0^4}+O\left(\frac{Y^4}{\lb_0^8}\right)=-\frac{1}{6\lb_0^4}h_2(Y)-\left(2+\frac{1}{3\lb_0^4}\right)h_0(Y)+O\left(\frac{Y^4}{\lb_0^8}\right),
$$
$$
\lb_0^2Y\pa_Z G_1 \left(\frac{Y}{\lb_0^2}\right)=-\frac{Y^2}{2}+O\left(\frac{Y^4}{\lb_0^4}\right)=-\frac{1}{6}h_2(Y)-\frac{1}{3}h_0(Y)+O\left(\frac{Y^4}{\lb_0^4}\right),
$$
and
$$
\lb_0^2 \pa_Z G_1  \left(\frac{Y}{\lb_0^2}\right) =-\frac Y2+O\left(\frac{|Y|^3}{\lb_0^4}\right)=-\frac{1}{2\sqrt 3}h_1(Y)+O\left(\frac{|Y|^3}{\lb_0^4}\right).
$$
Therefore:
$$
J\Phi (0,0,1,0)= \begin{pmatrix} \la \cdot,h_0 \ra & -2\| h_0\|_{L^2_\rho}^2+O(\lb_0^{-4}) &-\frac{1}{3} \| h_0\|_{L^2_\rho}^2+O(\lb_0^{-4}) & O(\lb_0^{-4}) \\
\la \cdot,h_1 \ra & O(\lb_0^{-4}) &O(\lb_0^{-4}) & -\frac{1}{2\sqrt 3}\| h_1\|_{L^2_\rho}^2+O(\lb_0^{-4}) \\
\la \cdot,h_2 \ra & O(\lb_0^{-4}) & -\frac{1}{6}\| h_2\|_{L^2_\rho}^2+O(\lb_0^{-4}) & O(\lb_0^{-8}) \end{pmatrix}.
$$
This implies that the restriction of the differential to $\{0\}\times \mathbb R^3$ is invertible for $\lambda_0$ large enough, with a uniform size. Moreover, one can also check similarly that the second differential of $\Phi$ is bounded near $(0,0,1,0)$, and this uniformly for large $\lambda$. Therefore the implicit function theorem applies uniformly for all $\lambda_0\geq \lambda^*$ large enough and $Y_0\leq -\lb_0^2$. There exists $\delta,K>0$ such that for each $\e \in L^2_\rho$ with $\| \e \|_{L^2_\rho}\leq \delta$, there exist unique parameters $(\lambda,\mu,\tilde Y_0)$ with $|\lambda|+|\mu-1|+|\tilde Y_0|\leq K$ such that $\Phi(\e,\lambda,\mu,\tilde Y_0)=0$. Moreover, they define $C^1$ functions with respect to the $L^2_\rho$ topology.\\

\noindent Let $\lambda_0\geq \lambda^*$ and $\| \varepsilon \|_{L^2_\rho}\leq \delta \lambda_0^{-4}$. The above discussion yields the existence, uniqueness, and differentiability of $(\lambda,\mu,Y_0)$ such that $\Phi (\lambda_0^4 \e,\lambda,\mu,Y_0)=0$. Let $(\tilde \lambda,\tilde \mu,\tilde Y_0)=(1+\lambda_0^{-4}\lambda,\mu,Y_0)$. Then they produce indeed
$$
G_1\left(\frac{Y}{\lambda_0^2}\right)+\varepsilon (Y)=\tilde \lambda^2G_1\left(\frac{Y-\tilde Y_0}{\tilde \lambda^2\tilde \mu}\right) +\tilde \varepsilon (Y-\tilde Y_0)\ \ \text{with} \ \ \tilde \varepsilon \perp h_0,h_1,h_2 \ \text{in} \ L^2_\rho 
$$
and one has $|\tilde \lambda-1 |\leq K\lb_0^{-4}$ and $|\mu-1|+|Y_0|\leq K$. The uniqueness when requiring these bounds follows similarly, and implies the smoothness from the above discussion. This ends the proof.

\end{proof}

%%%%%%%%%%%%%%%%%%%%%%%%%%%%%%%%%%%%%%%%%%%%%%%%%%%%%%%%%%%%%%%%%
%%%%%%%%%%%%%%%%%%%%%%%%%%%%%%%%%%%%%%%%%%%%%%%%%%%%%%%%%%%%%%%%%
%%%%%%%%%%%%%%%%%%%%%%%%%%%%%%%%%%%%%%%%%%%%%%%%%%%%%%%%%%%%%%%%%
%%%%%%%%%%%%%%%%%%%%%%%%%%%%%%%%%%%%%%%%%%%%%%%%%%%%%%%%%%%%%%%%%

\section{Estimates for the heat kernel} \label{ap:sec:heat}

\begin{lemma} \label{lwp:lem:heatkernel}

Let $K_t$ be given by \fref{eq:defheatkernel}. First, for any $T>0$, there exists $C(T)>0$ such that for any $t\in [0,T]$ and $y\in \mathbb R$:
\be \label{bd:heatanalytic21}
(K_t * \omega)(y) \leq C \omega(y).
\ee
Second, there exist $C,\kappa,\tau_0>0$ such that for all $k\in \mathbb N$, for any $t>0$ and $y\in \mathbb R$:
\be \label{bd:heatanalytic2}
| \pa_t^k K_t (y) |\leq C\tau_0^{-k} t^{-k}  k! K_{\kappa t}(y), \qquad | \pa_t^k \pa_y K_t (y) |\leq C\tau_0^{-k} t^{-k-\frac 12}  k! K_{\kappa t}(y).
\ee

\end{lemma}

\begin{proof}[Proof of Lemma \ref{lwp:lem:heatkernel}]

From a direct computation, $ \int_{z\in \mathbb R} \omega(y-z) K_{t} (z)dz\leq C \omega(y)$ for all $y\in \mathbb R$ and $t\in [0,T]$ for some universal constant $C(T)>0$, we omit the details. This shows \eqref{bd:heatanalytic21}.

Below we will denote by $C>0$ some universal constant whose value may change from one line to another.

Let $z\in \mathbb C$ denote a complex variable, and $(\rho,\theta)$ be its $(\mbox{radius},\mbox{angle})$ variables. Let $\phi(z)=e^{-\frac{1}{z}}$. Then, $\phi$ is an analytic function on $\mathbb C\backslash \{ 0\}$. Consider for any $t>0$ the circle  $\mathcal C_t:=\{ z\in \mathbb C, \ |z-t|=\frac{t}{10}\}$. Then, for all $z\in \mathcal C_t$, we have $\frac{9}{10}t \leq \rho \leq \frac{11}{10}t $ and $|\theta| \leq \theta_0$ for some $\theta_0<\frac{\pi}{2}$. Hence, denoting $\tilde z=\frac 1z$, we have for $z\in \mathcal C_t$ that:
$$
\frac{10}{11 t}\leq \tilde \rho \leq\frac{10}{9t} \qquad \mbox{and}\qquad |\tilde \theta|\leq \theta_0.
$$
Therefore, there exists a constant $c_0>0$ independent of $t$ such that $\frac{c_0}{ t} \leq \Re (\frac 1z)\leq \frac{1}{c_0 t} $ for all $z\in \mathcal C_t$. Consequently, for all $z\in \mathcal C_t$, there holds:
$$
|\phi(z)|\leq e^{-\frac{c_0}{t}}.
$$
Applying Cauchy contour formula to the holomorphic function $\phi$ with the contour $\mathcal C_t$, and differentiating, one finds that for some constant $C>0$, for all $j\in \mathbb N$:
$$
|\pa_t^j \phi (t) |\leq C t^{-j} 10^{j} j! e^{-\frac{c_0}{t}}.
$$
Let now $y\in \mathbb R\backslash \{0\}$ and $\phi_y(t)=e^{-\frac{y^2}{4t}}=\phi(\frac{4 t}{y^2})$. Then:
$$
|\pa_t^j \phi_y(t)|=\frac{4^j}{y^{2j}} |\pa_t^j \phi (\frac{4 t}{y^2})|\leq \frac{4^j}{y^{2j}} C (\frac{4 t}{y^2})^{-j} 10^{j} j! e^{-\frac{c_0y^2}{4t}} = C t^{-j}10^{j} j! e^{-\frac{c_0y^2}{4t}}.
$$
Combining the above bound with the bounds $|\pa_t^j (t\mapsto \frac{1}{\sqrt t})|\leq j! t^{-j-\frac 12}$ and $|\pa_t^j (t\mapsto \frac{1}{t^{\frac 32}})|\leq (j+1)! t^{-j-\frac 32}$, using Leibniz rule for differentiation, one finds that for all $j\in \mathbb N$, $t>0$ and $y\in \mathbb R$:
$$
|\pa_t^j K_t(y)| \leq C t^{-j-\frac 12} 11^{j} j! e^{-\frac{c_0y^2}{4t}} \quad \mbox{and} \quad |\pa_t^j (\frac{1}{t} K_t(y))| \leq C t^{-j-\frac 32} 11^{j} j! e^{-\frac{c_0y^2}{4t}}
$$
The first bound above is precisely the first bound in \eqref{bd:heatanalytic2}, while the second one above gives the second one in \eqref{bd:heatanalytic2}, using that $\pa_y K_t=-\frac{y}{2t}K_t$, and that for any $0<c_0'<c_0$ there exists $C>0$ with $\frac{|y|}{\sqrt{t}}e^{-\frac{c_0y^2}{4t}}\leq Ce^{-\frac{c_0'y^2}{4t}}$.

\end{proof}

%%%%%%%%%%%%%%%%%%%%%%%%%%%%%%%%%%%%%%%%%%%%%%%%%%%%%%%%%%%%%%%%%
%%%%%%%%%%%%%%%%%%%%%%%%%%%%%%%%%%%%%%%%%%%%%%%%%%%%%%%%%%%%%%%%%
%%%%%%%%%%%%%%%%%%%%%%%%%%%%%%%%%%%%%%%%%%%%%%%%%%%%%%%%%%%%%%%%%
%%%%%%%%%%%%%%%%%%%%%%%%%%%%%%%%%%%%%%%%%%%%%%%%%%%%%%%%%%%%%%%%%

\section{Combinatorics estimates}

\begin{lemma}[Combinatorics estimates]

There exists $C>0$ such that for any $A\in \mathbb N$:
\be \label{bd:technicalcombinatorial}
\sum_{a=0}^{A} \langle a\rangle^{-2}\langle A-a \rangle^{-2}\leq  C \langle A \rangle^{-2}
\ee
For any $A_1,A_2,A_3\in \mathbb N$, $r_1\leq A_1+A_2$ and $r_2\leq A_1+A_2+A_3$:
\begin{align} 
\label{bd:technicalcombinatorial2} & \sum_{a_1\leq A_1, \ a_2\leq A_2, \ a_1+a_2=r_1} \binom{A_1}{a_1}\binom{A_2}{a_2} = \binom{A_1+A_2}{r_1},\\
\label{bd:technicalcombinatorial3} & \sum_{a_1\leq A_1, \ a_2\leq A_2, \ a_3\leq A_3, \ a_1+a_2+a_3=r_2} \binom{A_1}{a_1}\binom{A_2}{a_2}\binom{A_3}{a_3} = \binom{A_1+A_2+A_3}{r_2}.
\end{align}

\end{lemma}

\begin{proof}

\underline{Proof of \eqref{bd:technicalcombinatorial}}. We decompose:
$$
\sum_{k=0}^{K} \langle k\rangle^{-2}\langle K-k \rangle^{-2}=\sum_{k=0}^{\lfloor \frac K2\rfloor} \langle k\rangle^{-2}\langle K-k \rangle^{-2}+\sum_{k=0}^{\lceil \frac K2\rceil} \langle k\rangle^{-2}\langle K-k \rangle^{-2}\lesssim \langle K\rangle^{-2} \sum_{k\geq 0}\langle k\rangle^{-2}\lesssim \langle K\rangle^{-2}.
$$
\underline{Proof of \eqref{bd:technicalcombinatorial2} and \eqref{bd:technicalcombinatorial3}}. These equalities are obtained from a standard counting argument.

\end{proof}

%%%%%%%%%%%%%%%%%%%%%%%%%%%%%%%%%%%%%%%%%%%%%%%%%%%%%%%%%%%%%%%%%
%%%%%%%%%%%%%%%%%%%%%%%%%%%%%%%%%%%%%%%%%%%%%%%%%%%%%%%%%%%%%%%%%
%%%%%%%%%%%%%%%%%%%%%%%%%%%%%%%%%%%%%%%%%%%%%%%%%%%%%%%%%%%%%%%%%
%%%%%%%%%%%%%%%%%%%%%%%%%%%%%%%%%%%%%%%%%%%%%%%%%%%%%%%%%%%%%%%%%

\end{appendix}

\end{document}